\documentclass[12pt]{article}
{\begin{Sbox}\begin{minipage}}%
{\end{minipage}\end{Sbox}\fbox{\TheSbox}}%

\usepackage[psamsfonts]{amssymb}
\usepackage{amsmath, amsthm, anysize, enumerate, color, url}
\usepackage{graphicx}
\usepackage{epstopdf}
\usepackage{epsfig}
\usepackage{subfigure}
\usepackage[ruled,linesnumbered]{algorithm2e}
\usepackage[sectionbib]{natbib}

\usepackage{threeparttable}

\usepackage[colorlinks, breaklinks,
                   linkcolor=red,
                  citecolor=blue
                  ]{hyperref}

\usepackage{breakurl}

\newtheorem{theorem}{Theorem} %%%[section]
\newtheorem{corollary}{Corollary}

\newtheorem{lemma}{Lemma} %%%[theorem]
\theoremstyle{definition}
\newtheorem{remark}{Remark}  %%%[theorem]

\newcommand{\hh}{\mathbf{h}}
\newcommand{\dd}{\mathbf{d}}
\newcommand{\E}{\mathbb{E}}

\newcommand{\R}{\mathbb{R}}

\newcommand{\PP}{\mathbb{P}}

\newcommand{\bs}{\boldsymbol}

\usepackage{setspace}

\begin{document}

  \title{\bf Likelihood ratio tests in random graph models with increasing dimensions\thanks{This paper supersedes arxiv article arXiv:2211.10055
titled ``Wilks' theorems in the $\beta$-model" by T. Yan, Y. Zhang, J. Xu, Y. Yang and J. Zhu.
In this arxiv preprint, we have changed the title, made substantial changes in the introduction and simulation sections and improved some theoretical results and the writing.
In addition, we compared likelihood ratio tests with Wald-type tests for fixed dimensional parameter hypothesis testing problems.}}

  \author{Ting Yan\footnote{Department of Statistics, Central China Normal University, Wuhan, 430079, China. \texttt{Email:} tingyanty@mail.ccnu.edu.cn.
  Yan is partially supported by the National Natural Science Foundation of China (No. 12171188, 12322114).
    },~~
    Yuanzhang Li\footnote{Department of Statistics, George Washington University,
Washington, D.C. 20052, USA. \texttt{Email:} yuanzhang.li@yahoo.com%%, Wuhan, 430079, China.
},~~
Jinfeng Xu\footnote{Department of Biostatistics, City University of Hong Kong, Hong Kong. \texttt{Email:} jinfengxu@gmail.com},~~
Yaning Yang\footnote{Department of Statistics and Finance,
University of Science and Technology of China, Anhui, 230026, China. \texttt{Email:}  ynyang@ustc.edu.cn},~~
Ji Zhu\footnote{Department of Statistics, University of Michigan, Ann Arbor, Michigan, USA. \texttt{Email:} jizhu@umich.edu}
\medskip
\\
Central China Normal University$^\dag$,\\
George Washington University$^\ddag$,\\
City University of Hong Kong$^\S$, \\
University of Science and Technology of China$^\P$, \\
University of Michigan$^\|$
  }
\date{}
\maketitle

\bigskip
\begin{abstract}
We explore the Wilks phenomena in
two random graph models: the $\beta$-model and the Bradley--Terry model.
For two increasing dimensional null hypotheses, including
a specified null $H_0: \beta_i=\beta_i^0$ for $i=1,\ldots, r$ and a homogenous null $H_0: \beta_1=\cdots=\beta_r$,
we reveal high dimensional Wilks' phenomena  that the normalized log-likelihood ratio statistic,
$[2\{\ell(\widehat{\bs{\beta}}) - \ell(\widehat{\bs{\beta}}^0)\} -r]/(2r)^{1/2}$,
converges in distribution to the standard normal distribution as $r$ goes to infinity.
Here, $\ell( \bs{\beta})$ is the log-likelihood function on the model parameter $\bs{\beta}=(\beta_1, \ldots, \beta_n)^\top$,
$\widehat{\bs{\beta}}$ is its maximum likelihood estimator (MLE) under the full parameter space,
and $\widehat{\bs{\beta}}^0$ is the restricted MLE under the null parameter space.
For the homogenous null with a fixed $r$,
we establish Wilks-type theorems that $2\{\ell(\widehat{\bs{\beta}}) - \ell(\widehat{\bs{\beta}}^0)\}$ converges in distribution to a chi-square distribution with $r-1$ degrees of freedom, as the total number of parameters, $n$, goes to infinity.
When testing the fixed dimensional specified null, we find that
its asymptotic null distribution is a chi-square distribution in the $\beta$-model. However, unexpectedly, this is not true in the Bradley--Terry model.
By developing several novel
technical methods for asymptotic expansion, we
explore Wilks-type results in a principled manner; these principled methods should be applicable
to a class of random graph models beyond the $\beta$-model and the Bradley--Terry model. Simulation
studies and real network data applications further demonstrate the theoretical results.
\end{abstract}

\noindent%

%\tableofcontents

\renewcommand{\P}{\mathbb{P}}

\section{Introduction}

The $\beta$-model [\cite{Chatterjee:Diaconis:Sly:2011}] is an exponential family distribution on an undirected graph
with the degree sequence as the sufficient statistic. Specifically, the model assigns
an intrinsic degree parameter $\beta_i$ to
 each node $i$ and postulates that random edges $a_{ij}\in\{0,1\}$ for $1\le i< j \le n$, with $n$ denoting the number of nodes in the graph, occur independently with connection probabilities
\begin{equation}\label{model-beta}
\mathrm{logit}\{ \PP(a_{ij}=1) \} = \beta_i + \beta_j,
\end{equation}
where $\mathrm{logit}(x) = \log \{ x/(1-x) \}$ for $x\in (0,1)$.
The $\beta$-model can be viewed as the undirected version of an earlier $p_1$-model [\cite{Holland:Leinhardt:1981}].
Another closely related model is the Bradley--Terry model [\cite{bradley-terry1952}],
which is a fundamental model for ranking network data involving paired comparison [\cite{Han-chen2020}].
The Bradley--Terry model assigns a merit parameter, denoted by the same notation $\beta_i$,  to each item and  assumes that the probability of item $i$ beating
item $j$ depends only on  the relative difference $\beta_i-\beta_j$:
\begin{equation}\label{model-bt}
\mathrm{logit}\{ \PP( \mbox{~item $i$ beats item $j$~} ) \} = \beta_i - \beta_j,
\end{equation}

Random graph models such as the $\beta$-model and the Bradley--Terry model have a wide range of applications.
For example, the $\beta$-model has been used to
model degree heterogeneity in real-world networks [e.g., \cite{Park:Newman:2004,Blitzstein:Diaconis:2011,Chen:2020}],
or make statistical inference in a situation in which only the degree sequence is available owing to
privacy considerations [\cite{Elmer-2020}].
Conversely, the applications of the Bradley--Terry model include the rankings of classical sports teams
[\cite{Whelan-2020-Hockey}]
and scientific journals in citation networks [\cite{Varin-2016-jrsa}],
the ranking of web pages in hyperlink networks [\cite{selby2024pagerank}],
and the quality of product brands [\cite{radlinski2007active}],
among others.

In both models with increasing
dimensions,  theoretical properties of their maximum likelihood estimators (MLEs) such
as consistency and asymptotic normality have been
derived [e.g., \cite{Chatterjee:Diaconis:Sly:2011,simons-yao1999}].
However, asymptotic behaviors of likelihood ratio tests (LRTs) have not been explored yet.
In its absence, we cannot obtain approximate $p$-values of LRTs for testing some
interesting hypotheses, for example, testing
the existence of degree heterogeneity in the $\beta$-model or whether there are significant differences among a large set of items in the Bradley--Terry model. The questions to be addressed are as follows: Is the limiting distribution of the LRT a chi-square distribution when the number of parameters being tested is fixed?
Is it an approximate normal distribution for the scaled LRT when the number of hypothetic parameters increases?
Notably, \cite{Holland:Leinhardt:1981} made the conjecture of a chi-square approximation when testing a single parameter in the related
$p_1$ model. \cite{FienbergWasserman1981} suggested a normal approximation for the scaled LRT under the null that all degree parameters are equal to zero. However, these conjectures have not been resolved in the existing literature.

In classical parametric hypothesis testing problems where the parameter space is finite dimensional, the LRT has the appealing property
that its asymptotic null distribution is a chi-square distribution independent of nuisance parameters [\cite{wilks1938}],
which is referred to as the Wilks theorem. If this property occurs when the dimension of the parameter space increases or is infinite,
it is termed the Wilks phenomenon [\cite{fan2001}].
Here, we explore the Wilks phenomena for
increasing and fixed dimensional parameter testing problems in the $\beta$-model and the Bradley--Terry model.
Our findings are as follows:

\begin{itemize}
\item
For two increasing dimensional null hypotheses $H_0: \beta_i = \beta_i^0$ for $i=1, \ldots, r$ and $H_0: \beta_1=\cdots=\beta_r$,
we show that the normalized log-likelihood ratio statistic, $[2\{\ell(\widehat{\bs{\beta}}) - \ell(\widehat{\bs{\beta}}^0)\} -r]/(2r)^{1/2}$,
converges in distribution to the standard normal distribution as $r\to\infty$, where $\beta_i^0$ is a known number.
Here, $\ell( \bs{\beta})$ is the log-likelihood function on the model parameter $\bs{\beta}=(\beta_1, \ldots, \beta_n)^\top$, $\widehat{\bs{\beta}}$ is its MLE under the full parameter space $\Theta= \R^n$,
and $\widehat{\bs{\beta}}^0$ is the restricted MLE under the null parameter space. %% associated with the null hypothesis.
 In other words,
$2(\ell(\widehat{\bs{\beta}}) - \ell(\widehat{\bs{\beta}}^0))$ is approximately a chi-square distribution with a large degree $r$ of freedom.

\item
Under the homogenous null $H_0$: $\beta_1=\cdots=\beta_r$ with a fixed $r$, we show that $2\{\ell(\widehat{\bs{\beta}}) - \ell(\widehat{\bs{\beta}}^0)\}$ converges in distribution to a chi-square distribution with
$r-1$ degrees of freedom, as the number of nodes $n$ goes to infinity.
Thus, the high dimensional likelihood ratio statistics behave like classical ones as long as the difference between the dimension of the full space and that of the null space is fixed, although both dimensions increase simultaneously.

\item
For testing a specified null $H_0: \beta_i=\beta_i^0$, $i=1,\ldots, r$ with a fixed $r$, we discover a different phenomenon
that $2[\ell(\widehat{\bs{\beta}}) - \ell(\widehat{\bs{\beta}}^0)]$ asymptotically follows a chi-square distribution with $r$ degrees of freedom  in the $\beta$-model, while this is not true in the Bradley--Terry model.
Simulation studies show that its distribution depends crucially on the number of items in the Bradley--Terry model.
This reveals some inherent differences between these two closely related models.

\end{itemize}

To the best of our knowledge, this is the first attempt to establish Wilks-type theorems with increasing dimensions for either model.
We shall elaborate related literature later.
Further, we conduct simulations to demonstrate our theoretical findings
and apply our results to test homogeneity for a set of parameters in some real-world network data sets.

Establishing the asymptotic distribution of the LRT faces some technical challenges: (1)
the classic proof needs to derive high-dimensional limiting distribution of the $n$-dimensional MLE; (2)
how to bound over $n^3$ remainder terms  in the difference between asymptotic expansions of the naive and restricted log-likelihood functions
since the second-order Taylor expansion in the classic proof does not work in our setting.
To address these challenges, we use fourth-order asymptotic expansions of the log-likelihood function
and develop novel technical results to analyze the expansion terms.
The first is the central limit theorem for the sum of quadratic normalized degrees that leads to the asymptotic distribution of
$(\boldsymbol{\widehat{\beta}} - \boldsymbol{\beta})^\top V^{-1} (\boldsymbol{\widehat{\beta}} - \boldsymbol{\beta})$ with $V$ denoting the Fisher information matrix.
This solves the difficulty of obtaining the limiting distribution of $\boldsymbol{\widehat{\beta}}$.
The second is a small upper bound of  a weighted cubic sum $\sum_i f_i(\widehat{\beta}_i - \beta_i)^3$,
which has an additional vanishing factor $n^{-1/2}$ in contrast to the order of $\sum_i |f_i| |\widehat{\beta}_i - \beta_i|^3$.
The third contains the consistency rate of the restricted MLE $\widehat{\bs{\beta}}^0$,
the approximate inverse of the Fisher information matrix under the null space, and
an upper bound of the absolute entry-wise maximum norm between two approximate inverses under the full space and the restricted null space,
which are used to bound remainder terms.

Note that a Wald-type test can be constructed to test hypotheses for a set of parameters with a fixed dimension
according to central limit theorems, as in \cite{simons-yao1999} and \cite{Yan:Xu:2013}. However, it cannot be extended to increasing dimensional hypotheses because their central limit theorems are finite dimensional.
Furthermore, we compare powers of Wald-type tests with LRTs via numerical studies and find that the LRT exhibits substantially higher powers, indicating its unique advantages.

\subsection{Related work}

As pointed out by \cite{Goldenberg2010} and \cite{Fienberg2012}, it is challenging to obtain asymptotic inferences in random graph models such as
the $\beta$-model, as the number of parameters grows with that of nodes, and the sample is only one realized graph.
\cite{Chatterjee:Diaconis:Sly:2011} proved the consistency of the MLE in the $\beta$-model;
%\cite{Yan:Xu:2013} derived  its central limit theorems.
\cite{Rinaldo2013} gave conditions of the MLE existence.
Asymptotic theories of MLEs are also established in generalized $\beta$-models [e.g., \cite{Hillar:Wibisono:2013,Yan:Leng:Zhu:2016,Chen:2020}],
and the Bradley--Terry model [\cite{simons-yao1999,Chen_2019,Han-chen2020}].
Nevertheless, none of these works investigate the asymptotic properties of LRTs, which remain unknown.

For an adjusted $\beta$-model with a rescaled factor $\lambda/n$ ($\lambda$ is a known parameter),
\cite{mukherjee2018} studied thresholds for detecting sparse signals with explicitly degree-based test statistics under the null
$\beta_i=0$ for all $i$ against an alternative with a subset of $\{\beta_i\}$ greater than $0$.
Their problem settings are different from ours.
Here,  we study LRTs under specified and homogenous null hypotheses and do not require that the parameters are nonnegative under the alternative.

Moreover, the $\beta$-model and the Bradley--Terry model can be recast into a logistic regression form.
Under the ``large $N$, diverging $p_N$'' framework in generalized linear models,
\cite{wang2011} obtained a Wilks type of result for the Wald test under a simple null when $p_N^3/N \to 0$.
In our case, $p_N^3/N \to \infty$, not $0$, where  the dimension of parameter space is $p_N = n$ , and the total number of observations is $N=O(n^2)$.
In a different setting, by assuming that a sequence of independent and identical distributed samples
from a regular exponential family,
\cite{portnoy1988} showed a high dimensional Wilks type of result for the log-likelihood ratio statistic under the simple null.
For logistic regression models with asymptotic regime $p_N/N \to \kappa \in(0, 1/2)$, \cite{sur2019the} showed that the log-likelihood ratio statistic for testing a single parameter under the null $\beta_i=0$
converges to a rescaled chi-square with an inflated factor greater than one.
Conversely, our results do not have such inflated factors and cover a wider class of hypothesis testing problems.

Finally, we mention some literature on hypothesis testing problems in random graph models, including detecting a planted clique in an Erd\"{o}s--R\'{e}nyi graph [\cite{verzelen2015community}],
goodness-of-fit tests in stochastic block models [\cite{Hu.2020.1722676}] %%,gao2017testing
or testing whether there are one or multiple communities [\cite{jin2019optimal}],
testing between two inhomogeneous Erd\"{o}s--R\'{e}nyi graphs [\cite{10.1214/19-AOS1884}].
However, LRTs are not investigated in these works.

The rest of the paper is organized as follows. Wilks-type theorems for the $\beta$-model  and the Bradley--Terry model are presented in Sections \ref{section-beta-model}
and \ref{section-bt-model}, respectively.
Simulation studies and applications
are given in Section \ref{section-numerical}. Some further discussions are given in Section \ref{section:discussion}.
Section \ref{section:proof} presents proof outlines of all theorems and the proof of Theorems \ref{theorem-LRT-beta} (a).
All other proofs, including those of Theorems \ref{theorem-LRT-beta} (b), \ref{theorem-LRT-beta-fixed}, \ref{theorem-ratio-bt-3}, and \ref{theorem-ratio-bt-fixed},
as well as those of supported lemmas, are relegated to the Supplemental Material.

\section{Wilks-type theorems for the $\beta$-model}
\label{section-beta-model}

We consider an undirected graph $\mathcal{G}_n$ with $n$ nodes labelled as ``$1, \ldots, n$.''
Let $A=(a_{ij})_{n\times n}$ be the adjacency matrix of $\mathcal{G}_n$, where
$a_{ij}=1$ if there is an edge connecting nodes $i$ and $j$, and $a_{ij}=0$ otherwise.
Let $d_i = \sum_{j\neq i} a_{ij}$ be the degree of node $i$ and $\mathbf{d}=(d_1, \ldots, d_n)^\top$ be the degree sequence of $\mathcal{G}_n$.

Recall that the $\beta$-model postulates that $a_{ij}$, $1\le i< j \le n$, are mutually independent
Bernoulli random variables with edge probabilities given in \eqref{model-beta}.
The logarithm of the likelihood function under the $\beta$-model can be written as
\begin{equation*}
\ell(\boldsymbol{\beta}) = \sum_{1\le i < j \le n} \left\{a_{ij}(\beta_i + \beta_j) - \log(1 + e^{\beta_i + \beta_j})\right\} = \sum_{i=1}^n \beta_i d_i - \sum_{1\le i<j\le n} \log(1 + e^{\beta_i + \beta_j}),
\end{equation*}
where $\boldsymbol{\beta}=(\beta_1, \ldots, \beta_n)$.
Notably, the degree sequence is  the natural sufficient statistic in the $\beta$-model.
Setting the derivatives with respect to $\beta_i$ to zero, we obtain the likelihood equations
\begin{equation} \label{eq-likelihood-beta}
d_i = \sum_{j\neq i} \frac{e^{\widehat{\beta}_i + \widehat{\beta}_j}}{1 + e^{\widehat{\beta}_i + \widehat{\beta}_j}},~~i=1,\ldots,n,
\end{equation}
where $\boldsymbol{\widehat{\beta}}=(\widehat{\beta}_1, \ldots, \widehat{\beta}_n)^\top$ is
the MLE of $\boldsymbol{\beta}=(\beta_1, \ldots, \beta_n)^\top$.
The fixed-point iterative algorithm in \cite{Chatterjee:Diaconis:Sly:2011} can be used to solve $\boldsymbol{\widehat{\beta}}$.

With some ambiguity of notations, we use $V$ to denote the Hessian matrix of
the negative log-likelihood function under the $\beta$-model and the Bradley--Terry model.
In the case of the $\beta$-model, the elements of $V$ ($=(v_{ij})_{n\times n}$) are
\begin{equation}\label{definition-v-beta}
v_{ii} = \sum_{j\neq i}v_{ij}, ~~ v_{ij} =  \frac{e^{\beta_i+\beta_j}}{(1 + e^{\beta_i+\beta_j})^2}, ~~i\neq j; ~i,j=1,\ldots, n.
\end{equation}
Note that $V$ is also the Fisher information matrix of $\bs{\beta}$ and the covariance matrix of $\mathbf{d}$.
We define two notations that play important roles in guaranteeing good properties of the MLE $\boldsymbol{\widehat{\beta}}$:
\begin{equation}\label{definition-bncn}
b_n = \max_{i,j} \frac{(1 + e^{\beta_i+\beta_j})^2}{e^{\beta_i+\beta_j}}, \quad c_n=\min_{i,j}\frac{(1 + e^{\beta_i+\beta_j})^2}{e^{\beta_i+\beta_j}}.
\end{equation}
Remarkably, $b_n\ge c_n \ge 4$, and $b_n^{-1}$ and $c_n^{-1}$ express the minimum and maximum variances of $a_{ij}$ over $i\neq j$, respectively.

We first present Wilks-type theorems  in parameter testing problems with increasing dimensions. %%
We consider a specified null $H_0: \beta_i=\beta_i^0$ for $i=1,\ldots,r$ and a homogeneous null
$H_0: \beta_1=\cdots=\beta_r$ with $r\to\infty$, where $\beta_i^0$ for $i=1,\ldots,r$ are known numbers.
We assume that the random adjacency matrix $A$ is generated under the model with the parameter $\bs{\beta}$.
When $r=n$, the null $H_0: \beta_i=\beta_i^0$, $i=1,\ldots,r$ becomes the simple null.
Recall that $\boldsymbol{\widehat{\beta}^{0}}$ denotes the restricted MLE of $\bs{\beta}$ under the null parameter space.
%%and therefore $\bs{\beta}=\bs{\beta}^0$, where $\bs{\beta}^0=(\beta_1^0, \ldots, \beta_n^0)^\top$.

\begin{theorem}
\label{theorem-LRT-beta}
%Suppose that $b_n^5/c_n^4 = o( n^{1/2}/(\log n)^2 )$ and $r/n\ge \tau>0$, where $\tau$ is a positive constant.
\begin{itemize}
\item[(a)]
Under the specified null
$H_0: \beta_i=\beta_i^0$, $i=1,\ldots,r$, if $b_n^5/c_n^2 = o( r^{1/2}/(\log n)^2 )$,
the log-likelihood ratio statistic $\ell(\boldsymbol{\widehat{\beta}}) - \ell(\boldsymbol{\widehat{\beta}}^0)$
is asymptotically normally distributed in the sense that
\begin{equation} \label{statistics-beta}
\frac{2\{ \ell(\boldsymbol{\widehat{\beta}}) - \ell(\boldsymbol{\widehat{\beta}}^0)\} - r}{\sqrt{2r}} \stackrel{L}{\rightarrow} N(0,1), ~~\mbox{as}~~ r\to\infty,
\end{equation}
where $\boldsymbol{\widehat{\beta}^{0}}=\arg\max_{\bs{\beta}\in \Theta_0} \ell(\bs{\beta})$ and $\Theta_0=\{
 \bs{\beta}: \bs{\beta}\in \R^n, (\beta_1, \ldots, \beta_r) = (\beta_1^0, \ldots, \beta_r^0) \}$.
\item[(b)]
Under the homogenous null
$H_0: \bs{\beta}\in \Theta_0=\{ \bs{\beta}: \bs{\beta}\in \R^n, \beta_1=\cdots=\beta_r\}$,
if $b_n^{15}/c_n^9 = o\left( r^{1/2}/(\log n)^3 \right)$,
the normalized log-likelihood ratio statistic
in \eqref{statistics-beta} also converges in distribution to the standard normality.
\end{itemize}
\end{theorem}

The likelihood equations under the specified or homogenous null are similar to \eqref{eq-likelihood-beta}. The restricted MLE $\boldsymbol{\widehat{\beta}^{0}}$
can also be solved via the fixed-point iterative algorithm or the Newton-Raphson algorithm.

We discuss the requirements for $b_n, c_n$, and $r$.
The condition $ b_n^5 / c_n^2  = o\left(  r^{1/2} /(\log n)^2 \right)$
captures the trade-off between network sparsity, node degree heterogeneity, and the number of hypotheses. Specifically, as the network becomes sparser and node degrees become more heterogeneous, the value of $r$ must increase.
First, if all parameters are bounded by a fixed constant,
then the Wilks-type result holds as long as \(r \gg (\log n)^6\). %This is a relatively mild condition for the dimension \(r\).
Second, when the network is sparse, the requirement for the dimension \(r\) becomes linked to the network density through \(b_n\) and \(c_n\). For example, if $\beta_i$ tends to $-\infty$ at the same rate for $i=1,\ldots,n$, then \(b_n \asymp c_n\) and $b_n\to\infty$. The network density, a defined as \(\frac{\sum_{i,j} A_{i,j}}{n(n-1)}\), is then  \(1/b_n\). The condition in Theorem \ref{theorem-LRT-beta} (a) becomes
\(b_n = o\left( r^{1/6}/(\log n)^{2/3} \right)\). This indicates:
when $r$ increases, $b_n$ could be larger such that the network density could be smaller, indicating that the network could be sparser.
In particular,
when $r\asymp n$, the network density could be close to $n^{-1/6}$.
Alternatively, the condition in Theorem \ref{theorem-LRT-beta} (b) is stronger than that in Theorem \ref{theorem-LRT-beta} (a).
This is because we need a stronger condition on $b_n$ to obtain a unified consistency rate, regardless of $r$, as given in Lemma 7 of Supplementary Material A. %\ref{lemma-con-beta-b}.

We note that the consistency of the MLE in \cite{Chatterjee:Diaconis:Sly:2011}
is based on the condition that all parameters are bounded above by a constant, while asymptotic normality of the MLE in \cite{Yan:Xu:2013} needs the following condition: $\max_i |\beta_i| = o( \log (\log n))$, which implies $b_n \ll \log n$.
Conversely, the condition here is much weaker.
Furthermore, some intermediate results in Lemmas \ref{lemma:weighte-degree-al},
\ref{lemma-clt-beta-W}, \ref{lemma-consi-beta}, and 7, %\ref{lemma-con-beta-b},
are built under weaker conditions.
For instance, the consistency rate of $\bs{\widehat{\beta}}^0$ in Lemma \ref{lemma-consi-beta}
only requires  $b_n^2/c_n=o( n^{1/2}/(\log n)^{1/2})$.

The following corollary gives the smallest $r$ to guarantee Wilks-type theorems,
which only requires $r$ much larger than a logarithm factor to the power of $6$.

\begin{corollary} \label{corollary-theorem1}
If $b_n$ is bounded by a constant and $r/(\log n)^6\to\infty$,
the normalized log-likelihood ratio statistic in \eqref{statistics-beta} converges in distribution to the standard normal distribution under specified and homogenous null hypotheses.
\end{corollary}

Next, we present a central limit theorem for normalized degrees mentioned previously. %, which  bridges a technical gap between LRTs and Wilks-type theorems in the $\beta$-model and the Bradley--Terry model.
With the use of a fourth-order Taylor expansion for the log-likelihood functions $\ell(\boldsymbol{\widehat{\beta}})$ and $\ell(\boldsymbol{\widehat{\beta}}^0)$, % at the true parameter $\boldsymbol{\beta}$,
the main term in the difference $\ell(\boldsymbol{\widehat{\beta}}) - \ell(\boldsymbol{\widehat{\beta}}^0)$ is $\sum_{i=1}^r \bar{d}_i^{\,2}/v_{ii}$.
Therefore, %the central limit theorem for the normalized degrees in
%Lemma \ref{lemma:weighte-degree-al} bridges a technical gap between LRTs and Wilks-type theorems.
Lemma \ref{lemma:weighte-degree-al} establishes the key
step in demonstrating the asymptotic normality of LRTs, thereby proving Wilks-type theorems in the
diverging $r$-setting.
It also gives an intuition why the Wilks-type result holds.

\begin{lemma}\label{lemma:weighte-degree-al}
Under the $\beta$-model,
if $b_n^4/c_n^3=o(n)$, then
$\sum_{i=1}^r \bar{d}_i^{\,2}/v_{ii}$ is asymptotically normally distributed with mean $r$ and variance $2r$,
where $\bar{d}_i= d_i - \E d_i$ and $v_{ii}$ is the variance of $d_i$.
\end{lemma}

The above lemma shows that the normalized sum
for $\{ \bar{d}_i^{\,2}/v_{ii} \}_{i=1}^r$
converges in distribution to the standard normal distribution for arbitrary $r$ tending to infinity in the case of that $b_n$ is a constant.
The proof of Lemma \ref{lemma:weighte-degree-al} is highly nontrivial. %technical.
The quadratic centered degree sequence $\{\bar{d}_i^{\,2}\}_{i=1}^r$ is not independent and is also not a commonly seen mixing sequence such as
$\alpha$-mixing, $\phi$-mixing and so on. Consequently, classical central limit theorems for independent random variables or
dependent random variables [e.g., \cite{Peligrad1987}] % [e.g., \cite{Peligrad1987,withers1987central}]
cannot be applied. Further, it is not a natural martingale.
Observe that $\bar{d}_i^{\,2} = \sum_{j,k\neq i} \bar{a}_{ij}\bar{a}_{ik}$ and $\E(\bar{a}_{ij}\bar{a}_{ik}|\bar{a}_{ij})=0$, where $\bar{a}_{ij}=a_{ij} - \E a_{ij}$.
This is analogous to the property of vanishing conditional expectations in one-sample $U$-statistics [e.g., \cite{HALL19841}] and the quadratic form
$w(X_i, X_j)$ [e.g., \cite{Jong1987PTRF}], where $\E\{ w(X_i, X_j)|X_i \}=0$ for independent random variables $\{X_i\}$.
Motivated by this property, we divide $\sum_{i=1}^r \bar{d}_i^{\,2}/v_{ii}$ into two parts:
\begin{eqnarray}\label{eq-lemma1-a}
\sum_{i=1}^r \frac{ (\bar{d}_i^{\,2} - \E \bar{d}_i^{\,2}) }{ v_{ii} }  & = &  \sum_{i=1}^r \sum_{j=1}^n \frac{ (\bar{a}_{ij}^2 - \E \bar{a}_{ij}^2 ) }{ v_{ii} }
+ \sum_{i=1}^r \sum_{j=1,j\neq i}^n \sum_{k=1, k\neq i,j}^n \frac{ \bar{a}_{ij} \bar{a}_{ik} }{ v_{ii} }.
\end{eqnarray}
The first summation in the right-hand side of the above equation scaled by $r^{1/2}$ vanishes, while the second summation can be represented as
a sum of martingale differences with a delicate construction.
Subsequently, we can use Martingale theory [e.g., \cite{Brown1971}] to obtain its central limit theorem, whose details are given in the Supplementary Material.

Now, we present Wilks-type theorems for fixed dimensional parameter hypothesis testing problems.

\begin{theorem}
\label{theorem-LRT-beta-fixed}
Assume that $b_n^3/c_n=o(n^{1/6}/(\log n))$ and $r$ is a fixed positive integer.
\begin{itemize}
\item[(a)]
Under the specified null $H_0: \beta_i=\beta_i^0, i=1, \ldots, r$,
the log-likelihood ratio statistic $2\{\ell(\widehat{\bs{\beta}}) - \ell( \widehat{\bs{\beta}}^0) \}$
converges in distribution to a chi-square distribution with $r$ degrees of freedom as $n$ goes to infinity.

\item[(b)]
Under the homogenous null $H_0: \beta_1 = \cdots = \beta_r$,
$2\{\ell(\widehat{\bs{\beta}}) - \ell( \widehat{\bs{\beta}}^0) \}$
converges in distribution to a chi-square distribution with $r-1$ degrees of freedom as $n$ goes to infinity.
\end{itemize}

\end{theorem}

Theorem \ref{theorem-LRT-beta-fixed} states that the log-likelihood ratio statistic behaves like the classical Wilks theorem in the case that the difference between the full and null spaces of the tests is fixed. As previously mentioned, the condition imposed on $b_n$ restricts the increasing rate of $b_n$ and is fully filled when $b_n$ is a constant.

\begin{remark}
When $r$ is fixed, a Wald-type test can be constructed from the central limit theorem for $\widehat{\beta}$ in \cite{Yan:Xu:2013}.
Take, for example, the null $H_0:\beta_1 = \cdots = \beta_r$.
Let $\boldsymbol{\nu}=( \widehat{\beta}_1 - \widehat{\beta}_2, \ldots, \widehat{\beta}_{r-1} - \widehat{\beta}_r)^\top$ and $\Omega=(\omega_{ij})$ be the $(r-1)\times (r-1)$ dimensional matrix, where
$\omega_{ii}=1/\hat{v}_{ii}+1/\hat{v}_{i+1,i+1}$, $\omega_{ij}=-1/\hat{v}_{ii}$ if $i=j+1$, $\omega_{ij}=-1/\hat{v}_{jj}$ if $j=i+1$ and $\omega_{ij}=0$ if $|i-j|>1$.
Following \cite{Yan:Xu:2013} and \cite{simons-yao1999}, we can use the statistic $\boldsymbol{\nu}^\top Q \boldsymbol{\nu}$ to test $H_0$.
\iffalse
 the Wald-type test statistic
\begin{equation}
\label{eq-wald-test}
%(\widehat{\beta}_1 - \widehat{\beta}_2, \widehat{\beta}_{2} - \widehat{\beta}_3, \ldots, \widehat{\beta}_{r-1} - \widehat{\beta}_r )
\begin{pmatrix}
\widehat{\beta}_1 - \widehat{\beta}_2 \\ \widehat{\beta}_{2} - \widehat{\beta}_3\\
 \ldots\\
 \widehat{\beta}_{r-1} - \widehat{\beta}_r
\end{pmatrix}^\top
\begin{pmatrix}
\frac{1}{ \hat{v}_{11} } + \frac{1}{ \hat{v}_{22} }, & - \frac{1}{ \hat{v}_{22} }, &  0, & \ldots, &¡¡0 \\
- \frac{1}{ \hat{v}_{22} }, & \frac{1}{ \hat{v}_{22} } + \frac{1}{ \hat{v}_{33} }, &  -\frac{1}{ \hat{v}_{33} } , & \ldots, & 0 \\
& & \vdots & & \\
0 , &\ldots & 0, &  -\frac{1}{ \hat{v}_{r-1,r-1} }, & \frac{1}{ \hat{v}_{r-1,r-1} } + \frac{1}{\hat{v}_{rr}}
\end{pmatrix}
\begin{pmatrix}
\widehat{\beta}_1 - \widehat{\beta}_2 \\ \widehat{\beta}_{2} - \widehat{\beta}_3\\
 \ldots\\
 \widehat{\beta}_{r-1} - \widehat{\beta}_r
\end{pmatrix}
\end{equation}
\fi
We compare powers of LRTs and this Wald-type test via numerical studies and find that the former has higher powers than the latter.
\end{remark}

\section{Wilks-type theorems for the Bradley--Terry model}
\label{section-bt-model}

In the above section, we considered an undirected graph. Now we consider a weighted directed graph $\mathcal{G}_n$,
where % nodes denote items joining in paired comparisons and
the element of the adjacency matrix $A$ denotes the number of times that one item is preferred to another item.
Let $k_{ij}$ be the number of comparisons between items $i$ and $j$.
Similar to \cite{simons-yao1999}, we assume $1\le k_{ij}\le K$ for all $i\neq j$, where $K$ is a fixed positive constant.
%For easy exposition, similar to \cite{simons-yao1999}, we assume $k_{ij}=K$ for all $i\neq j$, where $K$ is a fixed positive constant.
Subsequently, $a_{ij}$ is the number of times that $i$ beats $j$ out of a total number of $k_{ij}$ comparisons.

The Bradley--Terry model postulates that $a_{ij}$, $1\le i<j \le n$, are mutually independent
binomial random variables---$a_{ij} \sim \mbox{Binomial}(k_{ij}, p_{ij})$, where
$\mathrm{logit}(p_{ij})= \beta_i - \beta_j$.
\iffalse
\begin{equation}\label{btmodel}
p_{ij} = \frac{ e^{ \beta_i - \beta_j } }{ 1 + e^{ \beta_i - \beta_j } }.
\end{equation}
\fi
Here, $\beta_i$ measures the intrinsic strength of item $i$, and $d_i=\sum_{j\neq i} a_{ij}$ denotes the total number of wins for item $i$.
The win-loss probabilities for any two items only depend on the difference of their strength parameters.
The bigger the strength parameter, the higher the probability of item $i$ achieving a win over other items.

As the probability is invariable by adding a common constant to all strength parameters $\beta_i$, $i=1, \ldots, n$,
we need a restriction for the identifiability of model. Following \cite{simons-yao1999}, we set $\beta_1=0$ as a constraint.
Notice that the number of free parameters here is $n-1$, different from the $\beta$-model with $n$ free parameters.
The logarithm of the likelihood function under the Bradley--Terry model is
\begin{equation}\label{likelihood-bt}
\ell_{bt}(\boldsymbol{\beta}) = \sum_{i,j=1;i\neq j}^n a_{ij} \left\{\beta_i - \log(e^{\beta_i}+e^{\beta_j})\right\} = \sum_{i=1}^n \beta_id_i  -  \sum_{1\le i<j\le n} k_{ij} \log(e^{\beta_i}+e^{\beta_j}),
\end{equation}
where $\boldsymbol{\beta} = (\beta_2, \ldots, \beta_n)^\top$ and $\beta_1=0$.
To distinguish the log-likelihood function in the $\beta$-model, we use a subscript $bt$ in this section.
As we can see, it is an exponential family distribution on the directed graph $\mathcal{G}_n$ with the out-degree sequence as its natural sufficient statistic.
The likelihood equations are
\begin{equation} \label{estimated-eq-bt-a}
d_i = \sum_{j=1,j\neq i}^n \frac{k_{ij} e^{\hat{\beta}_i}}{e^{\hat{\beta}_i} + e^{\hat{\beta}_j}}, ~~i=2,\ldots,n,
\end{equation}
where $\boldsymbol{\widehat{\beta}} = (\widehat{\beta}_2, \ldots, \hat{\beta}_n)$ is the MLE of $\boldsymbol{\beta}$ with $\widehat{\beta}_1=0$.
If the directed graph $\mathcal{G}_n$ is strongly connected,
then the MLE uniquely exists [\cite{Ford1957}].
Note that $d_1$ is not involved in \eqref{estimated-eq-bt-a}; indeed, given $d_2, \ldots, d_n$ and $\{k_{ij}\}_{i,j=1,i\neq j}^n$, $d_1$ is determined.

Now, we present Wilks-type theorems for the Bradley--Terry model with a growing number of hypothetic parameters.
The corresponding definitions of $b_n$ and $c_n$ are as follows:
\begin{equation*}\label{definition-bncn}
b_n = \max_{i,j} \frac{(1 + e^{\beta_i-\beta_j})^2}{e^{\beta_i-\beta_j}},
\quad c_n=\min_{i,j}\frac{(1 + e^{\beta_i-\beta_j})^2}{e^{\beta_i-\beta_j}}.
\end{equation*}
With some ambiguity, we use the same notations $b_n$ and $c_n$ as in the $\beta$-model, where their expressions are based on $\beta_i+\beta_j$.
%In the Bradley--Terry model, they are functions on $\beta_i-\beta_j$.
%and is relegated to the Supplementary Material.

\begin{theorem} \label{theorem-ratio-bt-3}
Suppose $b_n^7/c_n^4 = o( r^{1/2}/(\log n)^{2})$.
\begin{itemize}
\item[(a)]
Under the specified null
$H_0: \beta_i=\beta_i^0$, $i=2,\ldots,r$, the log-likelihood ratio statistic $\ell_{bt}(\boldsymbol{\widehat{\beta}}) - \ell_{bt}(\boldsymbol{\widehat{\beta}}^0)$
is asymptotically normally distributed in the sense that
\begin{equation} \label{statistics-bt}
\frac{2\{ \ell_{bt}(\boldsymbol{\widehat{\beta}}) - \ell_{bt}(\boldsymbol{\widehat{\beta}}^0)\} - r}{\sqrt{2r}} \stackrel{L}{\rightarrow} N(0,1), ~~\mbox{as}~~ r\to\infty,
\end{equation}
where $\boldsymbol{\widehat{\beta}^{0}}=\arg\max_{\bs{\beta}\in \Theta_0} \ell_{bt}(\bs{\beta})$ and $\Theta_0=\{
 \bs{\beta}: \bs{\beta}\in \R^{n-1}, (\beta_2, \ldots, \beta_r) = (\beta_2^0, \ldots, \beta_r^0) \}$.
\item[(b)]
Under the homogenous null
$H_0: \bs{\beta}\in \Theta_0=\{ \bs{\beta}: \bs{\beta}\in \R^{n-1}, \beta_2=\cdots=\beta_r\}$, the normalized log-likelihood ratio statistic
in \eqref{statistics-bt} also converges in distribution to the standard normality.
\end{itemize}
\iffalse
then the log-likelihood ratio statistic  $\ell_{bt}(\boldsymbol{\widehat{\beta}}) -
\ell_{bt}(\boldsymbol{\beta})$ is asymptotically normally distributed in the sense that
\begin{equation} \label{statistics-bt1}
\frac{2\{ \ell_{bt}(\boldsymbol{\widehat{\beta}}_n) - \ell_{bt}(\boldsymbol{\beta}_n) \} - (n-1)}{\sqrt{2(n-1)}}
\stackrel{L}{\rightarrow} N(0,1), ~~\mbox{as}~~ n\to\infty.
\end{equation}
(b)Without loss of generality, suppose the composite null hypothesis takes the following form---$H_0$: $\beta_1 = \cdots = \beta_r$.  Let $\boldsymbol{\widehat{\beta}}^0 = (\widehat{\beta}_{1}^{0}, \ldots, \widehat{\beta}_{n}^{0})$ be the MLE of $\boldsymbol{\beta}$ under $H_0^*$, with $\widehat{\beta}_1^*=0$. Assume that $r/n \ge \tau>0$, where $\tau$ is a positive constant. If \eqref{bt-condition1} holds, then
$\ell_{bt}(\boldsymbol{\hat{\beta}}) - \ell_{bt}(\boldsymbol{\hat{\beta}}^{0})$ is asymptotically normally distributed in the sense that
\begin{equation} \label{redudant-bt}
\frac{2\{ \ell_{bt}(\boldsymbol{\hat{\beta}}_n) - \ell_{bt}(\boldsymbol{\hat{\beta}}_{n}^*)\} - (r-1)}{\sqrt{2(r-1)}} \stackrel{L}{\rightarrow} N(0,1), ~~\mbox{as}~~ n\to\infty.
\end{equation}
\fi
\end{theorem}

To guarantee the existence of the MLE which in turns defines the LRT well, it is necessary to
control the increasing rate of $b_n$, as discussed in \cite{simons-yao1999}.
If all items could be divided into two sets
with the first one having large merit parameters $\beta_i$'s and the second one having extremely small merits,
corresponding to a large value of $b_n$, then the items in the second set
will stand little chance of beating those in the first set. Once such event occurs,
the MLE will not exist [\cite{Ford1957}].

Next, we present the Wilks-type theorem under the homogenous null with a fixed dimension.

\begin{theorem}
\label{theorem-ratio-bt-fixed}
If $b_n^{11}/c_n^6 = o( n^{1/2}/(\log n)^{5/2})$,
under the homogenous null $H_0: \beta_2 = \cdots = \beta_r$ with a fixed $r > 2$,
the twice log-likelihood ratio statistic $2\left\{\ell_{bt}(\widehat{\bs{\beta}}) - \ell_{bt}( \widehat{\bs{\beta}}^0) \right\}$
converges in distribution to a chi-square distribution with $r-2$ degrees of freedom.
\end{theorem}

Different to Theorem \ref{theorem-LRT-beta-fixed} in the $\beta$-model,
the above theorem does not contain a Wilks-type result under
the fixed dimensional specified null $H_0: \beta_i=\beta_i^0, i=2, \ldots, r$.
An intuitive explanation is that the difference between the naive MLE and restricted MLE under this null
is larger than that in the $\beta$-model, which leads to third- or higher-order expansion terms not being negligible.
Detailed explanations are as follows:
%With the use of $S_{22}=\mathrm{diag}( 1/v_{r+1,r+1}, \ldots, 1/v_{nn}) + 1/\tilde{v}_{11}$ to approximate the Fisher information matrix $V_{22}$
%under the specified null and
With similar arguments as in the proof of \eqref{eq-ell-difference} and \eqref{eq-theorem2-B10}, we have
\begin{equation*}
2\left\{ \ell_{bt}(\widehat{\bs{\beta}}) - \ell_{bt}( \widehat{\bs{\beta}}^0) \right\}=
\sum_{i=1}^r \frac{ \bar{d}_i^{\,2} }{ v_{ii}}  - \frac{ \{ \sum_{i=1}^r \bar{d}_i \}^2 }{ \tilde{v}_{11} }+ \frac{1}{3}(B_2-B_2^0)+ \frac{1}{12}(B_3-B_3^0),
\end{equation*}
where $\tilde{v}_{11}=\sum_{i=1}^r v_{ii}$,
$(B_2-B_2^0)$ is the difference of the third-order expansion term of the log-likelihood function between the full and null spaces, and $(B_3-B_3^0)$ is the corresponding difference of the fourth-order expansion term.
If $v_{11}=\cdots=v_{rr}$, then $\sum_{i=1}^r \bar{d}_i^{\,2}/v_{ii} - \{ \sum_{i=1}^r \bar{d}_i \}^2/\tilde{v}_{11}$ asymptotically
follows a chi-square distribution. Even if $v_{11}=\cdots=v_{rr}$, $2\left\{ \ell_{bt}(\widehat{\bs{\beta}}) - \ell_{bt}( \widehat{\bs{\beta}}^0) \right\}$ is not approximately a chi-square distribution.
This is because $B_2-B_2^0$ does not go to zero, whereas it vanishes in the $\beta$-model.
In the case of fixed $r$, a key quantity to measure $B_2-B_2^0$ is
$\max_{i=r+1, \ldots, n}|\widehat{\beta}_i - \widehat{\beta}_i^0|$.
It has the order of $\log n/n$ in the $\beta$-model, whereas in the Bradley-Terry model, the differences have the following representation:
\[
\widehat{\beta}_i - \widehat{\beta}_i^0 = \frac{ \bar{d}_1 }{ v_{11} } - \frac{ \sum_{i=1}^r \bar{d}_i  }{ \tilde{v}_{11} } + O_p\left( \frac{ b_n^2 \log n }{ n} \right),~~ i=r+1, \ldots, n,
\]
under the specified null.
The difference between the two distributions of $\bar{d}_1/v_{11}$ and $(\sum_{i=1}^r \bar{d}_i )^2/\tilde{v}_{11}$ is much larger than the order of $\log n/n$.
Under the homogenous null,
$\widehat{\beta}_i - \widehat{\beta}_i^0$ does not contain the difference of the above two terms.
It leads to the result that $B_2-B_2^0$ vanishes in the homogenous null, while it does not vanish in the specified null.
Therefore, the Wilks-type result does not hold in the fixed dimensional specified null in the Bradley--Terry model.

\begin{figure}[h]
\centering
%\captionstyle{flushleft}
%\onelinecaptionsfalse
\caption{Comparisons of density curves of LRT and chi-square distributions with  2 or 3 degrees of freedom.}
\subfigure[Performance of minus twice log-likelihood ratio statistic under the specified null (n=200)]{\includegraphics[width=0.9\textwidth, height=1.8in]{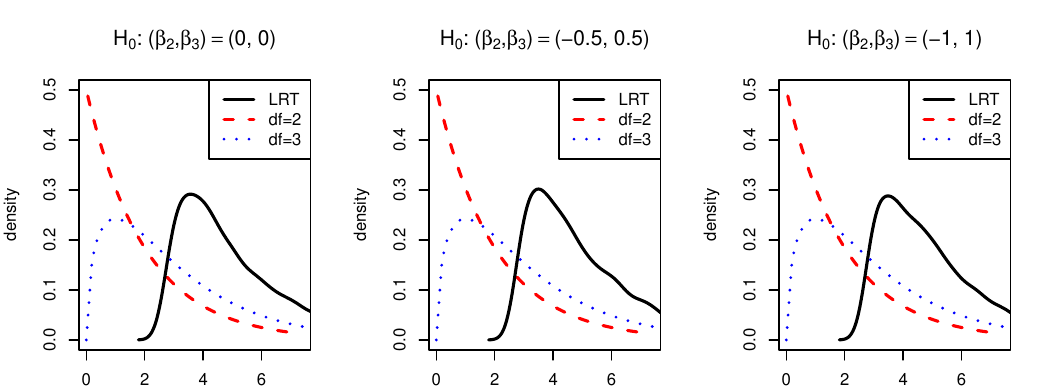}}
\subfigure[n=100]{\includegraphics[width=0.9\textwidth, height=1.8in]{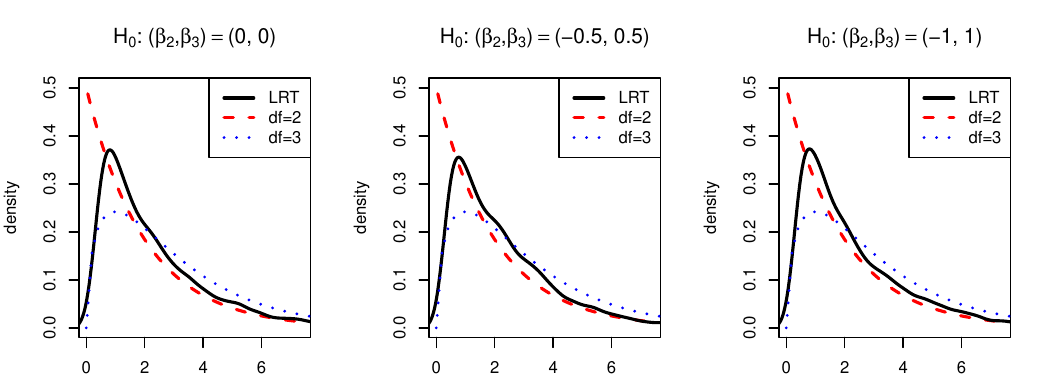}}
\label{fig:bt-a}
\end{figure}

To provide some insight on the distribution of $2\left\{ \ell_{bt}(\widehat{\bs{\beta}}) - \ell_{bt}( \widehat{\bs{\beta}}^0) \right\}$ under the specified null, we draw its density curve against that of the chi-square distribution. We consider several specified null $H_0: (\beta_2, \beta_3)=(-c, c)$ with $c=0, 0.5, 1$ and other parameters are $\beta_i= 0.2(i-1)\log n/(n-1)$, where $\beta_0=0$.
The plots are shown in Figure \ref{fig:bt-a}, where the simulation is repeated $5,000$ times.
We noted three findings: (1) the distribution is far away from chi-square distributions with degree $2$ or $3$;
(2) even if $\beta_2=\beta_3=0$, the density curve is considerably different from that of a chi-square distribution;
(3) the density curve depends crucially on $n$ and appears to be insensitive to changes in the parameter choices.

\section{Numerical Results}
\label{section-numerical}
In this section, we illustrate the theoretical results via numerical studies.

\subsection{Simulation studies}
\label{sub-sec-simu}

We conduct simulations to evaluate the performance of the log-likelihood ratio statistics for a finite number of nodes.
We considered the four null hypotheses: (1) $H_{01}$: $\beta_i=(i-1)L_n/(n-1)$, $i=1,\ldots,n$;
(2) $H_{02}$: $\beta_1=\cdots=\beta_r$, $r=n/2$;
(3) $H_{03}$: $\beta_i= (i-1)r/5$, $r=5$;
(4) $H_{04}$: $\beta_1=\cdots=\beta_r$, $r=10$, where $L_n$ is set to evaluate different asymptotic regimes.
\( H_{01} \) corresponds to the simple null, while \( H_{03} \) tests whether a fixed number of parameters are equal to specified values.
These two specified nulls explore the intriguing question of whether the parameters for a subset of nodes assume specific values.	
For instance, we may seek to determine whether the connectivity parameters of the current local network align with those estimated from a previous time period in
time-varying network data. Similarly, when examining specific subgroups, such as teams in a sports league,
it is practically relevant to evaluate whether the performance levels of teams in the current season correspond to those measured from past performances.
In contrast, \( H_{02} \) and \( H_{04} \) % are homogeneous nulls that
assess whether a given set of parameters, whether with increasing or fixed dimensions, are equal. The framework for \( H_{02} \) can be utilized to determine whether subgraphs behave like Erd\H{o}s-R\'{e}nyi graphs, where the probability of connecting any two nodes is uniform. Meanwhile, the framework for \( H_{04} \) can be employed to evaluate whether the associated teams belong to a subgroup with comparable merits or substantially different calibers.

For $H_{02}$, $H_{03}$, and $H_{04}$, we set the left $n-r$ true parameters as follows: $\beta_i = (i-1)L_n/(n-1)$ for $i=r+1,\ldots,n$.
For $H_{02}$ and $H_{04}$, we set the first $r$ true parameters as $\beta_1 = \cdots=\beta_r=0$.
In the Bradley--Terry model, $\beta_1$ $(=0)$ is a reference parameter and is excluded in the above null.
Four values for $L_n$ were chosen: $L_n= 0$, $0.2\log n$, $0.4\log n$, and $0.5\log n$.
Three values for $n$ were considered: $n=200$, $n=500$ and $n=1000$.
For the Bradley--Terry model, we assumed that each pair has one comparison---$K=1$.
Further, motivated by the schedules of the National Basketball Association (NBA) regular season briefly described in next section,
we additionally considered a relatively small size $n=30$ and let the number of paired comparisons $k_{ij}$ equal to 3 for all $1\le i\neq j\le n$.
Each simulation was repeated $1,000$ times.

We first draw the quantile-quantile (QQ) plots to evaluate the finite-sample performance of the LRTs. % and $0.10$.
%For the increasing dimensional null hypotheses, we use the chi-square approximation instead of the normal approximation because the former performs better than the latter in finite sample sizes.
For two increasing dimensional hypothesis $H_{01}$ and $H_{02}$,
we compared the normal and $\chi^2$ approximations for the normalized LRT,
where we draw the sample quantile of $\{ 2[\ell(\boldsymbol{\widehat{\beta}}) - \ell(\boldsymbol{\widehat{\beta}}^0)] - r \}/(2r)^{1/2}$
against the quantile of the standard normality or  $(X_r - r)/(2r)^{1/2}$
with $X_r$ denoting a $\chi^2_r$ random variable with $r$ degrees of freedom.
Owing to the limited space, we only show the QQ-plots in the case of $n=1000$ under the $\beta$-model, and other QQ plots are in Supplementary Material A.
From figure \ref{fig:beta}, we can see that the sample quantiles agree well with theoretical quantiles when $L\le 0.4\log n$.
The QQ-plots based on the normal and $\chi^2$ approximations are very close when $n\ge 200$.
This is due to that $(X_r - r)/(2r)^{1/2}$ is asymptotically normal when $r$ is large.
The $\chi^2$ approximation looks a little better than the normal approximation when $n \le 100$.
Alternatively, when $L=0.5\log n$ and $n=1000$, there are evident deviations from the reference line $y=x$ under the null $H_{01}$, indicating that it is necessary to control the increasing rate of $b_n$ for guaranteeing
good asymptotic properties of the LRT.

\begin{figure}[htbp]
\centering
%\captionstyle{flushleft}
%\onelinecaptionsfalse
\caption{QQ plots for the $\beta$-model under the null. The horizontal and vertical axes in each QQ-plot are the respective theoretical (based on the normal or chi-square distribution) and empirical quantiles.
The straight lines correspond to $y=x$. The first, second, and third columns correspond to $L_n=0.2\log n, 0.4\log n, 0.5\log n$, respectively.
The QQ plots for $L_n=0$ are very similar to those for $L_n=0.2\log n$ and thus are not shown.
Furthermore, the QQ plots under the fixed dimensional null $H_{03}$ and $H_{04}$ are similar and only those under $H_{03}$ are shown here to save space. }
\subfigure[QQ-plot for normalized log-likelihood ratio statistic in \eqref{statistics-beta} under $H_{01}$]{\includegraphics[width=0.9\textwidth]{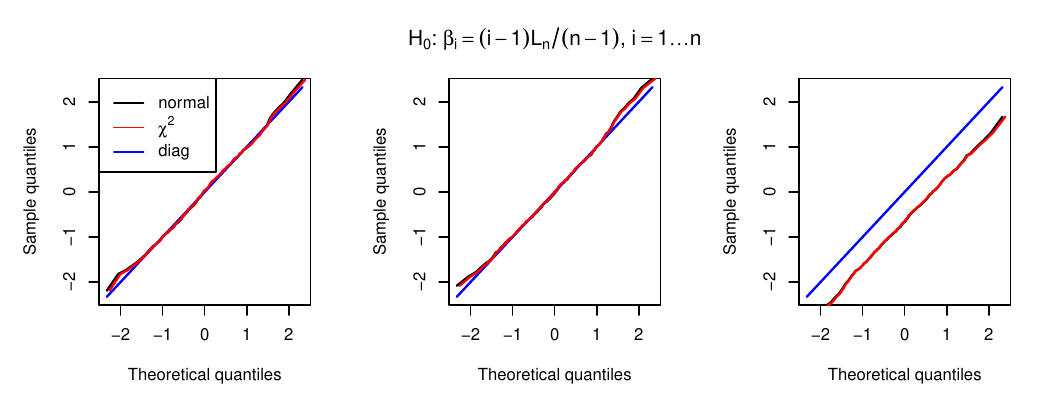}}
\subfigure[QQ-plot for normalized log-likelihood ratio statistic in \eqref{statistics-beta} under $H_{02}$]{\includegraphics[width=0.9\textwidth]{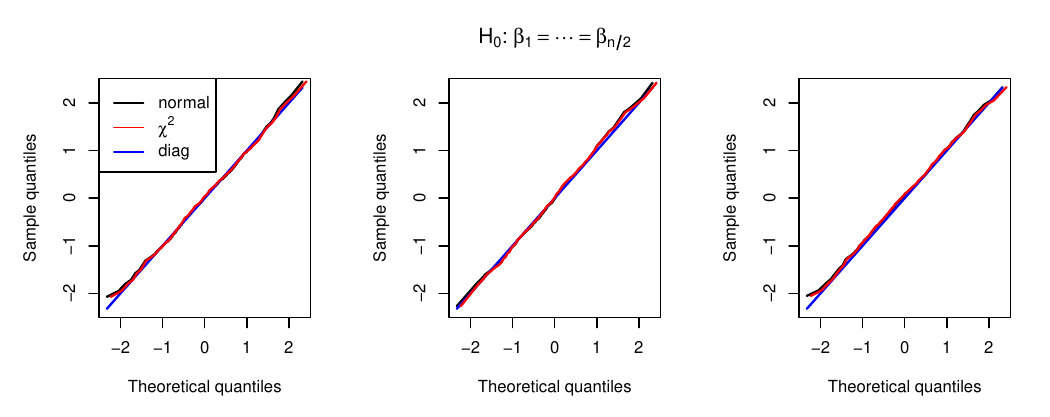}}
\subfigure[QQ-plot for log-likelihood ratio statistic under $H_{03}$]{\includegraphics[width=0.9\textwidth]{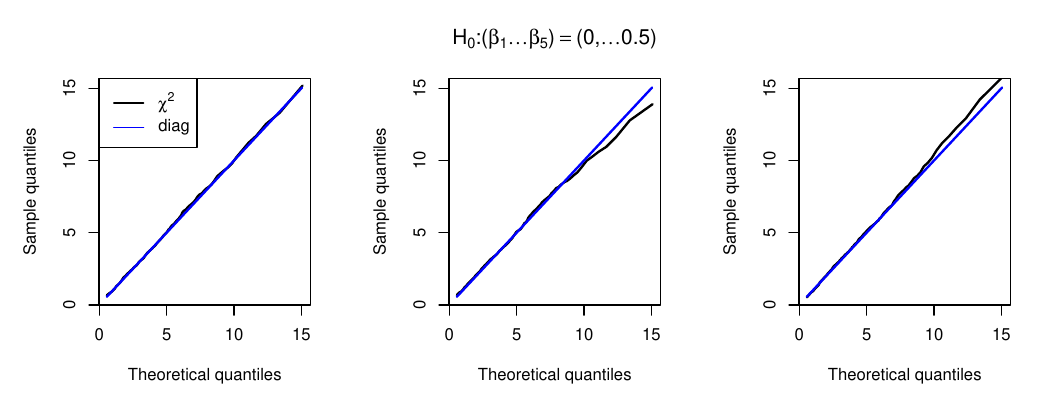}}
%\subfigure[QQ-plot for log-likelihood ratio statistic under $H_{04}$]{\includegraphics[width=0.9\textwidth]{beta-com-200-5.eps}}
\label{fig:beta}     %% label for entire figure
\end{figure}

\begin{table}[h!]
\centering
\caption{Type I errors $(\times 100)$ of the LRTs. % under the nominal levels $\alpha=0.05$~/frequencies $(\times 100)$ that the MLE does not exist ($L_n = c\log n$).
The first and second entries in each triple are the simulated type I errors, corresponding to the respective nominal levels $\alpha=0.05$ and $\alpha=0.1$, where
the last entry denotes the frequency $(\times 100)$ that the MLE does not exist.}
\label{table-type-I}
\vskip5pt
\small
\begin{tabular}{cc ccc c}
\hline
\multicolumn{6}{c}{Type I errors under the $\beta$-model}\\
\hline
NULL      &   $n$       & $c =0$  & $c = 0.2$  & $c = 0.4 $  & $c=0.5 $  \\
\hline
$H_{01}$  &   $200$    &$( 4.9 , 10.4 , 0 )$&$( 4.3 , 8.9 , 0 )$&$( 5.4 , 10.6 , 0 )$&$( 5.35 , 11.37 , 10.3 )$  \\
          &   $500$    &$( 5.9 , 12.0 , 0 )$&$( 6.2 , 10.9 , 0 )$&$( 5.9 , 10.4 , 0 )$&$( 5.91 , 12.54 , 1.9 )$     \\
          &   $1000$   &$( 5.1 , 10.6 , 0 )$&$( 5 , 10.1 , 0 )$&$( 5.8 , 11.7 , 0 )$&$( 9.72 , 16.53 , 0.2 )$ \\
$H_{02}$  &  $200$     &$( 5.1 , 9.8 , 0 )$&$( 5.7 , 11.2 , 0 )$&$( 7.1 , 12.7 , 0 )$&$( 4.71 , 9.73 , 0.3 )$ \\
          &  $500$     &$( 7.5 , 12.2 , 0 )$&$( 4.1 , 9.0 , 0 )$&$( 5.4 , 10.8 , 0 )$&$( 6.4 , 11.8 , 0 )$ \\
          &  $1000$    &$( 5.1 , 10.3 , 0 )$&$( 5.2 , 9.8 , 0 )$&$( 5.3 , 11.1 , 0 )$&$( 4.8 , 10.2 , 0 )$ \\
$H_{03}$  &  $200$     &$( 4.3 , 9.2 , 0 )$&$( 4.8 , 10.7 , 0 )$&$( 4.6 , 11.0 , 0 )$&$( 5.26 , 10.64 , 10.7 )$\\
          &  $500$     &$( 4.7 , 9.2 , 0 )$&$( 4.2 , 9.5 , 0 )$&$( 6.5 , 10.4 , 0 )$&$( 4.39 , 8.47 , 2.0 )$ \\
          &  $1000$    &$( 5.6 , 10.9 , 0 )$&$( 5.6 , 10.6 , 0 )$&$( 3.9 , 8.9 , 0 )$&$( 6.31 , 10.82 , 0.2 )$       \\
$H_{04}$  &  $200$     &$( 5.0 , 9.4 , 0 )$&$( 4.9 , 10.1 , 0 )$&$( 5.1 , 12.2 , 0 )$&$( 5.44 , 11.78 , 10 )$ \\
          &  $500$     &$( 6.0 , 10.6 , 0 )$&$( 5.5 , 10.7 , 0 )$&$( 6.0 , 11.3 , 0 )$&$( 4.99 , 9.28 , 1.9 )$  \\
          &  $1000$    &$( 5.4 , 10.3 , 0 )$&$( 5.7 , 10.8 , 0 )$&$( 4.8 , 10.5 , 0 )$&$( 7.41 , 12.42 , 0.2 )$ \\
\hline
\multicolumn{6}{c}{Type I errors under the Bradley--Terry model}\\
\hline
$H_{01}$  &   $200$     &$( 4.6 , 10.3 , 0 )$&$( 5.4 , 10.9 , 0 )$&$( 4.8 , 9.3 , 0 )$&$( 3.4 , 7.7 , 0 )$          \\
          &   $500$     &$( 5.4 , 10.5 , 0 )$&$( 4.2 , 8.5 , 0 )$&$( 4.8 , 9.2 , 0 )$&$( 3.7 , 8.8 , 0 )$         \\
          &   $1000$    &$( 4.3 , 9.8 , 0 )$&$( 5.5 , 9.7 , 0 )$&$( 5.1 , 9.6 , 0 )$&$( 5.4 , 10.6 , 0 )$           \\
$H_{02}$  &  $200$      &$( 5.1 , 9.1 , 0 )$&$( 4.8 , 10.7 , 0 )$&$( 3.6 , 9.4 , 0 )$&$( 4.2 , 10.5 , 0 )$         \\
          &  $500$      &$( 5.3 , 9.3 , 0 )$&$( 4.4 , 9.6 , 0 )$&$( 4.5 , 9.4 , 0 )$&$( 5.2 , 9.6 , 0 )$    \\
          &  $1000$     &$( 3.9 , 9.2 , 0 )$&$( 3.7 , 8.6 , 0 )$&$( 5.6 , 11.1 , 0 )$&$( 5 , 9.4 , 0 )$     \\
$H_{04}$  &  $200$      &$( 6 , 11.7 , 0 )$&$( 4.9 , 9.6 , 0 )$&$( 5.4 , 10.2 , 0 )$&$( 5.6 , 10 , 0 )$ \\
          &  $500$      &$( 6.1 , 9.8 , 0 )$&$( 4.9 , 10.6 , 0 )$&$( 4.7 , 9.3 , 0 )$&$( 4.9 , 10.4 , 0 )$ \\
          &  $1000$    &$( 3.9 , 9.2 , 0 )$&$( 3.7 , 8.6 , 0 )$&$( 5.6 , 11.1 , 0 )$&$( 5 , 9.4 , 0 )$ \\
\hline
\end{tabular}
\end{table}

The simulated Type I errors are reported in Table \ref{table-type-I}.
Most simulated type I errors are close to the target nominal level when $L_n\le 0.4\log n$.
The differences between simulated values
and nominal levels are relatively smaller when $n=1000$, contrary to those when $n=200$.
Conversely, the MLE failed to exist with positive frequencies in the $\beta$-model when $L_n= 0.5\log n$.
In other cases, the frequencies that the MLE does not exist are considerably small, less than $0.1$.

Next, we compare the powers of the LRTs with the Wald-type tests given in Remark 1 when the number of parameters is fixed.
We consider the homogenous hypothesis testing problems: $H_0$: $\beta_1=\cdots=\beta_r$ in the $\beta$-model and
$\beta_2=\cdots=\beta_r$ in the Bradley-Terry model ($\beta_1=0$ is the reference parameter).
The true model was set to be $\beta_i=(i-1)c/(r-1)$, $i=1,\ldots,r$.
The other parameters were set as $\beta_i=0.2(i-r)\log n/n$ for $i=r+1,\ldots,n$.
In the Bradley--Terry model, we assume that each pair has only one comparison.
The results are shown in Table \ref{powers-a}.
Notably, when $c=0$, all simulated type I errors agree reasonably well with the nominal level $0.05$.
Further, when $n$ and $r$ are fixed, and as $c$ increases, the power tends to increase and is close to $100\%$ when $c=1.2$.
A similar phenomenon can be observed when $r$ increases while $n$ and $c$ are fixed or when $n$ increases while $c$ and $r$ are fixed.
Alternatively, the powers of the LRTs are a little higher than those of the Wald-type tests. This may be because the LRT contains the full information involved with data.

Further, we conducted additional simulations in the Bradley--Terry model under the situation that imitates the schedule of the NBA regular season.
The number of nodes is $n=30$, and each pair of nodes has $3$ comparisons. Other parameters are the same as previously stated.
The results are shown in Table \ref{table-bt-powers}. This table shows that Type I errors are well controlled, and powers are visibly high when $c=1.2$. This shows that the asymptotic approximation is good even when $n$ is small.
%, as long as the number of comparisons in each pair is over $3$.

\begin{table}[h!]
\centering
\caption{Comparison between powers of the LRTs and the Wald-type tests}
\label{powers-a}
\vskip5pt
\small
\begin{tabular}{cc cc cc cc cc  cc cccc}
\hline
  $n$            & $r$  & \multicolumn{2}{c}{$c=0$}  &     & \multicolumn{2}{c}{$c=0.3$} & & \multicolumn{2}{c}{$c=0.6$} & & \multicolumn{2}{c}{$c=0.9$} & & \multicolumn{2}{c}{$c=1.2$} \\
            \cline{3-4}     \cline{6-7} \cline{9-10}  \cline{12-13} \cline{15-16}
                 &      & LRT   & Wald         &          & LRT  & Wald          &       & LRT   & Wald     &           & LRT  & Wald         &        & LRT  & Wald  \\
\hline
&&\multicolumn{14}{c}{Powers in the $\beta$-model} \\
\hline
 $100$          & $5$    &$ 5.60 $&$ 5.34  $ &&$ 13.34 $&$ 12.76  $&&$ 37.56 $&$ 36.8  $&&$ 72.04 $&$ 70.82  $&&$ 92.12 $&$ 91.48  $              \\
                & $10$   &$ 5.58 $&$ 5.32  $&&$ 13.54 $&$ 12.86  $&&$ 45.60 $&$ 44.26  $&&$ 81.80 $&$ 80.40  $&&$ 97.66 $&$ 97.32  $                \\
 $200$          & $5$    &$ 5.02 $&$ 4.94  $&&$ 34.28 $&$ 33.82  $&&$ 91.10 $&$ 90.74  $&&$ 99.80 $&$ 99.80  $&&$ 100 $&$ 100  $            \\
                & $10$   &$ 5.18 $&$ 5.04  $&&$ 22.18 $&$ 21.74  $&&$ 77.34 $&$ 77.14  $&&$ 98.82 $&$ 98.78  $&&$ 99.98 $&$ 99.98  $                \\
\hline
&&\multicolumn{14}{c}{Powers in the Bradley--Terry model}\\
\hline
$100$           & $6$    &$ 5.42 $&$ 4.92  $&&$ 9.64 $&$ 8.96  $&&$ 29.26 $&$ 27.68  $&&$ 61.28 $&$ 59.34  $&&$ 85.62 $&$ 84.14  $                           \\
                & $11$   &$ 5.30 $&$ 4.44  $&&$ 12.56 $&$ 11.02  $&&$ 42.74 $&$ 39.50  $&&$ 82.88 $&$ 81.00  $&&$ 98.24 $&$ 98.04  $                           \\
$200$           & $6$    &$ 5.32 $&$ 5.00  $&&$ 17.00 $&$ 16.26  $&&$ 53.34 $&$ 52.46  $&&$ 90.26 $&$ 89.88  $&&$ 99.28 $&$ 99.28  $                         \\
                & $11$   &$ 5.52 $&$ 5.22  $&&$ 20.04 $&$ 19.02  $&&$ 75.92 $&$ 75.10  $&&$ 99.10 $&$ 99.00  $&&$ 100 $&$ 100  $                          \\
     \hline
\end{tabular}
\end{table}

\begin{table}[h!]
\centering
\caption{Powers in the Bradley-Terry model with $n=30$ (Values in parentheses are powers of the Wald-type test).}
\label{table-bt-powers}
\vskip2pt
\small
\begin{tabular}{ccc ccc c ccccc cccccc}
\hline
 $r$    &\multicolumn{2}{c}{$c=0$}      & & \multicolumn{2}{c}{$c=0.3$} & & \multicolumn{2}{c}{$c=0.6$} & & \multicolumn{2}{c}{$c=0.9$} & & \multicolumn{2}{c}{$c=1.2$} & & \multicolumn{2}{c}{$c=1.5$}
 \\
\cline{2-3}     \cline{5-6} \cline{8-9}  \cline{11-12} \cline{14-15} \cline{17-18}
                       & LRT   & Wald         & & LRT  & Wald          & & LRT   & Wald     & & LRT  & Wald         & & LRT  & Wald  && LRT  & Wald \\
\hline
 $3$  &$ 5.44 $&$ 4.15  $&&$ 8.70 $&$ 6.93  $&&$ 18.04 $&$ 15.14  $&&$ 34.20 $&$ 29.14  $&&$ 50.33 $&$ 47.35  $&&$ 64.97 $&$ 63.10  $       \\
 $4$  &$ 5.42 $&$ 4.43  $&&$ 9.05 $&$ 7.95  $&&$ 20.49 $&$ 18.74  $&&$ 42.15 $&$ 39.10  $&&$ 65.48 $&$ 62.04  $&&$ 82.27 $&$ 80.14  $       \\
 $5$  &$ 5.36 $&$ 4.40  $&&$ 9.70 $&$ 8.28  $&&$ 23.32 $&$ 20.88  $&&$ 49.67 $&$ 45.99  $&&$ 75.51 $&$ 72.48  $&&$ 90.85 $&$ 88.87  $       \\
 $6$  &$ 5.41 $&$ 4.28  $&&$ 9.91 $&$ 8.15  $&&$ 26.40 $&$ 23.12  $&&$ 55.95 $&$ 52.05  $&&$ 83.28 $&$ 80.28  $&&$ 95.25 $&$ 94.10  $       \\
 $7$  &$ 5.24 $&$ 3.92  $&&$ 10.75 $&$ 8.59  $&&$ 29.09 $&$ 25.7  $&&$ 61.35 $&$ 57.05  $&&$ 87.39 $&$ 85.11  $&&$ 97.38 $&$ 96.62  $ \\
 $8$  &$ 5.40 $&$ 4.17  $&&$ 10.94 $&$ 8.56  $&&$ 31.76 $&$ 27.78  $&&$ 65.85 $&$ 61.97  $&&$ 91.03 $&$ 88.63  $&&$ 98.85 $&$ 98.46  $  \\
     \hline
\end{tabular}
\end{table}

\subsection{An application to the NBA data}
National Basketball Association (NBA) is one of the most successful men's professional basketball league in the world.
%%There are a total of 30 teams in NBA, which are organized into two conferences: the western conference and the eastern conference.
The current league organization divides its total thirty teams into two conferences: the Western Conference and the Eastern Conference.
In the regular season, every team plays against every other team three or four times.
It would be of interest to test whether some teams have the same merits.
Here, we use the recent 2020--21 NBA season data as an illustrative example.

The fitted merits in the Bradley--Terry model are presented in Table \ref{table-merits}, in which the Houston Rockets are the reference team.
The ranking based on the win-loss percentage and that based on the fitted merits are similar.
As shown in the simulations, the asymptotic chi-square distribution for the likelihood ratio statistic provides
good approximation, even when $n=30$.
We use the log-likelihood ratio statistic to test whether there are significant differences among the top 3 and 6 teams in respective
conferences.

As the first three teams in the Eastern Conference---``Philadelphia 76ers,'' ``Brooklyn Nets,'' and ``Milwaukee Bucks''---have similar win-loss percentages, we may want to test their equality.
By using a chi-square approximation, we get a value $0.551$ for the log-likelihood ratio with a p-value $0.759$.
For testing the equality of ``Philadelphia 76ers'' and ``Boston Celtics,'' it yields
a p-value $0.031$, showing there exists significant difference between these two teams.
For testing equality among the top 4 teams in the Western Conference, according to the ranking of the win-loss percentage,
we get a value $1.892$ for the log-likelihood ratio with a p-value $0.595$,
showing that the differences of the 4 teams do not exhibit statistical significance.
We can use the results to predict the outcomes when the Philadelphia 76ers face the Boston Celtics in the upcoming regular season.
Moreover, the champion does not show significant differences from the second and third place teams.
Therefore, caution is warranted when forecasting the champion for the next season, and this analysis provides valuable insights for betting on which team is likely to win in the postseason.

\begin{table}[h!]
\centering
\caption{Fitted merits based on the 2020--21 NBA season data. The column $\hat{\sigma}$ corresponds to standard errors.}
\label{table-merits}
\vskip2pt
\small
\begin{tabular}{l llll | llll }
\hline
& \multicolumn{4}{c|}{Eastern Conference } & \multicolumn{4}{c}{Western Conference  }\\
&    Team      &  W-L    & $\widehat{\beta}_i$   & $\hat{\sigma}$       &     Team      & W-L       & $\widehat{\beta}_i$   & $\hat{\sigma}$  \\
\hline
1 & Philadelphia        & 49-23  &$ 1.863 $&$ 0.377 $& Utah J.      & 52-20 &$ 2.164 $&$ 0.383 $ \\
2 & Brooklyn N.        & 48-24  &$ 1.804 $&$ 0.375 $& Phoenix S.   & 51-21 &$ 1.836 $&$ 0.374 $\\
3 & Milwaukee B.       & 46-26  &$ 1.617 $&$ 0.365 $& Denver N.    & 47-25 &$ 1.806 $&$ 0.369 $\\
4 & New Y. K.          & 41-31  &$ 1.383 $&$ 0.364 $& LA C.        & 47-25 &$ 1.709 $&$ 0.367 $\\
5 & Miami H.           & 40-32  &$ 1.363 $&$ 0.363 $& Los A. L.    & 42-30 &$ 1.647 $&$ 0.364 $\\
6 & Atlanta H.         & 41-31  &$ 1.355 $&$ 0.363 $& Portland T. B. &42-30 &$ 1.380 $&$ 0.364 $\\
7 & Boston C.          & 36-36  &$ 1.123 $&$ 0.365 $& Memphis G.     &38-34 &$ 1.348 $&$ 0.360 $\\
8 & Washington W.     & 34-38  &$ 1.006 $&$ 0.359 $& Golden S. W.   &39-33 &$ 1.294 $&$ 0.360 $\\
9 & Indiana P.         & 34-38  &$ 0.891 $&$ 0.361 $& Dallas M.      &42-30 &$ 1.233 $&$ 0.356 $\\
10 & Chicago B.        & 31-41  &$ 0.827 $&$ 0.363 $& New O. P.      &31-41 &$ 0.941 $&$ 0.364 $\\
11 & Charlotte H.      & 33-39  &$ 0.817 $&$ 0.362 $& San A. S.      &33-39 &$ 0.845 $&$ 0.363 $\\
12 & Toronto Raptors   & 27-45  &$ 0.616 $&$ 0.367 $& Sacramento K.  &31-41 &$ 0.758 $&$ 0.368 $\\
13 & Orlando M.        & 21-51  &$ 0.328 $&$ 0.370 $& Minnesota T.   &23-49 &$ 0.355 $&$ 0.374 $\\
14 & Cleveland C.      & 22-50  &$ 0.305 $&$ 0.373 $& Oklahoma C.    &22-50 &$ 0.317 $&$ 0.373 $\\
15 & Detroit P.        & 20-52  &$ 0.274 $&$ 0.371 $& Houston R.     &17-55 &$ 0 $ & \\
\hline
\end{tabular}

\end{table}

\subsection{Testing degree heterogeneity in network data}

We apply likelihood ratio statistics to test degree heterogeneity in four real-world network data sets,
which are the fb-pages-food network [\cite{rozemberczki2019gemsec}], email-univ network, ca-GrQc collaboration network [\cite{leskovec2007graph}], and ego-Facebook network [\cite{leskovec2012learning}].
The former three networks are openly available from \url{https://networkrepository.com/index.php} [\cite{networkdata2015}] while the last one can be downloaded from \url{https://snap.stanford.edu/data/ego-Facebook.html}.
The fb-pages-food network is a social network with $620$ nodes and $2,102$ edges, where nodes represent the pages, and edges are mutual likes among them.
The email-univ network is an email network of $1,133$ users, where edges denote email communications between users.
The ca-GrQc collaboration network is a collaboration network of $5,242$ authors, with edges indicating scientific collaborations.
The ego-Facebook network comprises $4,039$ individuals, with edges representing friendships.

There is obvious degree heterogeneity among all nodes because the largest degree is more than $100$ times than the smallest degree in all the four networks.
%Thus, we test whether there is degree heterogeneity in the subset of nodes with the $r$ smallest id number.
We test whether there is degree heterogeneity in a subset of nodes.
When degree homogeneity among a set of nodes is rejected, it indicates that the local network structure does not follow an Erd\H{o}s-R\'{e}nyi random graph.
Conversely, if homogeneity is not rejected, it can enhance the quality of predictions by grouping nodes with similar connection tendencies or merits.
Since we apply the test to four different real-world data sets, for convenience and illustration, we choose the subset to be the nodes with the r smallest ID numbers in a unified manner.
%Since we study four different real-world data sets, we choose nodes with the $r$ smallest ID numbers for illustrating examples in a unified manner.
We could also test other sets of nodes of interest.
We chose two values for $r$---$r=10$ and $r=100$.
Table \ref{table-networks} shows p-values in these networks, as well as the largest and smallest degrees among the $r$ nodes.
All these p-values, except for the one obtained from testing degree homogeneity of the $10$ nodes in the  fb-pages-food network,
are less than $10^{-4}$, indicating  significant effects of degree heterogeneity .
This reflects the diversity of nodes in observed networks.

\begin{table}[h!]
\centering
\begin{threeparttable}
\caption{Testing degree heterogeneity in several network data sets.}
\label{table-networks}
\vskip2pt
\small
\begin{tabular}{lll  lll l lll }
\hline
          &         &         &     \multicolumn{3}{c}{  r=10 } &  & \multicolumn{3}{c}{  $r=100$  }\\
          \cline{4-6} \cline{8-10}
Data set       &  $|V|$   & $|E|$    & $d_{r,\max}$  &  $d_{r,\min}$   &   p-value  &  & $d_{r,\max}$  &  $d_{r,\min}$  & p-value  \\
\hline
fb-pages-food  &   $620$  &  $2102$  & $10$          & $1$             &   $ 0.149$           & & $56$          & $1$             & $<10^{-5}$  \\
email-univ     &   $1133$  & $5451$  & $38$          & $8$             &   $2\times 10^{-5}$   & &  $51$         & $1$             & $<10^{-5}$  \\
ca-GrQc        &  $5242$  & $28980$   & $58$        & $2$               &   $<10^{-5}$         & &  $162$        & $2$  & $<10^{-5}$  \\
ego-Facebook       &   $4039$ & $88234$  &  $347$       & $6$               &   $<10^{-5}$         & &  $1045$       & $2$  & $<10^{-5}$   \\
\hline
\end{tabular}
\begin{tablenotes}[para, flushleft]
\item $d_{r,\max}$  denotes the largest degree, and $d_{r,\min}$ is the smallest degree among the given $r$ nodes.
\end{tablenotes}
\end{threeparttable}
\end{table}

\section{Discussion}
\label{section:discussion}

We study asymptotic behaviors of LRTs in the $\beta$-model and the Bradley--Terry model with a diverging number of nodes. %increasing dimensions.
The Wilks-type theorems have been established for fixed and increasing dimensional hypothesis testing problems.
Notably, the conditions imposed on $b_{n}$ and $c_n$ may not be the best possible.
The simulation results indicate that asymptotic approximations can still be effective for a wide range of $b_n$, and
%there are good asymptotic approximations when these conditions are violated, and
%asymptotic approximations can still be effective even when some of these conditions are not satisfied
the chi-square approximation for the LRT is comparable to the normal approximation in increasing dimensional hypothesis testing problems.
Therefore, it is of interest to investigate whether the conditions in theorems could be relaxed and to study whether
 the Kolmogorov-Smirnov distance between the distribution of $2[\ell(\boldsymbol{\widehat{\beta}}) - \ell(\boldsymbol{\widehat{\beta}}^0)]$
and the chi-square distribution with large degrees of freedom $r$, i.e., $\sup_x |\mathbb{P}( \ell(\boldsymbol{\widehat{\beta}}) - \ell(\boldsymbol{\widehat{\beta}}^0) \le x ) - \mathbb{P}( X_r \le x)| $,
tends to zero and calculate its error rate, where $X_r$ is a chi-square random variable with $r$ degrees of freedom.

%Although we present Wilks-type theorems only in the $\beta$-model and the Bradley--Terry model,
Our paper sheds light on how to explore Wilks-type results in a principled manner.
These principled methods should be applicable to a class of random graph models.
Specifically, the $\beta$-model and the Bradley--Terry model can be recast into a unified framework.
Following \cite{Rinaldo2013}, %we propose a general framework that encompasses these two models as special cases.
we assume that the occurrence of edges is observed \(K_{ij}\) times for each pair of nodes \(i\) and \(j\) with \(i < j\). Let \(a_{ij}\) represent the number of times that the edge \((i,j)\) appears out of \(K_{ij}\) samples, while \(a_{ji} = K_{ij} - a_{ij}\) denotes the number of times the edge is missing. In undirected graphs, \(a_{ji}\) indicates the absence of the edge \((i,j)\). In paired comparison data, \(a_{ij}\) represents the number of times that \(i\) is preferred over \(j\) out of \(K_{ij}\) comparisons for any \(i \neq j\).
Suppose that the probability of observing \(a_{ij}\) has the following mass function:
\[
\mathbb{P}(a_{ij} = a) = f(a | \beta_i, \beta_j),
\]
where \(f(\cdot)\) is a probability mass function that depends on the node-specific parameters \(\beta_i\) and \(\beta_j\).
When \(f(\cdot)\) takes the form of a logistic function, with \(\beta_i\) and \(\beta_j\) entering \(f\) in either an additive or subtractive manner, it becomes the \(\beta\)-model or the Bradley--Terry model, respectively.
Establishing a cohesive theoretical framework for likelihood ratio test theories in the unified model is quite challenging.
As discussed in section \ref{subsec-61},
there are several significant differences in the technical steps required to prove the Wilks-type theorems (e.g., Theorem 1 in the \(\beta\)-model and Theorem 3 in the Bradley--Terry model). %These differences include distinct approximate inverses for the Fisher information matrices, different asymptotic representations of the MLE and restricted MLE, and different methods for obtaining consistency rates.
Given that the proofs for the Wilks-type results are already highly non-trivial and lengthy, it is beyond of this paper to investigate this issue. We would like to defer it for future work.

We only consider dense paired comparisons, in which all pairs have comparisons. This assumption holds some practical applications. For example, the Major League Baseball schedule in the United States and Canada arranges that all teams play each other in a regular season.
In some other applications, not all possible comparisons are available. For example, some games might be cancelled owing to bad weather. If only a few comparisons are unavailable, it has little impact on the results developed in this paper. An interesting scenario is that paired comparisons are sparse, in which numerous items do not have direct comparisons.
%The errors for the MLEs depend crucially on the sparse condition.
%This has impact on the remainder terms in the asymptotic expansion of the log-likelihood function.
The sparse condition has an impact on the existence of the MLE, as well as asymptotic behaviors of LRTs.
Extending the method to handle extremely sparse paired comparisons does not appear to be a trivial task. It is of interest to investigate this problem.

\section{Appendix}
\label{section:proof}

%\section{Proofs for Theorems \ref{theorem-LRT-beta} and \ref{theorem-LRT-beta-fixed}}
%\label{section:proof}

In this section, we present proofs for Theorems \ref{theorem-LRT-beta} (a).
The proofs of Theorems \ref{theorem-LRT-beta} (b), \ref{theorem-LRT-beta-fixed}, \ref{theorem-ratio-bt-3} and \ref{theorem-ratio-bt-fixed} are presented in the Supplementary Material.

We introduce some notations.
For a vector $\mathbf{x}=(x_1, \ldots, x_n)^\top\in \R^n$, denote by $\|\mathbf{x}\|$ for a general norm on vectors with the special cases
$\|\mathbf{x}\|_\infty = \max_{1\le i\le n} |x_i|$ and $\|\mathbf{x}\|_1=\sum_i |x_i|$ for the $\ell_\infty$- and $\ell_1$-norm of $\mathbf{x}$ respectively.
For an $n\times n$ matrix $J=(J_{ij})$, let $\|J\|_\infty$ denote the matrix norm induced by the $\ell_\infty$-norm on vectors in $\R^n$, i.e.,
\[
\|J\|_\infty = \max_{\mathbf{x}\neq 0} \frac{ \|J\mathbf{x}\|_\infty }{\|\mathbf{x}\|_\infty}
=\max_{1\le i\le n}\sum_{j=1}^n |J_{ij}|,
\]
and $\|J\|$ be a general matrix norm. $\|J\|_{\max}$ denotes the maximum absolute entry-wise norm, i.e., $\|J\|_{\max}=\max_{i,j} |J_{ij}|$.
The  notation $f(n)=O\left(g(n)\right)$ or
$f(n)\lesssim g(n)$ means  there is a constant $c>0$ such
that $\left|f(n)\right|\leq c|g(n)|$. $f(n) \asymp g(n)$ means that $f(n)\lesssim g(n)$ and $g(n)\lesssim f(n)$.
%there exist two constants $c_1>0$ and $c_2>0$ such that $c_1|f(n)| \le |g(n)| \le c_2|f(n)|$.
$f(n)=o(g(n))$ means  $\lim_{n\rightarrow\infty}f(n)/g(n)=0$.
The notation $\sum_{j<i}$  is a shorthand for $\sum_{i=1}^n \sum_{j=1}^{i-1}$.

\iffalse
, $f(n)=\Omega\left(g(n)\right)$
or $f(n)\gtrsim g(n)$ means  there is a constant $c>0$ such
that $|f(n)|\geq c\left|g(n)\right|$, $f(n)=\Theta\left(g(n)\right)$
or $f(n)\asymp g(n)$ means that there exist constants $c_{1},c_{2}>0$
such that $c_{1}|g(n)|\leq|f(n)|\leq c_{2}|g(n)|$, and $f(n)=o(g(n))$
means  $\lim_{n\rightarrow\infty}\frac{f(n)}{g(n)}=0$.
\fi

We define a matrix class $\mathcal{L}_n(m, M)$ with two positive numbers $m$ and $M$.
We say an $n\times n$ matrix $V=(v_{ij})$ belongs to the matrix class
$\mathcal{L}_n( m, M)$ if
\[
v_{ii}=\sum_{j\neq i} v_{ij},  i=1, \ldots, n;\quad
m \le v_{ij} \le M,  i,j=1, \ldots, n; i\neq j.
\]
\iffalse
\[
\begin{array}{cl}
v_{ii}=\sum_{j\neq i} v_{ij}, & i=1, \ldots, n \\
m \le v_{ij} \le M, & i,j=1, \ldots, n; i\neq j.
\end{array}
\]
\fi
Define two diagonal matrices:
\begin{equation}\label{definition-S}
S=\mathrm{diag}(1/v_{11}, \ldots, 1/v_{nn}),\quad  S_{22}=\mathrm{diag}(1/v_{r+1,r+1}, \ldots, 1/v_{nn}),
\end{equation}
where $S_{22}$ is the bottom right $(n-r)\times (n-r)$ block of $S$ for $r\in\{ 1, \ldots, n-1\}$.
\cite{Yan:Xu:2013} proposed to use the diagonal matrix $S$ to approximate $V^{-1}$.
\begin{lemma}\label{lemma-appro-beta-VS}
For $V\in \mathcal{L}_n(1/b_n, 1/c_n)$ with $n\ge 3$ and its bottom right $(n-r)\times (n-r)$ block $V_{22}$ with $r\in\{1,\ldots, n-1\}$, we have
\begin{equation}\label{ineq-V-S-appro-upper-b}
\max\{ \|V^{-1} - S \|_{\max}, \| V_{22}^{-1} - S_{22} \|_{\max} \} \le \frac{2b_n^2 }{ c_n (n-1)^2 } \left( \frac{nb_n}{2(n-2)c_n} + \frac{1}{2} \right).
\end{equation}
\end{lemma}
Lemma \ref{lemma-appro-beta-VS} is an extension of Proposition 1 in \cite{Yan:Xu:2013}. % and presented in the Supplementary Material B.
In Theorem 6.1 of \cite{hillar2012inverses}, they obtained a tight upper bound of $\|J\|_\infty$ for symmetric diagonally dominant $m\times m$ dimensional  matrices $J$.
\iffalse
 satisfying $J \ge J(\alpha, \ell)= \alpha I_m + \ell \mathbf{1}_m \mathbf{1}_m^\top$:
\begin{equation*} %\label{ineq-tight-V}
\|J^{-1}\|_\infty  \le \|[J(\alpha,\ell)]^{-1}\|_\infty \le \frac{ \alpha + 2 \ell (m-1) }{ \alpha (\alpha + \ell m ) },
\end{equation*}
where $A \ge B$ means $A - B$ is a nonnegative matrix, $\alpha\ge (m-2)\ell$,  $I_m$ denotes the $m\times m$ identity matrix, and
$\mathbf{1}_m$ denotes the $m$-dimensional column vector consisting of all ones.
\fi
As applied here, we have that for $V\in \mathcal{L}_n(1/b_n, 1/c_n)$ with $n\ge 3$ and its bottom right $(n-r)\times (n-r)$ block $V_{22}$ with $r\in\{1, \ldots, n-r\}$,
\begin{equation}\label{ineq-tight-V}
 \|V^{-1}\|_\infty \le \frac{3b_n}{2n-1}, \quad  \|V_{22}^{-1} \|_\infty \le \frac{ b_n}{ n-1 }\left(1+ \frac{n-r-2}{2n-r-1} \right)\le \frac{3b_n}{2(n-1)}.
\end{equation}
\iffalse
With similar arguments as in the proof of \cite{hillar2012inverses},  for , we have
\begin{equation}\label{ineq-tight-V2}
 \|V_{22}^{-1} \|_\infty \le \frac{ b_n}{ n-1 }\left(1+ \frac{n-r-2}{2n-r-1} \right).
\end{equation}
\fi
It is noteworthy that the upper bounds in \eqref{ineq-V-S-appro-upper-b} and \eqref{ineq-tight-V} are independent of $r$.
This property implies some remainder terms in the proofs of Theorems \ref{theorem-LRT-beta}  are in regardless of $r$. %and \ref{theorem-LRT-beta-fixed}

\iffalse
Theorem 6.1. Let $n \ge 3$ and suppose $S = \alpha I_n + \ell \mathbf{1}_n \mathbf{1}_n^\top$
is diagonally dominant with $\alpha, \ell> 0$. For all $n\times n$ symmetric diagonally dominant matrices $J \ge S$, we have
\[
\|J^{-1}\|_\infty \le \| S \|_\infty \le \frac{ \alpha + 2 \ell (n-1) }{ \alpha (\alpha + \ell n ) }
\]
\frac{c_n}{2(n-1)} \le \|V^{-1}\|_\infty \le \frac{(3n-4)b_n}{2(n-1)(n-2)}
\fi
%Furthermore, equality is achieved if and only if J = S.
\iffalse
\cite{Yan:Xu:2013} proposed to use a diagonal matrix $S=\mathrm{diag}(1/v_{11}, \ldots, 1/v_{nn})$ to approximate $V^{-1}$, where
\begin{equation}\label{ineq-V-S-appro-upper-b}
\|V^{-1} - S \|_{\max} \lesssim \frac{b_n^3}{n^2c_n}.
\end{equation}
\fi
\iffalse
The diagonal matrix
\begin{equation}\label{definition-S}
S=\mathrm{diag}(1/v_{11}, \ldots, 1/v_{nn}),
\end{equation}
can be used to approximate $V^{-1}$ [\cite{Yan:Xu:2013}].
%as an approximation for further simplification.
\fi
\iffalse
\cite{Yan:Xu:2013} proposed to use a simple matrix $\mathrm{diag}(1/v_{11}, \ldots, 1/v_{nn})+1/v_{\cdot\cdot}$ to approximate $V^{-1}$, where
$v_{\cdot\cdot}=\sum_i v_{ii}$. Since $1/v_{\cdot\cdot} =  O(b_n/n^2)$, we use the diagonal matrix
\begin{equation}\label{definition-S}
S=\mathrm{diag}(1/v_{11}, \ldots, 1/v_{nn}),
\end{equation}
as an approximation for further simplification. They proved
\begin{equation}\label{ineq-V-S-appro-upper-b}
\|V^{-1} - S \|_{\max} \lesssim \frac{b_n^3}{n^2}.
\end{equation}
\fi

We define a function $\mu(x) =  e^x/(1 + e^x)$ and a notation $\pi_{ij}=\beta_i+\beta_j$ for easy of exposition.
A direct calculation gives that the derivative of $\mu(x)$ up to the third order are
\begin{eqnarray}\label{eq-derivative-mu-various}
\mu^\prime(x) = \frac{e^x}{ (1+e^x)^2 },~~  \mu^{\prime\prime}(x) = \frac{e^x(1-e^x)}{ (1+e^x)^3 },~~ \mu^{\prime\prime\prime}(x) =  \frac{ e^x [ (1-e^x)^2 - 2e^x] }{ (1 + e^x)^4 }.
\end{eqnarray}
According to the definition of $c_n$ in \eqref{definition-bncn}, we have the following inequalities:
\begin{equation}\label{ineq-mu-deriv-bound}
|\mu^\prime(\pi_{ij})| \le \frac{1}{c_n}, ~~ |\mu^{\prime\prime}(\pi_{ij})| \le \frac{1}{c_n},~~ |\mu^{\prime\prime\prime}(\pi_{ij})| \le \frac{1}{c_n}.
\end{equation}
The above inequalities will be used in the proofs repeatedly.
Recall that $\bar{a}_{ij} = a_{ij} - \E(a_{ij})$. Define $\bar{a}_{ii}=0$ for all $i=1, \ldots, n$.
Correspondingly, denote $\bar{d}_i = d_i - \E(d_i)$ and $\bs{\bar{d}}=(\bar{d}_1, \ldots, \bar{d}_n)^\top$.

\subsection{Proof outline of Theorems \ref{theorem-LRT-beta}-\ref{theorem-ratio-bt-fixed}}
\label{subsec-61}

We first describe the idea for proving Theorem \ref{theorem-LRT-beta}  briefly.
We apply a fourth-order Taylor expansion to  $\ell(\boldsymbol{\widehat{\beta}})$
and $\ell(\boldsymbol{\widehat{\beta}}^0)$ at point $\bs{\beta}$, respectively.
With the use of the maximum likelihood equations and the asymptotic representations of $\widehat{\bs{\beta}}$ and $\widehat{\bs{\beta}}^0$
(see \eqref{eq-expansion-hatbeta-beta} and \eqref{eq-beta0-exapnsion}),
the first-order and second-order expansion terms in the difference $\ell(\boldsymbol{\widehat{\beta}})-\ell(\boldsymbol{\widehat{\beta}}^0)$
can be expressed as the difference between
$\bs{\bar{d}}^\top V^{-1} \bs{\bar{d}}$ and $\bs{\bar{d}}_2^\top V_{22}^{-1} \bs{\bar{d}}_2$
($\bs{\tilde{d}}^\top \widetilde{V}^{-1}\bs{\tilde{d}}$ under the homogenous null; see (39) in Supplementary Material A) % \eqref{B10-homo-expression})
and several remainder terms,
where $V_{22}$ is the  bottom right  $(n-r)\times (n-r)$ block of $V$,
$\bs{\bar{d}}= \mathbf{d} - \E \mathbf{d}$ and $\bs{\bar{d}}_2$ is the last $n-r$ elements of $\bs{\bar{d}}$.
%With the use of the diagonal matrix $S=\diag(1/v_{11}, \ldots, 1/v_{nn})$
%and $S_{22}=\diag(1/v_{r+1,r+1}, \ldots, 1/v_{nn})$  in \eqref{definition-S} to respectively approximate $V^{-1}$ and $V_{22}^{-1}$,
By using diagonal matrices $S$ and $S_{22}$ in \eqref{definition-S} to respectively approximate $V^{-1}$ and $V_{22}^{-1}$,
one can find that $\bs{\bar{d}}^\top V^{-1} \bs{\bar{d}}-\bs{\bar{d}}_2^\top V_{22}^{-1} \bs{\bar{d}}_2$ can be divided into
the quadratic sum of a normalized degrees, i.e.,
$\sum_{i=1}^r \bar{d}_i^{\,2}/v_{ii}$, whose central limit theorem is established in Lemma \ref{lemma:weighte-degree-al}, and several remainder terms that can be bounded by using approximate errors in Lemma \ref{lemma-appro-beta-VS}.
The idea for proving Lemma \ref{lemma:weighte-degree-al} is described before and we do not repeat it here.
The third-order and fourth-order expansion terms in the difference $\ell(\boldsymbol{\widehat{\beta}})-\ell(\boldsymbol{\widehat{\beta}}^0)$
are bounded by consistency rates of the restricted MLE $\widehat{\bs{\beta}}^0$ in Lemma \ref{lemma-consi-beta} and
a very small bound of a weighted cubic sum of $\{  \widehat{\beta}_i - \beta_i \}_{i=1}^n$ in Lemma \ref{lemma:beta3:err}.
%As mentioned before, this small bound has a vanishing factor $n^{-1/2}$ in contrast to its weighted absolute cubic sum.

The principled strategy for proving Theorem \ref{theorem-LRT-beta} described above is extended to prove Theorems \ref{theorem-LRT-beta-fixed}, \ref{theorem-ratio-bt-3} and \ref{theorem-ratio-bt-fixed}. Here, we emphasize the different places.
In contrast to  the proof of Theorem \ref{theorem-LRT-beta}, the proof of Theorem \ref{theorem-LRT-beta-fixed} needs additional technical steps.
It requires to bound $\max_{i=r+1,\ldots,n}| \widehat{{\beta}}_i - \widehat{{\beta}}^0_i |$ given in Lemmas 12 %\ref{lemma-hat-beta-diff}
and 15, %\ref{lemma-hat-beta-diff-2b},
and to evaluate the maximum absolute entry-wise difference between two approximate inverse matrices,
which are stated in Lemmas 10 and 13 of Supplementary Material A. %Lemmas 10--15 are presented in Supplementary Material A. % \ref{lemma:w2-error} and 13. % and \ref{lemma:w2-error-2b}.
Further, it needs to carefully analyze the differences between remainder terms under the
full space and the null space because we do not have a scaled vanishing factor $r^{-1/2}$, in contrast to the proof of Theorem \ref{theorem-LRT-beta}.

% The principled strategy for proving Theorem \ref{theorem-LRT-beta} is extended to prove the above theorem.
The main different technical steps for the proofs of Theorems \ref{theorem-ratio-bt-3} and \ref{theorem-ratio-bt-fixed}
are three-fold.
%include different approximate inverses for the Fisher information matrices under the null space, different asymptotic representations of the MLE and restricted MLE and different methods for obtaining consistency rates. Roughly speaking,
First, we use a diagonal matrix to approximate the Fisher information matrix in the $\beta$-model while the approximate inverse is
a diagonal matrix plus a commonly exceptive number in the Bradley--Terry model.
Second, the main term in the asymptotic representation of $\widehat{\beta}_i$ is $\bar{d}_i/v_{ii}$ in the $\beta$-model while
it is $\bar{d}_i/v_{ii} - \sum_{i=1}^r \bar{d}_i /\tilde{v}_{11}$ under the specified null or $\bar{d}_i/v_{ii} - \bar{d}_1 /v_{11}$
in the homogenous null in the Bradley--Terry model, where $\tilde{v}_{11}=\sum_{i=1}^r v_{ii}$.
Third, the methods for bounding errors of the MLEs or its restricted version under the null are different.
We obtain consistency rates in the $\beta$-model via deriving a geometrical convergence rate of the Newton iterative sequence that solves the MLE.
In the Bradley-Terry model, as in \cite{simons-yao1999}, we use
the common neighbors between any two of items as middleman, who have
ratios being simultaneously close to $\max_i (\widehat{\beta}_i - \beta_i)$ and $\min_i (\widehat{\beta}_i - \beta_i)$, to establish the error bound of the MLE.

\subsection{Proofs for Theorem \ref{theorem-LRT-beta} (a)}
\label{section:theorem1-a}

%In this section, we will suppress the superscript $0$ in $\bs{\beta}^0$ for convenience when causing no confusions and
%denote $\bs{\bar{d}}_2=(\bar{d}_{r+1}, \ldots, \bar{d}_n)^\top$.
%Recall that $V = -\partial^2 \ell( \bs{\beta} )/\partial \bs{\beta} \partial \bs{\beta}^\top$,
%where  the elements $v_{ij}$, $i,j=1,\ldots, n$ of $V$ are given in \eqref{definition-v-beta}.
To prove Theorem \ref{theorem-LRT-beta} (a), we need three lemmas below.

\begin{lemma} \label{lemma-clt-beta-W}
Recall that $V_{22}$ is the bottom right $(n-r)\times (n-r)$ block of $V$.
Let $W=V^{-1}-S$,
$\widetilde{W}_{22}=V_{22}^{-1}-S_{22}$ and $\bs{\bar{d}}_2=(\bar{d}_{r+1}, \ldots, \bar{d}_n)^\top$. For any given $r\in \{0, \ldots, n-1\}$, we have
\[
\bs{\bar{d}}_2^\top \widetilde{W}_{22} \bs{\bar{d}}_2 = O_p\left( \frac{b_n^3}{c_n^3}(1-\frac{r}{n})^3 \right),
\]
where $r=0$ implies $\bs{\bar{d}}_2=\bs{\bar{d}}$, $V_{22}=V$ and $\widetilde{W}_{22}=W$.
\end{lemma}

Lemma \ref{lemma-clt-beta-W} shall be used to bound the remainder terms $\bs{\bar{d}}_2^\top \widetilde{W}_{22} \bs{\bar{d}}_2$ and $\bs{\bar{d}}^\top W \bs{\bar{d}}$ in \eqref{eq-theorem2-B10}.
%states that the remainder terms $\bs{\bar{d}}_2^\top \widetilde{W}_{22} \bs{\bar{d}}_2$ and $\bs{\bar{d}}^\top W \bs{\bar{d}}$ in \eqref{eq-theorem2-B10}
% is in the order of $O_p\left( (b_n^3/c_n^3)(1-r/n)^3 \right)$ for any given $r\ge 0$.

\begin{lemma}\label{lemma-consi-beta}
Under the null $H_0: (\beta_1, \ldots, \beta_r)=(\beta_1^0, \ldots, \beta_r^0)$ for any given $r\in\{0, \ldots, n-1\}$,
if
\iffalse
\begin{equation}\label{con-lemma3-beta-c}
\frac{ b_n^2 }{ c_n } = o\left( \frac{n}{(n-r)} \times \sqrt{ \frac{ n }{\log n} } \right),
\end{equation}
\fi
$
 b_n^2/c_n  = o\left( n/(n-r) \times (n/\log n)^{1/2} \right),
$
 then with probability at least $1-2/n$, the restricted MLE $\widehat{\bs{\beta}}^0$ exists and satisfies
\begin{equation*}\label{ineq-En-beta}
\| \widehat{\bs{\beta}}^0 - \bs{\beta} \|_\infty  \le  \frac{3nb_n}{(2n-1)}\sqrt{ \frac{\log n}{n} },
\end{equation*}
where $r=0$ means % there is no any restriction on $\bs{\beta}$ and implies
$\widehat{\bs{\beta}}^0=\widehat{\bs{\beta}}$. %% and $\widehat{\bs{\beta}}^0$ is the MLE in this case.
Further, if the restricted MLE exists, it must be unique.
\end{lemma}

%Lemma \ref{lemma-consi-beta} gives the error bound for the r MLE $\bs{\widehat{\beta}}^0$.

From Lemma \ref{lemma-consi-beta}, we can see that the consistency rate for the restricted MLE $\bs{\widehat{\beta}}^0$ in terms of the $L_\infty$-norm
is independent of $r$ while the condition depends on $r$. %The larger $r$ is, the weaker the condition is. When $r=0$, the lemma gives the error bound for the MLE $\bs{\widehat{\beta}}$.
When $b_n$ is a constant, this corresponds to the assumption in \cite{Chatterjee:Diaconis:Sly:2011} and the $L_\infty$-norm error bound of the MLE reduces to their error bound.

\begin{lemma}
\label{lemma:beta3:err}
%Recall $\pi_{ij}=\beta_i+\beta_j$.
%Suppose that $\beta_1, \ldots, \beta_r$ are known with a given $r\in\{0, \ldots, n-1\}$.
If $b_n^2/c_n  = o\left( n/(n-r) \times (n/\log n)^{1/2} \right)$,  then for an arbitrarily given $r\in\{0, \ldots, n-1\}$, %%with probability at least $1-O(1/n)$, we have the following bounds:
\begin{eqnarray*}
\sum_{i=r+1}^n  (\widehat{\beta}_i-\beta_i)^3 \sum_{j=1,j\neq i}^n \mu^{\prime\prime}( \beta_i+\beta_j ) & = & O_p\left( \frac{b_n^4\log n}{ c_n^2} \left(\frac{n-r}{n}\right)^{1/2} \right), \\
\sum_{i,j=r+1, j\neq i}^n  (\widehat{\beta}_i-\beta_i)^2(\widehat{\beta}_j-\beta_j)\mu^{\prime\prime}( \beta_i+\beta_j ) & = & O_p
\left( \frac{ (n-r)b_n^5 (\log n)^2 }{ nc_n^2}  \right).
\end{eqnarray*}
If $\widehat{\beta}_i$ is replaced with $\widehat{\beta}_i^0$ for $i=r+1,\ldots,n$, then the above upper bound still holds.
\end{lemma}

If all $\beta_i$s are positive constant and we directly use the error bound  in \eqref{ineq-En-beta} for $\| \bs{\widehat{\beta}} - \bs{\beta}\|_\infty$,
the summation in the above lemma will be bounded above by $(n\log n)^{1/2}$ that does not tend to zero.  To improve it,
%The proof of Lemma \ref{lemma:beta3:err}
we uses the asymptotic representation of $\bs{\widehat{\beta}}$ in \eqref{eq-expansion-hatbeta-beta},
which leads to that the summarization is involved with a main term having the form of the weighted cubic sum $\sum_i \bar{d}_i^{\,3}/v_{ii}^3$. The variance of $\bar{d}_i^{\,3}$ is in order of $n^3$,
although $\E \bar{d}_i^{\,6}$ contains $n^6$ mixed items for $\bar{a}_{ij}$.
Lemma \ref{lemma:beta3:err} plays an important role in \eqref{eq-ell-difference} for proving $B_2/r^{1/2}\to 0$.

We are now ready to prove the first part of Theorem \ref{theorem-LRT-beta}.

\begin{proof}[Proof of Theorem \ref{theorem-LRT-beta} (a)]
Under the null $H_0: (\beta_1, \ldots, \beta_r)=(\beta_1^0, \ldots, \beta_r^0)$, the data generating parameter $\bs{\beta}$ is equal to $(\beta_1^0, \ldots, \beta_r^0, \beta_{r+1}, \ldots, \beta_n)^\top$.
For convenience, we suppress the superscript $0$ in $\beta_i^0, i=1,\ldots,r$ when causing no confusion.
The following calculations are based on the event $E_n$ that $\bs{\widehat{\beta}}$ and $\bs{\widehat{\beta}}^0$ simultaneously exist and satisfy
\begin{equation}\label{ineq-beta-beta0-upp}
\max\left\{ \| \widehat{\bs{\beta}} - \bs{\beta} \|_\infty, \| \widehat{\bs{\beta}}^0 - \bs{\beta} \|_\infty \right\}
\le \frac{3nb_n}{(2n-1)}\sqrt{ \frac{\log n}{n} }.
\end{equation}
By Lemma \ref{lemma-consi-beta}, $\PP(E_n) \ge 1 - O(n^{-1})$ if $b_n^2/c_n=o\left\{ (n/\log n)^{1/2} (1-r/n)^{-1}\right\}$.

Applying a fourth-order Taylor expansion to $\ell(\widehat{\bs{\beta}} ) $ at point $ \bs{\beta}$, it yields
\begin{eqnarray*}
\ell(\widehat{\bs{\beta}} ) - \ell( \bs{\beta} ) & = &
\underbrace{
\frac{\partial \ell( \bs{\beta} ) }{ \partial \bs{\beta}^\top } ( \widehat{\bs{\beta}} - \bs{\beta} ) +
\frac{1}{2} ( \widehat{\bs{\beta}} - \bs{\beta} )^\top \frac{ \partial^2 \ell( \bs{\beta} ) }{ \partial \bs{\beta} \bs{\beta}^\top } ( \widehat{\bs{\beta}} - \bs{\beta} ) }_{B_1} \\
& & +  \frac{1}{6} \underbrace{\sum_{i=1}^n \sum_{j=1}^n \sum_{k=1}^n \frac{ \partial^3 \ell(\bs{\beta})}{ \partial \beta_i \partial \beta_j \partial \beta_k }
( \widehat{\beta}_i - \beta_i)( \widehat{\beta}_j - \beta_j)( \widehat{\beta}_k - \beta_k) }_{B_2} \\
&& +  \frac{1}{4!} \underbrace{ \sum_{t=1}^n \sum_{i=1}^n \sum_{j=1}^n \sum_{k=1}^n \frac{ \partial^4 \ell(\bs{\tilde{\beta}})}{ \partial \beta_t \partial \beta_i \partial \beta_j \partial \beta_k } ( \widehat{\beta}_t - \beta_t)
( \widehat{\beta}_i - \beta_i)( \widehat{\beta}_j - \beta_j)( \widehat{\beta}_k - \beta_k) }_{B_3},
\end{eqnarray*}
where $\bs{\tilde{\beta}} = t \bs{\beta} + (1-t ) \bs{\widehat{\beta}}$ for some $t\in(0,1)$.
Correspondingly, $\ell(\widehat{\bs{\beta}}^0 )$ has the following expansion:
\begin{equation*}
\ell(\widehat{\bs{\beta}}^0 ) - \ell( \bs{\beta} ) = B_1^0 + \frac{1}{6}B_2^0 + \frac{1}{4!}B_3^0,
\end{equation*}
where $B_i^0$ is the version of $B_i$ with $\widehat{\bs{\beta}}$ replaced by $\widehat{\bs{\beta}}^0$.
Therefore,
\begin{equation}\label{eq-ell-difference}
2\{ \ell(\widehat{\bs{\beta}} ) - \ell(\widehat{\bs{\beta}}^0 ) \}
= 2( B_1 - B_1^0)  + \frac{1}{3}(B_2 - B_2^0) + \frac{1}{12}(B_3 - B_3^0).
\end{equation}
Recall that $\partial \ell( \bs{\beta} )/\partial \bs{\beta}^\top   =  \bs{d} - \E \bs{d}$ and
$V= - \partial^2 \ell( \bs{\beta} )/\partial \bs{\beta} \partial \bs{\beta}^\top $.
\iffalse
\begin{eqnarray*}
\frac{\partial \ell( \bs{\beta} ) }{\partial \bs{\beta}^\top }  =  \bs{d} - \E \bs{d},
~~
V= - \frac{\partial^2 \ell( \bs{\beta} )
}{\partial \bs{\beta} \partial \bs{\beta}^\top }.
\end{eqnarray*}
\fi
Then, $B_1$ can be written as
\begin{equation}
\label{lrt-a-beta-B1}
B_1  =   ( \bs{\widehat{\beta} } - \bs{\beta})^\top \bs{\bar{d}}
- \frac{1}{2} ( \bs{\widehat{\beta} } - \bs{\beta})^\top V( \bs{\widehat{\beta} } - \bs{\beta}).
\end{equation}
By direct calculations, $B_2$ and $B_3$ have the following expressions:
\begin{eqnarray}
\label{lrt-a-beta-B2}
-B_2  & = &  \sum_{i=1}^n  (\widehat{\beta}_i-\beta_i)^3 \sum_{j=1,j\neq i}^n \mu^{\prime\prime}( \pi_{ij} )
+ 3 \sum_{i,j=1, j\neq i}^n  (\widehat{\beta}_i-\beta_i)^2(\widehat{\beta}_j-\beta_j)\mu^{\prime\prime}( \pi_{ij} ), \\
\nonumber
-B_3 & = &  \sum\limits_{i=1}^n  (\widehat{\beta}_i -\beta_i)^4 \sum\limits_{j=1,j\neq i}^n \mu^{\prime\prime\prime}( \bar{\pi}_{ij} )
+4\sum\limits_{i=1}^n \sum\limits_{j=1,j\neq i}^n \mu^{\prime\prime\prime}( \bar{\pi}_{ij} )  (\widehat{\beta}_i -\beta_i)^3(\widehat{\beta}_j -\beta_j) \\
\label{lrt-a-beta-B3}
&& + 3 \sum\limits_{i=1}^n \sum\limits_{j=1, j\neq i}^n  \mu^{\prime\prime\prime}( \bar{\pi}_{ij} )(\widehat{\beta}_i-\beta_i)^2(\widehat{\beta}_j-\beta_j)^2,
\end{eqnarray}
where $\bar{\pi}_{ij}$ lies between $\widehat{\pi}_{ij}$ and $\pi_{ij}$.

It is sufficient to demonstrate: (1) $\{2( B_1 - B_1^0)-r\}/(2r)^{1/2}$ converges in distribution to the standard normal distribution;
(2) $(B_2 - B_2^0)/r^{1/2}=o_p(1)$; (3) $(B_3-B_3^0)/{r^{1/2}}=o_p(1)$.
The second claim is a direct result of Lemma \ref{lemma:beta3:err}.
Note that $\widehat{\beta}_i^0 = \beta_i$, $i=1, \ldots, r$. So $B_3^0$ has less terms than $B_3$.
In view of \eqref{ineq-mu-deriv-bound} and \eqref{ineq-beta-beta0-upp}, if $b_n^4/c_n=o( r^{1/2}/(\log n)^2 )$, then
\begin{eqnarray}
\label{ineq-B3-upper}
\frac{ |B_3| }{ r^{1/2} } & \lesssim & \frac{1}{r^{1/2}} \cdot
\frac{n^2}{c_n} \cdot \| \bs{\widehat{\beta}} - \bs{\beta} \|_\infty^4 \lesssim \frac{ b_n^4(\log n)^2 }{ r^{1/2} c_n } = o(1), \\
\label{ineq-B30-upper}
\frac{ |B_3^0| }{ r^{1/2} } & \lesssim & \frac{1}{r^{1/2}} \cdot
\frac{r(n-r)}{c_n} \cdot \| \bs{\widehat{\beta}}^0 - \bs{\beta} \|_\infty^4 \lesssim \frac{ b_n^4(\log n)^2 }{ r^{1/2} c_n } = o(1),
\end{eqnarray}
which shows the third claim.
Therefore, the remainder of the proof is verify claim (1). This contains three steps.
Step 1 is about explicit expressions of $\widehat{\bs{\beta}}$ and $\widehat{\bs{\beta}}^0$.
Step 2 is about the explicit expression of $B_1-B_1^0$. Step 3 is a combination step.
%is about showing that the main term involved with $B_1-B_1^0$
%asymptotically follows a normal distribution and the remainder terms goes to zero.

Step 1. We characterize the asymptotic representations of $\widehat{\bs{\beta}}$ and $\widehat{\bs{\beta}}^0$.
Recall that $\pi_{ij}=\beta_i+\beta_j$.
To simplify notations, define $\widehat{\pi}_{ij} = \widehat{\beta}_i + \widehat{\beta}_j$.
A second-order Taylor expansion gives
\begin{eqnarray*}\label{eq-expansion-beta-a}
\mu( \widehat{\pi}_{ij} )
&=& \mu( \pi_{ij} ) + \mu^\prime(\pi_{ij}) (\widehat{\pi}_{ij} - \pi_{ij}) +
\frac{1}{2} \mu^{\prime\prime}( \tilde{\pi}_{ij} ) (\widehat{\pi}_{ij} - \pi_{ij})^2,
\end{eqnarray*}
where $\tilde{\pi}_{ij}$ lies between $\widehat{\pi}_{ij}$ and $\pi_{ij}$.
Let
\begin{equation}\label{eq:definition:h}
h_{ij}= \frac{1}{2}\mu^{\prime\prime}( \tilde{\pi}_{ij} )(\widehat{\pi}_{ij} - \pi_{ij})^2,~~
h_i=\sum_{j\neq i}h_{ij}, ~~\bs{h}=(h_1, \ldots, h_n)^\top.
\end{equation}
In view of \eqref{ineq-mu-deriv-bound} and \eqref{ineq-En-beta}, we have
\begin{equation}\label{ineq-beta-h}
\| \bs{h} \|_\infty \le \frac{1}{2}(n-1) \max_{i,j} |h_{ij} | \lesssim \frac{n}{c_n}\| \bs{\widehat{\beta}} - \bs{\beta} \|_\infty^2
\lesssim \frac{b_n^2 \log n}{c_n}.
\end{equation}
By \eqref{eq-likelihood-beta} and \eqref{eq-expansion-beta-a}, we have
\begin{equation*}
d_i-\E(d_i)=\sum_{j=1,j\neq i}^n v_{ij}\{ (\widehat{\beta}_i-\beta_i)+(\widehat{\beta}_j-\beta_j)\} + h_i, ~~~i=1, \ldots, n.
\end{equation*}
%Writing the above equations into the matrix form,
In terms of matrix form, we have $\bs{d} - \E( \bs{d} ) = V ( \widehat{\boldsymbol{\beta}} - \boldsymbol{\beta} ) + \bs{h}$.
\iffalse
\begin{equation*}
\bs{d} - \E( \bs{d} ) = V ( \widehat{\boldsymbol{\beta}} - \boldsymbol{\beta} ) + \bs{h}.
\end{equation*}
\fi
It yields
\begin{equation}\label{eq-expansion-hatbeta-beta}
\boldsymbol{\widehat{\beta}} - \boldsymbol{\beta} = V^{-1} \bs{\bar{d}} - V^{-1}\mathbf{h},
\end{equation}
where, by \eqref{ineq-tight-V} and \eqref{ineq-beta-h},
\begin{equation}\label{ineq-beta-h-b}
\| V^{-1}\mathbf{h} \|_\infty \le \| V^{-1} \|_\infty \| \mathbf{h} \|_\infty \lesssim \frac{ b_n^3 \log n }{ n c_n }.
\end{equation}
Recall  $\bs{\bar{d}}_2=(d_{r+1}, \ldots, d_n)^\top$. Let $\boldsymbol{\widehat{\beta}}_2^0= (\widehat{\beta}_{r+1}^0, \ldots, \widehat{\beta}_n^0)^\top$ and
$ \boldsymbol{\beta}_2 = (\beta_{r+1}, \ldots, \beta_n)$.
Similar to \eqref{ineq-beta-h} and \eqref{eq-expansion-hatbeta-beta}, we have
\begin{equation}\label{eq-beta0-exapnsion}
\boldsymbol{\widehat{\beta}}^0_2 - \boldsymbol{\beta}_2 = V_{22}^{-1}\bs{\bar{d}}_2  - V_{22}^{-1}\mathbf{\widetilde{h}}_2,
\end{equation}
where $\mathbf{\widetilde{h}}_2 = (\tilde{h}_{r+1}, \ldots, \tilde{h}_n)^\top$, $\mathbf{\widetilde{h}} = (\tilde{h}_{1}, \ldots, \tilde{h}_n)^\top$ and
\begin{equation}\label{defintion-tilde-h}
\tilde{h}_i  =   \sum_{j=1,j\neq i}^n \mu^{\prime\prime}( \tilde{\pi}_{ij}^0 ) ( \widehat{\pi}_{ij}^0 - \pi_{ij} )^2,~~
|\tilde{h}_i| \lesssim \frac{ b_n^2 \log n}{c_n},
 ~i=1, \ldots, n
\end{equation}
In the above equation, $\tilde{\pi}_{ij}^0$ lies between $\pi_{ij}$ and $\widehat{\pi}_{ij}^0=\widehat{\beta}_i^0 + \widehat{\beta}_j^0$ for all $i,j=1, \ldots, n$.

Step 2. We derive the explicit expression of $B_1-B_1^0$.
Substituting \eqref{eq-expansion-hatbeta-beta} and \eqref{eq-beta0-exapnsion} into the expressions of $B_1$ in \eqref{lrt-a-beta-B1} and $B_1^0$ respectively, it yields
\begin{eqnarray*}\label{B1-expression}
2B_1  =  \bs{\bar{d}}^\top V^{-1} \bs{\bar{d}} - \bs{h}^\top V^{-1} \bs{h}, \quad %\\
%\label{likelihood-beta-composite}
 2B_1^0  =  \bs{\bar{d}}_2^\top V_{22}^{-1} \bs{\bar{d}}_2 -  \bs{\widetilde{h}}_2^\top V_{22}^{-1} \bs{\widetilde{h}}_2.
\end{eqnarray*}
By setting $V^{-1}=S+W$ and $V_{22}^{-1}=S_{22}+ \widetilde{W}_{22}$, we have
\begin{equation}\label{eq-theorem2-B10}
2(B_1 - B_1^0) = \sum_{i=1}^r \frac{ \bar{d}_i^{\,2} }{v_{ii}}  + \bs{\bar{d}}^\top W \bs{\bar{d}} - \bs{\bar{d}}_2^\top \widetilde{W}_{22} \bs{\bar{d}}_2
+ \bs{\bar{h}}^\top V^{-1} \bs{\bar{h}} - \bs{\bar{h}}_2^\top V_{22}^{-1} \bs{\bar{h}}_2.
\end{equation}

Step 3.  %, where we show the main term asymptotically follows a normal distribution and the remainder term goes to zero
We show three claims: (i) $(\sum_{i=1}^r  \bar{d}_i^{\,2}/v_{ii}-r)/(2r)^{1/2}\stackrel{L}{\to} N(0,1)$;
(ii) $\bs{\bar{d}}^\top W \bs{\bar{d}}/r^{1/2}=o_p(1)$ and $\bs{\bar{d}}_2^\top \widetilde{W}_{22} \bs{\bar{d}}_2/r^{1/2}=o_p(1)$; (iii) $\bs{\bar{h}}^\top V^{-1} \bs{\bar{h}}/r^{1/2}=o_p(1)$ and $\bs{\bar{h}}_2^\top V_{22}^{-1} \bs{\bar{h}}_2/r^{1/2}=o_p(1)$.
The first and second claims directly follows from Lemma \ref{lemma:weighte-degree-al} and Lemma \ref{lemma-clt-beta-W}, respectively.
By \eqref{ineq-beta-h} and \eqref{ineq-beta-h-b}, if $b_n^5/c_n^2 =o( r^{1/2}/(\log n)^{2})$, we have
\begin{equation}\label{eq-simi-aVh}
\frac{1}{r^{1/2}} |\mathbf{h}^\top V^{-1} \mathbf{h}| \le \frac{1}{r^{1/2}} \times n \| \mathbf{h} \|_\infty \| V^{-1} \mathbf{h} \|_\infty
\lesssim
\frac{1}{r^{1/2}} \cdot n \cdot \frac{b_n^2\log n}{c_n} \cdot \frac{ b_n^3 \log n}{nc_n} \lesssim \frac{b_n^5 (\log n)^2}{r^{1/2}c_n^2}= o(1).
\end{equation}
\iffalse
If $b_n^5/c_n^2 =o( r^{1/2}/(\log n)^{2})$, then
\begin{equation}\label{eq-simi-aVh}
\frac{1}{r^{1/2}} |\mathbf{h}^\top V^{-1} \mathbf{h}| \lesssim \frac{ b_n^5(\log n)^2 }{n^{1/2}}  = o(1).
\end{equation}
\fi
In view of \eqref{ineq-tight-V} and \eqref{defintion-tilde-h}, with the same arguments as in the proof of \eqref{eq-simi-aVh}, we have
\begin{equation}\label{eq-simi-a}
\frac{1}{r^{1/2}} |\bs{\tilde{h}}_2^\top V_{22}^{-1} \bs{\tilde{h}}_2| \lesssim \frac{ (n-r)b_n^5(\log n)^2 }{n r^{1/2}c_n^2}  = o(1).
\end{equation}
This demonstrates claim (iii).
It completes the proof.
\end{proof}

\section*{Acknowledgements}
We are very grateful to three referees, the associated editor, and the editor for their valuable
comments that have greatly improved the manuscript.

\section*{Disclosure Statement}
The authors report there are no competing interests to declare.

\bigskip
\begin{center}
{\large\bf SUPPLEMENTARY MATERIAL}
\end{center}

\begin{description}

\item[Supplementary material A].
It contains the proofs of Theorem \ref{theorem-LRT-beta} (b) and Theorem \ref{theorem-LRT-beta-fixed},
 the proofs of those supported lemmas,
 as well as five figures in section \ref{sub-sec-simu}.

\item[Supplementary material B].
It contains the proofs of Theorems \ref{theorem-ratio-bt-3} and \ref{theorem-ratio-bt-fixed} in the Bradley--Terry model as well as the proofs of their supported lemmas.
It is available upon request.

\item[Supplementary material C].
It contains the proofs of Lemmas \ref{lemma-appro-beta-VS}, 15, %\ref{lemma-beta-approx-ho}
26 and 33  about the approximation error for approximate inverse of Fisher information matrices.
It is available upon request.

\end{description}

\setlength{\itemsep}{1.5pt}
\setlength{\bibsep}{0ex}
\bibliography{reference3}
\bibliographystyle{apa}

\newpage

\begin{center}
{\Large Supplementary Material A for ``Likelihood ratio tests in random graph models with increasing dimensions"}
\end{center}
\vskip20pt

Supplementary Material A contains the proof of Theorem 1 (b) and Theorem 2,
the proofs of supported lemmas in the proofs of Theorems 1 and 2,
 detailed calculations for \eqref{eq-thereom1b-z43}, \eqref{ineq-B2-B20-a} and \eqref{eq-th2a-B3B30},
 as well as five figures for the simulation. %% (59) and (60).
This supplementary material is organized as follows.

Section \ref{section:theorem1b} presents the proof of Theorem 1 (b).
Section \ref{section-theorem2} presents the proofs for Theorem 2.

Section \ref{section-variance} gives the variances of the weighted quadratic sum $\sum_i f_i \bar{d}_i^{\,2}$,
the weighted cubic sum $\sum_i f_i \bar{d}_i^{\,3}$ and an upper bound of a mixed sum $\sum_{i,j} f_{ij} \bar{d}_i^{\,2} \bar{d}_j$, which will be used in the proofs
of supported lemmas repeatedly.

Section \ref{section:beta-th1a} contains the proofs of supported lemmas in the proof of Theorem 1 (a).
This section is organized as follows.
Sections \ref{subsection-proof-lemma2} and \ref{subsection-proof-lemma3}
present the proofs of Lemmas 3 and 4.
Section \ref{subsection-L2norm} gives an additional result about the upper bound for $\bs{\widehat{\beta}}$ in terms of the $L_2$-norm.
Section \ref{subsec-asy-widehatbeta} gives an asymptotically explicit expression for $\bs{\widehat{\beta}}^0$ that will be used in the proof of Lemma 5.
Section \ref{subsec-prooflemma4} presents the proof of Lemma 5.
The proof of Lemma 2 about approximation error of using $S_{22}$ to approximate $V_{22}^{-1}$ is present in Supplementary Material C.
We defer the proof of Lemma 1 to Section \ref{section-lemma1} since it contains many long calculations.

Section \ref{section-th1b} presents proofs of supported lemmas in the proof of Theorem 1 (b)
This section is organized as follows.
Sections \ref{section-lemma5}, \ref{subsection-proof-lemma6} and \ref{section-lemma7}
present the proofs of Lemmas 7, 8 and 9, respectively.
Section \ref{section-proof-39-B20} presents the proof of \eqref{eq-thereom1b-z43} in the main text.

Section \ref{section-theorem2a} presents proofs of supported Lemmas in the proof of Theorem 2 (a)
as well as  proof of \eqref{ineq-B2-B20} and \eqref{eq-th2a-B3B30}. % (59) and (60) in the main text.
This section is organized as follows.
Sections \ref{section-proof-lemma1010},  \ref{section-proof-lemma11} and \ref{section-proof-lemma12}
present the proofs of Lemmas 10, 11 and 12, respectively.
Sections \ref{subsection:B2B20} and \ref{subsection-B3B30} presents the proofs of orders of two remainder terms
$B_2-B_2^0$  in \eqref{ineq-B2-B20-a}  and $B_3-B_3^0$ in \eqref{eq-th2a-B3B30} in the main text, respectively.

Section \ref{section-lemma1} presents the proof of Lemma 1.
Section \ref{section-bernstein} reproduces Bernstein's inequality and a Martingale central limit theorem for easy readability.
Section \ref{sec-figure} presents additional figures in the simulation section.

All notation is defined in the main text unless explicitly noted otherwise. Equation and lemma numbering continues
in sequence with those established in the main text.

\section{Preliminaries}
\label{sec-prelim}

We first recall useful inequalities on the derivatives of $\mu(x)$, which will be used in the proofs repeatedly.
Recall that
\[
\mu(x) = \frac{ e^x }{ 1 + e^x}.
\]
A direct calculation gives that the derivative of $\mu(x)$ up to the third order are
\begin{eqnarray}\label{eq-derivative-mu-various}
\mu^\prime(x) = \frac{e^x}{ (1+e^x)^2 },~~  \mu^{\prime\prime}(x) = \frac{e^x(1-e^x)}{ (1+e^x)^3 },~~ \mu^{\prime\prime\prime}(x) =  \frac{ e^x [ (1-e^x)^2 - 2e^x] }{ (1 + e^x)^4 }.
\end{eqnarray}
Note that $\bs{\beta}=(\beta_1, \ldots, \beta_n)$ denotes the data generating parameter, under which the data are generated.
Recall that
\[
\pi_{ij} = \beta_i + \beta_j, ~~\widehat{\pi}_{ij}=\widehat{\beta}_i + \widehat{\beta}_j,~~
\widehat{\pi}_{ij}^0 = \widehat{\beta}_i^0 + \widehat{\beta}_j^0.
\]
According to the definition of $c_n$, we have
\begin{equation}\label{ineq-mu-deriv-bound}
|\mu^\prime(\pi_{ij})| \le \frac{1}{c_n}, ~~
|\mu^{\prime\prime}(\pi_{ij} )| \le \frac{1}{c_n},~~ |\mu^{\prime\prime\prime}(\pi_{ij})| \le \frac{1}{c_n}.
\end{equation}
For a $\bs{\widetilde{\beta} }$ satisfying $\| \bs{\widetilde{\beta} } - \bs{\beta} \|_\infty =o(1)$, we also have
\begin{equation}\label{ineq-mu-tilde}
|\mu^\prime(\tilde{\pi}_{ij} )| \lesssim \frac{1}{c_n}, ~~ |\mu^{\prime\prime}(\tilde{\pi}_{ij})| \lesssim \frac{1}{c_n},~~ |\mu^{\prime\prime\prime}(\tilde{\pi}_{ij})| \lesssim \frac{1}{c_n}.
\end{equation}
These facts will be used in the proofs repeatedly.
Recall that $\bar{a}_{ij}=a_{ij}-\E(a_{ij})$
is the centered random variable of $a_{ij}$ and $\bar{a}_{ii}=0$ for all $i=1, \ldots, n$.
Correspondingly, $\bar{d}_i = d_i - \E(d_i)$ and $\bs{\bar{d}}=(\bar{d}_1, \ldots, \bar{d}_n)^\top$.

\section{Proofs for Theorem 1 (b)}
\label{section:theorem1b}
%\subsection{Proofs for Theorem 1 (b)}
\label{subsection-proof-th1b}
Let  $\bs{\widetilde{d}}=(\sum_{i=1}^r d_i, d_{r+1},\ldots, d_n)$ and
 $\widetilde{V}$ denote the Fisher information matrix of $\widetilde{\bs{\beta}}=(\beta_1, \beta_{r+1}, \ldots, \beta_n)^\top$
under the null $H_0: \beta_1 = \cdots= \beta_r$, where
\begin{equation}\label{definition-tilde-V}
\widetilde{V}=\begin{pmatrix} \tilde{v}_{11} & \bs{\tilde{v}}_{12}^\top \\ \bs{\tilde{v}}_{12} & V_{22} \end{pmatrix}.
\end{equation}
In the above, $V_{22}$ is the lower right $(n-r)\times (n-r)$ block of $V$, $\bs{\tilde{v}}_{12} =
(\tilde{v}_{1,r+1}, \ldots, \bar{v}_{1, n})^\top$, and
\[
\tilde{v}_{11}= 2r(r-1)\cdot \frac{ e^{2\beta_1} }{ ( 1 + e^{2\beta_1})^2 } + r\sum_{j=r+1}^n \tilde{v}_{1j}, ~~
\tilde{v}_{1j} =  \frac{ r e^{\beta_1 + \beta_j } }{ ( 1 + e^{\beta_1 + \beta_j})^2 },~j=r+1, \ldots, n.
\]
%Note that $\widetilde{V}$ is also the covariance matrix of $\bs{\widetilde{d}}$.
%Similar to approximate $V^{-1}$ by $S$,
We use $\widetilde{S}=\mathrm{diag}(1/\tilde{v}_{11}, 1/v_{r+1, r+1}, \ldots, 1/v_{nn})$ to approximate $\widetilde{V}^{-1}$, whose
approximation error is in the following lemma.
\begin{lemma}\label{lemma-beta-approx-ho}
For any $r\in\{0, \ldots, n-1\}$ and $n\ge 3$, we have
\begin{equation}\label{approxi-inv2-beta-ho}
\|\widetilde{W}:= \widetilde{V}^{-1}-\widetilde{S} \|_{\max} \le \frac{ b_n }{ (n-1)^2 c_n^2 }\left( \frac{b_n n}{2(n-2)c_n} + \frac{1}{2} \right).
\end{equation}
\end{lemma}

The order of the above approximation error is the same as that in \eqref{ineq-V-S-appro-upper-b}, regardless of $r$. %%  and is independent of $r$.
%The proof of Lemma \ref{lemma-beta-approx-ho} is given in the Supplementary Material B.

Recall that $\bs{\widehat{\beta}}^0$ denotes
the restricted MLE of $\bs{\beta} =(\beta_1,  \ldots, \beta_n)^\top$. Under the null $H_0: \beta_1=\cdots = \beta_r$,
we have $\widehat{\beta}_1^0 = \cdots = \widehat{\beta}_r^0$.
Similar to the proof of Lemma 4, we have the following consistency result.

\begin{lemma}\label{lemma-con-beta-b}
Under the null $H_0: \beta_1=\cdots = \beta_r$, if
\[
\left( \frac{b_n}{c_n} + \frac{b_n^3}{c_n^3}\cdot \frac{r(n-r)}{n^2}  \right)\left(  b_n + \frac{ b_n^3 }{ c_n^2} ( \frac{r^{1/2}(n-r)^{1/2}}{n^{3/2}} + \frac{n-r}{n} ) \right) = o\left( \sqrt{\frac{n}{\log n}} \right),
\]
then with probability at least $1-2(n-r+1)/n^2$, $\bs{\widehat{\beta}}^0$ exists and satisfies
\begin{equation*}
\| \bs{\widehat{\beta}}^0 - \bs{\beta} \|_\infty \lesssim \left(  b_n + \frac{ b_n^3 }{ c_n^2} ( \frac{r^{1/2}(n-r)^{1/2}}{n^{3/2}} + \frac{n-r}{n} )  \right)\sqrt{ \frac{\log n}{n} } .
\end{equation*}
Further, if  $\bs{\widehat{\beta}}^0$ exists, it must be unique.
\end{lemma}

From the above lemma, we can see that the condition to guarantee consistency and the error bound depends on
$r$.
Larger $r$ means a weaker condition and a smaller error bound.
%A rough condition in regardless of $r$ to guarantee consistency is $b_n^6/c_n^5=o( (n/\log n)^{1/2})$. %, which shall be used in the proof of Theorem \ref{theorem-LRT-beta} (b).
To simplify calculations, we shall
use the condition $b_n^6/c_n^5=o( (n/\log n)^{1/2})$ in regardless of $r$,
which implies an error bound $(b_n^3/c_n^2)(\log n/n)^{1/2}$ that is generally larger than that in Lemma 4.
Similar to Lemma 3, we have the following bound, which is independent of $r$.

\begin{lemma}\label{lemma-tilde-W}
For any given $r\in\{1, \ldots, n-1\}$, we have
$(\bs{\widetilde{d}} - \E\bs{\widetilde{d}} )^\top \widetilde{W}  (\bs{\widetilde{d}} - \E\bs{\widetilde{d}} ) = O_p(  b_n^3 / c_n^3  )$.
\iffalse
\[
(\bs{\widetilde{d}} - \E\bs{\widetilde{d}} )^\top \widetilde{W}  (\bs{\widetilde{d}} - \E\bs{\widetilde{d}} ) = O_p\left( \frac{ b_n^3 }{ c_n^3 } \right).
\]
\fi
\end{lemma}

The asymptotic representation of $\bs{\widehat{\beta}}^0$ is given below.
\begin{lemma}
\label{lemma-beta-homo-expan}
Under the null $H_0: \beta_1=\cdots=\beta_r$,
if $b_n^6/c_n^5 = o( (n/\log n)^{1/2})$, then for any given $r\in\{0, \ldots, n-1\}$, we have
\begin{equation*}
\begin{array}{rcl}
\widehat{\beta}_1^0 - \beta_1 & = & \frac{ \sum_{i=1}^r \bar{d}_i }{ \tilde{v}_{11} }+g_1, \\
\widehat{\beta}_i^0 - \beta_i & = & \frac{ \bar{d}_i }{ v_{ii} } + g_i,~~ i=r+1, \ldots, n,
\end{array}
\end{equation*}
where $g_1, g_{r+1}, \ldots, g_n$ with probability at least $1- O(n^{-1})$ satisfy
\[
g_i = (\widetilde{V}^{-1}\bs{\tilde{h}})_i + [\widetilde{W} (\bs{\tilde{d}}- \E \bs{\tilde{d}})]_i = O\left( \frac{b_n^9 \log n}{ nc_n^7 } \right),
\]
 uniformly, and  $\bs{\tilde{h}}=(\tilde{h}_1, \tilde{h}_{r+1}, \ldots, \tilde{h}_n)^\top$ satisfies
\begin{equation}\label{eq-homo-tildeh}
\begin{array}{rcl}
|\tilde{h}_1| & \lesssim & \frac{rb_n^6\log n}{c_n^5}, \\
\max_{i=r+1, \ldots, n} |\tilde{h}_i| & \lesssim & \frac{ b_n^6 \log n}{c_n^5}.
\end{array}
\end{equation}
\end{lemma}

With some ambiguity of notations, we still use the notation $\bs{\tilde{h}}$ here that is a little different from $\bs{\tilde{h}}$ defined in (27). % in Section \ref{section:theorem1-a}
Specifically, the first element of $\bs{\tilde{h}}$ can be viewed as the sum of $\tilde{h}_i$, $i=1,\ldots, r$ in  (27).
This difference leads to that some remainder term here is larger than that in (24).
% since the order of the first element of $\bs{\widetilde{h}} $ is different from others.

Now, we are ready to prove  Theorem 1 (b).

\begin{proof}[Proof of Theorem 1 (b)]
Under the null $H_0: \beta_1=\cdots=\beta_r$, the data generating parameter $\bs{\beta}$ is equal to $(\underbrace{\beta_1, \ldots, \beta_1}_r, \beta_{r+1}, \ldots, \beta_n)^\top$.
The following calculations are based on the event $E_n$ that $\bs{\widehat{\beta}}$ and $\bs{\widehat{\beta}}^0$ simultaneously exist and satisfy
\begin{equation}\label{ineq-beta-beta0-upp}
 \| \widehat{\bs{\beta}} - \bs{\beta} \|_\infty\le \frac{3nb_n}{(2n-1)}\sqrt{ \frac{\log n}{n} }, \mbox{~~and~~}
  \| \widehat{\bs{\beta}}^0 - \bs{\beta} \|_\infty\lesssim \frac{b_n^3}{c_n^2}\sqrt{ \frac{\log n}{n} }
\end{equation}
By Lemmas 4 and \ref{lemma-con-beta-b}, $\P(E_n) \ge 1 - O(n^{-1})$ if $b_n^3/c_n^2 = o( (n/\log n)^{1/2})$.

Similar to the proof of Theorem 1 (a),
it is sufficient to demonstrate: (i) $\{2( B_1 - B_1^0)-r\}/(2r)^{1/2} \stackrel{L.}{\to} N(0,1)$; %converges in distribution to the standard normal distribution;
(ii) $(B_2 - B_2^0)/r^{1/2}=o_p(1)$; (iii) $(B_3-B_3^0)/r^{1/2}=o_p(1)$.
The fourth-order Taylor expansion for $\ell(\bs{\widehat{\beta}}^0)$ here is with regard to the vector $(\beta_1, \beta_{r+1}, \ldots, \beta_n)^\top$
because $\beta_1, \ldots, \beta_r$ are the same under the null here. As we shall see, the expressions of $B_1^0$ and $B_2^0$ are a little different from $B_1$ and $B_2$
except from the difference $\bs{\widehat{\beta}}$ and $\bs{\widehat{\beta}}^0$.

With the same arguments as in (20) and (21), we have claim (iii) under the condition $b_n^{12}/c_n^9 = o\left( r^{1/2}/(\log n) \right)$.
In Lemma 5, we show $B_2/r^{1/2} = o_p(1)$. For claim (ii), it is sufficient to show $B_2^0/r^{1/2}=o_p(1)$.
Under the null $H_0:  \beta_1=\ldots=\beta_r$,  $B_2^0$ can be written as
\begin{eqnarray*}
B_2^0  & = & \{ 4r(r-1)\mu^{\prime\prime}(\pi_{11})+r\sum_{j=r+1}^n \mu^{\prime\prime}(\pi_{ij}) \} (\widehat{\beta}_1^0 - \beta_1)^3
+3r \sum_{i=r+1}^n \mu^{\prime\prime}( \pi_{1i}) ( \widehat{\beta}_1^0  - \beta_1)^2 ( \widehat{\beta}_i^0 - \beta_i)  \\
&&+3r \sum_{i=r+1}^n \mu^{\prime\prime}(\pi_{1i}) ( \widehat{\beta}_1^0  - \beta_1) ( \widehat{\beta}_i^0 - \beta_i)^2
+ 3\sum_{i,j=r+1, i\neq j}^n \mu^{\prime\prime}(\pi_{ij}) ( \widehat{\beta}_i^0  - \beta_i) ( \widehat{\beta}_j^0 - \beta_j)^2.
\end{eqnarray*}
With the use of the asymptotic representation of $\bs{\widehat{\beta}}^0$ in Lemma \ref{lemma-beta-homo-expan}, if
\iffalse
\[
\frac{ b_n^{9} }{ c_n^7 } = o\left( \frac{ (rn)^{1/3} }{ \log n } \right)
 \mbox{~~and~~} \frac{ b_n^{5}}{ c_n^{2} } = o\left( \frac{r^{1/2}}{ (\log n)^2 } \right),
\]
\fi
$b_n^{9}/c_n^7  = o\left( (rn)^{1/3} / \log n  \right)$ and $b_n^{5}/c_n^{2}  = o\left( r^{1/2}/(\log n)^2  \right)$,
then we have
\begin{equation}\label{eq-thereom1b-z43}
\frac{B_2^0}{r^{1/2}} = o_p( 1 ),
\end{equation}
whose detailed calculations are given in Section \ref{section-proof-39-B20}. % the Supplementary Material A.

Next, we show claim (i). Recall $\bs{\tilde{d}}=(\sum_{i=1}^r d_i, d_{r+1},\ldots, d_n)$. The expression of $B_1^0$ is
\begin{equation}\label{B10-homo-expression}
B_1^0  = \frac{1}{2}( \bs{\widetilde{d}} - \E \bs{\widetilde{d}} )^\top \widetilde{V}^{-1} ( \bs{\widetilde{d}} - \E \bs{\widetilde{d}} )
-\frac{1}{2}\bs{\widetilde{h}}^\top \widetilde{V}^{-1} \bs{\widetilde{h}},
\end{equation}
where  $\mathbf{\widetilde{h}} = (\tilde{h}_1, \tilde{h}_{r+1}, \ldots, \tilde{h}_n)^\top$  satisfies
\eqref{eq-homo-tildeh}.
In view of \eqref{approxi-inv2-beta-ho} and \eqref{eq-homo-tildeh}, setting $\widetilde{V}^{-1}=\widetilde{S}+\widetilde{W}$ yields
\begin{eqnarray}
\nonumber
\bs{\tilde{h}}^\top \widetilde{V}^{-1} \bs{\tilde{h}} & \le &
\underbrace{\frac{\tilde{h}_1^2}{\tilde{v}_{11}} + \sum_{i=r+1}^n \frac{ h_i^2 }{ v_{ii} }} + \underbrace{ |\tilde{w}_{11}| \tilde{h}_1^2 + \| \widetilde{W} \|_{\max} \left(
2 |\tilde{h}_1| \sum_{i=r+1}^n |h_i| + \sum_{i,j=r+1}^n |h_i| |h_j| \right)}\\
\nonumber
& \lesssim &  b_n \left( \frac{b_n^6 \log n}{ c_n^5}  \right)^2 + \frac{ b_n^3 }{ n^2 c_n^2 } \left\{ r^2 + 2r(n-r)+(n-r)^2 \right\} \left( \frac{b_n^6 \log n}{ c_n^5}  \right)^2 \\
\label{eq-thereom1b-a}
& \lesssim & \frac{ b_n^{15} (\log n)^2 }{ c_n^{12}}.
\end{eqnarray}
This shows that if $b_n^{15}/c_n^{12} = o\left( r^{1/2}/(\log n)^2 \right)$, then
$
 |\bs{\tilde{h}}^\top \bs{\widetilde{V}}^{-1} \bs{\tilde{h}} | / r^{1/2} = o_p(1).
$
\iffalse
\begin{equation*}\label{eq-thereom1b-a}
\frac{ |\bs{\tilde{h}}^\top \bs{\widetilde{V}}^{-1} \bs{\tilde{h}} | }{\sqrt{r}} = o_p(1).
\end{equation*}
\fi

Now, we evaluate the difference between $( \mathbf{\widetilde{d}} - \E \mathbf{\widetilde{d}} )^\top \widetilde{V}^{-1} ( \mathbf{\widetilde{d}} - \E \mathbf{\widetilde{d}} )$
and $\bs{\bar{d}}^\top V^{-1} \bs{\bar{d}}$.
By using $\widetilde{S}$ and $S$ to approximate $\widetilde{V}^{-1}$ and $V^{-1}$ respectively,
we have
\begin{eqnarray}
\nonumber
&&\bs{\bar{d}}^\top V^{-1} \bs{\bar{d}} -
( \mathbf{\widetilde{d}} - \E \mathbf{\widetilde{d}} )^\top \widetilde{V}^{-1} ( \mathbf{\widetilde{d}} - \E \mathbf{\widetilde{d}} ) \\
\label{eq-theorem1b-b}
& = & \sum_{i=1}^r \frac{ \bar{d}_i^{\,2} }{ v_{ii} } - \frac{ ( \tilde{d}_1- \E \tilde{d}_1) ^2 }{ \tilde{v}_{11} }
+ \bs{\bar{d}}^\top W \bs{\bar{d}} - ( \mathbf{\widetilde{d}} - \E \mathbf{\widetilde{d}} )^\top \widetilde{W} ( \mathbf{\widetilde{d}} - \E \mathbf{\widetilde{d}} ).
\end{eqnarray}
By Lemmas 3 and \ref{lemma-tilde-W}, if $b_n^3/c_n^3=o(r^{1/2})$, then
\begin{equation}\label{eq-theorem1b-c}
\frac{1}{r^{1/2}} \max\{ \bs{\bar{d}} W \bs{\bar{d}}, ( \mathbf{\widetilde{d}} - \E \mathbf{\widetilde{d}} )^\top \widetilde{W} ( \mathbf{\widetilde{d}} - \E \mathbf{\widetilde{d}} ) \} = o_p(1).
\end{equation}
Since $\sum_{i=1}^r d_i = 2 \sum_{1\le i<j \le r} a_{ij} + \sum_{i=1}^r \sum_{j=r+1}^n a_{ij}$,
by the central limit theorem for the bounded case (\citet{Loeve:1977}, page 289),
$\tilde{v}_{11}^{-1/2}\sum_{i=1}^r (d_i - \E d_i)$ converges in distribution to the standard normal distribution if $\tilde{v}_{11}\to\infty$.
Therefore, as $r\to\infty$,
\[
\frac{ [\sum_{i=1}^r \{ d_i-\E(d_i) \}]^2/\tilde{v}_{11} }{ r } = o_p(1).
\]
By combining (28), \eqref{B10-homo-expression}, (29), \eqref{eq-theorem1b-b} and \eqref{eq-theorem1b-c}, it yields
\[
\frac{2(B_1-B_1^0)}{\sqrt{2r}}
= \frac{1}{\sqrt{2r}} \sum_{i=1}^r \frac{ ( d_i- \E d_i )^2 }{ v_{ii} }  + o_p(1).
\]
Therefore, claim (i) immediately follows from Lemma 1.
This completes the proof.
\end{proof}

\section{Proofs for Theorem 2}
\label{section-theorem2}

\subsection{Proofs for Theorem 2 (a)}
\label{subsection-theorem2a}

Let $\bar{\dd}_1 = ( \bar{d}_1, \ldots, \bar{d}_r)^\top$, $\bar{\dd}_2 = ( \bar{d}_{r+1}, \ldots, \bar{d}_n)^\top$ and
\begin{equation}\label{VW-divide}
V =
\begin{pmatrix} V_{11}  & V_{12} \\
V_{21} & V_{22}
\end{pmatrix}, ~~
W = \begin{pmatrix} W_{11} & W_{12} \\
W_{21} & W_{22}
\end{pmatrix},
\end{equation}
where $V_{11}$ and $W_{11}$ are respective $r\times r$ dimensional sub-matrices of $V$ and $W$, and $W=V^{-1}-S$.
Recall that $V_{22}$ denotes the Fisher information matrix of $\bs{\beta}_2=(\beta_{r+1}, \ldots, \beta_n)^\top$
under the null $H_0: (\beta_1, \ldots, \beta_r)=(\beta_1^0, \ldots, \beta_r^0)$ and $S_{22}=\mathrm{diag}( 1/v_{r+1, r+1}, \ldots, 1/v_{nn})$.
Remark that $r$ is a fixed constant in this section.

To prove Theorem 2 (a), we need the following three lemmas.
The lemma below gives an upper bound of $\| W_{22} - \widetilde{W}_{22} \|_{\max}$,
whose magnitudes are $b_n^6/(n^3c_n^5)$. It is much smaller than the error bounds of $W_{22}$ in \eqref{ineq-V-S-appro-upper-b} and $\widetilde{W}_{22}$
in \eqref{ineq-V-S-appro-upper-b} by a vanishing factor $n^{-1}$.

\begin{lemma}\label{lemma:w2-error}
%Let $W_{22}$ be the bottom right block of $W$ with dimension $(n-r)\times (n-r)$ and $r$ a fixed constant.
For a fixed constant $r$, the error between $W_{22}$ and $\widetilde{W}_{22}$ in terms of the maximum absolute entry-wise norm has the following bound:
\begin{equation}\label{ineq-W-diff-upper}
\| W_{22} - \widetilde{W}_{22} \|_{\max} \lesssim \frac{ b_n^6 }{ n^3c_n^5 }.
\end{equation}
\end{lemma}

The following lemma gives the upper bounds of three remainder terms in \eqref{eq-theorem2a-B1022} that tend to zero.

\begin{lemma}\label{lemma-W-widetilde-d}
Suppose $r$ is a fixed constant. \\
(a)If $b_n^3/c_n^2=o( n^{3/2}/(\log n)^{1/2})$, then $\bar{\dd}_1^\top W_{11} \bar{\dd}_1 = o_p(1)$. \\
(b)If $b_n^3/c_n^3=o( n^{1/2} )$, then $\bar{\dd}_1^\top W_{12} \bar{\dd}_2 = o_p(1)$. \\
(c)If $b_n^3/c_n^3=o( n^{3/4} )$, then
\[
\bar{\dd}_2^\top ( W_{22} - \widetilde{W}_{22}) \bar{\dd}_2 = o_p(1).
\]
\end{lemma}

The lemma below establishes the upper bound of $\max_{i=r+1, \ldots, n} | \widehat{\beta}_i - \widehat{\beta}_i^0 |$.

\begin{lemma}\label{lemma-hat-beta-diff}
Under the null $H_0: \beta_i=\beta_i^0$, $i=1,\ldots, r$ with a fixed $r$,
if $b_n^3/c_n = o( n/\log n)$, then with probability at least $1-O(n^{-1})$,
\[
\max_{i=r+1, \ldots, n} | \widehat{\beta}_i - \widehat{\beta}_i^0 | \lesssim \frac{b_n^3 \log n}{nc_n}.
\]
\end{lemma}

The above error bound is in the magnitude of $(n/\log n)^{-1}$, up to a factor $b_n^3/c_n$, which makes the remainder terms in \eqref{eq-theorem2-whh} be asymptotically neglected.
Note that this error bound is much smaller than those for $\| \bs{\widehat{\beta}}^0 - \bs{\beta}^0 \|_\infty$
and $\| \bs{\widehat{\beta}} - \bs{\beta}^0 \|_\infty$ by a vanishing factor $n^{-1/2}$, whose magnitudes are $b_n(\log n/n)^{1/2}$.

Now, we are ready to prove Theorem 2 (a).

\begin{proof}[Proof of Theorem 2 (a)]

The following calculations are based on the event $E_n$ that $\bs{\widehat{\beta}}$ and $\bs{\widehat{\beta}}^0$ simultaneously exist and satisfy
\begin{equation}\label{ineq-beta-beta0-upp-f}
\max\left\{ \| \widehat{\bs{\beta}} - \bs{\beta} \|_\infty, \| \widehat{\bs{\beta}}^0 - \bs{\beta} \|_\infty \right\}
\le \frac{3nb_n}{(2n-1)}\sqrt{ \frac{\log n}{n} },~~\max_{i=r+1, \ldots, n} | \widehat{\beta}_i - \widehat{\beta}_i^0 | \lesssim \frac{b_n^3 \log n}{nc_n}.
\end{equation}
By Lemmas 4 %\ref{lemma-consi-beta}
and \ref{lemma-hat-beta-diff}, $\P(E_n) \ge 1 - O(n^{-1})$ if $b_n^3/c_n=o\left( n/\log n \right)$.

Similar to the proof of Theorem 1 %\ref{theorem-LRT-beta}
(a),
it is sufficient to demonstrate: (1) $2( B_1 - B_1^0)$ converges in distribution to the chi-square distribution with $r$ degrees of freedom;
(2) $B_2 - B_2^0$ and $B_3-B_3^0$ are asymptotically neglected remainder terms,
where %$B_1-B_1^0$ is given in %\eqref{eq-theorem2-B10}
%(32),
$B_2$ and $B_3$ are given in \eqref{lrt-a-beta-B2} and \eqref{lrt-a-beta-B3}, % (22) and (23), %
$B_2^0$ and $B_3^0$ are respective versions of $B_2$ and $B_3$ by replacing
$\bs{\widehat{\beta}}$ with $\bs{\widehat{\beta}}^0$.
Claims (1) and (2) are shown in three steps in turns.

Step 1. We show $2( B_1 - B_1^0)\stackrel{L}{\to} \chi^2_r$. Using the matrix form in \eqref{VW-divide}, $ B_1 - B_1^0$ in (32) %\eqref{eq-theorem2-B10}
can be written as
\begin{eqnarray}
\nonumber
2(B_1 - B_1^0) & = & \sum_{i=1}^r \frac{ \bar{d}_i^{\,2} }{v_{ii}}  + \underbrace{\mathbf{\bar{d}}_1^\top W_{11} \mathbf{\bar{d}}_1 +
 2\mathbf{\bar{d}}_1^\top W_{11} \mathbf{\bar{d}}_2 + \mathbf{\bar{d}}_2^\top ( W_{22} - \widetilde{W}_{22}) \mathbf{\bar{d}}_2}_{Z_1}\\
 \label{eq-theorem2a-B1022}
&&+ \underbrace{\mathbf{\widetilde{h}}_2^\top V_{22}^{-1} \mathbf{\widetilde{h}}_2 - \mathbf{h}^\top V^{-1} \mathbf{h}}_{Z_2}.
\end{eqnarray}
It is sufficient to demonstrate: (i) $\sum_{i=1}^r  \bar{d}_i^{\,2}/v_{ii}$ converges in distribution to a Chi-square distribution with $r$ degrees of freedom;
(ii) $Z_1 = o_p(1)$; (iii) $Z_2=o_p(1)$.
Claim (ii) directly follows from Lemma  \ref{lemma-W-widetilde-d}.
Because $\tilde{d}_i = \sum_{j=r+1}^n a_{ij}$ is independent over $i=1,\ldots, r$ and $r$ is a fixed constant,
the classical central limit theorem for the bounded case (\cite{Loeve:1977}, p. 289) gives that
the vector $(\bar{d}_1/v_{11}^{1/2}, \ldots, \bar{d}_r/v_{rr}^{1/2})$
follows a $r$-dimensional standard normal distribution. This verifies claim (i).
Now, we show $Z_2=o_p(1)$.
Recalling the definition of $\mathbf{h}$ in \eqref{eq:definition:h}, %(26),  %. \eqref{eq:definition:h}.
 $V^{-1}=S+W$ and $V_{22}^{-1}=S_{22}+ \widetilde{W}_{22}$, we have
\begin{eqnarray}
\label{eq-th2a-hVh}
\mathbf{h}^\top V^{-1} \mathbf{h} &  = & \sum_{i=1}^n \frac{ h_i^2 }{ v_{ii} } + \hh_1^\top W_{11} \hh_1
+ 2 \hh_1^\top W_{12} \hh_2 + \hh_2^\top W_{22} \hh_2, \\
\label{eq-th2a-h2V22h2}
\widetilde{\hh}_2^\top V_{22}^{-1} \widetilde{\hh}_2 & = & \sum_{i=r+1}^n \frac{ \tilde{h}_i^2 }{ v_{ii} }
+ \widetilde{\hh}_2^\top \widetilde{W}_{22} \widetilde{\hh}_2,
\end{eqnarray}
where $\mathbf{h}_1=(h_1, \ldots, h_r)^\top$, $\mathbf{h}_2=(h_{r+1}, \ldots, h_n)^\top$, and $h_i$ and  $\tilde{h}_i$ are given in (22) and (27), %\eqref{eq:definition:h} and \eqref{defintion-tilde-h},
respectively.
Since $\max\{\|\mathbf{h}\|_\infty, \| \bs{\widetilde{h}} \|_\infty\} \lesssim b_n^2\log n /c_n$ (see (23) and (27) %\eqref{ineq-beta-h} and\eqref{defintion-tilde-h}
), we have
\[
\sum_{i=1}^n \frac{ h_i^2 }{ v_{ii} } - \sum_{i=r+1}^n \frac{ \tilde{h}_i^2 }{ v_{ii} } = \sum_{i=r+1}^n \frac{ h_i^2 - \tilde{h}_i^2  }{ v_{ii} } + O_p\left( \frac{ b_n^5(\log n)^2 }{nc_n^2} \right).
\]
With the use of the mean value theorem and by \eqref{ineq-mu-deriv-bound},
the difference $h_i-\tilde{h}_i$ is bounded as follows:
\begin{eqnarray}
\nonumber
\max_{i=r+1,\ldots, n} |h_i-\tilde{h}_i| & \le &   \sum_{j\neq i} \left\{ |\mu^{\prime\prime}(\widetilde{\pi}_{ij})( \widehat{\pi}_{ij} - \pi_{ij})^2 - \mu^{\prime\prime}(\tilde{\pi}_{ij}^0)
( \widehat{\pi}_{ij} - \pi_{ij})^2| \right. \\
\nonumber
&&
\left.+ |\mu^{\prime\prime}(\tilde{\pi}_{ij}^0) ( \widehat{\pi}_{ij} - \pi_{ij})^2 - \mu^{\prime\prime}(\tilde{\pi}_{ij}^0)(( \widehat{\pi}_{ij} - \pi_{ij})^2)^2]| \right\} \\
\nonumber
& \lesssim & \frac{n}{c_n} \left\{ | \tilde{\pi}_{ij}-\tilde{\pi}_{ij}^0| \cdot ( \widehat{\pi}_{ij} - \pi_{ij})^2 +
 | \widehat{\pi}_{ij}-\widehat{\pi}_{ij}^0| \cdot ( |\widehat{\pi}_{ij} - \pi_{ij}|  +  |\widehat{\pi}_{ij}^0- \pi_{ij}|) \right\} \\
\nonumber
&\lesssim & \frac{n}{c_n} \cdot \left( b_n\sqrt{\frac{\log n}{n}}\right)^3 + \frac{n}{c_n} \cdot \frac{ b_n^3 \log n}{nc_n} \cdot \left( b_n\sqrt{\frac{\log n}{n}}\right)^2\\
\label{ineq-h-h-r1}
& \lesssim & \frac{ b_n^3 (\log n)^2 }{n^{1/2} c_n},
\end{eqnarray}
where the third inequality follows from \eqref{ineq-beta-beta0-upp-f}.
Therefore, if $ b_n^5/c_n^2 = o( n/(\log n)^2 )$, then
\begin{equation}\label{eq-theorem2-hhd}
| \sum_{i=1}^n \frac{ h_i^2 }{ v_{ii} } - \sum_{i=r+1}^n \frac{ \tilde{h}_i^2 }{ v_{ii} } | = O_p(  \frac{ b_n^5 (\log n)^2 }{nc_n^2} ) + O(\frac{ b_n^4 (\log n)^2 }{n^{3/2} c_n})  = o_p(1).
\end{equation}
By \eqref{ineq-V-S-appro-upper-b} and (27) %\eqref{ineq-beta-h}
, we have
\begin{eqnarray}
| \hh_1^\top W_{11} \hh_1 | & \le & r^2 \| W_{11} \|_{\max} \| \hh_1 \|_\infty \lesssim r^2 \cdot b_n^2 \log n \cdot \frac{ b_n^3 }{ n^2c_n^2 }\lesssim  \frac{ b_n^5 \log n}{ nc_n^2},
\\
| \hh_1^\top W_{12} \hh_2 | & \lesssim & r(n-r) \cdot b_n^2\log n \cdot \frac{ b_n^3 }{ n^2c_n^2 }
\lesssim \frac{ b_n^5 \log n}{ nc_n^2 }.
\end{eqnarray}
To evaluate the bound of $\hh_2^\top W_{22} \hh_2 - \widetilde{\hh}_2^\top \widetilde{W}_{22}^{-1} \widetilde{\hh}_2$,
we divide it into three terms:
\begin{eqnarray}
\nonumber
& & \hh_2^\top W_{22} \hh_2 - \widetilde{\hh}_2^\top \widetilde{W}_{22}^{-1} \widetilde{\hh}_2 \\
\label{eq-theorem2-whh}
& = & \underbrace{\hh_2^\top W_{22} \hh_2 - \hh_2^\top \widetilde{W}_{22} \hh_2}_{C_1} +
\underbrace{\hh_2^\top \widetilde{W}_{22} \hh_2
- \widetilde{ \hh}_2 \widetilde{W}_{22} \hh_2}_{C_2}
+
\underbrace{\widetilde{\hh}_2^\top W_{22} \hh_2 - \widetilde{\hh}_2 \widetilde{W}_{22} \widetilde{\hh}_2}_{C_3}.
\end{eqnarray}
The first term $C_1$ is bounded as follows.
By \eqref{ineq-W-diff-upper} and (27) %\eqref{ineq-beta-h}
, we have
\begin{eqnarray}
\nonumber
| \hh_2^\top ( W_{22} - \widetilde{W}_{22} ) \hh_2 | & \le & (n-r)^2 \| W_{22} - \widetilde{W}_{22} \|_{\max} \| \hh \|_\infty \\
\label{ineq-upper-B1}
& \lesssim & n^2 \cdot \frac{ b_n^6 }{ n^3c_n^5 } \cdot b_n^2 \log n  \lesssim \frac{ b_n^8}{ n c_n^5}.
\end{eqnarray}
In view of \eqref{ineq-V-S-appro-upper-b} and \eqref{ineq-h-h-r1}, the upper bounds of $C_2$ and $C_3$ are derived as follows:
\begin{eqnarray}
\nonumber
|C_2|& = %%& | \bs{h}_2^\top \widetilde{W}_{22} \bs{h}_2 - \widetilde{ \bs{h} }_2 \widetilde{W}_{22} \bs{h}_2 |
&   |( \bs{h}_2- \widetilde{\bs{h} }_2) ^\top \widetilde{W}_{22} \bs{h}_2 |
 \le  (n-r)^2 \| \widetilde{W}_{22} \|_{\max} \| \bs{h}_2- \widetilde{\bs{h}}_2 \|_\infty \| \bs{h}_2 \|_\infty \\
\label{ineq-upper-B2}
& \lesssim & n^2 \cdot \frac{b_n^3}{n^2c_n^2} \cdot \frac{ b_n^3 (\log n)^{2} }{ n^{1/2}c_n } \cdot b_n^2 \log n
 \lesssim \frac{ b_n^8 (\log n)^{3} }{ n^{1/2}c_n^3 }.
\end{eqnarray}
and
\begin{eqnarray}
\nonumber
|C_3|& = & |\widetilde{\bs{h}}_2^\top \widetilde{W}_{22} (\bs{h}_2 -  \widetilde{\bs{h}}_2) |
 \le  (n-r)^2 \|\widetilde{\bs{h}}_2\|_\infty  \| \widetilde{W}_{22} \|_{\max} \| \bs{h}_2 -  \widetilde{\bs{h}}_2 \|_\infty \\
\nonumber
& \le & (n-r)^2 \cdot \| W_{22} \|_{\max} \cdot \| \widetilde{\hh}_2 \|_\infty \cdot \| \widetilde{\hh}_2 - \hh_2 \|_\infty
+ (n-r)^2 \cdot \| W_{22} - \widetilde{W}_{22}  \|_{\max} \cdot \| \widetilde{\hh}_2 \|_\infty^2 \\
\label{ineq-upper-B3}
& \lesssim & \frac{ b_n^8 (\log n)^{3} }{ n^{1/2}c_n^3 }.
\end{eqnarray}
By combining \eqref{eq-th2a-hVh}--\eqref{ineq-upper-B3}, it yields
\begin{equation}
|\mathbf{h}^\top V^{-1} \mathbf{h} -
\widetilde{\hh}_2^\top V_{22}^{-1} \widetilde{\hh}_2|
\lesssim \frac{ b_n^8 (\log n)^{5/2} }{ n^{1/2}c_n } .
\end{equation}
This completes the proof of the first step.

Step 2. We bound $B_2 - B_2^0$.
\iffalse
Similar to \eqref{cacluation-z}, we have
\begin{equation}
B_2^0 = \frac{1}{6}\left\{ \sum_{i=1}^n  (\widehat{\beta}_i^0 -\beta_i^0)^3 \sum_{j\neq i} \mu^{\prime\prime}( \beta_i^0 + \beta_j^0)
+ 2 \sum_{i=1}^n \sum_{j=1, j\neq i}^n  (\widehat{\beta}_i^0-\beta_i^0)^2(\widehat{\beta}_j^0-\beta_j^0)
\mu^{\prime\prime}( \beta_i^0 + \beta_j^0) \right\}.
\end{equation}
Note that $\widehat{\beta}_i^0=\beta_i^0$, $i=1, \ldots, r$, and
\begin{eqnarray*}
B_2 & = &\frac{1}{6}\{ \sum_{i=1}^n  (\widehat{\beta}_i-\beta_i^0)^3 \sum_{j\neq i} \mu^{\prime\prime}( \beta_i^0 + \beta_j^0 )
+ 2 \sum_{i,j=1, j\neq i}^n  (\widehat{\beta}_i-\beta_i^0)^2(\widehat{\beta}_j-\beta_j)\mu^{\prime\prime}( \beta_i^0 + \beta_j^0 ) \}.
\end{eqnarray*}
\fi
For a cubic term in $B_2-B_2^0$, a direct scaling method gives that
\begin{eqnarray}
\nonumber
&&| \sum_{i=1}^n [ (\widehat{\beta}_i - \beta_i^0)^3 - ( \widehat{\beta}_i^0 - \beta_i^0)^3 ] [\sum_{j\neq i} \mu^{\prime\prime}( \pi_{ij} ) ]|\\
\nonumber
& \le & \frac{n(n-1)}{c_n} \cdot \| \bs{\widehat{\beta}} - \bs{\widehat{\beta}}^0 \|_\infty \cdot 2[\sum_i (\widehat{\beta}_i - \beta_i^0)^2 +   \sum_i (\widehat{\beta}_i^0 - \beta_i^0)^2] \\
\nonumber
& \lesssim & \frac{ b_n^9 (\log n)^3 }{ c_n^3 }.
\end{eqnarray}
The term in the above right hand does not tend to zero. Because $r$ is fixed, the approach for showing $(B_2-B_2^0)/r^{1/2}=o_p(1)$ in the proof of Theorem 1 does not work yet.
 To prove that this term does go to zero, we did a careful analysis on the difference
$B_2-B_2^0$ by using asymptotic representations of $\widehat{\beta}_i - \beta_i$ and $\widehat{\beta}_i^0 - \beta_i$.
With the use of Lemmas 5 and \ref{lemma-hat-beta-diff}, we have
\begin{equation}\label{ineq-B2-B20-a}
|B_2 - B_2^0| \lesssim \frac{ b_n^9 (\log n)^3 }{ n^{1/2} c_n^3 },
\end{equation}
whose detailed proofs are given in Section \ref{subsection:B2B20}. % 4.4 in Supplementary Material A.

Step 3. We bound $B_3 - B_3^0$. With the same reason as in Step 2, we could not yet use the method of proving $(B_3-B_3^0)/r^{1/2}=o_p(1)$
in the proof of Theorem 1.
With the use of asymptotic representations of $\widehat{\beta}_i - \beta_i$ and $\widehat{\beta}_i^0 - \beta_i$ and Lemma \ref{lemma-hat-beta-diff}, we can show
\begin{equation}\label{eq-th2a-B3B30}
|B_3 - B_3^0| \lesssim \frac{b_n^6 (\log n)^{3} }{ n^{1/2}c_n},
\end{equation}
whose detailed proofs are given in \ref{subsection-B3B30}. % Section 4.5 in Supplementary Material A.
This completes the proof.
\end{proof}

\subsection{Proofs for Theorem 2 (b)}  %%\ref{}

Recall that $\bs{\bar{d}}_1 = ( \bar{d}_1, \ldots, \bar{d}_r)^\top$, $\bs{\bar{d}}_2 = ( \bar{d}_{r+1}, \ldots, \bar{d}_n)^\top$, $\bs{\widetilde{d}}=(\sum_{i=1}^r d_i, d_{r+1},\ldots, d_n)$ and $\widetilde{V}$ is given in \eqref{approxi-inv2-beta-ho}.
$\widetilde{V}$ is the Fisher information matrix of $\widetilde{\bs{\beta}}=(\beta_1, \beta_{r+1}, \ldots, \beta_n)^\top$
under the null $H_0: \beta_1 = \cdots= \beta_r$.
Remark that $r$ is a fixed constant in this section.
Partition $\widetilde{W}$ into four blocks
\begin{equation}\label{VW-divide}
\widetilde{W} = \begin{pmatrix} \tilde{w}_{11} & \bs{\tilde{w}}_{12} \\
\bs{\tilde{w}}_{21} & \widetilde{W}_{22}
\end{pmatrix},
\end{equation}
where $\tilde{w}_{11}$ is a scalar and the dimension of $\widetilde{W}_{22}$ is $(n-r)\times (n-r)$.
It should be noted $\widetilde{W}_{22}$ is different from $\widetilde{W}_{22}$
in the proof of Theorem 1 (a),
where $\widetilde{W}_{22}=V_{22}^{-1}-S_{22}$.
With some ambiguity of notation, we use the same notation here. However, both share very similar properties.

To prove Theorem 2 (b), we need the following three lemmas,
whose proofs are respectively similar to those of Lemmas \ref{lemma:w2-error}, \ref{lemma-W-widetilde-d} and \ref{lemma-hat-beta-diff} and omitted.

Recall that $W_{22}$ is the bottom right $(n-r)\times (n-r)$ block of $W=V^{-1}-S$.
The lemma below gives an upper bound of $\| W_{22} - \widetilde{W}_{22} \|_{\max}$,
whose magnitudes are $b_n^3/(n^3c_n^5)$.

\begin{lemma}\label{lemma:w2-error-2b}
For a fixed constant $r$, the error between $W_{22}$ and $\widetilde{W}_{22}$ in terms of the maximum absolute entry-wise norm has the following bound:
\begin{equation}\label{ineq-WW-diff-2b}
\| W_{22} - \widetilde{W}_{22} \|_{\max} \lesssim \frac{ b_n^6 }{ n^3c_n^5 }.
\end{equation}
\end{lemma}

It is remark that the absolute entry-wise error between $W_{22}$ and $\widetilde{W}_{22}$ in the Bradley-Terry model is not in the order of $O(n^{-3})$ that holds in the $\beta$-model (i.e., \eqref{ineq-WW-diff-2b}),
but $O(n^{-3})$ adding a special matrix whose order is $O(n^{-2})$ (see (289) on page 20 in Supplementary Material B).
The following lemma gives the upper bounds of three remainder terms in \eqref{eq-theorem2b-B1022} that tend to zero,
whose proof is similar to the proof of Lemma \ref{lemma-W-widetilde-d}  and is omitted.

\begin{lemma}\label{lemma-W-widetilde-d-2b}
Suppose $r$ is a fixed constant. \\
(a)If $b_n^3/c_n^2=o( n^{3/2}/(\log n)^{1/2})$, then $(\sum_{i=1}^r\bar{d}_i) \tilde{w}_{11} (\sum_{i=1}^r\bar{d}_i) = o_p(1)$. \\
(b)If $b_n^3/c_n^3=o( n^{1/2} )$, then $(\sum_{i=1}^r\bar{d}_i) \bs{\tilde{w}}_{12}^\top \bar{\dd}_2 = o_p(1)$. \\
(c)If $b_n^3/c_n^3=o( n^{3/4} )$, then
\[
\bar{\dd}_2^\top ( W_{22} - \widetilde{W}_{22}) \bar{\dd}_2 = o_p(1).
\]
\end{lemma}

The lemma below establishes the upper bound of $ \widehat{\beta}_i - \widehat{\beta}_i^0 $.

\begin{lemma}\label{lemma-hat-beta-diff-2b}
Under the null $H_0: \beta_1=\cdots = \beta_r$ with a fixed $r$,
if $b_n^3/c_n = o( n/\log n)$, then with probability at least $1-O(n^{-1})$,
\[
\max_{i=r+1, \ldots, n} | \widehat{\beta}_i - \widehat{\beta}_i^0 | \lesssim \frac{b_n^3 \log n}{nc_n}.
\]
\end{lemma}

The above error bound is in the magnitude of $n^{-1}$, up to a factor $b_n^3\log n$,
which makes the remainder terms in \eqref{eq-theorem2b-B1022} be asymptotically neglected.
Note that this error bound is the same as that in Lemma \ref{lemma-hat-beta-diff}.

Now, we are ready to prove Theorem 2 (b).

\begin{proof}[Proof of Theorem 2 (b)]
Note that $\bs{\widehat{\beta}}^0$ denotes the restricted MLE under the null space $\Theta_0 = \{ \bs{\beta}\in \R^n: \beta_1= \cdots = \beta_r \}$.
The following calculations are based on the event $E_n$ %that is defined in \eqref{ineq-beta-beta0-upp}.
 that $\bs{\widehat{\beta}}$ and $\bs{\widehat{\beta}}^0$ simultaneously exist and satisfy
\begin{equation}\label{ineq-beta-beta0-upp-f}
\max\left\{ \| \widehat{\bs{\beta}} - \bs{\beta} \|_\infty, \| \widehat{\bs{\beta}}^0 - \bs{\beta} \|_\infty \right\}
\le \frac{3nb_n}{(2n-1)}\sqrt{ \frac{\log n}{n} },~~\max_{i=r+1, \ldots, n} | \widehat{\beta}_i - \widehat{\beta}_i^0 | \lesssim \frac{b_n^3 \log n}{nc_n}.
\end{equation}
By Lemmas 4 and \ref{lemma-hat-beta-diff-2b}, $\P(E_n) \ge 1 - O(n^{-1})$ if $b_n^3/c_n^2=o( n/\log n)$.

Similar to the proof of Theorem 2 (a),
it is sufficient to demonstrate: (1) $2( B_1 - B_1^0)$ converges in distribution to the Chi-square distribution with $r$ degrees of freedom;
(2)
\[
B_2-B_2^0 = O_p\left( \frac{b_n^9(\log n)^3 }{ n^{1/2} c_n^3 }\right), ~~ B_3 - B_3^0 = O_p\left( \frac{ b_n^6 (\log n)^3 }{ n^{1/2} c_n } \right).
\]
The proof of claim (2) is similar to those of \eqref{ineq-B2-B20} and \eqref{eq-th2a-B3B30} and omitted.
We only present the proof of claim (1) here. %% and the proof of claim (2) is presented in the supplementary material.

We show $2( B_1 - B_1^0)\stackrel{L}{\to} \chi^2_r$.
\iffalse
Recall that the expression of $B_1^0$ in the proof of Theorem \ref{theorem-LRT-beta} (b) is
\begin{equation*}\label{B10-homo-fixed-2b}
B_1^0  = \frac{1}{2}( \bs{\widetilde{d}} - \E \bs{\widetilde{d}} )^\top \widetilde{V}^{-1} ( \bs{\widetilde{d}} - \E \bs{\widetilde{d}} )
-\frac{1}{2}\bs{\widetilde{h}}^\top \widetilde{V}^{-1} \bs{\widetilde{h}},
\end{equation*}
where  $\bs{\tilde{d}}=(\sum_{i=1}^r d_i, d_{r+1},\ldots, d_n)$ and $\mathbf{\widetilde{h}} = (\tilde{h}_1, \tilde{h}_{r+1}, \ldots, \tilde{h}_n)^\top$ and
\begin{equation}
\label{eq-2b-tilde-h}
%\renewcommand{\arraystretch{1.5}}
\begin{array}{rcl}
\tilde{h}_1 & \lesssim &  \frac{rb_n^6\log n}{ c_n^5}, \\
\max_{i=r+1, \ldots, n} \tilde{h}_i & \lesssim & \frac{b_n^6\log n}{ c_n^5}.
\end{array}
\end{equation}
\fi
Using the matrix form in \eqref{VW-divide}, $ B_1 - B_1^0$  can be written as %%in \eqref{eq-ell-difference}
\begin{eqnarray}
\nonumber
& &2(B_1 - B_1^0) \\
\nonumber
&  = & \underbrace{\sum_{i=1}^r \frac{ \bar{d}_i^{\,2} }{v_{ii}}  - \frac{ ( \sum_{i=1}^r \bar{d}_i )^2 }{ \tilde{v}_{11} }}_{Z_1}
 + \mathbf{\bar{d}}_1^\top W_{11} \mathbf{\bar{d}}_1 + 2\mathbf{\bar{d}}_1^\top W_{12} \mathbf{\bar{d}}_2 + (\sum_{i=1}^r \bar{d}_i )^2 \tilde{w}_{11} \\
 \label{eq-theorem2b-B1022}
&&  + 2 ( \sum_{i=1}^r \bar{d}_i) \bs{\tilde{w}}_{12}^\top  \mathbf{\bar{d}}_2
+  \mathbf{\bar{d}}_2^\top ( W_{22} - \widetilde{W}_{22}) \mathbf{\bar{d}}_2
+ \underbrace{\mathbf{h}^\top V^{-1} \mathbf{h} - \mathbf{\widetilde{h}}^\top \widetilde{V}^{-1} \mathbf{\widetilde{h}}}_{Z_2},
\end{eqnarray}
where $\bs{h}$ is defined in (22) and $\mathbf{\widetilde{h}}$ is in (33). % \eqref{eq:definition:h} and $\mathbf{\widetilde{h}}$ is in \eqref{eq-homo-tildeh}.
In view of Lemmas \ref{lemma-W-widetilde-d} and \ref{lemma-W-widetilde-d-2b}, it is sufficient to demonstrate: (i) $Z_1$ converges in distribution to a Chi-square distribution with $r-1$ degrees of freedom;
(ii) $Z_2=o_p(1)$.
Because $\tilde{d}_i = \sum_{j=r+1}^n a_{ij}$ is independent over $i=1,\ldots, r$ and $r$ is a fixed constant,
the classical central limit theorem for the bounded case (\cite{Loeve:1977}, p. 289) gives that
the vector $(\bar{d}_1/v_{11}^{1/2}, \ldots, \bar{d}_r/v_{rr}^{1/2})$
follows a $r$-dimensional standard normal distribution.
Because $v_{11}=\cdots=v_{rr}$ under the null $H_0:\beta_1=\cdots=\beta_r$ and $\tilde{v}_{11}=rv_{11}$, we have
\[
\sum_{i=1}^r \frac{ \bar{d}_i^{\,2} }{v_{ii}}  - \frac{ ( \sum_{i=1}^r \bar{d}_i )^2 }{ \tilde{v}_{11} }
= \left( \frac{\bar{d}_1}{v_{11}^{1/2}}, \ldots, \frac{\bar{d}_r}{v_{11}^{1/2}}\right)\left( I_r - \frac{1}{r} \mathbf{1}_r \mathbf{1}_r^\top \right)
\left( \frac{\bar{d}_1}{v_{11}^{1/2}}, \ldots, \frac{\bar{d}_r}{v_{11}^{1/2}}\right)^\top.
\]
Because $\mathrm{rank}( I_r - \mathbf{1}_r \mathbf{1}_r^\top/r )=r-1$, it follows that we have claim (i).
Now, we show $Z_2=o_p(1)$.
By setting $V^{-1}=S+W$ and $V_{22}^{-1}=S_{22}+W_{22}$, we have
\begin{eqnarray*}
\mathbf{h}^\top V^{-1} \mathbf{h} &  = & \sum_{i=1}^n \frac{ h_i^2 }{ v_{ii} } + \hh_1^\top W_{11} \hh_1
+ 2 \hh_1^\top W_{12} \hh_2 + \hh_2^\top W_{22} \hh_2, \\
\widetilde{\hh}_2^\top V_{22}^{-1} \widetilde{\hh}_2 & = & \frac{ \tilde{h}_1^2 }{ \tilde{v}_{11}}
+ \sum_{i=r+1}^n \frac{ \tilde{h}_i^2 }{ v_{ii} }
+ \tilde{w}_{11} \tilde{h}_1^2 + 2 \tilde{h}_1 \tilde{w}_{12} \bs{\tilde{h}}_2 + \bs{\tilde{h}}_2^\top \widetilde{W}_{22} \bs{\tilde{h}}_2,
\end{eqnarray*}
where $\mathbf{h}_1=(h_1, \ldots, h_r)^\top$, $\mathbf{h}_2=(h_{r+1}, \ldots, h_n)^\top$, and $h_i$ and  $\tilde{h}_i$ are given  (22) and (33) in the main text, %%in \eqref{eq:definition:h} and \eqref{eq-homo-tildeh},
respectively.
%Because $\max\{\|\mathbf{h}\|_\infty, \| \bs{\widetilde{h}} \|_\infty\} \lesssim b_n^2\log n /c_n$ (see \eqref{ineq-beta-h} and %\eqref{eq-homo-tildeh}) and $r$ is a fixed constant,
In view of (22) and (33), % \eqref{ineq-beta-h} and \eqref{eq-homo-tildeh},
we have
\[
\sum_{i=1}^n \frac{ h_i^2 }{ v_{ii} } - \sum_{i=r+1}^n \frac{ \tilde{h}_i^2 }{ v_{ii} } = \sum_{i=r+1}^n \frac{ h_i^2 - \tilde{h}_i^2  }{ v_{ii} } + O( \frac{ b_n^5(\log n)^2 }{nc_n^2}).
\]
With the same arguments as in the proof of \eqref{ineq-h-h-r1}, we have
\begin{equation*}
\label{ineq-h-h-r1-2b}
\max_{i=r+1,\ldots, n} |h_i-\tilde{h}_i|   \lesssim  \frac{ b_n^3 (\log n)^2 }{n^{1/2} c_n}.
\end{equation*}
Therefore, if $ b_n^6/c_n = o( n/(\log n)^2 )$, then
\begin{equation}\label{eq-theorem2-hhd}
| \sum_{i=1}^n \frac{ h_i^2 }{ v_{ii} } - \sum_{i=r+1}^n \frac{ \tilde{h}_i^2 }{ v_{ii} } | = O_p(  \frac{ b_n^5 (\log n)^2 }{nc_n^2} ) + O(\frac{ b_n^4 (\log n)^2 }{n^{3/2} c_n})  = o(1).
\end{equation}
By (23), %\eqref{ineq-beta-h},
\eqref{ineq-V-S-appro-upper-b} and \eqref{approxi-inv2-beta-ho}, we have
\begin{eqnarray*}
| \hh_1^\top W_{11} \hh_1 | & \le & r^2 \| W_{11} \|_{\max} \| \hh_1 \|_\infty \lesssim r^2 \cdot b_n^2 \log n \cdot \frac{ b_n^3 }{ n^2c_n^2 }\lesssim  \frac{ b_n^5 \log n}{ nc_n^2},
\\
| \hh_1^\top W_{12} \hh_2 | & \lesssim & r(n-r) \cdot b_n^2\log n \cdot \frac{ b_n^3 }{ n^2c_n^2 }
\lesssim \frac{ b_n^5 \log n}{ nc_n^2 },  \\
\frac{ \tilde{h}_1^2 }{ \tilde{v}_{11}} & \lesssim & \frac{ r^2 ( b_n^2\log n/c_n)^2 }{ rn/b_n} \lesssim \frac{ b_n^3 (\log n)^2 }{ nc_n^2}, \\
|\tilde{w}_{11} \tilde{h}_1^2| & \lesssim &  ( b_n^2\log n/c_n)^2 \cdot \frac{ b_n^3 }{ n^2 c_n^2 } \lesssim \frac{ b_n^5 (\log n)^2 }{ n^2 c_n^4}, \\
 |\tilde{h}_1 \tilde{w}_{12} \bs{\tilde{h}}_2| & \lesssim & r(n-r) \cdot ( b_n^2\log n/c_n)^2 \cdot \frac{ b_n^3 }{ n^2 c_n^2 }\lesssim \frac{ b_n^5 (\log n)^2 }{ n c_n^4}.
\end{eqnarray*}
As in the proofs of \eqref{eq-theorem2-whh}--\eqref{ineq-upper-B3}, we have
\[
|\bs{h}_2 W_{22} \bs{h}_2 - \widetilde{\hh}_2^\top \widetilde{W}_{22} \widetilde{\hh}_2| \lesssim  \frac{ b_n^8 (\log n)^3 }{ n^{1/2} c_n^3 }.
\]
Combining the above inequalities and \eqref{eq-theorem2-hhd}, it yields
\begin{equation*}
|\mathbf{h}^\top V^{-1} \mathbf{h} -
\widetilde{\hh}^\top \widetilde{V}^{-1} \widetilde{\hh}|
\lesssim \frac{ b_n^8 (\log n)^3 }{ n^{1/2}c_n }.
\end{equation*}
which shows claim (ii).

\end{proof}

\section{Variances of weighted sums for $\{\bar{d}_{i}^{\,2}\}_{i=1}^n$ and $\{ \bar{d}_i^{\,3} \}_{i=1}^n$}
\label{section-variance}

This section presents the expressions of the variances of the weighted quadratic sum $\sum_i f_i \bar{d}_i^{\,2}$
and  the weighted cubic sum $\sum_i f_i \bar{d}_i^{\,3}$, as well as the upper bound of a mixed sum $\sum_{i,j} f_{ij} \bar{d}_i^{\,2} \bar{d}_j$. They are stated in Lemmas \ref{lemma:var:quadra}, \ref{lemma:var:cuibic} and \ref{lemma:var:cuibic2}, respectively.
Recall that $\bar{a}_{ij}=a_{ij}-\E(a_{ij})$ for $i\neq j$ and $\bar{a}_{ii}=0$ for all $i=1, \ldots, n$, and
\[
\bar{d}_i = d_i - \E (d_i) = \sum_j \bar{a}_{ij}.
\]
For a given sequence $\{f_i\}_{i=1}^n$,
the variance of the weighted quadratic sum  $\sum_i f_i\bar{d}_i^{\,2}$ is given below.

\begin{lemma}\label{lemma:var:quadra}
Let $u_{ij}= \mathrm{Cov}(\bar{a}_{ij}^{\,2}, \bar{a}_{ji}^{\,2} )$ and $v_{ii}=\sum_j\mathrm{Var}(\bar{a}_{ij})$.
For a given sequence $\{f_i\}_{i=1}^r$, we have
\begin{equation}\label{eq:expression:variance}
\mathrm{Var}( \sum_{i=1}^r f_i \bar{d}_i^{\,2} )=\sum_{i=1}^r f_i^2 (2v_{ii}^2 + \sum_{j=1,j\neq i}^n u_{ij} )
+ 2\sum_{ 1\le i<j \le r } f_i f_j u_{ij}.
\end{equation}
\end{lemma}

\begin{proof}
The calculation of the variance of $\sum_i f_i \bar{d}_i^{\,2}$ can be divided into two parts:
\begin{equation}\label{eq:variance:weight}
\mathrm{Var}( \sum_{i=1}^r f_i \bar{d}_i^{\,2} )=
\underbrace{\sum_{i=1}^r f_i^2 \mathrm{Var}( \bar{d}_i^{\,2} )}_{\mbox{part 1}} + \underbrace{2 \sum_{1\le i<j\le r} f_i f_j \mathrm{Cov}( \bar{d}_i^{\,2}, \bar{d}_j^{\,2} )}_{\mbox{part 2}}.
\end{equation}
The first part can be calculated as follows:
\begin{eqnarray*}
\mathrm{Var}( \bar{d}_i^{\,2}) = \mathrm{Cov} ( (\sum_{\alpha=1 }^n \bar{a}_{i\alpha})^{\,2}, (\sum_{h=1 }^n \bar{a}_{ih} )^{\,2} )
= \mathrm{Cov}( \sum_{\alpha=1}^n \sum_{\beta=1}^n \bar{a}_{i\alpha} \bar{a}_{i\beta}, \sum_{h=1}^n \sum_{g=1}^n \bar{a}_{ih} \bar{a}_{ig}).
\end{eqnarray*}
Note that the random variables $\bar{a}_{ij}$ for $1\le i < j \le n$ are mutually independent.
There are only two cases in terms of $(\alpha, \beta, h, g)$ for which
$\mathrm{Cov}( \bar{a}_{i\alpha} \bar{a}_{i\beta}, \bar{a}_{ih} \bar{a}_{ig} )$ is not equal to zero:
(Case A) $\alpha=\beta=h=g \neq i$;
(Case B) $\alpha=h, \beta=g$ or $\alpha=g, \beta=h$.
By respectively considering Case A and Case B,  a direct calculation gives that
\begin{equation}\label{eq:variance:2}
\mathrm{Var}( \bar{d}_i^{\,2} )= 2 v_{ii}^2 + \sum_{j=1,j\neq i}^n u_{ij}.
\end{equation}

The second part of \eqref{eq:variance:weight} can be calculated as follows:
\begin{equation*}
\mathrm{Cov}( \bar{d}_i^{\,2}, \bar{d}_j^{\,2} ) = \mathrm{Cov}( (\sum_{\alpha=1}^n \bar{a}_{i\alpha})^{\,2}, ( \sum_{h=1}^n \bar{a}_{jh})^{\,2} )
=\mathrm{Cov}( \sum_{\alpha=1}^n \sum_{\beta=1}^n \bar{a}_{i\alpha} \bar{a}_{i\beta}, \sum_{h=1}^n \sum_{g=1}^n \bar{a}_{jh} \bar{a}_{jg}).
\end{equation*}
In the above, the only case for $\mathrm{Cov}( \bar{a}_{i\alpha} \bar{a}_{i\beta}, \bar{a}_{jh} \bar{a}_{jg})$ not being equal to $0$,
is $\alpha=\beta=j$. % and $h=g=i$.
According to the definition of $u_{ij}$, we have
\begin{equation}\label{eq:vairance:3}
\mathrm{Cov}( \bar{d}_i^{\,2}, \bar{d}_j^{\,2} ) = \E(\bar{a}_{ij}^{\,2} \bar{a}_{ji}^{\,2} )- \E(\bar{a}_{ij}^{\,2})\E(\bar{a}_{ji}^{\,2})=u_{ij}.
\end{equation}
By combing \eqref{eq:variance:2}, \eqref{eq:vairance:3} and \eqref{eq:variance:weight}, it yields equation \eqref{eq:expression:variance}.
\end{proof}

Now, we present the variance of the cubic weighted sum.

\begin{lemma}\label{lemma:var:cuibic}
For a given sequence $\{f_i\}_{i=1}^r$,
the variance of $\sum_{i=1}^r f_i \bar{d}_i^{\,3}$ has the following expression:
\begin{eqnarray} %\label{eq:expression:variance3}
\nonumber
\mathrm{Var}( \sum_{i=1}^r f_i \bar{d}_i^{\,3} )  =  \sum_{i=1}^r f_i^2 \left\{ \sum_{t=1}^n (\E \bar{a}_{i t}^6 - (\E \bar{a}_{i t}^3)^2) +3 \sum_{1\le h\neq g \le n} \E \bar{a}_{i h}^4 \E \bar{a}_{ig}^2 + \right.~~~~~~~~~~~~~~~~~~ \\
\label{eq:expression:variance3}
~~~~~~~~~~~~~~~~~ \left.6 \sum_{1\le h\neq g \neq t \le n } \E \bar{a}_{i g}^2 \E \bar{a}_{ih}^2 \E \bar{a}_{it}^2 \right\} + 2\sum_{ 1\le i<j \le r } f_i f_j \{ \E(\bar{a}_{ij}^{\,3} \bar{a}_{ji}^{\,3} )- \E(\bar{a}_{ij}^{\,3})\E(\bar{a}_{ji}^{\,3})\} \\
\nonumber
~~~~~~~~~~~~~~~~~ +8\sum_{ 1\le i<j \le r } \left(f_{ij}\E \bar{a}_{ij}\bar{a}_{ji} \right) \left( \sum_{g=1,g\neq i,j}^n \E \bar{a}_{i g}^{\,2} \right) \left( \sum_{h=1,h\neq i, j}^n \E \bar{a}_{j h}^{\,2}\right).~~~~~~~~~~~~~~~~~~~
\end{eqnarray}
\end{lemma}

\begin{proof}
Similar to the proof of Lemma \ref{lemma:var:quadra},
the calculation of the variance of $\sum_i f_i \bar{d}_i^{\,3}$ can also be divided into two parts:
\begin{equation}\label{eq:variance:weight3}
\mathrm{Var}( \sum_i f_i \bar{d}_i^{\,3} )=
\underbrace{\sum_{i=1}^n f_i^2 \mathrm{Var}( \bar{d}_i^{\,3} )}_{\mbox{part 1}} + \underbrace{2 \sum_{1\le i<j\le n} f_i c_j \mathrm{Cov}( \bar{d}_i^{\,3}, \bar{d}_j^{\,3} )}_{\mbox{part 2}}.
\end{equation}
The first part can be expressed as
\begin{equation}\label{eq-cal3-a}
\mathrm{Var}( \bar{d}_i^{\,3})
= \sum_{\alpha=1}^n \sum_{\beta=1}^n \sum_{\gamma=1}^n \sum_{h=1}^n \sum_{g=1}^n \sum_{t=1}^n \mathrm{Cov}(  \bar{a}_{i\alpha} \bar{a}_{i\beta}\bar{a}_{i\gamma},  \bar{a}_{ih} \bar{a}_{ig} \bar{a}_{it} ).
\end{equation}
Note that the random variables $\bar{a}_{ij}$ for $1\le i < j \le n$ are mutually independent
and $\bar{a}_{ii}=0$ when $i=j$.
The first part can be calculated as follows.
There are six cases to consider according to the number of distinct values of six indices: $\alpha, \beta, \gamma, h, g, t$.\\
(Case A) All six indices, $\alpha, \beta, \gamma, h, g, t$, are equal. In this case, the summation in \eqref{eq-cal3-a} becomes
\[
\sum_{t=1}^n (\E \bar{a}_{i t}^6 - (\E \bar{a}_{i t}^3)^2).
\]
(Case B) All six indices, $\alpha, \beta, \gamma, h, g, t$, have exactly two distinct values. By considering all possible pairs, e.g.,
$(\bar{a}_{i g}^3, \bar{a}_{ig}\bar{a}_{ih}^2), (\bar{a}_{ig}^3, \bar{a}_{ig}^2 \bar{a}_{ih}), (\bar{a}_{ig}^3, \bar{a}_{ih}^3), \ldots $,  pairs where the covariance is not zero are
those like $(\bar{a}_{i g}^3, \bar{a}_{ig}\bar{a}_{ih}^2)$ for $g\neq h$. In this case, the summation in \eqref{eq-cal3-a} becomes
\[
3\sum_{g=1}^n \sum_{h=1,h\neq g}^n \E \bar{a}_{i g}^4 \E \bar{a}_{ih}^2.
\]
(Case C) All six indices, $\alpha, \beta, \gamma, h, g, t$, have exactly three distinct values. By considering all possible pairs, e.g.,
$(\bar{a}_{i g}^3, \bar{a}_{ig}\bar{a}_{ih}\bar{a}_{it}), (\bar{a}_{ig}^2\bar{a}_{ih}, \bar{a}_{ig}^2 \bar{a}_{it}), (\bar{a}_{ig}^2\bar{a}_{ih}, \bar{a}_{ih}^2\bar{a}_{it}), \ldots $,  pairs where the covariance is not zero are
those like $(\bar{a}_{i g}\bar{a}_{ih}\bar{a}_{it}, \bar{a}_{ig}\bar{a}_{ih}\bar{a}_{it})$ for distinct $g, h, t$. In this case, the summation in \eqref{eq-cal3-a} becomes
\[
6 \sum_{1\le h\neq g \neq t \le n } \E \bar{a}_{i g}^2 \E \bar{a}_{ih}^2 \E \bar{a}_{it}^2.
\]
(Case D) All six indices, $\alpha, \beta, \gamma, h, g, t$, have exactly four, five, or six distinct values. In all these cases,
$\E\bar{a}_{i\alpha} \bar{a}_{i\beta}\bar{a}_{i\gamma}\bar{a}_{ih} \bar{a}_{ig} \bar{a}_{it}$ and $\E\bar{a}_{i\alpha} \bar{a}_{i\beta}\bar{a}_{i\gamma}\E \bar{a}_{ih} \bar{a}_{ig} \bar{a}_{it}$
are equal zero because at least such one $\bar{a}_{i\beta}$ is independent of others.
\iffalse
There are only four cases in terms of $(\alpha, \beta, \gamma, h, g, t)$ for which
$\mathrm{Cov}( \bar{a}_{i\alpha} \bar{a}_{i\beta} \bar{a}_{i\gamma}, \bar{a}_{ih} \bar{a}_{ig} \bar{a}_{it} )$ is not equal to zero: \\
(Case A) All six indices, $\alpha, \beta, \gamma, h, g, t$, are equal; \\
(Case B) Four indices among all six indices are equal and the left two indices are also equal and distinct from
 the other four ones, e.g., $\alpha=\beta=\gamma=g$, $h=t$, $g\neq h$. \\
(Case C) Three indices among all six indices are equal and the left three indices are also equal and distinct from
 the other three ones, e.g., $\alpha=\beta=\gamma$, $h=g=t$, $\alpha \neq h$. \\
(Case D) All six indices are evenly divided into three different groups and the indices of each group are equal,
$\alpha=\beta$, $\gamma=h$, $g=t$, $\alpha\neq \gamma$, $\alpha \neq g$, $h \neq g$.
By respectively considering Cases A, B, C and D,  a direct calculation gives that
\fi
By combining the above cases, it yields,
\begin{equation}\label{eq:variance:22}
\mathrm{Var}( \bar{d}_i^{\,3} )= \sum_{t=1}^n (\E \bar{a}_{i t}^6 - (\E \bar{a}_{i t}^3)^2) +3 \sum_{1\le h\neq g \le n} \E \bar{a}_{i h}^4 \E \bar{a}_{ig}^2 + 6 \sum_{1\le h\neq g \neq t \le n } \E \bar{a}_{i g}^2 \E \bar{a}_{ih}^2 \E \bar{a}_{it}^2.
\end{equation}

With the similar arguments as in the calculation of part $1$, part $2$ in \eqref{eq:variance:weight3} has the following expression:
\begin{equation}\label{eq:vairance:33}
\mathrm{Cov}( \bar{d}_i^{\,3}, \bar{d}_j^{\,3} ) = \E(\bar{a}_{ij}^{\,3} \bar{a}_{ji}^{\,3} )- \E(\bar{a}_{ij}^{\,3})\E(\bar{a}_{ji}^{\,3})
+ 4\E \bar{a}_{ij}\bar{a}_{ji} \left( \sum_{g\neq i,j} \E \bar{a}_{i g}^{\,2} \right) \left( \sum_{h\neq i, j} \E \bar{a}_{j h}^{\,2}\right).
\end{equation}
By combing \eqref{eq:variance:22}, \eqref{eq:vairance:33} and \eqref{eq:variance:weight}, it yields equation \eqref{eq:expression:variance3}.
\end{proof}

Now, we present an upper bound of the variance of a mixed weighted sum.

\begin{lemma}\label{lemma:var:cuibic2}
For a fixed array $\{ f_{ij} \}_{i,j=1}^n$,
an upper bound of the variance of $\sum_{i\neq j} f_{ij} \bar{d}_i^{\,2}\bar{d}_j$ is below:
\begin{eqnarray*}
\mathrm{Var}( \sum_{i\neq j} f_{ij}\bar{d}_i^{\,2}\bar{d}_j )  \lesssim \frac{n^6}{c_n^3} \max_{i,j} |f_{ij}|^2.
\end{eqnarray*}
\end{lemma}

\begin{proof}
Note that
\[
\mathrm{Var}( \sum_{i\neq j} f_{ij}\bar{d}_i^{\,2}\bar{d}_j ) \le \max_{i,j} |f_{ij}|^2 \sum_{i\neq j} \sum_{\alpha\neq \gamma}
\mathrm{Cov}( \bar{d}_i^{\,2} \bar{d}_j, \bar{d}_{\alpha}^{\,2} \bar{d}_\gamma ).
\]
For $i\neq j$ and $\alpha\neq \gamma$,
the calculation of the covariance between $\bar{d}_i^{\,2}\bar{d}_j$ and $\bar{d}_{\alpha}^{\,2} \bar{d}_{\gamma}$ can also be divided into eight cases:
(Case 1) $i=\alpha, j=\gamma$; (Case 2) $i=\alpha, j\neq \gamma$; (Case 3) $i\neq \alpha, j=\gamma$; (Case 4) $i\neq \alpha, j\neq \gamma$;
(Case 5) $i=\gamma, j=\alpha$; (Case 6) $i=\gamma, j\neq \alpha$; (Case 7) $i\neq \gamma$, $j= \alpha$; (Case 8) $i\neq \gamma$, $j\neq \alpha$.
By writing the covariance between $\bar{d}_i^{\,2}\bar{d}_j$ and $\bar{d}_{\alpha}^{\,2} \bar{d}_{\gamma}$ into the following form
\begin{equation*}\label{eq-cal3-aa}
\mathrm{Cov}( \bar{d}_i^{\,2}\bar{d}_j, \bar{d}_{\alpha}^{\,2} \bar{d}_\gamma)
= \sum_{k=1}^n \sum_{s=1}^n \sum_{t=1}^n \sum_{\zeta=1}^n \sum_{\eta=1}^n \sum_{\xi=1}^n \left( \E  \bar{a}_{ik} \bar{a}_{is}\bar{a}_{jt}  \bar{a}_{\alpha \zeta} \bar{a}_{\alpha \eta} \bar{a}_{\gamma\xi}
-\E  \bar{a}_{ik} \bar{a}_{is}\bar{a}_{jt} \E \bar{a}_{\alpha \zeta} \bar{a}_{\alpha \eta} \bar{a}_{\gamma\xi} \right),
\end{equation*}
and using the similar arguments as in the proof of Lemma \ref{lemma:var:cuibic}, we have
\begin{eqnarray}
|\mathrm{Cov}( \bar{d}_i^{\,2}\bar{d}_j, \bar{d}_{\alpha}^{\,2} \bar{d}_\gamma)|
\begin{cases}
 = \E \bar{a}_{ij}^6 - (\E \bar{a}_{ij}^3)^2, &  \mbox{Case 1}, \\
= 0, & \mbox{Case 2}, \\
= \sum_k \sum_{\zeta} \sum_{\eta} \E \bar{a}_{ik}^2 \E \bar{a}_{j\zeta}^2 \E \bar{a}_{\alpha \eta}^2,  & \mbox{Case 3}, \\
= \sum_k \sum_{\zeta} \E \bar{a}_{ik}^2 \E \bar{a}_{\alpha\zeta}^2 \E \bar{a}_{j\beta}\bar{a}_{\beta j}, & \mbox{Case 4}, \\
= \sum_k \sum_{t} \E \bar{a}_{ik}^3 \E \bar{a}_{jt}^3, & \mbox{Case 5}, \\
= \sum_k \sum_{t} \E \bar{a}_{ik}^2 \E\bar{a}_{\alpha t}^2 \E \bar{a}_{ji}\bar{a}_{ij} + \sum_t \E \bar{a}_{it}^3 \E\bar{a}_{j\alpha}\bar{a}_{\alpha j}, & \mbox{Case 6}, \\
= \sum_{t} \E \bar{a}_{i\gamma}^2\bar{a}_{\gamma i} \E \bar{a}_{jt}^3,  & \mbox{Case 7}, \\
= \left(\sum_{k\neq i, j} \E \bar{a}_{ik}^{\,2} \right) \left(\sum_{\zeta \neq \alpha, \gamma} \E \bar{a}_{\alpha \zeta}^{\,2} \right)
\E \bar{a}_{j\gamma} \bar{a}_{\gamma j} & \mbox{Case 8}.
\end{cases}
\end{eqnarray}
\iffalse
Thus, we have
\[
\mathrm{Var}( \sum_{i\neq j} f_{ij} \bar{d}_i^{\,2}\bar{d}_j ) \lesssim \frac{n^6 }{c_n}.
\]
\fi
Let
\[
p_{ij} = \frac{e^{\beta_i+\beta_j}}{ ( 1 + e^{\beta_i+\beta_j}) },~~q_{ij}=1-p_{ij}
\]
In view of that
\begin{eqnarray*}
\E \bar{a}_{ij}^{\,6} - (\E \bar{a}_{ij}^{\,3})^2 & = &  p_{ij}q_{ji}( p_{ij}^5+q_{ji}^5 - p_{ij}q_{ji}(p_{ij}^2-q_{ji}^2)^2), \\
\E \bar{a}_{ij}^{\,3} & = & q_{ij}^3p_{ij} - p_{ij}^3 q_{ij}, \\
\max_{i,j} p_{ij}q_{ij} & = & \max_{i,j} \frac{e^{\beta_i+\beta_j}}{ ( 1 + e^{\beta_i+\beta_j})^2 }  \le \frac{1}{c_n},
\end{eqnarray*}
by combining the above cases, it completes the proof.
\end{proof}

\section{Proofs of supported lemmas in the proof of Theorem 1 (a)} %% and equations under the $\beta$-model
\label{section:beta-th1a}

We first reproduce some basic results here in the main text.
Recall that an $n\times n$ matrix $V=(v_{ij})$ belongs to the matrix class
$\mathcal{L}_n(m, M)$ if
\[
\begin{array}{cl}
v_{ii}=\sum_{j\neq i} v_{ij}, & i=1, \ldots, n \\
m \le v_{ij} \le M, & i,j=1, \ldots, n; i\neq j.
\end{array}
\]
%\cite{Yan:Xu:2013} proposed to use a simple matrix $\mathrm{diag}(1/v_{11}, \ldots, 1/v_{nn})+1/v_{\cdot\cdot}$ to approximate $V^{-1}$, where
%$v_{\cdot\cdot}=\sum_i v_{ii}$. Since $1/v_{\cdot\cdot} =  O(b_{n1}/n^2)$,
We use the diagonal matrix
\begin{equation}\label{definition-S}
S=\mathrm{diag}(1/v_{11}, \ldots, 1/v_{nn}),
\end{equation}
to approximate $V^{-1}$. For $V\in \mathcal{L}_n(1/b_n, 1/c_n)$, \cite{Yan:Xu:2013} proved
\begin{equation}\label{ineq-V-S-appro-upper-b}
\|W:=V^{-1} - S \|_{\max} \lesssim \frac{b_n^3}{n^2c_n}.
\end{equation}
Further, for its bottom right $(n-r)\times (n-r)$ block $V_{22}$ of $V\in \mathcal{L}_n(1/b_n, 1/c_n)$ and $n\ge 3$, we have
\begin{equation}\label{ineq-V22-S22-app}
\| \widetilde{W}_{22}:= V_{22}^{-1} - S_{22} \|_{\max} \le \frac{ b_n^2 }{ (n-1)^2 c_n } \left( 1 + \frac{ nb_n }{ (n-2)c_n } \right)\lesssim \frac{ b_n^3}{ n^2 c_n^2 }, ~~r=0, \ldots, n-1,
\end{equation}
where
\begin{equation}\label{definition-S22}
S_{22}=\mathrm{diag}(1/v_{r+1, r+1}, \ldots, 1/v_{nn}).
\end{equation}
%The proof of \eqref{ineq-V22-S22-app} is presented in Section ?.
From \eqref{ineq-V22-S22-app}, we can see that the error bound by using $S_{22}$ to approximate $V_{22}^{-1}$ is independent of $r$ and depends only on $b_n$, $c_n$ and n.
Moreover, by Theorem 6.1 of \cite{hillar2012inverses},
we have that for $V\in \mathcal{L}_n(1/b_n, 1/c_n)$ and its bottom right $(n-r)\times (n-r)$ block $V_{22}$,
\begin{equation}\label{ineq-tight-V}
\frac{c_n}{2(n-1)} \le \|V^{-1}\|_\infty \le \frac{3b_n}{2n-1}, \quad \|V_{22}^{-1} \|_\infty \le \frac{ 3b_n}{ 2n-1 }.
\end{equation}

Recall that $V = -\partial^2 \ell( \bs{\beta} )/\partial \bs{\beta} \partial \bs{\beta}^\top$, where row $i$ column $j$ element $v_{ij}$ of $V$ is
\begin{equation*}
v_{ii} = \sum\nolimits_{j\neq i} v_{ij}, ~~ v_{ij} = \frac{e^{\beta_i + \beta_j}}{(1 + e^{\beta_i + \beta_j})^2}=\mu^\prime(\pi_{ij}), ~~i \neq j; i,j=1,\ldots, n,
\end{equation*}
which is also the covariance matrix of $\mathbf{d}$.

This section is organized as follows.
Sections \ref{subsection-proof-lemma2} and \ref{subsection-proof-lemma3} presents
the proofs of Lemmas 3 and 4, respectively.
Section \ref{subsection-L2norm} contains an additional result about an $L_2$-norm error bound for $\bs{\widehat{\beta}}$. %, stated as Lemma \ref{lemma-L2-con}.
Section \ref{subsec-asy-widehatbeta} presents an asymptotically explicit expression for $\bs{\widehat{\beta}}$ that is used in the proof of Lemma 4. %, stated as Lemma \ref{lemma-hatbeta-exp}.
Section \ref{subsec-prooflemma4} presents the proof of Lemma 5.

\subsection{Proof of Lemma 3}
\label{subsection-proof-lemma2}

\begin{proof}[Proof of Lemma 3]
Recall that  $\bs{d}_2=(d_{r+1},\ldots, d_n)^\top$, $\bs{\bar{d}}_2=( \bar{d}_{r+1}, \ldots, \bar{d}_n)^\top$, and
$\widetilde{W}_{22}=V_{22}^{-1} - S_{22}$. Note that when $r=0$, $\bs{\bar{d}}_2=\bs{\bar{d}}$ and $\widetilde{W}_{22}=W$.
The aim is to prove
\begin{equation}\label{Wd-op1}
 \bs{\bar{d}}_2^\top \widetilde{W}_{22} \bs{\bar{d}}_2 = O_p\left( \frac{b_n^3( 1- r/n)^{3/2} }{ c_n^3 } \right).
\end{equation}
We first have
\begin{equation}\label{wd-expectation}
\E[  \bs{\bar{d}}_2^\top \widetilde{W}_{22}  \bs{\bar{d}}_2 ] =0,
\end{equation}
which is due to that
\[
\E[  \mathbf{\bar{d}_2}^\top \widetilde{W}_{22}  \mathbf{\bar{d}_2} ]
= \mathrm{tr} (\E[  \mathbf{\bar{d}_2}^\top \mathbf{\bar{d}_2}  ] \widetilde{W}_{22} )
= \mathrm{tr} (V_{22}\widetilde{W}_{22}) = \mathrm{tr} ( I_{n-r} - V_{22}S_{22} ) = 0.
\]

Let $\widetilde{W}_{22}=(\tilde{w}_{ij})_{(n-r)\times (n-r) }$.
Next, we bound the variance of $\sum_{i,j=r+1}^n \bar{d}_i \tilde{w}_{(i-r)(j-r)} \bar{d}_j$.
Recall that $v_{ij}=\mathrm{Var}(\bar{a}_{ij})=\mu^\prime(\beta_i + \beta_j)$.
There are four cases for calculating the covariance
\[
g_{ij\zeta\eta}=\mathrm{Cov}\big( \bar{d}_i  \tilde{w}_{(i-r)(j-r)} \bar{d}_j, \bar{d}_\zeta w_{(\zeta-r)(\eta-r)} \bar{d}_\eta ).
\]
Case 1: $i=j=\zeta=\eta$. In view of \eqref{eq:variance:2}, we have
\begin{equation}\label{equ:varii}
\mathrm{var}( \bar{d}_i^{\,2} ) = \sum_{j}\mathrm{Var}( \bar{a}_{ij}^4 ) + 2 v_{ii}^2.
\end{equation}
Let $p_{ij}=\mu(\beta_i+\beta_j)$ and $q_{ij}=1-p_{ij}$. By \eqref{ineq-mu-deriv-bound},
\[
\max_{i,j} p_{ij}(1-p_{ij}) \le \frac{1}{c_n}.
\]
Note that the function $x^4 + (1-x)^4$ with $x\in[0,1]$ attains its maximum value at points $0$ or $1$.
It follows that
\[
\mathrm{Var}( \bar{a}_{ij}^{\,4} ) = p_{ij}q_{ij}\{ p_{ij}^7 + q_{ij}^7 - p_{ij}q_{ij}(p_{ij}^3+q_{ij}^3)^2 \}
= p_{ij}q_{ij}( p_{ij}^8 + q_{ij}^8 - 2p_{ij}^4q_{ij}^4 )
\le \frac{1}{c_n}.
\]
Thus, we have
\begin{equation*}
|g_{iiii}|\le \tilde{w}_{(i-r)(i-r)}^2\cdot \left( \frac{2(n-1)^2}{c_n^2} + \frac{n-1}{c_n} \right).
\end{equation*}
%Similarly, we have the upper bounds in other cases.\\
Case 2: Three indices among the four indices are the same. Without loss of generality, we assume that
$j=\zeta=\eta$ and $i\neq j$.
Observe that
\[
\mathrm{Cov}( \bar{d}_i \bar{d}_j, \bar{d}_j^{\,2}) = \sum_{k,h,\alpha,\gamma} ( \E \bar{a}_{ik}\bar{a}_{jh} \bar{a}_{j\alpha}
\bar{a}_{j\gamma} - \E \bar{a}_{ik}\bar{a}_{jh} \E \bar{a}_{j\alpha} \bar{a}_{j\gamma} )
\]
and, for distinct $k,h,\alpha,\gamma$,
\begin{eqnarray*}
\E \bar{a}_{ik}\bar{a}_{jh} \bar{a}_{j\alpha}
\bar{a}_{j\gamma} & = & 0, \\
\E \bar{a}_{ij}\bar{a}_{ji} \bar{a}_{j\alpha}
\bar{a}_{j\gamma} & = & 0, \\
\E \bar{a}_{ih}\bar{a}_{jh} \bar{a}_{j\alpha}
\bar{a}_{j\gamma} & = & 0.
\end{eqnarray*}
It follows that
\[
\mathrm{Cov}( \bar{d}_i \bar{d}_j, \bar{d}_j^{\,2}) = \E \bar{a}_{ij}\bar{a}_{ji}^3 - \E \bar{a}_{ij}\bar{a}_{ji}\E \bar{a}_{ji}^2
+ 2\sum_{h\neq i,j} \E \bar{a}_{ij}\bar{a}_{ji} \E \bar{a}_{jh}^{\,2}.
\]
Therefore,  by \eqref{ineq-mu-deriv-bound},
\begin{eqnarray*}
|g_{ijjj}|&\le & |\tilde{w}_{(i-r)(j-r)}\tilde{w}_{(j-r)(j-r)}|\cdot \frac{n}{c_n^2}.
\end{eqnarray*}
Similarly, we have the upper bounds in other cases.\\
Case 3. Two indices among the four are the same (e.g. $i=j$ or $j=\zeta$):
\begin{eqnarray*}
|g_{ii\eta\zeta}|&=&|\tilde{w}_{(i-r)(i-r)}\tilde{w}_{(\zeta-r)(\eta-r)}(2v_{i\zeta}v_{i\eta}+v_{ii}v_{\zeta\eta})|
\le |\tilde{w}_{(i-r)(i-r)}\tilde{w}_{(\zeta-r)(\eta-r)}|\cdot \frac{n}{c_n^2};\\
|g_{ijj\eta}|&=&|\tilde{w}_{(i-r)(i-r)}\tilde{w}_{(j-r)(\eta-r)}(2v_{ji}v_{j\eta}+v_{ij}v_{j\eta})|
\le 3|\tilde{w}_{(i-r)(i-r)}\tilde{w}_{(j-r)(\eta-r)}|\cdot \frac{1}{c_n^2}.
\end{eqnarray*}
Case 4: All four indices are different
\begin{eqnarray*}
|g_{ij\zeta\eta}|&=& |\tilde{w}_{(i-r)(j-r)}\tilde{w}_{(\zeta-r)(\eta-r)}(v_{i\zeta}v_{j\eta}+v_{i\eta}v_{j\zeta})|\le 2|\tilde{w}_{(i-r)(j-r)}\tilde{w}_{(\zeta-r)(\eta-r)}|\frac{1}{c_n^2}.
\end{eqnarray*}
Consequently, by \eqref{ineq-V-S-appro-upper-b}, we have
\begin{eqnarray*}
&&\mathrm{Var} ( \bs{\bar{d}}_2^{\top} \widetilde{W}_{22} \bs{\bar{d}}_2 )  \\
& = & \sum_{i,j,\zeta,\eta={r+1}}^n \mathrm{Cov}\big( \bar{d}_i  \tilde{w}_{(i-r)(j-r)} \bar{d}_j,
\bar{d}_\zeta \tilde{w}_{(\zeta-r)(\eta-r)} \bar{d}_\eta ) \\
& \lesssim &  \left( \frac{b_n^3}{n^2 c_n^2 } \right)^2 \times \left( (n-r)\cdot \frac{ n^2 }{c_n^2 }
+ (n-r)^2 \cdot \frac{n}{c_n^2} + (n-r)^3 \cdot \frac{n}{c_n^2} + (n-r)^4 \cdot \frac{1}{c_n^2 } \right) \\
& \lesssim & \frac{b_n^6(1-r/n)^3}{c_n^6}.
\end{eqnarray*}
It follows that from Chebyshev's inequality and \eqref{wd-expectation}, we have
\begin{align*}
 & \P( \left( |\bs{\bar{d}}_2^\top \widetilde{W}_{22}  \bs{\bar{d}}_2  |
\ge  \rho_n\frac{ b_n^3(1-r/n)^{3/2} }{c_n^6} \right) \\
\le & \frac{ c_n^6 }{ b_n^3(1-r/n)^3\rho_n^2 }\times \mathrm{Var} ( \mathbf{\bar{d}_2}^{\top} W_{22} \mathbf{\bar{d}_2} ) \\
\lesssim & \frac{1}{\rho_n^2} \to 0,
\end{align*}
where $\{\rho_n\}_{n=1}^\infty$ is any positive sequence tending to infinity.
This completes the proof.
\end{proof}

\subsection{Proof of Lemma 4}
\label{subsection-proof-lemma3}

Before beginning the proof of Lemma 4, we introduce one useful lemma.
Let $F(\mathbf{x}): \R^n \to \R^n$ be a function vector on $\mathbf{x}\in\R^n$.
We say that a Jacobian matrix $F^\prime(\mathbf{x})$ with $\mathbf{x}\in \R^n$ is Lipschitz continuous on a convex set $D\subset\R^n$ if
for any $\mathbf{x},\mathbf{y}\in D$, there exists a constant $\lambda>0$ such that
for any vector $\mathbf{v}\in \R^n$ the inequality
\begin{equation*}
\| [F^\prime (\mathbf{x})] \mathbf{v} - [F^\prime (\mathbf{y})] \mathbf{v} \|_\infty \le \lambda \| \mathbf{x} - \mathbf{y} \|_\infty \|\mathbf{v}\|_\infty
\end{equation*}
holds.
We will use the Newton iterative sequence to establish the existence and consistency of the MLE.
\cite{Gragg:Tapia:1974} gave the optimal error bound for the Newton method under the Kantovorich conditions
[\cite{Kantorovich1948Functional}]. We only show partial results here that are enough for our applications.

\begin{lemma}[\cite{Gragg:Tapia:1974}]\label{lemma:Newton:Kantovorich}
Let $D$ be an open convex set of $\R^n$ and $F:D \to \R^n$ be Fr\'{e}chet differentiable on $D$
with a Jacobian $F^\prime(\mathbf{x})$ that is Lipschitz continuous on $D$ with Lipschitz coefficient $\lambda$.
Assume that $\mathbf{x}_0 \in D$ is such that $[ F^\prime (\mathbf{x}_0) ]^{-1} $ exists,
\[
\| [ F^\prime (\mathbf{x}_0 ) ]^{-1} \|  \le \aleph,~~ \| [ F^\prime (\mathbf{x}_0) ]^{-1} F(\mathbf{x}_0) \| \le \delta, ~~ h= 2 \aleph \lambda \delta \le 1,
\]
and
\[
B(\mathbf{x}_0, t^*) \subset D, ~~ t^* = \frac{2}{h} ( 1 - \sqrt{1-h} ) \delta = \frac{ 2\delta }{ 1 + \sqrt{1-h} }.
\]
Then: (1) The Newton iterations $\mathbf{x}_{k+1} = \mathbf{x}_k - [ F^\prime (x_\mathbf{k}) ]^{-1} F(\mathbf{x}_k)$ exist and $\mathbf{x}_k \in B(\mathbf{x}_0, t^*) \subset D$ for $k \ge 0$. (2)
$\mathbf{x}^* = \lim \mathbf{x}_k$ exists, $\mathbf{x}^* \in \overline{ B(\mathbf{x}_0, t^*) } \subset D$ and $F(\mathbf{x}^*)=0$.
\end{lemma}

\begin{proof}[Proof of Lemma 4]
Under the null $H_0: (\beta_1, \ldots, \beta_r)^\top=(\beta_1^0, \ldots, \beta_r^0)^\top$, $\beta_1, \ldots, \beta_r$ are known and $\beta_{r+1}, \ldots, \beta_n$ are unknown.
Recall that $\bs{\widehat{\beta}}^0$ denotes the restricted MLE under the null space,
where $\widehat{\beta}_i = \beta_i^0$, $i=1,\ldots,r$.
For convenience, we will use $\bs{\beta}$ and $\bs{\widehat{\beta}}^0$ to denote
the vectors $(\beta_{r+1}, \ldots, \beta_n)^\top$ and $(\widehat{\beta}_{r+1}^0, \ldots, \widehat{\beta}_n^0)^\top$ in this proof, respectively.
Note that $\bs{\widehat{\beta}}^0=\bs{\widehat{\beta}}$ when $r=0$.

Define a system of score functions based on likelihood equations:
\begin{equation}\label{eqn:def:F}
 F_i(\bs{\beta})=  \sum\limits_{j=1, j\neq i}^n \mu(\beta_i + \beta_j) - d_i ,~~i=1, \ldots, n,
\end{equation}
and $F(\bs{\beta})=(F_{r+1}(\bs{\beta}), \ldots, F_n(\bs{\beta}))^\top$.

Let $B(\bs{\beta}, 1/(2b_n))=\{\bs{\gamma}=(\gamma_{r+1}, \ldots, \gamma_n)\in \R^{n-r}: \| \bs{\beta} - \bs{\gamma}\|_\infty \le 1/(2b_n) \}$ be a convex set containing $\bs{\beta}$.
We will derive the error bound between $\bs{\widehat{\beta}}^0$ and $\bs{\beta}$ through
obtaining the convergence rate of the Newton iterative sequence $\bs{\beta}^{(n+1)}= \bs{\beta}^{(n)} - [F^\prime (\bs{\beta}^{(n)})]^{-1}
F (\bs{\beta}^{(n)})$,
where we choose the true parameter $\bs{\beta}$ as the starting point $\bs{\beta}^{(0)}:=\bs{\beta}$.
To this end, it is sufficient to demonstrate the Kantovorich conditions in Lemma \ref{lemma:Newton:Kantovorich}, where we set $D=B(\bs{\beta}, 1/(2c_n))$.
The Kantororich conditions require the Lipschitz continuous of $F^\prime(\bs{\beta})$ and the upper bounds of
$F(\bs{\beta}^*)$.
%The fact that when $r=0$, $V=F'(\bs{\beta}) \in \mathcal{L}_n( b_{n}^{-1}, c_n^{-1})$ will be used in the proof repeatedly.
The proof proceeds three steps. Step 1 is about the Lipschitz continuous property of the Jacobian matrix $F'( \bs{\beta} )$.
Step 2 is about the tail probability of $F(\bs{\beta})$. Step 3 is a combining step.

Step 1. We claim that
the Jacobian matrix $F'( \bs{\gamma} )$ of $F(\bs{\gamma})$ on $\bs{\gamma}$ is Lipschitz continuous on $B(\bs{\beta}, 1/(2c_n))$ with the Lipschitz coefficient  $3(4n-4-3r)/2c_n$.
This is verified as follows.
Let $(\gamma_1, \ldots, \gamma_r)=(\beta_1^0, \ldots, \beta_r^0)$.
The Jacobian matrix $F^\prime(\bs{\gamma})$ of $F(\bs{\gamma})$ can be calculated as follows.
By finding the partial derivative of $F_i$ with respect to $\bs{\gamma}$, for $i\neq j\in \{r+1, \ldots, n\}$, we have
\[
\frac{\partial F_i(\bs{\gamma}) }{ \partial \gamma_j} =  \mu^\prime (\gamma_i+\gamma_j), ~~
\frac{ \partial F_i(\bs{\gamma})}{ \partial \gamma_i} =  \sum_{j\neq i} \mu^\prime (\gamma_i+\gamma_j),
\]
\[
\frac{\partial^2 F_i( \bs{\gamma}) }{ \partial \gamma_i \partial \gamma_j} =  \mu^{\prime\prime} (\gamma_i+\gamma_j),~~
\frac{ \partial^2 F_i(\bs{\gamma})}{\partial \gamma_i^2} =  \sum_{j\neq i} \mu^{\prime\prime} (\gamma_i+\gamma_j),
\]
\begin{equation}\label{eq-le-con-a}
\frac{\partial^2 F_i(\bs{\gamma}) }{ \partial \beta_k \partial \beta_j} = 0,~~k\in \{\ell: \ell \neq i,j; \ell =1, \ldots, n\}.
\end{equation}
By the mean value theorem and \eqref{eq-derivative-mu-various}, we have
\[
|\mu^{\prime\prime} (\beta_i+\beta_j) - \mu^{\prime\prime} (\gamma_i+\gamma_j)|\le \frac{1}{4} \| \bs{\gamma}-\bs{\beta}\|_\infty \le \frac{1}{2c_n}.
\]
For $\bs{\gamma}\in D$, this shows
\[
\max_{i,j} |\mu^{\prime\prime} (\gamma_i+\gamma_j)| \le \frac{3}{2c_n}.
\]
It follows that
\begin{equation}
\label{inequ:second:deri}
|\frac{\partial^2 F_i( \bs{\gamma} ) }{\partial \gamma_i^2} |\le \frac{ 3(n-1)}{2c_n},~~
|\frac{\partial^2 F_i( \bs{\gamma} ) }{\partial \gamma_j\partial \gamma_i}| \le \frac{3}{2c_n},~~i\neq j.
\end{equation}
Let
\[
\mathbf{g}_{ij}(\bs{\gamma})=(\frac{\partial^2 F_i(\bs{\gamma}) }{ \partial \gamma_{r+1} \partial \gamma_j}, \ldots,
\frac{\partial^2 F_i(\bs{\gamma}) }{ \partial \gamma_n \partial \gamma_j})^\top.
\]
It leads to that
$\|\mathbf{g}_{ii}(\bs{\gamma})\|_1 \le 3(2n-r-2)/(2c_n)$. %, where $\|\mathbf{x}\|_1=\sum_i |x_i|$ for a general vector $\mathbf{x}$.
Note that when $i\neq j$ and $k\neq i, j$,
\[
\frac{\partial^2 F_i(\bs{\gamma}) }{ \partial \gamma_k \partial \gamma_j} =0.
\]
Therefore, we have
$\|\mathbf{g}_{ij}(\bs{\gamma})\|_1 \le 3/(2c_n)$, for $j\neq i$. Consequently, for vectors $\mathbf{x}, \mathbf{y}, \mathbf{v}\subset D$, we have
\begin{eqnarray*}
& & \| [F^\prime(\mathbf{x})]\mathbf{v} - [F^\prime(\mathbf{y})] \mathbf{v} \|_\infty \\
& \le & \max_{i=r+1, \ldots, n} \{\sum_{j=r+1}^n [ \frac{\partial F_i}{\partial \beta_j }(\mathbf{x}) - \frac{\partial F_i}{\partial \beta_j }(\mathbf{y})] v_j \} \\
& \le & \|\mathbf{v}\|_\infty \max_{i=r+1, \ldots,n} \sum_{j=r+1}^n |\frac{\partial F_i}{\partial \beta_j }
(\mathbf{x}) - \frac{\partial F_i}{\partial \beta_j }(\mathbf{y}) |  \\
& = & \|\mathbf{v}\|_\infty \max_{i=r+1,\ldots,n} \sum_{j=r+1}^n |\int_0^1 [\mathbf{g}_{ij}(t\mathbf{x}+(1-t)\mathbf{y})]^\top (\mathbf{x}-\mathbf{y})dt | \\
%& \le & \|\mathbf{v}\|_\infty \max_{i=r+1,\ldots, n} \sum_{j=r+1}^n \int_0^1 \sum_{k=1}^n \left |  \frac{ \partial^2 F_i (t\mathbf{x} + (1-t)\mathbf{y}) }{ \partial \beta_k \partial \beta_j } \right | dt \cdot \| \mathbf{x} - \mathbf{y} \|_\infty \\
& \le & \frac{3(2n-2-r+2(n-r-1))}{2c_n} \|\mathbf{v}\|_\infty\|\mathbf{x}-\mathbf{y}\|_\infty \\
& = & \frac{3(4n-4-3r)}{2c_n} \|\mathbf{v}\|_\infty\|\mathbf{x}-\mathbf{y}\|_\infty,
\end{eqnarray*}
where $t\in(0,1)$ is some real number.

Step 2. We give the tail probability of $\|F(\bs{\beta}) \|_\infty$ satisfying
\begin{equation}\label{ineq-union-d}
\P\Bigg( \max_{i=r+1, \ldots, n} |F_i(\bs{\beta})| \le \sqrt{n\log n} \Bigg)  \ge  1 - \frac{ 2 }{ n }.
\end{equation}
This is verified as follows.
Recall that $a_{ij}$, $1\le i<j \le n$, are independent Bernoulli random variables
and $F_i(\bs{\beta}) = \sum_{j\neq i} (\E a_{ij} - a_{ij})$.
By \citeauthor{Hoeffding:1963}'s \citeyearpar{Hoeffding:1963} inequality, we have
\begin{equation*}
\P\left( |F_i(\bs{\beta}) | \ge \sqrt{n\log n}  \right) \le 2\exp (- 2\frac{n\log n}{n} ) \le  \frac{2}{n^2},~~i=1, \ldots, n.
\end{equation*}
By the union bound, we have
\begin{eqnarray}
\label{ineq-d-upper}
\P\Bigg( \max_{i=1, \ldots, n} |F_i(\bs{\beta})| \ge \sqrt{n\log n} \Bigg)
\le  \sum_{i=1}^n \P\left(|F_i(\bs{\beta})| \geq \sqrt{n\log n} \right)
\le  \frac{2n}{n^2 },
\end{eqnarray}
such that
\[
\P\Bigg( \max_{i=r+1, \ldots, n} |F_i(\bs{\beta})| \le \sqrt{n\log n} \Bigg) \ge \P\Bigg( \max_{i=1, \ldots, n} |F_i(\bs{\beta})| \le \sqrt{n\log n} \Bigg) \ge  1 - \frac{ 2 }{ n }.
\]

Step 3. This step is one combining step.
The following calculations are based on the event $E_n$:
\[
E_n = \{ \max_{i=1,\ldots, n} |F_i(\bs{\beta})| \le   (n\log n)^{1/2}  \}.
\]
Recall that $V_{22}=(v_{ij})=  F^\prime(\bs{\beta})$.
By \eqref{ineq-tight-V}, we have
$\aleph =\|V_{22}^{-1}\|_\infty \le  3b_n/(2n-1)$.
By the event $E_n$, we have
\begin{eqnarray*}
\| F(\bs{\beta}) \|_\infty  \le (n\log n)^{1/2}.
\end{eqnarray*}
Repeatedly utilizing  \eqref{ineq-tight-V}, we have
\begin{eqnarray*}
\delta=\| [F'(\bs{\beta})]^{-1}F(\bs{\beta}) \|_\infty \le \| [F'(\bs{\beta})]^{-1}\|_\infty \|F(\bs{\beta}) \|_\infty
\le   \frac{3nb_n}{2n-1} \sqrt{\frac{\log n}{n}}.
\end{eqnarray*}
In Step 1, we show that $F^\prime(\bs{\beta})$ is Lipschitz continuous with Lipschitz coefficient $\lambda=3(4n-4-3r)/2c_n$.
Note that for any $r\in [0,n-1]$, $4n-4-3r\le n-4$.
Therefore, if $ b_n^2/c_n=o( (n/\log n)^{1/2} )$, then
\begin{eqnarray*}
h =2\aleph \lambda \delta & \le & \frac{3b_n}{2n-1} \times  \frac{3(4n-4-3r)}{2c_n}
\times    \frac{3nb_n}{2n-1} \sqrt{\frac{\log n}{n}}  \\
& = &    \frac{27n(4n-4-3r)b_n^2}{2(2n-1)^2c_n} \sqrt{ \frac{\log n}{n} } =o(1).
\end{eqnarray*}
The above arguments verify the Kantovorich conditions.
By Lemma \ref{lemma:Newton:Kantovorich}, it yields that
\begin{equation}
\label{eq-hatbeta-upper}
\| \bs{\widehat{\beta}}^0 - \bs{\beta} \|_\infty \le   \frac{3nb_n}{2n-1} \sqrt{\frac{\log n}{n}}.
\end{equation}
Step 2 implies $\P (E_n^c) \le 1 - 2/n$. This completes the proof.
\end{proof}

\subsection{The upper bound of $\| {\widehat{\beta}} - {\beta} \|_2$}
\label{subsection-L2norm}

We derive the upper bound for $\bs{\widehat{\beta}}$ in terms of the $L_2$-norm.

\begin{lemma}\label{lemma-L2-con}
If $b_n^2/c_n = o( ( n/\log n)^{1/2} )$, with probability at least $1-2/n$, we have
\[
\| \bs{\widehat{\beta}} - \bs{\beta} \|_2  \lesssim  b_n (\log n)^{1/2}.
\]
\end{lemma}

\begin{proof}
By \eqref{eq-hatbeta-upper}, if $b_n^2/c_n = o( ( n/\log n)^{1/2} )$, with probability at least $1-2/n$, $\bs{\widehat{\beta}}$ exists.
Because $\bs{\widehat{\beta}}$ minimizes $-\ell(\bs{\beta})$, by the mean value theorem, we have
\[
-\ell( \bs{\beta} ) \ge -\ell( \bs{\widehat{\beta}} ) = -\ell( \bs{\beta} )
- \frac{ \partial \ell( \bs{\beta} ) }{ \partial \bs{\beta}^\top } ( \bs{\widehat{\beta}}-\bs{\beta} )
- \frac{1}{2} ( \bs{\widehat{\beta}} - \bs{\beta} )^\top \frac{ \partial \ell( \bs{\tilde{\beta}} ) }{ \partial \bs{\beta} \partial \bs{\beta}^\top }( \bs{\widehat{\beta}}-\bs{\beta} ),
\]
where
\[
\| \bs{\tilde{\beta}}-\bs{\beta}\|_\infty \le \| \bs{\widehat{\beta}} - \bs{\beta} \|_\infty = O\left(  b_n \sqrt{\frac{\log n}{n}} \right).
\]
It follows from the Cauchy-Schwarz inequality that
\[
\frac{1}{2} ( \bs{\widehat{\beta}} - \bs{\beta} )^\top V( \bs{\tilde{\beta}} )( \bs{\widehat{\beta}} - \bs{\beta} ) \le -( \bs{d} - \E \bs{d} )^\top ( \bs{\widehat{\beta}} - \bs{\beta} ) \le \| \bs{d} - \E \bs{d}\|_2 \cdot \| \bs{\widehat{\beta}}- \bs{\beta} \|_2.
\]
This shows that
\[
\| \bs{\widehat{\beta}} - \bs{\beta} \|_2 \le \frac{ 2\| \bs{d} - \E \bs{d}\|_2 }{ \lambda_{\min}(V(\bs{\tilde{\beta}})) },
\]
where $\lambda_{\min}(V(\bs{\tilde{\beta}}))$ denotes the smallest eigenvalue of $V(\bs{\tilde{\beta}})$.
Because for any vector $\bs{x}=(x_1, \ldots, x_n)^\top \in \R^n$,
\[
\bs{x}^\top V \bs{x} = \sum_i v_{ii} x_i^2 + 2\sum_{i<j} x_i v_{ij} x_j = \sum_{i<j} v_{ij} (x_i + x_j)^2,
\]
we have
\[
\lambda_{\min}(V(\bs{\tilde{\beta}})) \ge (n-1) \min_{i,j} \frac{ e^{\tilde{\beta}_i + \tilde{\beta}_j} }{ (1 + e^{\tilde{\beta}_i + \tilde{\beta}_j}) }  \gtrsim  \frac{n}{b_n}.
\]
By \eqref{ineq-union-d}, with probability at least $1-2/n$, we have
\[
\sum_i ( d_i - \E d_i )^2 \le n^2 \log n.
\]
Consequently,
\[
\| \bs{\widehat{\beta}}- \bs{\beta} \|_2  \lesssim  b_n (\log n)^{1/2}.
\]
\end{proof}

\subsection{Asymptotic expression for ${\widehat{\beta}}$}
\label{subsec-asy-widehatbeta}

The following lemma gives an asymptotically explicit expression for $\bs{\widehat{\beta}}^0$ that will be repeatedly used in the proof.

\begin{lemma}\label{lemma-hatbeta-exp}
Suppose that $\beta_1, \ldots, \beta_r$ with $r\in\{0,\ldots,n-1\}$ are known.
If $b_n^2/c_n=o( n/(n-r)\cdot (n/\log n)^{1/2} )$, then with probability at least $1 - 6/n$, the following holds uniformly:
\[
\widehat{\beta}_i^0 - \beta_i = \frac{ \bar{d}_i }{ v_{ii} } + g_i,~~ i=r+1, \ldots, n,
\]
where
\begin{equation}\label{ineq-g}
g_i = O\left( \frac{ b_n^3\log n }{ nc_n } \right).
\end{equation}
\end{lemma}

\begin{proof}[Proof of Lemma \ref{lemma-hatbeta-exp}]
\iffalse
In the proof of Theorem 1, we show
\[
\bs{\widehat{\beta}} - \bs{\beta} = V^{-1}\bs{\bar{d}} + O_p\left( \frac{ b_n^3\log n }{ nc_n } \right).
\]
For integrality, we  give its proof once more here.
\fi
Since $\beta_1, \ldots, \beta_r$ with $r\in\{0,\ldots,n-1\}$ are known, with some ambiguity of notations, here we use $\bs{\widehat{\beta}}^0$ and $\bs{\beta}$
to denote vectors $(\widehat{\beta}_{r+1}^0, \ldots, \widehat{\beta}_n^0)^\top$ and $(\beta_{r+1}, \ldots, \beta_n)^\top$, respectively.
By \eqref{eq-hatbeta-upper}, if $b_n^2/c_n=o( (n/\log n)^{1/2} )$, then $\P(E_n) \ge 1 -2/n$, where
\[
E_n : = \left\{ \| \bs{\widehat{\beta}}^0 - \bs{\beta}  \|_\infty \lesssim  b_n \sqrt{ \frac{\log n}{n} } \right\}.
\]
The following calculations are based on the event $E_n$.

Write $\mu_{ij}( \bs{\beta} )=\mu( \beta_i + \beta_j )$.
Let $F_i( \bs{\beta} ) = \sum_{j\neq i} \mu_{ij}( \bs{\beta} ) - d_i$, $i=1, \ldots, n$ and $F( \bs{\beta} )=(F_{r+1}( \bs{\beta} ), \ldots, F_n( \bs{\beta}))^\top$.
By applying a second order Taylor expansion to $F( \bs{\widehat{\beta}} )$, we have
\begin{equation}
\label{equ-lemma-gamma-b}
F( \bs{\widehat{\beta}}^0 )  = F( \bs{\beta}^0 ) +  \frac{\partial F(\bs{\beta}^0)}{\partial \bs{\beta}^\top } ( \bs{\widehat{\beta}}^0 - \bs{\beta})
+ \frac{1}{2} \left[\sum_{k=r+1}^{n} (\widehat{\beta}_k^0 - \beta_k) \frac{\partial^2F(\bs{\tilde{\beta}})}{\partial \beta_k \partial \bs{\beta}^\top} \right]\times ( \bs{\widehat{\beta}}^0 - \bs{\beta} ),
\end{equation}
where $\bs{\tilde{\beta}}$ lies between $\bs{\widehat{\beta}}$ and $\bs{\beta}$.
We evaluate the last term in the above equation row by row.
Its $k$th row is
\begin{equation}\label{definition-R}
R_k := \frac{1}{2} ( \bs{\widehat{\beta}}^0 - \bs{\beta} )^\top  \frac{\partial^2 F_k( \bs{\tilde{\beta}})}{\partial \bs{\beta} \partial \bs{\beta}^\top} ( \bs{\widehat{\beta}}^0 - \bs{\beta}),~~k=1, \ldots, n.
\end{equation}
A directed calculation gives that
\[
\frac{\partial^2 F_k( \bs{\tilde{\beta}} )}{\partial \beta_i \partial \beta_j} =
\begin{cases}
 \sum_{t\neq k} \mu^{\prime\prime}(\tilde{\beta}_k + \tilde{\beta}_t ),  &   i=j=k  \\
\mu^{\prime\prime}(\tilde{\beta}_k + \tilde{\beta}_j ), & i=k, i\neq j; j=k, i\neq j \\
0, &  i \neq j \neq k.
\end{cases}
\]
It follows that
\[
R_k = \frac{1}{2} \sum_{j=r+1,j\neq k}^n \mu^{\prime\prime}(\tilde{\beta}_k + \tilde{\beta}_j) |( \widehat{\beta}_k - \beta_k^* )^2
+ \sum_{j,k=r+1;j\neq k}^n | \mu^{\prime\prime}(\tilde{\beta}_k + \tilde{\beta}_j)| |( \widehat{\beta}_k - \beta_k^* )( \widehat{\beta}_j - \beta_j^*).
\]
By \eqref{ineq-mu-tilde} and event $E_n$, we have
\begin{equation}
\label{eq:rk}
\max_{k=1, \ldots, n} |R_k| \lesssim n \cdot \frac{1}{c_n} \cdot \| \bs{\widehat{\beta}} - \bs{\beta}\|_\infty \lesssim \frac{b_n^2\log n}{c_n}.
\end{equation}
Let $R=(R_{r+1}, \ldots, R_{n})^\top$ and $V=(v_{ij})=\partial F( \bs{\beta})/\partial \bs{\beta}^\top$.
Because $F(\widehat{\beta})=0$, by \eqref{equ-lemma-gamma-b}, we have
\begin{equation}\label{eq-expression-beta-star}
\bs{\widehat{\beta}} - \bs{\beta} =  V_{22}^{-1} \bs{\bar{d}}_2 + V_{22}^{-1} R.
\end{equation}
Note that $V\in \mathcal{L}_n(1/b_n, 1/c_n)$.
By \eqref{eq:rk} and \eqref{ineq-tight-V}, we have
\begin{eqnarray}\label{eq-hatbeta-cc}
\|V_{22}^{-1}R \|_\infty & \le & \|V_{22}^{-1} \|_\infty\|R \|_\infty \lesssim \frac{ b_n^3\log n }{ nc_n }.
\end{eqnarray}

Now, we bound the error term $\| (V_{22}^{-1} - S_{22}) \bs{\bar{d}}_2 \|_\infty$, where $S_{22}=\mathrm{diag}(1/v_{r+1,r+1}, \ldots, 1/v_{nn})$.
Note that
\begin{equation*}
\| (V^{-1} - S) \bs{\bar{d}} \|_\infty = \| V^{-1}( V - S^{-1} ) S \bs{\bar{d}} \|_\infty \le \| V^{-1} \|_\infty \| (V-S^{-1})S \bs{\bar{d}} \|_\infty,
\end{equation*}
and
\[
[(V-S^{-1})S\bs{\bar{d}}]_i = \sum_{j=r+1,j\neq i}^n \frac{v_{ij} }{ v_{jj} } \bar{d}_j = \sum_{j=r+1,j\neq i}^n \sum_{k=1, k\neq j}^n \frac{v_{ij}}{v_{jj}} \bar{a}_{jk}.
\]
The summation of the above right hand can be viewed as the sum of $r\times (n-r) + (n-r)(n-r-1)/2 $ independent random variables by noting that it is equal to
\[
\sum_{j=r+1,j\neq i}^n \sum_{k=r+1,k\neq j}^n ( \frac{v_{ij}}{v_{jj}} \bar{a}_{jk} + \frac{v_{ik}}{v_{kk}} \bar{a}_{kj} )
+ \sum_{j=r+1,j\neq i}^n \sum_{k=1}^r\frac{v_{ij}}{v_{jj}} \bar{a}_{jk}.
\]
For any $i\neq j$ and for any $k$, we have
\[
\frac{v_{ij}}{v_{kk} } \le \frac{ b_n }{ (n-1)c_n }.
\]
It follows that
\begin{equation}\label{eq-vv-a}
\sum_{j=r+1,j\neq i}^n \sum_{k=1, k\neq j}^n (\frac{v_{ij}}{v_{jj}})^2 \E a_{jk}^2 \le (n-1)(n-r-1) \cdot \left( \frac{ b_n }{ (n-1)c_n } \right)^2 \cdot \frac{1}{c_n}\lesssim \frac{b_n^2 }{c_n^3}.
\end{equation}
By Bernstein's inequality in Lemma \ref{lemma:bernstein} and \eqref{eq-vv-a}, with probability $1- 4/[n(n-1)]$, for large $n$, we have that
\[
| [(V-S^{-1})S\bar{d}]_i | \le \sqrt{ 2\log n \left( \sum_{j\neq i, k\neq j} \frac{v_{ij}}{v_{jj}} \E a_{jk}^2  \right)} + \frac{2}{3} \cdot  \frac{ b_n }{ (n-1)c_n } \cdot \log n
< 2.1 \frac{b_n}{c_n^{3/2}} \sqrt{\log n}.
\]
Therefore, with probability at least $1-4/n$, the following holds:
\[
\| (V_{22}-S_{22}^{-1})S_{22} \bs{\bar{d}}_2 \|_\infty < \frac{ 2.1b_n }{ c_n^{3/2} } \sqrt{\log n}.
\]
It yields that
\begin{equation}\label{eq-hatbeta-bb}
\| (V^{-1}_{22} - S_{22}) \bs{\bar{d}}_2 \|_\infty \le  \| V_{22}^{-1} \|_\infty \| (V_{22}-S_{22}^{-1})S \bs{\bar{d}}_2 \|_\infty \le \frac{ 2.1b_n^2  \sqrt{\log n} }{ nc_n^{3/2} }.
\end{equation}

Let
\[
g_i = \{(V_{22}^{-1} - S_{22}) \bs{\bar{d}}_2 \}_i + (V_{22}^{-1}R)_i, ~~i=r+1,\ldots, n.
\]
By combining \eqref{eq-expression-beta-star}, \eqref{eq-hatbeta-cc} and \eqref{eq-hatbeta-bb}, with probability at least $1-6 /n$, the following holds:
\[
\widehat{\beta}_i^0 - \beta_i = \frac{ \bar{d}_i }{ v_{ii} } + g_i, i=r+1, \ldots, n.
\]
where $g_i$ satisfies \eqref{ineq-g}.
\end{proof}

\subsection{Proof of Lemma 5}
\label{subsec-prooflemma4}
In this section, we give the proof of Lemma 5.

\begin{proof}[Proof of Lemma 5]
The bounds of $Q_{r1}$ and $Q_{r2}$ in Lemma 5 are reproduced here:
\begin{eqnarray}
\label{eq-S1-bound}
Q_{r1}:=\sum_{i=r+1}^n  (\widehat{\beta}_i-\beta_i)^3 \sum_{j=1,j\neq i}^n \mu^{\prime\prime}( \pi_{ij} ) & = &
O_p\left( \frac{ b_n^4 \log n ( 1- r/n)^{1/2} }{ c_n^2 } \right), \\
\label{eq-S2-bound}
Q_{r2}:=\sum_{i,j=r+1, j\neq i}^n  \widehat{\beta}_i-\beta_i)^2(\widehat{\beta}_j-\beta_j)\mu^{\prime\prime}( \pi_{ij} ) & = &
O_p\left( \frac{b_n^5(\log n)^2(n-r)}{ nc_n^2}  \right).
\end{eqnarray}

Note that $\beta_1, \ldots, \beta_r$ are known and $\beta_{r+1}, \ldots, \beta_n$ are unknown.
Let $E_{n1}$ be the event that
\begin{equation}\label{eq-a-zcc}
E_{n1} = \{ \widehat{\beta}_i - \beta_i = \frac{ \bar{d}_i }{ v_{ii} } + g_i, ~~i=r+1, \ldots, n, \},
\end{equation}
where
\begin{equation}\label{eq-a-zcg}
|g_i| \lesssim \frac{b_n^3 \log n }{ n c_n }.
\end{equation}
Let $E_{n2}$ be the event
\begin{equation}\label{eq-upp-bard}
E_{n2} =\{ \| \bs{\bar{d} } \|_\infty \le (n\log n)^{1/2}\}
\end{equation}
By Lemma \ref{lemma-hatbeta-exp} and \eqref{ineq-union-d},
$E_{n1}\bigcap E_{n2}$ holds with probability at least $1-8/n$.
The following calculations are based on $E_{n1}\bigcap E_{n2}$.

Let
\begin{equation}
f_i := \sum_{j=1,j\neq i}^n \mu^{\prime\prime}( \pi_{ij} ),\quad  f_{ij} :=  \mu^{\prime\prime}( \pi_{ij} )
\end{equation}
In view of \eqref{ineq-mu-deriv-bound}, we have
\begin{equation}\label{ineq-fi-fij}
\max_{i=r+1,\ldots,n} |f_i| \le \frac{ n-r-1}{ c_n }, ~~ \max_{i,j} |f_{ij} | \le \frac{ 1}{c_n}.
\end{equation}
By substituting the expression of $\widehat{\beta}_i - \beta_i$ in \eqref{eq-a-zcc} into  $f_i(\widehat{\beta}_i-\beta_i)^3$, we get
\begin{equation}\label{eq-a-za}
\sum_{i=r+1}^n  f_i(\widehat{\beta}_i-\beta_i)^3   =   \sum_{i=r+1}^n f_i \cdot \frac{ \bar{d}_i^3 }{ v_{ii}^3 }
+ 3 \sum_{i=r+1}^n f_i \cdot \frac{ \bar{d}_i^2 g_i }{ v_{ii}^2 } + 3 \sum_{i=r+1}^n f_i \cdot \frac{ \bar{d}_i g_i^2 }{v_{ii}}
+ \sum_{i=r+1}^n f_i g_i^3.
\end{equation}
We bound the four terms in the above right hand in an inverse order.
The fourth term can be bounded as follows:
\begin{eqnarray}
\nonumber
|\sum_{i=r+1}^n  f_i g_i^3 | & \le & (n-r) \max_{i=r+1,\ldots,n} |f_i| \max_i |g_i^3|, \\
\nonumber
& \lesssim & (n-r) \cdot \frac{ n }{c_n } \cdot (\frac{ b_n^3 \log n }{ nc_n})^3, \\
\label{eq-a-four}
& \lesssim &  \frac{ (n-r)b_n^9 (\log n)^3 }{ n^2 c_n^4 } .
\end{eqnarray}
In view of \eqref{eq-a-zcg}, \eqref{eq-upp-bard} and \eqref{ineq-fi-fij},  the upper bound of the third term is
\begin{eqnarray}
\nonumber
|\sum_{i=r+1}^n \frac{\bar{d}_i}{v_{ii}} f_i g_i^2 | & \le & (n-r) \cdot \max_i \frac{1}{v_{ii}} \cdot \| \bs{\bar{d}} \|_\infty \cdot \max_{i=r+1,\ldots,n} |f_i| \cdot \max_i |g_i^2| \\
\nonumber
& \lesssim & (n-r)\cdot \frac{b_n}{n} \cdot (n\log n)^{1/2} \cdot \frac{ n }{c_n } \cdot (\frac{ b_n^3 \log n }{ nc_n})^2 \\
\label{eq-a-three}
& \lesssim &  \frac{ (n-r)^2 b_n^7 (\log n)^{5/2} }{ n^{5/2} c_n^3 } .
\end{eqnarray}
By Corollary 1 in the main text, we have that
\[
\sum_{i=r+1}^n \frac{ \bar{d}_i^2 }{v_{ii}} = O_p( n-r).
\]
In view of \eqref{eq-a-zcg} and \eqref{ineq-fi-fij}, the second term can be bounded as follows:
\begin{eqnarray}
\nonumber
|\sum_{i=r+1}^n \frac{\bar{d}_i^2}{v_{ii}^2} f_i g_i | & \le & \max_i \frac{ 1 }{v_{ii}} \cdot \max_{i=r+1,\ldots,n} |f_i| \cdot \max_i |g_i|  \cdot \sum_{i=r+1}^n \frac{ \bar{d}_i^2 }{v_{ii} } \\
\nonumber
& \lesssim & \frac{b_n}{n} \cdot \frac{ n }{c_n } \cdot \frac{ b_n^3 \log n }{ nc_n} \cdot O_p(n-r) \\
\label{eq-a-second}
& = & O_p( \frac{ (n-r) b_n^4\log n }{ n c_n^2 } ).
\end{eqnarray}
Now, we bound the first term.
By Lemma \ref{lemma:var:cuibic}, we have that
\begin{eqnarray*}
&&\mathrm{Var}( \sum_{i=r+1}^n \frac{f_i}{v_{ii}^3} \bar{d}_i^{\,3} ) \\
& < &
\max_i \frac{ |f_i|^2 }{v_{ii}^6} \left\{ n(n-r) \cdot \max_{i,j} \mathrm{Var}( \bar{a}_{ij}^3 ) + 3n^2(n-r) \cdot \max_{i,j} \E \bar{a}_{ij}^4 \cdot \max_{i,j} \E \bar{a}_{ij}^2 \right\} \\
&&+ 6n^3(n-r) (\max_{i,j} \E \bar{a}_{ij}^2 )^3 + 2n(n-r) \max_{i,j} \mathrm{Var}(\bar{a}_{ij}^3) + 8n^3(n-r) ( \max_{i,j} \E \bar{a}_{ij} )^3.
\end{eqnarray*}
Because
\begin{eqnarray*}
\mathrm{Var}(\bar{a}_{ij}^3 ) & = & p_{ij} q_{ij} ( p_{ij}^5q_{ij} + q_{ij}^5p_{ij} - 2p_{ij}^3q_{ij}^3) \le  p_{ij} q_{ij} \le \frac{1}{c_n},\\
\E \bar{a}_{ij}^2 & = & p_{ij}q_{ij} \le \frac{1}{c_n}, \\
\E \bar{a}_{ij}^4 & = & p_{ij}^4 q_{ij} + q_{ij}^4 p_{ij} \le p_{ij}q_{ij} \le \frac{1}{c_n},
\end{eqnarray*}
an upper bound of $\mathrm{Var}( \sum_{i=r+1}^n f_i\bar{d}_i^3/v_{ii}^3 )$ is
\begin{equation}\label{ineq-d-firsta}
\mathrm{Var}( \sum_{i=r+1}^n \frac{f_i\bar{d}_i^3}{v_{ii}^3} ) %\le \max_i |f_i|^2 \mathrm{Var}( \sum_{i=r+1}^n \frac{\bar{d}_i^3}{v_{ii}^3} )
\lesssim \frac{ n^2}{c_n^2 } \cdot \frac{ b_n^6 }{ n^6 } \cdot \frac{ n^3(n-r) }{ c_n^3} \lesssim \frac{ (n-r)b_n^6 }{ n c_n^5 }.
\end{equation}
\iffalse
By Chebyshev's inequality, for any $\rho_n\to\infty$, it follows that
\begin{eqnarray}
\nonumber
&&\P\left( \sum_{i=1}^n  \left (  \frac{ \bar{d}_i^3 - \E \bar{d}_i^3 }{ v_{ii}^3 } \right)f_i \ge \frac{ (n-r)^{1/2}b_n^3 \rho_n }{ n^{1/2}c_n^{5/2}} \right) \\
\nonumber
&\le& \frac{(n-r)c_n^5}{nb_n^6 \rho_n^2} \mathrm{Var}\left( \sum_{i=1}^n f_i \left (  \frac{ \bar{d}_i^3 - \E \bar{d}_i^3 }{ v_{ii}^3 } \right) \right)  \\
&\lesssim & \frac{1}{ \rho_n^2 } \to 0.
\end{eqnarray}
\fi
Because
\[
\E \bar{d}_{i}^{\,3} = \sum_{\alpha, \gamma, \zeta} \E \bar{a}_{i\alpha}\bar{a}_{i\gamma}\bar{a}_{i\zeta} = \sum_{\alpha\neq i} \E \bar{a}_{i\alpha}^{\,3},
\]
we have
\[
|\E \bar{a}_{ij}^{\,3}|  = | p_{ij}q_{ij}( p_{ij}^2 - q_{ij}^2)|\le\frac{1}{c_n},
\]
such that
\begin{equation}\label{eq-ed3}
| \sum_{i=r+1}^n  \left (  \frac{\E \bar{d}_i^3 }{ v_{ii}^3 } \right)f_i | \lesssim (n-r) \cdot \frac{n}{c_n} \cdot \frac{b_n^3}{n^3} \cdot \frac{n-r}{c_n} \lesssim \frac{(n-r)^2b_n^3}{n^2c_n^2}.
\end{equation}
In view of \eqref{ineq-d-firsta} and \eqref{eq-ed3}, we have
\begin{equation}\label{ineq-d-first}
|\sum_{i=r+1}^n  \left (  \frac{ \bar{d}_i^3 }{ v_{ii}^3 } \right)f_i | = O_p\left( \frac{ b_n^3 }{c_n^{5/2}} \left( \frac{ n-r }{ n} \right)^{1/2} \right).
\end{equation}
By combining the upper bounds of the above four terms in \eqref{eq-a-four}, \eqref{eq-a-three}, \eqref{eq-a-second} and \eqref{ineq-d-first}, it yields that
\[
Q_{r1} = O_p\left(  \frac{ (n-r)b_n^9 (\log n)^3 }{ n^2 c_n^4 }
+ \frac{ (n-r)b_n^7 (\log n)^{3/2} }{ n^{3/2} c_n^3 } +
\frac{ (n-r) b_n^4 \log n}{ n c_n^2 } + \frac{ b_n^3 }{c_n^{5/2} } \left( \frac{n-r}{n} \right)^{1/2} \right).
\]
This leads to \eqref{eq-S1-bound}.

Now we bound the following terms in \eqref{eq-S2-bound}:
\[
Q_{r2} = \sum_{i,j=r+1, j\neq i}^n  (\widehat{\beta}_i-\beta_i)^2(\widehat{\beta}_j-\beta_j)\mu^{\prime\prime}( \pi_{ij} ).
\]
By substituting \eqref{eq-a-zcc} into the above expression, we get
\begin{eqnarray*}\label{eq-s2-expan}
Q_{r2} & = &  \sum_{r+1 \le  i\neq j \le n} f_{ij} \cdot \frac{\bar{d}_i^2\bar{d}_j}{v_{ii}^2v_{jj}} +
\sum_{r+1 \le i\neq j \le n} f_{ij} \cdot \frac{ \bar{d}_i^2 }{ v_{ii}^2 } \cdot g_i
 + 2\sum_{r+1 \le i\neq j \le n } \frac{ \bar{d}_i \bar{d}_j}{v_{ii}v_{jj}} \cdot g_i \cdot f_{ij} \\
 & &
+ 2 \sum_{r+1 \le i\neq j \le n} \frac{ \bar{d}_i }{ v_{ii} } \cdot g_i g_j f_{ij}
+ \sum_{r+1 \le i\neq j \le n}  \frac{ \bar{d}_j }{ v_{jj} } \cdot g_i^2 f_{ij}
+ \sum_{r+1\le i\neq j \le n} g_i^2 g_j f_{ij}.
\end{eqnarray*}
In view of \eqref{ineq-d-upper}, \eqref{ineq-fi-fij} and \eqref{eq-a-zcg}, we have the following bounds
\begin{eqnarray}
\label{ineq-s1-a}
|\sum_{r+1\le i\neq j \le n} g_i^2 g_j f_{ij}| & \lesssim & \frac{ b_n^9 (\log n)^3 (n-r)^2 }{ n^3c_n^3 }, \\
\label{ineq-s1-b}
\max\{|\sum_{r+1\le i\neq j \le n} \frac{ \bar{d}_i }{ v_{ii} } \cdot g_i g_j f_{ij}|, |\sum_{r+1\le i\neq j\le n}  \frac{ \bar{d}_j }{ v_{jj} } \cdot g_i^2 f_{ij} |  & \lesssim & \frac{ b_n^7 (\log n)^{1/2} (n-r)^2 }{ c_n^2 n^{5/2} }, \\
\label{ineq-s1-c}
~~\max\{ | \sum_{r+1\le i\neq j\le n} \frac{ \bar{d}_i \bar{d}_j}{v_{ii}v_{jj}} \cdot g_i f_{ij}|,
|\sum_{r+1\le i\neq j\le n} \frac{ \bar{d}_i^2 }{ v_{ii}^2 } \cdot g_i  f_{ij} | \} & \lesssim &
\frac{ b_n^5 (\log n)^2(n-r)^2 }{ n^2c_n^2}.
\end{eqnarray}
The left argument is to bound the first term in \eqref{eq-s2-expan}. Because
\[
|\E \bar{d}_i^2 \bar{d}_j |= |\E \bar{a}_{ij}^2 \bar{a}_{ji} | \le p_{ij}q_{ij} \le \frac{1}{c_n},
\]
by \eqref{ineq-fi-fij}, we have
\[
|\sum_{r+1 \le i\neq j \le n} f_{ij} \cdot \frac{\E \bar{d}_i^2\bar{d}_j}{v_{ii}^2v_{jj}}| \le (n-r)^2 \frac{n}{c_n} \cdot \frac{ b_n^3}{n^3} \lesssim \frac{ b_n^3(n-r)^2 }{ n^2c_n^2}.
\]
By Lemma \ref{lemma:var:cuibic2}, we have
\[
\mathrm{Var}\left( \sum_{r+1\le i\neq j\le n} f_{ij} \cdot \frac{ \bar{d}_i^2\bar{d}_j}{v_{ii}^2v_{jj}} \right) \lesssim \frac{ (n-r)^2b_n^6 }{n^2c_n^2}.
\]
Similar to \eqref{ineq-d-firsta},  Chebyshev's inequality gives that
\begin{equation}\label{ineq-s1-d}
\sum_{i\neq j} f_{ij} \cdot \frac{\bar{d}_i^2\bar{d}_j}{v_{ii}^2v_{jj}} = O_p( \frac{ b_n^3 \log n}{ c_n}  ).
\end{equation}
By combining \eqref{eq-s2-expan} and \eqref{ineq-s1-a}--\eqref{ineq-s1-d},  we have
\[
Q_{r2}=O_p\left(
\frac{ b_n^9 (\log n)^3 (n-r)^2 }{n^3 c_n^3 } + \frac{b_n^7 (\log n)^{1/2} (n-r)^2 }{ n^{5/2} c_n }
+ \frac{ b_n^5 (\log n)^2 (n-r)^2 }{ n^2 c_n^2 } +
\frac{ (n-r)b_n^3 }{ nc_n }\right)
\]
If
\[
\frac{ b_n^5 (\log n)^2 (n-r) }{ nc_n^2} = o( r^{1/2} ),
\]
it yields \eqref{eq-S2-bound}.
\end{proof}

\section{Proofs of supported lemmas in the proof of Theorem 1 (b)}
\label{section-th1b}

This section is organized as follows.
Sections \ref{section-lemma5}, \ref{subsection-proof-lemma6} and \ref{section-lemma7}
present the proofs of Lemmas 7, 8 and 9, respectively.
Section \ref{section-proof-39-B20} presents the proof of (39) in the main text.

We reproduce some notations and some useful results in Section 6.2 here.
Recall  $\bs{\widetilde{d}}=(\sum_{i=1}^r d_i, d_{r+1},\ldots, d_n)$ and
 $\widetilde{V}$ denote the Fisher information matrix of $\widetilde{\bs{\beta}}=(\beta_1, \beta_{r+1}, \ldots, \beta_n)^\top$
under the null $H_0: \beta_1 = \cdots= \beta_r$, where
\begin{equation}\label{definition-tilde-V}
\widetilde{V}=\begin{pmatrix} \tilde{v}_{11} & \bs{\tilde{v}}_{12}^\top \\ \bs{\tilde{v}}_{12} & V_{22} \end{pmatrix},
\end{equation}
where $V_{22}$ is the lower right $(n-r)\times (n-r)$ block of $V$, $\bs{\tilde{v}}_{12} =
(\tilde{v}_{1,r+1}, \ldots, \bar{v}_{1, n})^\top$, and
\[
\tilde{v}_{11}= 2r(r-1)\cdot \frac{ e^{2\beta_1} }{ ( 1 + e^{2\beta_1})^2 } + r\sum_{j=r+1}^n \tilde{v}_{1j}, ~~
\tilde{v}_{1j} =  \frac{ r e^{\beta_1 + \beta_j } }{ ( 1 + e^{\beta_1 + \beta_j})^2 },~j=r+1, \ldots, n.
\]
We use $\widetilde{S}=\mathrm{diag}(1/\tilde{v}_{11}, 1/v_{r+1, r+1}, \ldots, 1/v_{nn})$ to approximate $\widetilde{V}^{-1}$ and have the following approximation error
\begin{equation}\label{approxi-inv2-beta-ho}
\|\widetilde{W}:= \widetilde{V}^{-1}-\widetilde{S} \|_{\max} \lesssim \frac{b_n^3}{n^2c_n^2}.
\end{equation}

\subsection{Proof of Lemma 7}
\label{section-lemma5}

In this section, we present the proof of Lemma 7.
We introduce
an error bound in the Newton method by \cite{Kantorovich-Akilov1964}
%%established by \cite{Yamamoto1986}
under the Kantorovich conditions [\cite{Kantorovich1948Functional}].

\begin{lemma}[Theorem 6 in \cite{Kantorovich-Akilov1964}]\label{lemma:Newton:Kantovorich-b}
Let $D$ be an open convex subset of $\R^n$ and
$F:D \to \R^n$ be Fr\'{e}chet differntiable.
Assume that, at some $\bs{x}_0 \in D$, $F^\prime(\bs{x}_0)$ is invertible and that
\begin{eqnarray}
\label{eq-kantororich-a}
\| F^\prime(\bs{x}_0)^{-1} ( F^\prime(\bs{x}) - F^\prime(\bs{y}))\| \le K\|\bs{x}-\bs{y}\|,~~ \bs{x}, \bs{y}\in D, \\
\label{eq-kantororich-b}
\| F^\prime(\bs{x}_0)^{-1} F(\bs{x}_0) \| \le \eta, h=K\eta \le 1/2, \\
\nonumber
\bar{S}(\bs{x}_0, t^*) \subseteq D, t^*=2\eta/( 1+ \sqrt{ 1-2h}).
\end{eqnarray}
Then:
(1) The Newton iterates $\bs{x}_{n+1} = \bs{x}_n - F^\prime (\bs{x}_n)^{-1} F(\bs{x}_n)$, $n\ge0$ are well-defined,
lie in $\bar{S}(\bs{x}_0, t^*)$ and converge to a solution $\bs{x}^*$ of $F(\bs{x})=0$. \\
(2) The solution $\bs{x}^*$ is unique in $S(\bs{x}_0, t^{**})\cap D$, $t^{**}=(1 + \sqrt{1-2h})/K$ if $2h<1$
and in $\bar{S}(\bs{x}_0, t^{**})$ if $2h=1$. \\
(3) $\| \bs{x}^* - \bs{x}_n \| \le t^*$ if $n=0$ and $\| \bs{x}^* - \bs{x}_n \| \le 2^{1-n} (2h)^{ 2^n -1 } \eta $ if $n\ge 1$.
\end{lemma}

Now, we are ready to prove Lemma 7.

\begin{proof}[Proof of Lemma 7]
Recall that $\bs{\widehat{\beta}}^0$ denotes the restricted MLE under the null $H_0: \beta_1=\cdots =\beta_r$.
In what follows, $\bs{\widehat{\beta}}^0$ and $\bs{\beta}$ denote
respective vectors $(\widehat{\beta}_1^0, \widehat{\beta}_{r+1}^0, \ldots,
\widehat{\beta}_n^0)^\top$ and $(\beta_1, \beta_{r+1}, \ldots, \beta_n)^\top$ with some ambiguity of notations.
Define a system of score functions based on likelihood equations:
\begin{equation}\label{eqn:def:F}
\begin{array}{rcl}
F_1(\bs{\beta}) & = &  \sum\limits_{i=1}^r \sum\limits_{j=1, j\neq i}^n \mu(\beta_i + \beta_j) - \sum\limits_{i=1}^r d_i, \\
 F_i(\bs{\beta}) & = &  \sum\limits_{j=1, j\neq i}^n \mu(\beta_i + \beta_j) - d_i ,~~i=r+1, \ldots, n,
\end{array}
\end{equation}
and $F(\bs{\beta})=(F_1(\bs{\beta}), F_{r+1}(\bs{\beta}), \ldots, F_n(\bs{\beta}))^\top$, where
$\beta_1=\ldots=\beta_r$.

Let $B(\bs{\beta}, 1/(2c_n))=\{\bs{\gamma}=(\gamma_1, \gamma_{r+1}, \ldots, \gamma_n)\in \R^{n-r+1}: \| \bs{\beta} - \bs{\gamma}\|_\infty \le 1/(2c_n) \}$ be a neighbouring set containing $\bs{\beta}$.
We will derive the error bound between $\bs{\widehat{\beta}}$ and $\bs{\beta}$ through
obtaining the convergence rate of the Newton iterative sequence $\bs{\beta}^{(n+1)}= \bs{\beta}^{(n)} - [F^\prime (\bs{\beta}^{(n)})]^{-1}
F (\bs{\beta}^{(n)})$,
where we choose the true parameter $\bs{\beta}$ as the starting point $\bs{\beta}^{(0)}:=\bs{\beta}$.
To this end, it is sufficient to demonstrate the Kantovorich conditions in Lemma \ref{lemma:Newton:Kantovorich}, where we set $D=B(\bs{\beta}, 1/(2c_n))$.
The proof proceeds three steps. Step 1 is about verifying condition \eqref{eq-kantororich-a}.
Step 2 is about verifying \eqref{eq-kantororich-b}. Step 3 is a combining step.

Step 1. We claim that for any $\bs{x}, \bs{y} \in B(\bs{\beta}, 1/(2c_n))$,
\begin{equation}\label{eq-con-ho-a}
\| [F^\prime(\bs{\beta})]^{-1} \{ F^\prime(\bs{x}) - F^\prime(\bs{y}) \}\| \lesssim \left( b_n + \frac{(n-r)b_n^3}{nc_n^3} \right) \|\bs{x} - \bs{y}\|.
\end{equation}
This is verified as follows.
Let $\pi_{ij}=\gamma_i + \gamma_j$ and $\mu(\pi_{ij}) = e^{\pi_{ij}}/( 1 + e^{\pi_{ij}})$.
The Jacobian matrix $F^\prime(\bs{\gamma})$ of $F(\bs{\gamma})$ can be calculated as follows.
By finding the partial derivative of $F_i$ with respect to $\gamma_j$, we have
\[
\frac{ \partial F_i (\bs{\gamma}) }{ \partial \gamma_j } =
\begin{cases}
r(r-1)\mu^\prime(\pi_{11}) + r \sum_{k=2}^{n-r+1} \mu^\prime(\pi_{1k}), & i=1,j=1, \\
r \mu^\prime (\pi_{1j}), & i=1, j=2,\ldots, n-r+1, \\
r\mu^\prime (\pi_{i1}), & i=2,\ldots, n-r+1, j=1, \\
r\mu^\prime (\pi_{i1}) + \sum_{k=2}^{n-r+1} \mu^\prime( \pi_{ik} ), &i=2,\ldots, n-r+1, j=i, \\
 \mu^\prime( \pi_{ij} ), &i,j=2,\ldots, n-r+1, j\neq i,
\end{cases}
\]
and
\[
\frac{ \partial^2 F_i (\bs{\gamma}) }{ \partial \gamma_j\partial \gamma_k } =
\begin{cases}
r(r-1)\mu^{\prime\prime}(\pi_{11}) + r \sum\limits_{t=2}^{n-r+1} \mu^{\prime\prime}(\pi_{1t}), & i=1,j=1,k=1, \\
r \mu^\prime (\pi_{1j}), & i=1, k=j=2,\ldots, n-r+1, \\
0, & i=1; k,j=2,\ldots, n-r+1; k\neq j, \\
r\mu^{\prime\prime} (\pi_{i1}), & i=2,\ldots, n-r+1, j=k=1, \\
r\mu^{\prime\prime} (\pi_{i1}) + \sum\limits_{t=2, t\neq i}^{n-r+1} \mu^{\prime\prime}( \pi_{it} ), &i=j=k=2,\ldots, n-r+1, \\
\mu^{\prime\prime} (\pi_{ij}), &i,j,k=2,\ldots, n-r+1, i\neq j, j=k, \\
0, &i, j,k=2,\ldots, n-r+1, k\neq j, j\neq i, i\neq k.
\end{cases}.
\]
By the mean value theorem and \eqref{eq-derivative-mu-various}, we have
\[
|\mu^{\prime\prime} (\beta_i+\beta_j) - \mu^{\prime\prime} (\gamma_i+\gamma_j)|\le \frac{1}{4} \| \bs{\gamma}-\bs{\beta}\|_\infty \le \frac{1}{2c_n}.
\]
This shows
\[
\max_{i,j} |\mu^{\prime\prime} (\gamma_i+\gamma_j)| \le \frac{3}{2c_n}.
\]
It follows that
\begin{equation}
\label{inequ:second:deri}
\left| \frac{ \partial^2 F_i (\bs{\gamma}) }{ \partial \gamma_j\partial \gamma_k } \right|=
\begin{cases}
\frac{3r(n-1)}{2c_n}, & i=1,j=1,k=1, \\
\frac{3r}{2c_n}, & i=1, j= k=2,\ldots, n-r+1, \\
0, & i=1; k,j=2,\ldots, n-r+1; k\neq j, \\
\frac{3r}{2c_n}, & i=2,\ldots, n-r+1, j=k=1, \\
\frac{3(n-1)}{2c_n}, &i=j=k=2,\ldots, n-r+1, \\
\frac{3}{2c_n}, &i,j,k=2,\ldots, n-r+1, i\neq j, j=k, \\
0, &i, j,k=2,\ldots, n-r+1, k\neq j, j\neq i, i\neq k.
\end{cases}
\end{equation}
For any $i,j\in \{1, \ldots, n-r+1\}$, define
\[
\mathbf{g}_{ij}(\bs{\gamma})=\left(\frac{\partial^2 F_i(\bs{\gamma}) }{ \partial \gamma_{1} \partial \gamma_j}, \ldots,
\frac{\partial^2 F_i(\bs{\gamma}) }{ \partial \gamma_{n-r+1} \partial \gamma_j} \right)^\top.
\]
By \eqref{inequ:second:deri}, we have
\begin{equation}
\|\mathbf{g}_{ij}(\bs{\gamma})\|_1 \le
\begin{cases}
\frac{ 6r(n-1) }{ 2c_n }, & i=1, j=1, \\
\frac{ 6r }{ 2c_n }, & i=1, j=2, \ldots, n-r+1, \\
\frac{ 3(n-1) }{ 2c_n }, & i=j=2, \ldots, n-r+1, \\
\frac{ 6r }{ 2c_n }, & i,j=2,\ldots, n-r+1; i\neq j.
\end{cases}
\end{equation}
%By the mean value theorem for vector function [], we have
Consequently,  for any vectors $\bs{x}, \bs{y} \subset D$, we have
\begin{align*}
 & | [F^\prime(\bs{x})]_{ij} - [F^\prime(\bs{y})]_{ij} | \\
 = &   |\int_0^1 [\mathbf{g}_{ij}(t\bs{x}+(1-t)\bs{y})]^\top (\bs{x}-\bs{y})dt | \\
 \le &
 \begin{cases}
 \frac{ 6r(n-1) }{ 2c_n } \| \bs{x}-\bs{y} \|_\infty, & i=1, j=1, \\
\frac{ 6r }{ 2c_n }\| \bs{x}-\bs{y} \|_\infty, & i=1, j=2, \ldots, n-r+1, \\
\frac{ 6r }{ 2c_n }\| \bs{x}-\bs{y} \|_\infty, & i=2, \ldots, n-r+1, j=1, \\
\frac{ 3(n-1) }{ 2c_n }\| \bs{x}-\bs{y} \|_\infty, & i=j=2, \ldots, n-r+1, \\
\frac{ 6 }{ 2c_n }\| \bs{x}-\bs{y} \|_\infty, & i,j=2,\ldots, n-r+1; i\neq j.
 \end{cases}
\end{align*}
It follows that
\[
\sum_{j=1}^{n-r+1} | [F^\prime(\bs{x})]_{ij} - [F^\prime(\bs{y})]_{ij} |
\le \begin{cases}
\frac{ 6r(2n-r-1) }{ 2c_n } \| \bs{x}-\bs{y} \|_\infty, & i=1, \\
\frac{ 3n + 6r + 6(n-r-1) }{ 2c_n } \| \bs{x}-\bs{y} \|_\infty, & i=2, \ldots, n-r+1.
\end{cases}
\]
This gives that
\begin{equation}\label{eq-wf-a}
\sum_{k=1}^{n-r+1}\widetilde{S}_{ii} | [F^\prime(\bs{x})]_{ik} - [F^\prime(\bs{y})]_{ik}|
\le \begin{cases}
\frac{ 6b_n(2n-r-1) }{ 2(n-1)c_n }, & i=1, \\
\frac{ (6n-3)b_n }{ 2(n-1)c_n } \| \bs{x}-\bs{y} \|_\infty, & i=2, \ldots, n-r+1.
\end{cases}
\end{equation}
and, by \eqref{approxi-inv2-beta-ho},
\begin{equation}\label{eq-wf-b}
\sum_{k=1}^{n-r+1}\left|\widetilde{W}_{ik} \right| | [F^\prime(\bs{x})]_{kj} - [F^\prime(\bs{y})]_{kj}|
\lesssim
\frac{ b_n^3 }{ n^2c_n^2 } \times \frac{ 6r(2n-r-1)}{ c_n } \lesssim \frac{ (n-r) b_n^3 }{ nc_n^3 }.
\end{equation}
Note that $\widetilde{W}=\widetilde{V}^{-1} - \widetilde{S}$
and $\widetilde{S}=\mathrm{diag}(1/\tilde{v}_{11}, 1/v_{r+1,r+1}, \ldots, 1/v_{nn})$.
By combining \eqref{eq-wf-a} and \eqref{eq-wf-b}, we have \eqref{eq-con-ho-a}.

Step 2. We claim that with probability at least $1-2(n-r+1)/n^2$, we have
\begin{equation}\label{ineq-union-da}
\|\widetilde{V}^{-1}F^\prime(\bs{\beta})\|_\infty \lesssim
\left\{ b_n + \frac{ b_n^3 }{ c_n^2} \left( \frac{ r^{1/2} }{ n } + \frac{n-r}{n} \right) \right\} \sqrt{\frac{\log n}{n}}.
\end{equation}
Recall that $a_{ij}$, $1\le i < j \le n$, are independent Bernoulli random variables
and $\bar{d}_i = \sum_{j\neq i} \bar{a}_{ij}$.
By \citeauthor{Hoeffding:1963}'s \citeyearpar{Hoeffding:1963} inequality, we have
\begin{equation*}
\P\left( |\bar{d}_i | \ge \sqrt{n\log n}  \right) \le 2\exp \left(- 2\frac{n\log n}{n} \right) \le  \frac{2}{n^2},~~i=1, \ldots, n.
\end{equation*}
By the union bound, we have
\begin{eqnarray}
\label{ineq-d-upper}
\P\Bigg( \max_{i=r+1, \ldots, n} | \bar{d}_i | \ge \sqrt{ n\log n} \Bigg)
\le  \sum_{i=r+1}^n \P\left(|\bar{d}_i| \geq \sqrt{n\log n} \right)
\le  \frac{2(n-r)}{n^2 },
\end{eqnarray}
such that
\[
\P\Bigg( \max_{i=r+1, \ldots, n} |\bar{d}_i| \le \sqrt{n\log n} \Bigg) \ge \P\Bigg( \max_{i=1, \ldots, n} |\bar{d}_i| \le \sqrt{n\log n} \Bigg) \ge  1 - \frac{ 2(n-r) }{ n^2 }.
\]
Note that
\[
\sum_{i=1}^r \bar{d}_i = \sum_{1\le i\neq j \le r} 2\bar{a}_{ij} + \sum_{i=1}^r \sum_{j=r+1}^n \bar{a}_{ij},
\]
and the terms in the above summation are independent.
\citeauthor{Hoeffding:1963}'s \citeyearpar{Hoeffding:1963} inequality gives that
\begin{align*}
 & \P\left( |\sum_{i=1}^r \bar{d}_i | \ge \sqrt{2\left\{\frac{r(r-1)}{2} + r(n-r-1) \right\} \log n}  \right)  \\
\le & 2\exp \left(- 2\frac{4\left\{\frac{r(r-1)}{2} + r(n-r-1)\right\}\log n}{4\left\{\frac{r(r-1)}{2} + r(n-r-1)\right\}} \right)
\le  \frac{2}{n^2}.
\end{align*}
The above arguments imply that with probability at least $1-2(n-r+1)/n^2$,
\[
\widetilde{S}_{ii} |F_i(\bs{\beta})| \le \begin{cases}
\frac{b_n}{rn} \times \sqrt{ r(2n-r-3)\log n} \le \frac{b_n(2n-r-3)^{1/2}}{(rn)^{1/2}} \sqrt{\frac{\log n}{n}}, & i=1, \\
b_n \sqrt{\frac{\log n}{n}}, & i=2,\ldots, n-r+1.
\end{cases}
\]
and, by \eqref{approxi-inv2-beta-ho},
\begin{eqnarray*}
\sum_{j=1}^{n-r+1}|\widetilde{W}_{ij}||F_j(\bs{\beta})| & \lesssim &
\frac{b_n^3}{n^2c_n^2} \times (\sqrt{ r(2n-r-3)\log n} + (n-r)\sqrt{n\log n} )\\
& \lesssim &
\frac{ b_n^3 }{ c_n^2} \left( \frac{ r^{1/2} }{ n } + \frac{n-r}{n} \right) \sqrt{\frac{\log n}{n}}.
\end{eqnarray*}

Step 3. This step is one combining step. By \eqref{eq-con-ho-a}, we can set
\[
K=O\left( \frac{b_n}{c_n} + \frac{b_n^3}{c_n^3}\cdot \frac{r}{n}  \right),
\]
and
\[
\eta = O\left(  \left\{ b_n + \frac{ b_n^3 }{ c_n^2} \left( \frac{ r^{1/2} }{ n } + \frac{n-r}{n} \right) \right\} \sqrt{\frac{\log n}{n}} \right).
\]
 in Lemma \ref{lemma:Newton:Kantovorich-b}.
If $b_n^6 /c_n^5=o( (n/\log n)^{1/2})$,
then
\[
h=K\eta \lesssim  \left( \frac{b_n^2}{c_n} + \frac{ b_n^4 }{ c_n^3 } \cdot \left( \frac{ r^{1/2} }{ n } + \frac{n-r}{n} \right) + \frac{ b_n^4}{c_n^3 } \cdot \frac{r}{ n }
+ \frac{ b_n^6 }{ c_n^5 } \left( \frac{ r^{3/2} }{ n^2 } + \frac{r(n-r)}{n^2} \right) \right)\sqrt{\frac{\log n}{n}}\to 0.
\]
This completes the proof.
\end{proof}

\subsection{Proof of Lemma 8}
\label{subsection-proof-lemma6}

In this section, we present the proof of Lemma 8.

\begin{proof}[Proof of Lemma 8]
Recall that  $\bs{\tilde{d}}=(\sum_{i=1}^r d_i, d_{r+1},\ldots, d_n)^\top$ and
$\widetilde{W} = \widetilde{V} - \widetilde{S}$.
It is sufficient to demonstrate:
\begin{equation}\label{wd-expectation2}
\E[  (\bs{\tilde{d}} - \E\bs{\tilde{d}} )^\top \widetilde{W}  (\bs{\tilde{d}} - \E\bs{\tilde{d}} ) ] =0,
\end{equation}
and
\begin{equation}\label{Wd-op2}
\mathrm{Var}( (\bs{\tilde{d}} - \E\bs{\tilde{d}} )^\top \widetilde{W}  (\bs{\tilde{d}} - \E\bs{\tilde{d}} ) )=
O\left( \frac{ b_n^3 }{ c_n^3 } \right).
\end{equation}
The claim of \eqref{wd-expectation2} is due to that
\begin{eqnarray*}
\E[  (\bs{\tilde{d}} - \E\bs{\tilde{d}} )^\top \widetilde{W} (\bs{\tilde{d}} - \E\bs{\tilde{d}} ) ] & = &
\mathrm{tr} (\E[  (\bs{\tilde{d}} - \E\bs{\tilde{d}} )^\top (\bs{\tilde{d}} - \E\bs{\tilde{d}} )  ] \widetilde{W} ) \\
& = & \mathrm{tr} ( \widetilde{V} \widetilde{W} ) = \mathrm{tr} ( I_{n-r+1} - \widetilde{V} \widetilde{S} ) = 0.
\end{eqnarray*}
Let
\[
R = \begin{pmatrix} \overline{W}_{11} &  \overline{W}_{12} \\
\overline{W}_{21} & \widetilde{W}_{22}
\end{pmatrix}
\]
where $\widetilde{W}_{22}$ is the bottom right $(n-r)\times (n-r)$ block of $\widetilde{W}$,
$\overline{W}_{11}$ is the $r\times r$ matrix with all its elements being equal to $\tilde{w}_{11}$,
and $\overline{W}_{12}$ is the $r \times (n-r)$ matrix with all its row being equal to
the vector $(\tilde{w}_{12}, \ldots, \tilde{w}_{1,n-r+1})$, and $\overline{W}_{21}$ is the transpose of $\overline{W}_{12}$.
Therefore, we have
\[
(\bs{\tilde{d}} - \E\bs{\tilde{d}} )^\top \widetilde{W}  (\bs{\tilde{d}} - \E\bs{\tilde{d}} ) = \bs{\bar{d}}^\top R \bs{\bar{d}}.
\]
Because
\[
\| R \|_{\max}=\| \widetilde{W} \|_{\max}  \lesssim \frac{ b_n^3 }{ n^2c_n^2 },
\]
with the same arguments as in the proof of Lemma 3, we have \eqref{Wd-op2}.
This completes the proof.
\end{proof}

\subsection{Proof of Lemma 9}
\label{section-lemma7}
In this section, we present the proof of Lemma 9.

\begin{proof}[Proof of Lemma 9]

Since $\beta_1=\cdots=\beta_r$ and $\widehat{\beta}_1^0 = \cdots = \widehat{\beta}_r^0$ with $r\in\{1,\ldots,n-1\}$ under the null,
with some ambiguity of notations,
we still use $\bs{\widehat{\beta}}^0$ and $\bs{\beta}$
to denote vectors $(\widehat{\beta}^0_1, \widehat{\beta}_{r+1}^0, \ldots, \widehat{\beta}_n^0)^\top$ and $(\beta_1, \beta_{r+1}, \ldots, \beta_n)^\top$, respectively.
By Lemma 5, if $b_n^6/c_n^5=o( (n/\log n)^{1/2} )$, then $\P(E_n) \ge 1 -2/n$, where
\[
E_n : = \left\{ \| \bs{\widehat{\beta}}^0 - \bs{\beta}  \|_\infty \lesssim  \frac{b_n^3}{c_n^2} \sqrt{ \frac{\log n}{n} } \right\}.
\]
The following calculations are based on the event $E_n$.

A second order Taylor expansion gives that
\begin{eqnarray*}
\mu( 2 \widehat{\beta}_1^0 ) & = & \mu( 2\beta_1) + 2\mu^{\prime}(2\beta_1)( \widehat{\beta}_1 - \beta_1)
+ \frac{1}{2}\cdot 4 \mu^{\prime\prime}( 2\tilde{\beta}_{1} ) ( \widehat{\beta}_1^0 - \beta_1)^2,
\\
\mu( \widehat{\beta}_{ij}^0 ) & = & \mu( \pi_{ij} ) + \mu^{\prime}(\pi_{ij})( \widehat{\pi}_{ij}^0 - \pi_{ij})
+ \frac{1}{2}\cdot \mu^{\prime\prime}( \tilde{\pi}_{ij} ) ( \widehat{\pi}_{ij}^0 - \pi_{ij} )^2,
\end{eqnarray*}
where $\tilde{\pi}_{ij}$ lies between $\pi_{ij}$ and $\widehat{\pi}_{ij}^0$, and, for any $i,j$,
\[
\pi_{ij}=\beta_i+\beta_j, ~~\widehat{\pi}_{ij}^0 = \widehat{\beta}_i^0 + \widehat{\beta}_j^0, ~~
\tilde{\pi}_{ij}=\tilde{\beta}_i + \tilde{\beta}_j.
\]
It follows that
\begin{equation}
\label{eq-homo-dia}
\sum_{i=1}^r d_i - \sum_{i=1}^r \E d_i  =   2r(r-1)\mu^{\prime}(2\beta_1)( \widehat{\beta}_1^0 - \beta_1) + r\sum_{j=r+1}^n \mu^{\prime}(\pi_{1j})( \widehat{\pi}_{1j}^0 - \pi_{1j})
+  h_1,
\end{equation}
and, for $i=r+1, \ldots, n$,
\begin{equation}
\label{eq-homo-dib}
d_i - \E d_i  =   r\mu^{\prime}( \pi_{i1})(\widehat{\pi}_{i1} - \pi_{i1} ) +
\sum_{j=r+1, j\neq i}^n \mu^{\prime}( \pi_{ij})(\widehat{\pi}_{ij} - \pi_{ij} ) + h_i,
\end{equation}
where
\begin{eqnarray*}
h_1 = 2r(r-1) \mu^{\prime\prime}( 2\tilde{\pi}_{11} ) ( \widehat{\beta}_1^0 - \beta_1)^2
+r \sum_{j=r+1}^n \frac{1}{2} \mu^{\prime\prime}( \tilde{\pi}_{1j} ) ( \widehat{\beta}^0_1 - \beta_{1} )^2
\\
h_i = r \mu^{\prime\prime}( \tilde{\pi}_{i1} ) ( \widehat{\pi}_{i1}^0 - \pi_{i1})^2
+ \sum_{j=r+1}^n \frac{1}{2} \mu^{\prime\prime}( \tilde{\pi}_{ij} ) ( \widehat{\pi}_{ij} - \pi_{ij} )^2,~~i=r+1, \ldots, n.
\end{eqnarray*}
Writing \eqref{eq-homo-dia} and \eqref{eq-homo-dia} into a matrix form, we have
\begin{equation}\label{eq-expression-beta-star}
\bs{\widetilde{d}} - \E \bs{\widetilde{d}} = \widetilde{V}( \bs{\widehat{\beta}}^0 - \bs{\beta} ) + \bs{\tilde{h}},
\end{equation}
where $\bs{\tilde{h}}=(h_1, h_{r+1}, \ldots, h_n)^\top$.  It is equivalent to
\[
\bs{\widehat{\beta}}^0 - \bs{\beta} = \widetilde{V}^{-1} ( \bs{\widetilde{d}} - \E \bs{\widetilde{d}} ) - \widetilde{V}^{-1}\bs{\tilde{h}}.
\]
In view of that $\max_{ij}|\mu^{\prime\prime}(\pi_{ij})| \le 1/c_n$ and the event $E_n$, we have
\begin{equation}\label{ineq-home-be-h1}
|h_1|  \lesssim  \frac{rn}{c_n} \|  \bs{\widehat{\beta}}^0 - \bs{\beta} \|_\infty^2
 \lesssim  \frac{ rb_n^6\log n }{c_n^5 }.
\end{equation}
and, for $k=r+1, \ldots, n$,
\begin{equation}\label{ineq-home-be-hk}
|h_k| \lesssim  \frac{n}{c_n} \|  \bs{\widehat{\beta}}^0 - \bs{\beta} \|_\infty^2 \lesssim  \frac{ b_n^6\log n }{ c_n^5 }.
\end{equation}

By letting $\widetilde{V} = \widetilde{S} + \widetilde{W}$, in view of (35), \eqref{ineq-home-be-h1} and \eqref{ineq-home-be-hk}, we have
\begin{eqnarray}
\nonumber
\| \widetilde{V}^{-1}\bs{h} \|_\infty & \lesssim & \frac{ |h_1| }{ \tilde{v}_{11} } + \max_{i=r+1, \ldots, n} \frac{ |h_i| }{ v_{ii} }
+ \| \widetilde{W} \|_{\max} \left(\sum_{i=r+1}^n |h_i| + |h_1| \right) \\
\nonumber
& \lesssim & \frac{ b_n^6 \log n}{c_n^5 } \cdot \frac{ b_n }{n } + \frac{ b_n^3 }{ n^2 c_n^2 } \cdot \left\{ \frac{rb_n^6\log n }{ c_n^5}
+ (n-r)  \frac{ b_n^6 \log n}{c_n^5 } \right\} \\
\label{ineq-home-be-vh}
& \lesssim & \frac{ b_n^9 \log n}{ nc_n^7 }.
\end{eqnarray}
Now, we bound the error term $\| \widetilde{W} (\bs{\widetilde{d}}- \E \bs{\widetilde{d}}) \|_\infty$.
Note that
\begin{align*}
 &[\widetilde{W} (\bs{\widetilde{d}}- \E \bs{\widetilde{d}})]_i \\
 = & \tilde{w}_{i1} \sum_{j=1}^r \bar{d}_j + \sum_{j=r+1}^n \tilde{w}_{ij}\bar{d}_j, \\
 = & \tilde{w}_{i1} \sum_{1\le k<j \le r} (\bar{a}_{kj}+\bar{a}_{jk}) + \sum_{k=1}^r \sum_{j=r+1}^n (\tilde{w}_{i1}\bar{a}_{kj} +\tilde{w}_{jk}\bar{a}_{jk})
 \\
& +  \sum_{r+1 \le k < j \le n} ( \tilde{w}_{kj} \bar{a}_{kj} + \tilde{w}_{jk} \bar{a}_{jk} ).
\end{align*}
The summation of the above right hand can be viewed as the sum of $n(n-1)/2$ independent random variables.
Because $\E \bar{a}_{ij}^2 \le 1/c_n$, we have
\begin{align*}
& \E \{[\widetilde{W} (\bs{\tilde{d}}- \E \bs{\tilde{d}})]_i \}^2 \\
\le & \left\{\frac{r(r-1)}{2c_n} + \frac{4r(n-r)}{c_n} + \frac{4(n-r)(n-r+1)}{c_n} \right\}\|\widetilde{W} \|_{\max}^2 \\
 \lesssim & \frac{ n^2 }{c_n} \|\widetilde{W} \|_{\max}^2.
\end{align*}
It follows from Bernstein's inequality in Lemma \ref{lemma:bernstein} and inequality (35), with probability $1- N^{-2}$,  we have that
\begin{eqnarray*}
| [\widetilde{W} (\bs{\tilde{d}}- \E \bs{\tilde{d}})]_i | & \le & \sqrt{ 2\log N \left( \E \{[\widetilde{W} (\bs{\tilde{d}}- \E \bs{\tilde{d}})]_i \}^2   \right)} + \frac{2}{3} \cdot  \frac{ b_n }{ (n-1)c_n } \cdot \log n \\
& \lesssim & \frac{ b_n^3}{ n^2 c_n^2 } \cdot \frac{n (\log n)^{1/2}}{ c_n^{1.2}} \\
& \lesssim & \frac{ b_n^3(\log n)^{1/2} }{ n c_n^{5/2} },
\end{eqnarray*}
where $N=n(n-1)/2$.
By the uniform bound,  with probability at leas $1- 4/(n-1)^3$, we have
\begin{equation}\label{eq-hatbeta-bb}
\| \widetilde{W} (\bs{\tilde{d}}- \E \bs{\tilde{d}}) \|_\infty \lesssim \frac{ b_n^3(\log n)^{1/2} }{ n c_n^{5/2} }.
\end{equation}

By combining \eqref{eq-expression-beta-star}, \eqref{ineq-home-be-vh} and  \eqref{eq-hatbeta-bb}, with probability at least $1- O(n^{-1})$, we have
\begin{eqnarray}
\label{expan-hatbeta-hoa}
\widehat{\beta}_1^0 - \beta_1 & = & \frac{ \sum_{i=1}^r \bar{d}_i }{ \tilde{v}_{11} }+g_1, \\
\label{expan-hatbeta-hob}
\widehat{\beta}_i^0 - \beta_i & = & \frac{ \bar{d}_i }{ v_{ii} } + g_i, i=r+1, \ldots, n,
\end{eqnarray}
where $g_1, g_{r+1}, \ldots, g_n$  simultaneously satisfy
\[
g_i = (\widetilde{V}^{-1}\bs{h})_i + [\widetilde{W} (\bs{\tilde{d}}- \E \bs{\tilde{d}})]_i = O\left( \frac{b_n^9 \log n}{ nc_n^7 } \right).
\]
\end{proof}

\subsection{Proof of (39)}
\label{section-proof-39-B20}
Let $\tilde{\ell}( \bs{{\beta}} )=-\ell( \bs{{\beta}} )$.
The expression of $B_2^0$ can be written as
\begin{eqnarray*}
-B_2^0 & = & \underbrace{\frac{ \partial^3 \tilde{\ell}( \bs{{\beta}} ) }{ \partial \beta_1^3 } (\widehat{\beta}_1^0 - \beta_1)^3}_{Q_1}
+ 3\underbrace{\sum_{i=r+1}^n \frac{  \partial^3 \tilde{\ell}( \bs{{\beta}} ) }{ \partial \beta_1^2 \partial \beta_i }
( \widehat{\beta}_1 - \beta_1)^2( \widehat{\beta}_i - \beta_i )}_{Q_2} \\
&&+ 3 \underbrace{\sum_{i,j=r+1}^n \frac{  \partial^3 \tilde{\ell}( \bs{{\beta}} ) }{ \partial \beta_1 \partial \beta_i \partial \beta_j}
( \widehat{\beta}_1 - \beta_1)( \widehat{\beta}_i - \beta_i ) ( \widehat{\beta}_j - \beta_j )}_{Q_3} \\
&&+ \underbrace{\sum_{i,j,k=r+1}^n \frac{  \partial^3 \tilde{\ell}( \bs{{\beta}} ) }{ \partial \beta_i \partial \beta_j \partial \beta_k}
( \widehat{\beta}_i - \beta_i)( \widehat{\beta}_j - \beta_j ) ( \widehat{\beta}_k - \beta_k )}_{Q_4},
\end{eqnarray*}
where
\begin{eqnarray*}
Q_1 & = & \{ 4r(r-1)\mu^{\prime\prime}(\pi_{11})+r\sum_{j=r+1}^n \mu^{\prime\prime}(\pi_{ij}) \} (\widehat{\beta}_1^0 - \beta_1)^3, \\
Q_2&=& r \sum_{i=r+1}^n \mu^{\prime\prime}( \pi_{1i}) ( \widehat{\beta}_1^0  - \beta_1)^2 ( \widehat{\beta}_i^0 - \beta_i),  \\
Q_3&=& r \sum_{i=r+1}^n \mu^{\prime\prime}(\pi_{1i}) ( \widehat{\beta}_1^0  - \beta_1) ( \widehat{\beta}_i^0 - \beta_i)^2, \\
Q_4& = & \sum_{i,j=r+1, i\neq j}^n \mu^{\prime\prime}(\pi_{ij}) ( \widehat{\beta}_i^0  - \beta_i) ( \widehat{\beta}_j^0 - \beta_j)^2.
\end{eqnarray*}

We shall in turn bound each term in the above summation.
To simplify notations, let
\begin{eqnarray*}
f_1 & = & \frac{ \partial^3 \tilde{\ell}( \bs{\beta} ) }{ \partial \beta_1^3 }=4r(r-1)\mu^{\prime\prime}(\pi_{11})
+ r\sum_{j=r+1}^n \mu^{\prime\prime}(\pi_{1j}),  \\
f_{1j} & = & \frac{  \partial^3 \tilde{\ell}( \bs{\beta} ) }{ \partial \beta_1 \partial \beta_j^2 }=
r\mu^{\prime\prime}(\pi_{1j}),~~j=r+1,\ldots, n, \\
f_{ij} & = & \frac{  \partial^3 \tilde{\ell}( \bs{\beta} ) }{ \partial \beta_i^2 \partial \beta_j },~~i,j=r+1,\ldots, n.
\end{eqnarray*}
In view of \eqref{ineq-mu-deriv-bound}, we have
\begin{equation}\label{ineq-f1-fj-fij}
|f_1| \lesssim \frac{rn}{c_n}, \quad |f_{1j}| \lesssim \frac{r}{c_n}, \quad |f_{ij}| \lesssim \frac{1}{c_n}.
\end{equation}
Because $\sum_{i=1}^r \bar{d}_i$ can be expressed as the sum of $r(r-1)/2 + r(n-r)$ independent and bounded random variables,
\[
\sum_{i=1}^r \bar{d}_i = 2\sum_{1\le i<j\le r} \bar{a}_{ij} + \sum_{i=1}^r \sum_{j=r+1}^n \bar{a}_{ij}
\]
and
\[
\E (2\sum_{1\le i<j\le r} \bar{a}_{ij} + \sum_{i=1}^r \sum_{j=r+1}^n \bar{a}_{ij})^2
\le \frac{2r(r-1)}{c_n} + \frac{r(n-r)}{c_n},
\]
by Bernstern's inequality, with probability at least $ 1- 2(rn)^{-2}$, we have
\begin{equation}\label{eq-sum-hom-d}
\left| \sum_{i=1}^r \bar{d}_i \right| \lesssim \sqrt{ 2\log (rn) \times \frac{2r(r-1)+r(n-r)}{c_n} } + \frac{12}{3}\log n \lesssim
\sqrt{\frac{rn\log n}{c_n}}.
\end{equation}
By \eqref{expan-hatbeta-hoa} and \eqref{eq-sum-hom-d}, we have
\begin{eqnarray}
\nonumber
| f_1(\widehat{\beta}_1^0 - \beta_1)^3 | & = & \frac{rn}{c_n}  \left|\left( \frac{ \sum_{i=1}^r \bar{d}_i }{ \tilde{v}_{11} } + g_1 \right)^3 \right|\\
\nonumber
& = &  \frac{rn}{c_n}  \left|
 ( \frac{ \sum_{i=1}^r \bar{d}_i }{ \tilde{v}_{11} })^3 + 3 ( \frac{ \sum_{i=1}^r \bar{d}_i }{ \tilde{v}_{11} })^2 g_1
 + 3 ( \frac{ \sum_{i=1}^r \bar{d}_i }{ \tilde{v}_{11} }) g_1^2 + g_1^3
\right| \\
\nonumber
& \lesssim &  \frac{rn}{c_n} \left\{
\frac{b_n^3}{ (rn)^3}  \cdot (\frac{rn\log n}{c_n})^{3/2} +
\frac{b_n^2}{ (rn)^2} \cdot \frac{rn\log n}{c_n} \cdot \frac{ b_n^9 \log n}{ nc_n^7} \right. \\
\nonumber
&& \left. +
\frac{b_n}{ (rn)} \cdot \sqrt{\frac{rn\log n}{c_n}} \cdot (\frac{ b_n^9 \log n}{ nc_n^7})^2
+ (\frac{ b_n^9 \log n}{ nc_n^7})^3
\right\} \\
\nonumber
& \lesssim &  \frac{ b_n^3 }{ c_n^{5/2} } \cdot \frac{ (\log n)^{1/2} }{ (rn)^{1/2} } + \frac{ b_n^{11} }{ c_n^9 } \cdot \frac{ (\log n)^2 }{ n} \\
\label{ineq-Q1-ho-a}
&&+ \frac{ b_n^{19}}{c_n^{15/2}} \cdot \frac{ (rn)^{1/2}(\log n)^{5/2}}{ n^2 }
+ \frac{ b_n^{27} (\log n)^3r }{ n^2}.
\end{eqnarray}
Therefore, if $b_n^9/c_n^7=o( n^{1/3}/(\log n) )$, then
\begin{equation}\label{ineq-1b-Q1}
Q_1 = O\left( \frac{ b_n^{27} (\log n)^3 }{ n} \right) = o(1).
\end{equation}

We now bound $Q_2$.  By \eqref{expan-hatbeta-hoa} and \eqref{expan-hatbeta-hob}, we have
\begin{eqnarray*}
&&r \left| \sum_{i=r+1}^n \mu^{\prime\prime}(\pi_{1i}) ( \widehat{\beta}_1^0 - \beta_1)^2(\widehat{\beta}_i - \beta_i) \right| \\
& = & r \left| \sum_{i=r+1}^n \mu^{\prime\prime}(\pi_{1i}) ( \frac{ \tilde{d}_1 }{ \tilde{v}_{11}} + g_1)^2
( \frac{\bar{d}_i}{ v_{ii} } + g_i ) \right| \\
& = & r \left| \sum_{i=r+1}^n \mu^{\prime\prime}(\pi_{1i}) ( (\frac{ \tilde{d}_1 }{ \tilde{v}_{11}})^2  +2 \frac{ \tilde{d}_1 }{ \tilde{v}_{11}} g_1 + g_1^2)
( \frac{\bar{d}_i}{ v_{ii} } + g_i ) \right| \\
& \lesssim &  r \sum_{i=r+1}^n |\mu^{\prime\prime}(\pi_{1i})| \left( \left|
(\frac{ \tilde{d}_1 }{ \tilde{v}_{11}})^2\cdot \frac{ \bar{d}_i }{ v_{ii} } \right|
+ \left|(\frac{ \tilde{d}_1 }{ \tilde{v}_{11}})^2 \cdot g_i \right| %+ 2 \frac{ \tilde{d}_1 }{ \tilde{v}_{11}}) g_1 \frac{ \bar{d}_i }{ v_{ii}}
%+ 2 \frac{ \tilde{d}_1 }{ \tilde{v}_{11}}) g_1  g_i
+ \left| \frac{ \bar{d}_i }{ v_{ii}} \cdot g_1^2 \right| + \left| g_1^2 g_i \right|
\right)\\
& \lesssim &
\frac{ r(n-r)}{ c_n } \left[
\left( \frac{ b_n \sqrt{rn\log n/c_n} }{ rn } \right)^2 \frac{ b_n \sqrt{n\log n} }{ n}
+ \left( \frac{ b_n \sqrt{rn\log n/c_n} }{ rn } \right)^2 \cdot \frac{ b_n^9 \log n}{ nc_n^7 }
\right. \\
%&&+ \frac{ b_n \sqrt{ rn\log n}}{ rn} \cdot \frac{ b_n^9 \log n}{ nc_n^7 } \cdot \frac{ b_n \sqrt{n\log n}}{ n}
%+ \frac{ b_n \sqrt{ rn\log n}}{ rn} \cdot \left( \frac{ b_n^9 \log n}{ nc_n^7 } \right)^2\\
&&
\left.
+ \frac{ b_n \sqrt{n\log n}}{ n} \cdot \left( \frac{ b_n^9 \log n}{ nc_n^7 } \right)^2
+ \left( \frac{ b_n^9 \log n}{ nc_n^7 } \right)^3
\right] \\
& \lesssim &  \frac{ b_n^3 (n-r)(\log n)^{3/2} }{ n^{3/2} c_n^2} + \frac{ b_n^{11} (\log n)^2(n-r) }{ n^2c_n^9}\\
%+ \frac{ b_n^{11} (\log n)^2 }{ n^{1/2} c_n^8 } + \frac{ b_n^{19} (\log n)^{5/2} }{ nc_n^{15}}
&&+ \frac{ b_n^{19} (\log n)^{5/2} r(n-r) }{ n^{5/2} c_n^{15}} + \frac{ b_n^{27} (\log n)^{3}r(n-r)}{ n^3c_n^{21}}.
\end{eqnarray*}
Therefore, if $b_n^9/c_n^7 = o( n^{1/3}/(\log n) )$, then
\begin{equation}\label{eq-hom-Q2}
Q_2 \lesssim \frac{ b_n^{27} (\log n)^{3}}{ n c_n^{21}}  =o(1).
\end{equation}

We now consider $Q_3$.
\iffalse
Because
\[
\frac{  \partial^3 \tilde{\ell}( \bs{\widetilde{\beta}} ) }{ \partial \beta_1 \partial \beta_i \partial \beta_j}= 0, i\neq j \in \{ r+1, \ldots, n\},
\]
we have
\begin{eqnarray*}
Q_3 & = & \sum_{i=r+1}^n \frac{  \partial^3 \tilde{\ell}( \bs{\widetilde{\beta}} ) }{ \partial \beta_1 \partial \beta_i^2}
( \widehat{\beta}_1^0 - \beta_1)( \widehat{\beta}_i^0 - \beta_i )^2 \\
& = & \sum_{i=r+1}^n r\mu^{\prime\prime}( \tilde{\pi}_{1i} )
( \widehat{\beta}_1^0 - \beta_1)( \widehat{\beta}_i^0 - \beta_i )^2 \\
& = & \sum_{i=r+1}^n r\mu^{\prime\prime}( \pi_{1i} )
( \widehat{\beta}_1^0 - \beta_1)( \widehat{\beta}_i^0 - \beta_i )^2 \\
&& \sum_{i=r+1}^n r\mu^{\prime\prime}( \cdot{\pi}_{1i} )
( \widehat{\beta}_1^0 - \beta_1)( \widehat{\beta}_i^0 - \beta_i )^2(\tilde{\pi}_{1i}-\pi_{1i}) \\
\end{eqnarray*}
\fi
\begin{eqnarray*}
Q_3&=& r \sum_{i=r+1}^n \mu^{\prime\prime}(\pi_{1i}) ( \widehat{\beta}_1^0  - \beta_1) ( \widehat{\beta}_i^0 - \beta_i)^2 \\
& = & r \sum_{i=r+1}^n \mu^{\prime\prime}(\pi_{1i})
( \frac{ \sum_{k=1}^r \bar{d}_k }{ \tilde{v}_{11} } + g_1 )( \frac{ \bar{d}_i }{ v_{ii} } + g_i )^2 \\
& \lesssim & \frac{ r(n-r)}{ c_n} \left( \frac{ |\sum_{k=1}^r \bar{d}_k| }{ \tilde{v}_{11} } + |g_1| \right)( \frac{ \bar{d}_i }{ v_{ii} } + g_i )^2 \\
& \lesssim & \frac{ r(n-r)}{ c_n} \left( \frac{b_n}{ rn} \cdot \sqrt{ \frac{rn\log n}{ c_n } }  + \frac{ b_n^9 \log n}{ nc_n^7 } \right)( \frac{ b_n^2 n\log n }{ n^2 }  + ( \frac{b_n^9\log n}{ nc_n^7} )^2 ) \\
& \lesssim & \frac{ b_n^3(\log n)^{3/2} r^{1/2} (n-r) }{n^{3/2}c_n} + \frac{ b_n^{19} (\log n)^2 r^{1/2}(n-r) }{ n^{5/2}c_n^{14}}  \\
&&+
\frac{ b_n^{11} (\log n)^2 r(n-r)}{ n^2 c_n^7 }
+ \frac{ (n-r)b_n^{27} (\log n)^3 }{ n^2 c_n^{21}}.
\end{eqnarray*}
If $b_n^9/c_n^7 = o( n^{1/3}/(\log n) )$, then
\[
Q_3 \lesssim \frac{ b_n^3(\log n)^{3/2} }{ c_n} + \frac{ (n-r)b_n^{27} (\log n)^3 }{ n^2 c_n^{21}} = O(\frac{ b_n^3(\log n)^{3/2} }{ c_n})+o(1).
\]

Finally, we bound $Q_4$. It can be written as
\begin{eqnarray*}
Q_4& = & 3\sum_{i,j=r+1;i\neq j}^n \mu^{\prime\prime}( \pi_{ij})
( \widehat{\beta}_i - \beta_i)^2( \widehat{\beta}_j - \beta_j ) \\
&& + 3\sum_{i,j=r+1;i\neq j}^n \mu^{\prime\prime}( \pi_{ij})
( \widehat{\beta}_i - \beta_i)^2( \widehat{\beta}_j - \beta_j )
\end{eqnarray*}
With similar arguments as in the proof of Lemma 5, we have
\begin{eqnarray*}
&& \left| \sum_{i,j=r+1;i\neq j}^n \mu^{\prime\prime}( \pi_{ij})
( \widehat{\beta}_i - \beta_i)^2( \widehat{\beta}_j - \beta_j ) \right | \\
& \lesssim & \frac{ b_n^{27} (\log n)^3 }{ nc_n^{22}}( 1- \frac{r}{n})^2
+ \frac{ b_n^{19} }{ c_n^{15}} \cdot \frac{ (\log n)^{5/2} }{ n^{1/2}} \cdot ( 1- \frac{r}{n})^2 \\
&&+ \frac{ b_n^{11}(\log n)^{2}}{ c_n^8 } ( 1- \frac{r}{n})^2 + \frac{ b_n^3 \log n}{ c_n}( 1- \frac{r}{n}).
\end{eqnarray*}
This gives that
\begin{eqnarray*}
\frac{Q_4}{r^{1/2}} & \lesssim & \frac{1}{r^{1/2}} \left( \frac{ b_n^{27} (\log n)^3 }{ nc_n^{22}}( 1- \frac{r}{n})^2
+ \frac{ b_n^{19} }{ c_n^{15}} \cdot \frac{ (\log n)^{5/2} }{ n^{1/2}} \cdot ( 1- \frac{r}{n})^2 \right. \\
&& \left.+ \frac{ b_n^{11}(\log n)^{2}}{ c_n^8 } ( 1- \frac{r}{n})^2 + \frac{ b_n^3 \log n}{ c_n}( 1- \frac{r}{n}) + \frac{ b_n^{12} (\log n)^2 (n-r)^2 }{ n^2c_n^9} \right)
\end{eqnarray*}

If
\[
\frac{ b_n^{12} }{ c_n^9 } = o\left( \frac{ r^{1/2} }{ (\log n)^3 } \right), \mbox{~~and~~} \frac{ b_n^{19}}{ c_n^{14} } = o\left( \frac{n^{1/2}}{ (\log n)^3 } \right),
\]
then
\[
\max\left\{ \frac{|Q_1|}{ r^{1/2} }, \frac{|Q_2|}{ r^{1/2} }, \frac{|Q_3|}{ r^{1/2} }, \frac{|Q_4|}{ r^{1/2} } \right\} = o_p(1),
\]
which shows (39) in the main text.

\section{Proofs of supported Lemmas in the proof of Theorem 2 (a)}
\label{section-theorem2a}

This section presents the proofs of supported Lemmas in the proof of Theorem 2 (a) and two vanishing remainder terms.
This section is organized as follows.
Sections \ref{section-proof-lemma1010},  \ref{section-proof-lemma11} and \ref{section-proof-lemma12}
present the proofs of Lemmas 10, 11 and 12, respectively.
Sections \ref{subsection:B2B20} and \ref{subsection-B3B30} presents the proofs of orders of two remainder terms
$B_2-B_2^0$  in (59) and $B_3-B_3^0$ in (60) in the main text, respectively.

\subsection{Proof of Lemma 10}
\label{section-proof-lemma1010}

This section presents the proof of Lemma 10.

\begin{proof}[Proof of Lemma 10]
Note that
\[
\left[
\begin{pmatrix} S_{11} & \mathbf{0} \\
\mathbf{0} & S_{22}
\end{pmatrix}
+
\begin{pmatrix} W_{11} & W_{12} \\
W_{21} & W_{22}
\end{pmatrix}
\right]
\begin{pmatrix}
V_{11} & V_{12} \\
V_{21} & V_{22}
\end{pmatrix}
=I_{n\times n},
\]
where $V_{11}$ is the upper left $r\times r$ sub-matrix of $V$.
Because
\[
W_{21} V_{12} + W_{22}V_{22} + S_{22}V_{22} = I_{(n-r)\times (n-r) },
\]
and
\[
(\widetilde{W}_{22}+S_{22})V_{22} = I_{(n-r)\times (n-r) },
\]
we have
\[
W_{21}V_{12} + W_{22}V_{22} = \widetilde{W}_{22}V_{22} \Longrightarrow W_{22} - \widetilde{W}_{22}= - V_{22}^{-1}W_{21}V_{12},
\]
where
\[
\widetilde{W}_{22} = V_{22}^{-1} - S_{22}.
\]
With the similar arguments as in the proof of \eqref{ineq-V-S-appro-upper-b}, we have
\[
\| \widetilde{W}_{22} \|_{\max} \lesssim \frac{b_n^3}{n^2c_n^2}.
\]
Note that $r$ is a fixed positive integer.
A direct calculation gives that
\begin{eqnarray*}
|(S_{22} W_{21} V_{12})_{ij}| &  = & |\sum_{k=1}^{n-r}\sum_{h=1}^r (S_{22})_{ik} (W_{21})_{k h} (V_{12})_{hj}| \\
& = & | \sum_{h=1}^r \frac{ 1}{v_{i+r,i+r}} (W_{21})_{i h} (V_{12})_{hj} |\\
& \lesssim &  r \cdot \frac{ b_n }{n-1} \cdot \frac{ b_n^3 }{ n^2c_n^2 } \cdot \frac{1}{c_n} \lesssim \frac{b_n^4}{ n^3c_n^3},
\end{eqnarray*}
and
\begin{eqnarray*}
|(W_{22} W_{21} V_{12})_{ij}| & = & |\sum_{k=1}^{n-r}\sum_{h=1}^r (W_{22})_{ik} (W_{21})_{k h} (V_{12})_{hj}| \\
& \le & (n-r)r \cdot \| W \|_{\max}^2 \cdot \frac{1}{c_n} \lesssim \frac{ b_n^6 }{ n^3c_n^5 } .
\end{eqnarray*}
This shows that
\begin{equation}\label{ineq-W-diff-upper}
\| W_{22} - \widetilde{W}_{22} \|_{\max} \lesssim \frac{ b_n^6 }{ n^3c_n^5 },
\end{equation}
which has a much smaller error in contract to $\| W_{22} \|_{\max}$ and $\| \widetilde{W}_{22} \|_{\max}$ whose magnitudes are $b_n^3/(n^2c_n^2)$.
\end{proof}

\subsection{Proof of Lemma 11}
\label{section-proof-lemma11}

In this section, we present the proof of Lemma 11.
\begin{proof}[Proof of Lemma 11]
Note that $r$ is a fixed constant, $\bs{\bar{d}}_1= (\bar{d}_1, \ldots, \bar{d}_r)^\top$, and
$\bs{\bar{d}}_2 = (\bar{d}_{r+1}, \ldots, \bar{d}_n)^\top$.

(a) We  bound $\bs{\bar{d}}_1^\top W_{11} \bs{\bar{d}}_1$.
By \eqref{ineq-d-upper}, with probability at least $1-2/n$, we have
\[
\|\bs{\bar{d}}\|_\infty \le \sqrt{n\log n}.
\]
It follows that, by \eqref{ineq-V-S-appro-upper-b},
\begin{eqnarray*}
\bs{\bar{d}}_1^\top W_{11} \bs{\bar{d}}_1 & \lesssim  & \frac{ b_n^3 }{ n^2 c_n^2} \cdot (n\log n)^{1/2} \lesssim \frac{ b_n^3 (\log n)^{1/2} }{ n^{3/2}c_n^2}
\end{eqnarray*}

(b) We bound $\bs{\bar{d}}_1^\top W_{12} \bs{\bar{d}}_2$.
Note that
\[
\E \bs{\bar{d}}^\top (V^{-1}-S) \bs{\bar{d}} = \E \mathrm{tr}( V^{-1}-S) \bs{\bar{d}} \bs{\bar{d}}^\top )
= \mathrm{tr}( I - VS) = 0.
\]
It follows that
\[
\E \bs{\bar{d}}_1^\top W_{12} \bs{\bar{d}}_2 =0.
\]
Now, we calculate
\[
\mathrm{Var}( \bs{\bar{d}}_1^\top W_{12} \bs{\bar{d}}_2 ) = \sum_{i=1}^r \sum_{j=r+1}^n \sum_{\alpha=1}^r \sum_{\gamma=r+1}^n
\mathrm{Cov}( \bar{d}_i w_{ij}\bar{d}_j, \bar{d}_\alpha w_{\alpha\gamma} \bar{d}_\gamma).
\]
Note that
\[
|\mathrm{Cov}( \bar{d}_i w_{ij}\bar{d}_j, \bar{d}_\alpha w_{\alpha\gamma} \bar{d}_\gamma)|
\le \| W \|_{\max}^2 |\mathrm{Cov}( \bar{d}_i \bar{d}_j, \bar{d}_\alpha  \bar{d}_\gamma)|
\]
We evaluate $\mathrm{Cov}( \bar{d}_i \bar{d}_j, \bar{d}_\alpha  \bar{d}_\gamma)$ according to four cases:
(Case A) $i=\alpha\in\{1,\ldots, r\}$, $j=\gamma \in \{r+1,\ldots, n\}$;
(Case B) $i=\alpha\in\{1,\ldots, r\}$, $j \neq \gamma \in \{r+1,\ldots, n\}$;
(Case C) $i\neq \alpha\in\{1,\ldots, r\}$, $j=\gamma \in \{r+1,\ldots, n\}$;
(Case D) $i\neq \alpha\in\{1,\ldots, r\}$, $j\neq \gamma \in \{r+1,\ldots, n\}$.\\
Case A: the expression of $\mathrm{Cov}( \bar{d}_i \bar{d}_j, \bar{d}_i \bar{d}_j)$ is
\begin{eqnarray*}
\mathrm{Cov}( \bar{d}_i \bar{d}_j, \bar{d}_i \bar{d}_j) & = & \sum_s \sum_t \sum_\eta \sum_\zeta
( \E \bar{a}_{is} \bar{a}_{jt} \bar{a}_{i\eta} \bar{a}_{j\zeta} -
\E \bar{a}_{is} \bar{a}_{jt} \E \bar{a}_{i\eta} \bar{a}_{j\zeta} ) \\
& = & \sum_s \sum_t ( \E \bar{a}_{is} \bar{a}_{jt} \bar{a}_{is} \bar{a}_{jt} -
\E \bar{a}_{is} \bar{a}_{jt} \E \bar{a}_{is} \bar{a}_{jt} ).
\end{eqnarray*}
Case B: the expression of $\mathrm{Cov}( \bar{d}_i \bar{d}_j, \bar{d}_i \bar{d}_\gamma)$ is
\begin{eqnarray*}
\mathrm{Cov}( \bar{d}_i \bar{d}_j, \bar{d}_i \bar{d}_\gamma) & = & \sum_s \sum_t \sum_\eta \sum_\zeta
( \E \bar{a}_{is} \bar{a}_{jt} \bar{a}_{i\eta} \bar{a}_{\gamma \zeta} -
\E \bar{a}_{is} \bar{a}_{jt} \E \bar{a}_{i\eta} \bar{a}_{\gamma\zeta} ) \\
& = & \sum_s ( \E \bar{a}_{is} \bar{a}_{j\gamma} \bar{a}_{is} \bar{a}_{\gamma j} -
\E \bar{a}_{is} \bar{a}_{j\gamma} \E \bar{a}_{is} \bar{a}_{j\gamma} )
+ \E \bar{a}_{i\gamma}\bar{a}_{\gamma i} \bar{a}_{ij}\bar{a}_{ji}.
\end{eqnarray*}
Case C: the expression of $\mathrm{Cov}( \bar{d}_i \bar{d}_j, \bar{d}_\alpha \bar{d}_j)$ is
\begin{eqnarray*}
\mathrm{Cov}( \bar{d}_i \bar{d}_j, \bar{d}_\alpha \bar{d}_j) & = & \sum_s \sum_t \sum_\eta \sum_\zeta
( \E \bar{a}_{is} \bar{a}_{jt} \bar{a}_{\alpha\eta} \bar{a}_{j \zeta} -
\E \bar{a}_{is} \bar{a}_{jt} \E \bar{a}_{\alpha\eta} \bar{a}_{j \zeta} ) \\
& = & \sum_t ( \E \bar{a}_{i\alpha} \bar{a}_{j t} \bar{a}_{\alpha i} \bar{a}_{j t} -
\E \bar{a}_{i\alpha} \bar{a}_{j t} \E \bar{a}_{\alpha i} \bar{a}_{j t})
+ \E \bar{a}_{i j}\bar{a}_{j i} \bar{a}_{\alpha j}\bar{a}_{j\alpha}.
\end{eqnarray*}
Case D: the expression of $\mathrm{Cov}( \bar{d}_i \bar{d}_j, \bar{d}_\alpha \bar{d}_\gamma)$ is
\begin{eqnarray*}
\mathrm{Cov}( \bar{d}_i \bar{d}_j, \bar{d}_\alpha \bar{d}_\gamma) & = & \sum_s \sum_t \sum_\eta \sum_\zeta
( \E \bar{a}_{is} \bar{a}_{jt} \bar{a}_{\alpha\eta} \bar{a}_{\gamma \zeta}) -
\E \bar{a}_{is} \bar{a}_{jt} \E \bar{a}_{\alpha\eta} \bar{a}_{\gamma \zeta} ) \\
& = & 0.
\end{eqnarray*}
By combining the above four cases, it yields
\[
\mathrm{Var}( \bs{\bar{d}}_1^\top W_{12} \bs{\bar{d}}_2 ) \le r(n-r)\cdot \|W\|_{\max}^2 \cdot n^2 \max_{i,j}(\E \bar{a}_{ij}^2)^2
\lesssim n\cdot \left( \frac{b_n^3 }{ n^2 c_n^2} \right)^2 \cdot \frac{n^2}{c_n^2} \lesssim \frac{b_n^6}{c_n^2n}.
\]
By Chebyshev's inequality, we have
\[
\bs{\bar{d}}_1^\top W_{12} \bs{\bar{d}}_2 = O_p\left( \sqrt{\frac{b_n^6}{nc_n^6} } \right).
\]

(c) We bound $\bs{\bar{d}}_2^\top (W_{22}-\widetilde{W}_{22}) \bs{\bar{d}}_2$.
By Lemma \ref{lemma:var:quadra}, we have that
\[
\mathrm{ Var}( \sum_i \bar{d}_i^2 ) \lesssim n \sum_i v_{ii}^2 \lesssim \frac{ n^3 }{c_n^2}.
\]
Chebychev's inequality gives that
\[
| \sum_i \bar{d}_i^2 - \E \sum_i \bar{d}_i^2 | = O_p\left( \sqrt{ \frac{ n^3 }{c_n^2} } \right).
\]
It follows that
\[
| \bs{\bar{d}}_2^\top (W_{22}-\widetilde{W}_{22}) \bs{\bar{d}}_2 |
\lesssim \| W_{22}-\widetilde{W}_{22} \|_{\max} \cdot \bs{\bar{d}}_2^\top \bs{\bar{d}}_2
\lesssim \frac{ b_n^6 }{ c_n^5 n^{3/2} }\cdot \sqrt{ \frac{ n^3 }{c_n^2} } \lesssim \frac{ b_n^6 }{ n^{3/2}c_n^6 }.
\]
\end{proof}

\subsection{Proof of Lemma 12}
\label{section-proof-lemma12}

This section presents the proof of Lemma 12.

\begin{proof}[Proof of Lemma 12]

In (30) and (33), we have shown
\begin{eqnarray}
\label{ineq-2a-hatbeta}
\bs{\widehat{\beta}} - \bs{\beta} & = & V^{-1} \bs{\bar{d}} + V^{-1} \bs{h} \\
\label{ineq-2a-hatbeta-zero}
\bs{\widehat{\beta}}^0_2 - \bs{\beta}_2 & = & V_{22}^{-1} \bs{\bar{d}}_2 + V_{22}^{-1} \bs{\tilde{h}}_2,
\end{eqnarray}
where
\[
\| V^{-1} \bs{h} \|_\infty \lesssim \frac{ b_n^3\log n }{ nc_n},~~ \| V_{22}^{-1} \bs{\tilde{h}}_2 \|_\infty \lesssim \frac{ b_n^3\log n }{ nc_n}.
\]
Subtracting both sides in \eqref{ineq-2a-hatbeta-zero} from \eqref{ineq-2a-hatbeta} over $i=r+1, \ldots, n$ yields
\begin{eqnarray*}
\widehat{\beta}_i - \widehat{\beta}_i^0 & =  & (V^{-1} \bs{\bar{d}})_i - (V_{22}^{-1} \bs{\bar{d}}_2)_i + O\left( \frac{ b_n^3 }{ nc_n} \right) \\
& = & \sum_{j=1}^n W_{ij} \bar{d}_j - \sum_{j=r+1}^n (\widetilde{W}_{22})_{(i-r)(j-r)} \bar{d}_j + O\left( \frac{ b_n^3 }{ nc_n} \right),
\end{eqnarray*}
where the second equality is due to $V^{-1}=S+W$ in \eqref{ineq-V-S-appro-upper-b} and $V_{22}^{-1} = S_{22} + \widetilde{W}_{22}$ in
\eqref{ineq-V22-S22-app}.
By \eqref{eq-hatbeta-bb}, we have
\[
\| W \bs{\bar{d}} \|_\infty \lesssim O_p( \frac{ 2.1b_n^2  \sqrt{\log n} }{ nc_n^{3/2} } ),~~
\| \widetilde{W}_{22} \bs{\bar{d}}_2\|_\infty \lesssim O_p( \frac{ 2.1b_n^2  \sqrt{\log n} }{ nc_n^{3/2} } ).
\]
It completes the proof.
\end{proof}

\subsection{The proof of (59) for bounding $B_2-B_2^0$}
\label{subsection:B2B20}

In this section, we prove inequality (59), reproduced below:
\begin{equation}\label{ineq-2a-62}
B_2 - B_2^0 \lesssim \frac{ b_n^9 (\log n)^3 }{ n^{1/2} c_n^3 }.
\end{equation}

Before proving \eqref{ineq-2a-62}, we show one lemma.
Note that in (30), we show
\begin{equation}\label{betabeta}
\widehat{\beta}_i - \beta_i = (V^{-1} \bs{\bar{d}})_i + (V^{-1} \bs{h} )_i,
\end{equation}
where
\[
h_i = \sum_{j\neq i} h_{ij} =  \sum_{j\neq i} \frac{1}{2} \mu^{\prime\prime}( \tilde{\pi}_{ij} ) ( \widehat{\pi}_{ij} - \pi_{ij})^2,
\]
and $\widetilde{\pi}_{ij}$ lies between $\pi_{ij}=\beta_i+\beta_j$ and $\widehat{\pi}_{ij}=\widehat{\beta}_i+\widehat{\beta}_j$.
In (33), we show
\begin{equation}\label{betabeta0}
\widehat{\beta}_i^0 - \beta_i = (V_{22}^{-1} \bs{\bar{d}}_2)_{i-r} + (V_{22}^{-1} \bs{h}^{0})_{i-r}, ~~~i=r+1, \ldots, n,
\end{equation}
where $\bs{h}^0 : = ( h_{r+1}^0, \ldots, h_n^0)^\top$, and
\[
h_i^0 = \sum_{j\neq i} h_{ij}^0 = \sum_{j\neq i}  \frac{1}{2} \mu^{\prime\prime}( \tilde{\pi}_{ij}^0 ) ( \widehat{\pi}_{ij}^0 - \pi_{ij} )^2,~~i=1,\ldots, n.
\]
and $\tilde{\pi}_{ij}^0$ lies between $\widehat{\pi}_{ij}^0 = \widehat{\beta}_i^0 + \widehat{\beta}_j^0$ and $\pi_{ij}=\beta_i+\beta_j$.

\begin{lemma}\label{lemma-gg0}
If $b_n^3/c_n^2=o( (n/\log n)^{1/2})$, then $g_i - g_i^0$, $i=r+1, \ldots, n$ are bounded by
\begin{equation}\label{ineq-ggg0}
\max_{i=r+1, \ldots, n} | g_i - g_i^0 | \lesssim \frac{ b_n^7 (\log n)^{2} }{ n^{3/2} c_n^2 },
\end{equation}
where
\begin{eqnarray}
\nonumber
g_i & = & \widehat{\beta}_i - \beta_i - \frac{ \bar{d}_i }{v_{ii}} = (W\bs{\bar{d}})_i + (V^{-1}\bs{h})_i, \\
\nonumber
g_i^0 & = & \widehat{\beta}_i^0- \beta_i - \frac{\bar{d}_i }{v_{ii}} = (\widetilde{W}_{22}\bs{\bar{d}}_2)_{i-r} + (V_{22}^{-1} \bs{h}^{0} )_{i-r}.
\end{eqnarray}
\end{lemma}

\begin{proof}
If $b_n^3/c_n^2=o( (n/\log n)^{1/2})$, then with probability at least $1-O(n^{-1})$, we have
\eqref{betabeta} and \eqref{betabeta0}. Note that
\begin{equation}
\label{eq-gg0}
g_i - g_i^0  =  (W\bs{\bar{d}})_i - (\widetilde{W}_{22}\bs{\bar{d}}_2)_{i-r} +
(V^{-1} \bs{h} )_i - (V_{22}^{-1} \bs{h}^{0} )_{i-r}.
\end{equation}
We first bound $(W\bs{\bar{d}})_i - (\widetilde{W}_{22}\bs{\bar{d}}_2)_{i-r}$ over $i=r+1, \ldots, n$.
This term can be represented as
\begin{equation}\label{eq-WW22}
(W\bs{\bar{d}})_i - (\widetilde{W}_{22}\bs{\bar{d}}_2)_{i-r} = \sum_{j=1}^r w_{ij} \bar{d}_j  + \sum_{j=1}^{n-r} (W_{22} - \widetilde{W}_{22})_{i-r, j} \bar{d}_{j+r}.
\end{equation}
Because $r$ is a fixed constant and
\[
\sum_{j=1}^r w_{ij} \bar{d}_j = \sum_{1\le i<j \le r} ( w_{ij}\bar{a}_{ij}+w_{ji}\bar{a}_{ji} )
+ \sum_{i=1}^r \sum_{j=r+1}^n w_{ij}\bar{a}_{ij},
\]
by Bernstein's inequality, with probability at least $1-O(n^{-2})$, the first term in the above right-hand side satisfies
\begin{equation}\label{ineq-ww-aa1}
|\sum_{j=1}^r w_{ij} \bar{d}_j| \lesssim  \frac{ b_n^3 }{ n^2 c_n^2 } \sqrt{n\log n} \lesssim \frac{ b_n^3 (\log n)^{1/2} }{ n^{3/2} c_n^2 }.
\end{equation}
For the second term of the right hand side in \eqref{eq-WW22}, we use Benstern's inequality to bound it. Note that
\begin{align*}
   & \sum_{j=1}^{n-r} (W_{22} - \widetilde{W}_{22})_{i-r, j} \bar{d}_{j+r} \\
  = &  2\sum_{1\le k < j \le n-r} \{(W_{22} - \widetilde{W}_{22})_{ik} \}\bar{a}_{k+r,j+r}
 + \sum_{j=1}^{n-r} \sum_{k=1}^r (W_{22} - \widetilde{W}_{22})_{ij} \bar{a}_{j+r,k}
\end{align*}
Because the terms involved in the sum are independent and bounded,
Benstern's inequality in Lemma \ref{lemma:bernstein} gives that with probability at least $1- O(n^{-2})$, we have
\begin{eqnarray}\nonumber
|[(W_{22} - \widetilde{W}_{22}) \bs{\bar{d}}_2]_i | & \lesssim & q_n \sqrt{ 4 \log n \cdot \frac{n(n-r)}{c_n}}
+ \frac{4}{3} q_n \log n \\
\nonumber
& \lesssim & \frac{b_n^6}{n^3 c_n^5} \cdot \frac{n\log n}{ c_n^{1/2} } \\
\label{ew-WWd}
& \lesssim & \frac{ b_n^6 (\log n)^{1/2} }{ n^2 c_n^{11/2} },
\end{eqnarray}
where
\[
\E \left( \sum_{j=1}^{n-r} (W_{22} - \widetilde{W}_{22})_{i-r, j} \bar{d}_{j+r}  \right)^2 \lesssim \frac{ (n-r)(n-r-1)+r(n-r) q_n^2}{ c_n},
\]
and
\[
q_n =\| W_{22} - \widetilde{W}_{22} \|_{\max} \lesssim \frac{ b_n^6 }{ n^2 c_n}.
\]
By combining \eqref{eq-WW22}, \eqref{ineq-ww-aa1} and \eqref{ew-WWd}, with probability at least $1-O(n^{-2})$, we have
\begin{equation}\label{ineq-W22-2}
\max_{i=r+1,\ldots, n} |(W\bs{\bar{d}})_i - (\widetilde{W}_{22}\bs{\bar{d}}_2)_{i-r} | \lesssim \frac{ b_n^3 (\log n)^{1/2} }{ n^{3/2} c_n^2 } + \frac{ b_n^6 (\log n)^{1/2} }{ n^2 c_n^{11/2} }.
\end{equation}

Now, we bound the second term $(V^{-1}\bs{h})_{i} - (V_{22}^{-1} \bs{h}^{0} )_{i-r}$ in \eqref{eq-gg0}. For $i=r+1, \ldots, n$, observe that
\begin{eqnarray*}
h_i - h_{i}^0 &= & \frac{1}{2} \sum_{j\neq i} [\mu^{\prime\prime} (\tilde{\pi}_{ij})(\widehat{\pi}_{ij} - \pi_{ij} )^2
- \mu^{\prime\prime} (\tilde{\pi}_{ij}^0)(\widehat{\pi}_{ij}^0 - \pi_{ij} )^2 ].
\end{eqnarray*}
With the use of the mean value theorem and Lemma 10, we have
\begin{eqnarray*}
& &| \mu^{\prime\prime} (\tilde{\pi}_{ij})(\widehat{\pi}_{ij} - \pi_{ij} )^2
- \mu^{\prime\prime} (\tilde{\pi}_{ij}^0)(\widehat{\pi}_{ij}^0 - \pi_{ij} )^2 | \\
&\le & | \mu^{\prime\prime} (\tilde{\pi}_{ij}) - \mu^{\prime\prime} (\tilde{\pi}_{ij}^0) |(\widehat{\pi}_{ij} - \pi_{ij} )^2
+ |\mu^{\prime\prime} (\tilde{\pi}_{ij}^0)[(\widehat{\pi}_{ij} - \pi_{ij} )^2-(\widehat{\pi}_{ij}^0 - \pi_{ij} )^2] |
\\
& \le & \left( 4| \mu^{\prime\prime\prime}( \dot{\pi}_{ij})| \| \bs{\widehat{\beta}} - \bs{\beta}\|_\infty^2 +  |\mu^{\prime\prime} (\tilde{\pi}_{ij}^0)|\widehat{\pi}_{ij} - \widehat{\pi}_{ij}^0 | \right)
\times ( | \widehat{\pi}_{ij}^0 - \pi_{ij}| +  |\widehat{\pi}_{ij} - \pi_{ij}| ) \\
&\lesssim & \frac{1}{c_n} \left( b_n \sqrt{\frac{\log n}{n}} \right)^3 + \frac{1}{c_n} \cdot b_n \sqrt{\frac{\log n}{n}} \cdot \frac{ b_n^3 \log n}{ nc_n }\\
& \lesssim & \frac{ b_n^4 (\log n)^{3/2} }{ n^{3/2} c_n^2 }.
\end{eqnarray*}
This gives
\begin{equation}\label{eq-RR0}
|h_i - h_i^0 | \lesssim \frac{ b_n^4 (\log n)^{3/2} }{ n^{1/2} c_n^2 }.
\end{equation}
It follows that
\begin{equation}\label{eq-srr0}
(S\bs{h})_{i} - (S_{22} \bs{h}^0)_{i-r} \le \frac{ | h_i - h_i^0 |}{v_{ii}} \lesssim \frac{ b_n^5 (\log n)^{2}}{n^{3/2}c_n^2}.
\end{equation}
By Lemma 5 and \eqref{eq-RR0}, we have
\begin{eqnarray*}
(W \bs{h} )_{i} - (\tilde{W}_{22} \bs{h}^0)_i & = & \sum_{j=1}^r w_{ij} h_j + \sum_{j=r+1}^n ( w_{ij}h_j - (\widetilde{W}_{22})_{i-r,j-r} h_j^0 ) \\
& = & \sum_{j=1}^r w_{ij} h_j + \sum_{j=r+1}^n ( w_{ij}h_j - w_{ij}h_j^0 + w_{ij}h_j^0 - (\widetilde{W}_{22})_{i-r,j-r} h_j^0 ).
\end{eqnarray*}
Because $r$ is a fixed constant, we have
\[
| \sum_{j=1}^r w_{ij} h_j | \lesssim \frac{ b_n^3 }{ n^2 c_n^2 } \cdot b_n^2 \log n,
\]
\[
|\sum_{j=r+1}^n ( w_{ij}h_j - w_{ij}h_j^0| \lesssim n \cdot \frac{ b_n^3 }{ n^2 c_n^2 } \cdot \frac{ b_n^4 (\log n)^{3/2} }{ n^{1/2} c_n^2 }
\lesssim \frac{ b_n^7 (\log n)^{3/2} }{ n^{3/2} c_n^4 },
\]
\[
\sum_{j=r+1}^n ( w_{ij}h_j^0 - (\widetilde{W}_{22})_{i-r,j-r} h_j^0 )
\lesssim n \cdot \frac{ b_n^6 }{ n^3 c_n^5 } \cdot b_n^2 \log n \lesssim \frac{ b_n^6 \log n}{ n^2 c_n^5 },
\]
and
\begin{eqnarray*}
|(W \bs{h} )_{i} - (\tilde{W}_{22} \bs{h}^0)_i | & \lesssim & \| W_{22}- \widetilde{W}_{22} \|_{\max} \sum_i |R_i - R_i^0 | \\
&\lesssim & \frac{ b_n^6 }{ n^3 c_n} \cdot \frac{ b_n^5 (\log n)^2 }{ n c_n } \cdot n \\
& \lesssim & \frac{ b_n^{11} (\log n)^{2} }{ n^{3}c_n^2}.
\end{eqnarray*}
Consequently,
\[
|(W \bs{h} )_{i} - (\tilde{W}_{22} \bs{h}^0)_i | \lesssim \frac{ b_n^7 (\log n)^{3/2} }{ n^{3/2} c_n^4 }.
\]
By combining \eqref{eq-srr0} and the above inequality, it yields
\begin{equation}\label{ineq-vrr}
|(V^{-1}\bs{h})_i - (V_{22}^{-1} \bs{h}^{0} )_i| \lesssim \frac{ b_n^5 (\log n)^{2}}{n^{3/2}c_n^2} + \frac{ b_n^7 (\log n)^{3/2} }{ n^{3/2} c_n^4 }.
\end{equation}

Because
\begin{eqnarray*}
(V^{-1} \bs{h} )_i - (V_{22}^{-1} \bs{h}^0)_{i-r} & = & (S \bs{h})_i + (W\bs{h})_i - (S_{22} \bs{h}^0 )_{i-r} - ( \widetilde{W}_{22} \bs{h}^0 )_{i-r}, %\\
%& = & \frac{ | h_i - h_i^0 | }{ v_{ii} } + \sum_{j=1}^r w_{ij} h_j  + \sum_{j=r+1}^n ( w_{ij} h_j - (\widetilde{W}_{22})_{(i-r),(j-r)} h_{j}^0) \\
%& \le & \frac{ | h_i - h_i^0 | }{ v_{ii} } + \sum_{j=1}^r w_{ij} h_j  + \sum_{j=r+1}^n ( w_{ij} h_j - w_{ij} h_j^0 ) + (w_{ij}- (\widetilde{W}_{22})_{(i-r),(j-r)}) h_{j}^0) \\
%& \lesssim &
\end{eqnarray*}
in view of \eqref{eq-gg0}, \eqref{ineq-vrr} and \eqref{ineq-W22-2}, it yields
\begin{equation}\label{ineq-ggg0}
\max_{i=r+1, \ldots, n} | g_i - g_i^0 | \lesssim \frac{ b_n^7 (\log n)^{2} }{ n^{3/2} c_n^2 }.
\end{equation}
\end{proof}

Now, we are ready to prove \eqref{ineq-2a-62}.

\begin{proof}[Proof of \eqref{ineq-2a-62}]
$B_2 - B_2^0$ can be written as
\begin{eqnarray}
\nonumber
B_2 - B_2^0 & = & \underbrace{ \sum_{i=1}^r  (\widehat{\beta}_i-\beta_i)^3 \sum_{j=1,j\neq i}^n \mu^{\prime\prime}( \pi_{ij}^0 ) }_{T_1}
+ \underbrace{ \sum_{i,j=1, j\neq i}^r  (\widehat{\beta}_i-\beta_i)^2(\widehat{\beta}_j-\beta_j)\mu^{\prime\prime}( \pi_{ij}^0 ) }_{T_2} \\
\nonumber
&& + \underbrace{\left( \sum_{i=1}^r \sum_{j=1, j\neq i}^n + \sum_{j=1}^r \sum_{i=1,i\neq j}^n \right)  (\widehat{\beta}_i-\beta_i)^2(\widehat{\beta}_j-\beta_j)\mu^{\prime\prime}( \pi_{ij}^0 )}_{T_3} \\
\nonumber
&& + \underbrace{ \sum_{i=r}^n  \{ (\widehat{\beta}_i-\beta_i)^3 - (\widehat{\beta}_i^0-\beta_i)^3 \} \sum_{j\neq i} \mu^{\prime\prime}( \beta_i^0 + \beta_j^0 ) }_{T_4} \\
\label{ineq-B2-B20}
&& + \underbrace{ \sum_{i,j=r, j\neq i}^n  \left\{(\widehat{\beta}_i-\beta_i)^2(\widehat{\beta}_j-\beta_j)-(\widehat{\beta}_i^0-\beta_i)^2
(\widehat{\beta}_j^0-\beta_j)\right\} \mu^{\prime\prime}( \pi_{ij}^0 ) }_{T_5}.
\end{eqnarray}
Because $r$ is a fixed constant, the first three terms in the expression of $B_2 - B_2^0$ can be easily bounded by
\[
|T_1+T_2+T_3|\le \frac{1}{c_n} \| \bs{\widehat{\beta}} - \bs{\beta} \|_\infty \left( r(n-1) + r(r-1) + 2r(n-1) \right)
\lesssim \frac{ b_n^3(\log n)^{3/2} }{ n^{1/2} c_n }.
\]
We now bound $T_4$. In view of \eqref{ineq-ggg0}, we have
\begin{eqnarray*}
& & | (\widehat{\beta}_i-\beta_i)^3 - (\widehat{\beta}_i^0 -\beta_i)^3 | \\
&=& | ( \frac{\bar{d}_i}{v_{ii}} + g_i )^3 - ( \frac{\bar{d}_i}{v_{ii}} + g_i^0 )^3 | \\
& \le & 3 \frac{\bar{d}_i^2 }{ v_{ii}^2 }|g_i-g_i^0| + 3 \frac{ |\bar{d}_i| }{ v_{ii} }|g_i^2-(g_i^0)^2| +  |g_i^3 - (g_i^0)^3| \\
& \le &  \frac{ b_n^7 (\log n)^2 }{ n^{3/2} c_n^2 } \times \left(
\frac{ b_n^2 \log n}{ n} + \frac{ b_n^4 (\log n)^{3/2} }{ n^{3/2} } + \frac{ b_n^6 (\log n)^2 }{ n^2 c_n^2 } \right),~~i=r+1, \ldots, n,
\end{eqnarray*}
and
\begin{eqnarray*}
&&(\widehat{\beta}_i-\beta_i^0)^2(\widehat{\beta}_j-\beta_j^0) -  (\widehat{\beta}_i^0-\beta_i^0)^2(\widehat{\beta}_j^0-\beta_j^0) \\
& \le & \frac{ \bar{d}_i^2 }{ v_{ii}^2 } | g_j - g_j^0 |
+ 2 \frac{ \bar{d}_i \bar{d}_j }{ v_{ii} v_{jj} } |g_i - g_i^0 |
+ 2\frac{ \bar{d}_i }{ v_{ii} } | g_i g_j - g_i^0 g_j^0 |
+ \frac{ \bar{d}_j }{ v_{jj}} |g_i^2- (g_i^0)^2 | + |g_i^2 g_j - (g_i^0)^2 g_j^0 |
\\
& \lesssim & \frac{ b_n^7(\log n)^2 }{ n^{3/2} c_n^2 } \left\{
\frac{ b_n^2 (n\log n) }{ n^2 } + \frac{ b_n^3\log n}{nc_n} \cdot \frac{ b_n(n\log n)^{1/2} }{ n }
+  (\frac{ b_n^3\log n}{nc_n})^2
\right\},
\end{eqnarray*}
where the last inequality for terms $g_ig_j-g_i^0g_j^0$  and $g_i^2g_j - (g_i^0)^2 g_j^0$ are due to that
\[
|g_ig_j-g_i^0g_j^0| \le |g_i||g_j- g_j^0| + |g_j^0||g_i-g_i^0| \lesssim \frac{ b_n^3 \log n}{ nc_n} \max_{i=r+1,\ldots, n} | g_i - g_i^0 |,
\]
and
\[
| g_i^2 g_j - (g_i^0)^2 g_j^0| \le  g_i^2 | g_j-g_j^0| + |(g_i^2 - (g_i^0)^2)| |g_j^0|
\le (g_i^2+ |g_ig_j^0|+ | g_i^0 g_j^0|) \max_{i=r+1,\ldots,n} |g_i - g_i^0|.
\]
Therefore, $T_4$ and $T_5$ can be bounded by
\[
T_4 \lesssim \frac{ (n-r)n }{ c_n} \frac{ b_n^7 (\log n)^2 }{ n^{3/2} c_n^2 } \cdot b_n^2 \left( \frac{\log n}{n} \right)
\lesssim \frac{ b_n^9 (\log n)^3 }{ n^{1/2} c_n^3 }
\]
and
\[
T_5 \lesssim \frac{ (n-r)^2 }{ c_n } \cdot \frac{ b_n^7 (\log n)^2 }{ n^{3/2} c_n^2 } \cdot \frac{b_n^4(\log n)^2 }{ n^{3/2}c_n}
\lesssim  \frac{ b_n^{11} (\log n)^4 }{ nc_n^4 }.
\]
By combining inequalities for $T_i, i=1,\ldots, 5$, it yields \eqref{ineq-2a-62}.
\iffalse
\begin{equation}\label{ineq-B2-B20}
|B_2 - B_2^0|
\lesssim  \frac{ b_n^9 (\log n)^3 }{ n^{1/2} c_n^3 }.
\end{equation}
\fi
\end{proof}

\subsection{The upper bound of $B_3 - B_3^0$ in the proof of Theorem 2 (a)}
\label{subsection-B3B30}

In this section, we present the proof of the error bound of $B_3-B_3^0$ in (60) in the main text,
reproduced below:
\begin{equation}\label{ineq-2a-B330}
|B_3 - B_3^0| \lesssim \frac{b_n^6 (\log n)^{3} }{ n^{1/2}c_n}.
\end{equation}

\begin{proof}[Proof of \eqref{ineq-2a-B330}]
Because
\[
\frac{ \partial^4 \ell( \bs{\beta} )}{ \partial \beta_t \partial \beta_i \partial \beta_j \partial \beta_k } = 0 \mbox{~if there are at least three different indices among $i,j,t,k$},
\]
we have
\begin{equation}
\begin{array}{rcl}
B_3^0 & = &  \sum\limits_{i=1}^n  (\widehat{\beta}_i^0 -\beta_i)^3 \sum\limits_{j\neq i} \left[ \mu^{\prime\prime\prime}( \tilde{\pi}_{ij}^0 )
\left\{(\widehat{\beta}_i^0 -\beta_i)+4(\widehat{\beta}_j^0 -\beta_j)\right\} \right] \\
&& + 6 \sum\limits_{i=1}^n \sum\limits_{j=1, j\neq i}^n  (\widehat{\beta}_i^0-\beta_i)^2(\widehat{\beta}_j^0-\beta_j)^2
\mu^{\prime\prime\prime}( \tilde{\pi}_{ij}^0),
\end{array}
\end{equation}
and
\begin{equation}
\begin{array}{rcl}
B_3 & = &  \sum\limits_{i=1}^n  (\widehat{\beta}_i -\beta_i)^3 \sum\limits_{j\neq i} \left[ \mu^{\prime\prime\prime}( \tilde{\pi}_{ij} )
\left\{(\widehat{\beta}_i -\beta_i)+4(\widehat{\beta}_j -\beta_j)\right\} \right] \\
&& + 6 \sum\limits_{i=1}^n \sum\limits_{j=1, j\neq i}^n  (\widehat{\beta}_i - \beta_i)^2(\widehat{\beta}_j-\beta_j)^2
\mu^{\prime\prime\prime}( \tilde{\pi}_{ij}),
\end{array}
\end{equation}
where $\tilde{\pi}_{ij}=\tilde{\beta}_i + \tilde{\beta}_j$ and $\tilde{\pi}^0_{ij} = \tilde{\beta}_i^0 + \tilde{\beta}_j^0$.
Note that $\widehat{\beta}_i^0 -\beta_i=0$ over $i=1,\ldots, r$.
The difference between $B_3$ and $B_3^0$ can be expressed as the sum of the following four terms:
\begin{eqnarray*}
&&B_3 - B_3^0 \\
& = & \underbrace{\left( \sum\limits_{i=1}^r \sum\limits_{j=1,j\neq i}^n + \sum_{i=r+1}^n \sum_{j=1}^r \right) (\widehat{\beta}_i -\beta_i)^3  \left[ \mu^{\prime\prime\prime}( \tilde{\pi}_{ij} )
\left\{(\widehat{\beta}_i -\beta_i)+4(\widehat{\beta}_j -\beta_j)\right\} \right]}_{C_1} \\
&& + \underbrace{\left( 6 \sum\limits_{i=1}^r \sum\limits_{j=1, j\neq i}^n + 6\sum\limits_{i=r+1}^n \sum\limits_{j=1}^r \right)  (\widehat{\beta}_i - \beta_i)^2(\widehat{\beta}_j-\beta_j)^2
\mu^{\prime\prime\prime}( \tilde{\pi}_{ij})}_{C_2} \\
&& + \underbrace{\sum\limits_{i=r+1}^n \left\{ (\widehat{\beta}_i -\beta_i)^4 \sum\limits_{j=r+1,j\neq i}^n \mu^{\prime\prime\prime}( \tilde{\pi}_{ij} )
- (\widehat{\beta}_i^0 -\beta_i)^4 \sum\limits_{j=r+1,j\neq i}^n \mu^{\prime\prime\prime}( \tilde{\pi}_{ij}^0 ) \right\}}_{C_3}
\\
&&+ \underbrace{4\sum\limits_{i=r+1}^n \sum\limits_{j=r+1,j\neq i}^n \left\{ (\widehat{\beta}_i -\beta_i)^3  \mu^{\prime\prime\prime}( \tilde{\pi}_{ij} )(\widehat{\beta}_j -\beta_j)
- (\widehat{\beta}_i^0 - \beta_i)^3  \mu^{\prime\prime\prime}( \tilde{\pi}_{ij}^0 )(\widehat{\beta}_j^0 -\beta_j)\right\}}_{C_4}
 \\
&&+ \underbrace{6 \sum\limits_{i=r+1}^n \sum\limits_{j=r+1, j\neq i}^n \left\{ (\widehat{\beta}_i - \beta_i)^2(\widehat{\beta}_j-\beta_j)^2
\mu^{\prime\prime\prime}( \tilde{\pi}_{ij})- (\widehat{\beta}_i^0 - \beta_i)^2(\widehat{\beta}_j^0-\beta_j)^2
\mu^{\prime\prime\prime}( \tilde{\pi}_{ij}^0 ) \right\}}_{C_5}.
\end{eqnarray*}
We shall evaluate the above four terms in turn. Notice that $r$ is a fixed constant.
By Lemma , with probability at least, the upper bounds of $C_1$ and $C_2$ satisfies
\begin{equation}\label{remain-C1}
|C_1| \lesssim \frac{rn}{c_n} \| \bs{\widehat{\beta}} - \bs{\beta}\|_\infty^4 \lesssim  \frac{n}{c_n} \left( b_n \sqrt{\frac{\log n}{n}} \right)^4 \lesssim \frac{b_n^4(\log n)^2 }{ nc_n},
\end{equation}
and
\begin{equation}\label{remain-C2}
|C_2| \lesssim \frac{rn}{c_n} \| \bs{\widehat{\beta}} - \bs{\beta}\|_\infty^4 \lesssim  \frac{n}{c_n} \left( b_n \sqrt{\frac{\log n}{n}} \right)^4 \lesssim \frac{b_n^4(\log n)^2 }{ nc_n}.
\end{equation}

Before bounding $C_3$, $C_4$ and $C_5$,
we drive one useful inequality. By finding the fourth derivative of $\mu(x)$ with respect to $x$, we have
\begin{eqnarray*}
\mu^{\prime\prime\prime\prime}(x) & = & \frac{ e^x( 1- 8 e^x + 3e^{2x} ) }{ ( 1 + e^x )^4 } - \frac{ 4 e^{2x} ( 1 - 4 e^x + 4e^{2x} )}{ ( 1+ e^x)^5 } \\
& = & \frac{ e^x ( 1 -11e^x + 11 e^{2x} - e^{3x})  }{ ( 1 + e^x )^5 } \\
& = & \frac{ e^x }{ ( 1+e^x)^2 } \cdot \frac{ 1 -11e^x + 11 e^{2x} - e^{3x} }{ (1+e^x)^3 }.
\end{eqnarray*}
It is easy to see that
\[
\frac{ 11 }{ 3} (1+e^x)^3 \ge 1  + 11e^x + 11e^{2x} + e^{3x} \ge | 1 -11e^x + 11 e^{2x} - e^{3x} |.
\]
It follows that
\[
|\mu^{\prime\prime\prime\prime}(x)| \le \frac{ 11e^x }{ 3( 1+e^x)^2 }.
\]
Therefore, for any $\dot{\pi}_{ij}$ satisfying $|\dot{\pi}_{ij} - \pi_{ij}|\to 0$, we have
\begin{equation}\label{eq-fourth-mu}
| \mu^{\prime\prime\prime\prime}(\dot{\pi}_{ij} )| \le \frac{ 11\mu^{\prime}( \dot{\pi}_{ij} ) }{ 3 } \lesssim \mu^{\prime}( \pi_{ij} ) \lesssim \frac{1}{c_n}
\end{equation}
It follows from the mean value theorem that for any $\dot{\pi}_{ij}$ satisfying $|\dot{\pi}_{ij} - \pi_{ij}|\to 0$,
\begin{equation}\label{ineq-mu-thr-diff}
| \mu^{\prime\prime\prime}( \dot{\pi}_{ij} ) - \mu^{\prime\prime\prime}( \pi_{ij} ) | \lesssim \frac{1}{c_n} |\dot{\pi}_{ij}- \pi_{ij} |.
\end{equation}

By Lemmas 3 and 10, for $i=r+1, \ldots, n$,  we have
\begin{eqnarray}
\nonumber
&&\left|(\widehat{\beta}_i -\beta_i)^4\mu^{\prime\prime\prime}( \tilde{\pi}_{ij} )-
(\widehat{\beta}_i^0 -\beta_i)^4\mu^{\prime\prime\prime}( \tilde{\pi}_{ij}^0) \right| \\
\nonumber
& \le & \left| (\widehat{\beta}_i -\beta_i)^4 \left\{\mu^{\prime\prime\prime}( \tilde{\pi}_{ij} )
-  \mu^{\prime\prime\prime}( \tilde{\pi}_{ij}^0 )\right\} \right| + \left| \left\{(\widehat{\beta}_i -\beta_i)^4
- (\widehat{\beta}_i^0 -\beta_i)^4 \right\} \mu^{\prime\prime\prime}( \tilde{\pi}_{ij}^0 ) \right| \\
\nonumber
& \lesssim & \frac{1}{c_n} \left( | \widehat{\beta}_i - \beta_i |^4 \cdot ( | \widehat{\beta}_i - \beta_i | + |\widehat{\beta}_j - \beta_j | ) \right.\\
\nonumber
&& \left.+ | \widehat{\beta}_i - \widehat{\beta}_i^0|^2  \cdot ( | \widehat{\beta}_i - \beta_i |^2 + | \widehat{\beta}_i^0 - \beta_i|^2 ) \right)\\
\nonumber
& \lesssim & \frac{ 1 }{ c_n }\cdot \left( b_n \sqrt{\frac{\log n}{n}} \right)^5 + \frac{1}{c_n}\left(\frac{ b_n^3 \log n }{ n c_n }\right)^2 \cdot \left( b_n \sqrt{\frac{\log n}{n}} \right)^2 \\
\label{ineq-fourth-C1}
& \lesssim &  \frac{b_n^5(\log n)^{5/2} }{ n^{5/2}c_n} + \frac{ b_n^8 (\log n)^3 }{ n^3 c_n^3},
\end{eqnarray}
where the second inequality is due to \eqref{ineq-mu-thr-diff}.
Similarly, for $i,j=r+1, \ldots, n$, $i\neq j$, we have
\begin{eqnarray}
\nonumber
&&\left| (\widehat{\beta}_i -\beta_i)^3(\widehat{\beta}_j -\beta_j)\mu^{\prime\prime\prime}( \tilde{\pi}_{ij} )
 -(\widehat{\beta}_i^0 -\beta_i)^3(\widehat{\beta}_j^0 -\beta_j)\mu^{\prime\prime\prime}( \tilde{\pi}_{ij}^0 ) \right| \\
\nonumber
& \le & \underbrace{\left| (\widehat{\beta}_i -\beta_i)^3(\widehat{\beta}_j -\beta_j)\mu^{\prime\prime\prime}( \tilde{\pi}_{ij} )
-  (\widehat{\beta}_i^0 -\beta_i)^3(\widehat{\beta}_j^0 -\beta_j)\mu^{\prime\prime\prime}( \tilde{\pi}_{ij} ) \right|}_{ E_1} \\
\nonumber
&&+  \left|(\widehat{\beta}_i^0 -\beta_i)^3(\widehat{\beta}_j^0 -\beta_j)\mu^{\prime\prime\prime}( \tilde{\pi}_{ij} )
 -(\widehat{\beta}_i^0 -\beta_i)^3(\widehat{\beta}_j^0 -\beta_j)\mu^{\prime\prime\prime}( \tilde{\pi}_{ij}^0 ) \right|\\
\label{ineq-fourth-C2}
 & \lesssim & \frac{ 1 }{ c_n }\cdot \left( b_n \sqrt{\frac{\log n}{n}} \right)^5 + \frac{1}{c_n}\left( b_n \sqrt{\frac{\log n}{n}} \right)^3 \cdot \left(\frac{ b_n^3 \log n }{ n c_n }\right),
\end{eqnarray}
where the second inequality for $E_1$ follows from
\begin{eqnarray*}
&&|(\widehat{\beta}_i -\beta_i)^3(\widehat{\beta}_j -\beta_j) - (\widehat{\beta}_i^0 -\beta_i)^3(\widehat{\beta}_j^0 -\beta_j)| \\
& \le & |(\widehat{\beta}_i -\beta_i)^3(\widehat{\beta}_j -\beta_j) - (\widehat{\beta}_i^0 -\beta_i)^3(\widehat{\beta}_j -\beta_j)| \\
&& + |(\widehat{\beta}_i^0 -\beta_i)^3(\widehat{\beta}_j -\beta_j) - (\widehat{\beta}_i^0 -\beta_i)^3(\widehat{\beta}_j^0 -\beta_j)| \\
& \lesssim & | \widehat{\beta}_i - \widehat{\beta}_i^0 |\{ (\widehat{\beta}_i -\beta_i)^2 + (\widehat{\beta}_i^0 -\beta_i)^2 \} | \widehat{\beta}_j -\beta_j | \\
&& + |(\widehat{\beta}_i^0 -\beta_i)^3|\cdot |\widehat{\beta}_j - \widehat{\beta}_j^0| \\
& \lesssim & \left( b_n \sqrt{\frac{\log n}{n}} \right)^3 \cdot \left(\frac{ b_n^3 \log n }{ n c_n }\right).
\end{eqnarray*}
Again, for $i\neq j, i,j=r+1,\ldots, n$, we have
\begin{eqnarray}
\nonumber
&&  (\widehat{\beta}_i^0-\beta_i)^2(\widehat{\beta}_j^0-\beta_j)^2\mu^{\prime\prime\prime}( \tilde{\pi}_{ij}^0)
-(\widehat{\beta}_i - \beta_i)^2(\widehat{\beta}_j-\beta_j)^2\mu^{\prime\prime\prime}( \tilde{\pi}_{ij})  \\
\nonumber
& \le & \underbrace{| (\widehat{\beta}_i^0-\beta_i)^2(\widehat{\beta}_j^0-\beta_j)^2\mu^{\prime\prime\prime}( \tilde{\pi}_{ij}^0)
-(\widehat{\beta}_i - \beta_i)^2(\widehat{\beta}_j-\beta_j)^2\mu^{\prime\prime\prime}( \tilde{\pi}_{ij}^0) |}_{E_2} \\
\nonumber
&& + |(\widehat{\beta}_i-\beta_i)^2(\widehat{\beta}_j-\beta_j)^2\mu^{\prime\prime\prime}( \tilde{\pi}_{ij}^0)
-(\widehat{\beta}_i - \beta_i)^2(\widehat{\beta}_j-\beta_j)^2\mu^{\prime\prime\prime}( \tilde{\pi}_{ij}) \\
\label{ineq-fourth-C3}
& \lesssim & \frac{ 1 }{ c_n }\cdot \left( b_n \sqrt{\frac{\log n}{n}} \right)^5 + \frac{1}{c_n}\left( b_n \sqrt{\frac{\log n}{n}} \right)^3 \cdot \left(\frac{ b_n^3 \log n }{ n c_n }\right),
\end{eqnarray}
where the inequality for $E_2$ follows from
\begin{eqnarray*}
 & & | (\widehat{\beta}_i^0-\beta_i)^2(\widehat{\beta}_j^0-\beta_j)^2 - (\widehat{\beta}_i - \beta_i)^2(\widehat{\beta}_j-\beta_j)^2 | \\
& \le & | (\widehat{\beta}_i^0-\beta_i)^2(\widehat{\beta}_j^0-\beta_j)^2 - (\widehat{\beta}_i - \beta_i)^2(\widehat{\beta}_j^0-\beta_j)^2 | \\
& & + |(\widehat{\beta}_i-\beta_i)^2(\widehat{\beta}_j^0-\beta_j)^2 - (\widehat{\beta}_i - \beta_i)^2(\widehat{\beta}_j-\beta_j)^2 | \\
& \lesssim & | \widehat{\beta}_j^0-\beta_j |^2 | \widehat{\beta}_i -\widehat{\beta}_i^0 | ( |\widehat{\beta}_i^0-\beta_i| + | \widehat{\beta}_i-\beta_i | ) \\
&& + (\widehat{\beta}_i-\beta_i)^2 |\widehat{\beta}_j^0 - \widehat{\beta}_j| (|\widehat{\beta}_j^0-\beta_j|+ |\widehat{\beta}_j-\beta_j | ) \\
& \lesssim & \left( b_n \sqrt{\frac{\log n}{n}} \right)^3 \cdot \left(\frac{ b_n^3 \log n }{ n c_n }\right).
\end{eqnarray*}

By \eqref{ineq-fourth-C1}, we have
\[
|C_3| \lesssim (n-r)^2\cdot \left( \frac{b_n^5(\log n)^{5/2} }{ n^{5/2}c_n} + \frac{ b_n^8 (\log n)^3 }{ n^3 c_n^3} \right) \lesssim \frac{b_n^5(\log n)^{3} }{ n^{1/2}c_n}.
\]
By \eqref{ineq-fourth-C2}, we have
\[
|C_4| \lesssim (n-r)^2 \cdot \left( \frac{b_n^5(\log n)^{5/2} }{ n^{5/2}c_n} +  \frac{ b_n^6 (\log n)^{5/2} }{ n^{5/2} c_n^2 } \right) \lesssim \frac{b_n^6 (\log n)^{5/2} }{ n^{1/2}c_n}.
\]
By \eqref{ineq-fourth-C2}, we have
\[
|C_5| \lesssim (n-r)^2 \cdot \left( \frac{b_n^5(\log n)^{5/2} }{ n^{5/2}c_n} +  \frac{ b_n^6 (\log n)^{5/2} }{ n^{5/2} c_n^2 } \right) \lesssim \frac{b_n^6 (\log n)^{5/2} }{ n^{1/2}c_n}.
\]
By combining the above three inequalities with \eqref{remain-C1} and \eqref{remain-C2}, it yields
\[
|B_3 - B_3^0| \lesssim \frac{b_n^6 (\log n)^{3} }{ n^{1/2}c_n}.
\]
This completes the proof.
\end{proof}

\section{Proof of Lemma 1}
\label{section-lemma1}
This section presents the proof of Lemma 1.

\begin{proof}[Proof of Lemma 1]
The following inequalities will be repeatedly used in the proofs: for any $i\neq j$,
\[
\frac{1}{b_n} \le \E \bar{a}_{ij}^2 \le \frac{1}{c_n},
\]
\[
\frac{1}{b_n} \le |\E \bar{a}_{ij}^3| = p_{ij}(1-p_{ij})|\{ (1-p_{ij})^2 - p_{ij}^2 \}| \le \frac{1}{c_n},
\]
\[
\frac{1}{b_n} \le \E \bar{a}_{ij}^4 = p_{ij}(1-p_{ij})\{ (1-p_{ij})^3 + p_{ij}^3 \} \le \frac{1}{c_n},
\]
where $p_{ij}=\E a_{ij}$. We do not cite them explicitly.

Note that
\begin{eqnarray}\label{eq-lemma1-a}
\sum_{i=1}^r \frac{ (\bar{d}_i^{\, 2} - \E \bar{d}_i^{\, 2}) }{ v_{ii} }  & = &  \sum_{i=1}^r \sum_{j=1, j\neq i}^n \frac{ (\bar{a}_{ij}^2 - \E \bar{a}_{ij}^2 ) }{ v_{ii} }
+ \sum_{i=1}^r \sum_{j=1,j\neq i}^n \sum_{k=1, k\neq i,j}^n \frac{ \bar{a}_{ij} \bar{a}_{ik} }{ v_{ii} }.
\end{eqnarray}
%{\color{red}{
By Lemma \ref{lemma:var:quadra}, we have
\[
\mathrm{Var}( \sum_{i=1}^r \frac{ (\bar{d}_i^{\,2} - \E \bar{d}_i^{\,2}) }{ v_{ii} } )
 = \sum_{i=1}^r \frac{1}{v_{ii}^2 } ( 2 v_{ii}^2 + \sum_{j=1,j\neq i}^n u_{ij} ) + 2 \sum_{1\le i<j \le r} \frac{ u_{ij} }{ v_{ii} v_{jj} },
\]
where $u_{ij} = \mathrm{Var}( \bar{a}_{ij}^2 )$.
Because
\[
        \frac{ (n-1)}{ b_n } \le v_{ii} \le \frac{ (n-1)}{ c_n}
\]
and
\[
u_{ij} = \mathrm{Var}( \bar{a}_{ij}^2 )  \le \E \bar{a}_{ij}^4 = p_{ij}(1-p_{ij})\{(1-p_{ij})^3 + p_{ij}^3 \} \le \frac{1}{c_n},
\]
we have
\[
        0< \frac{1}{r} \sum_{i=1}^r \frac{\sum_{j=1,j\neq i}^n u_{ij}}{v_{ii}^2 }  \le \frac{ b_n^2 }{ (n-1)c_n },~~
      0< \frac{1}{r} \sum_{1\le i<j \le r} \frac{ u_{ij} }{ v_{ii} v_{jj} } \le  \frac{ r b_n^2 }{ (n-1)^2 c_n }.
\]
Therefore, if $b_n^2/c_n = o(n)$, then
\begin{equation}\label{eq-lemma1-b}
        \lim_{r\to\infty} \frac{1}{2r} \mathrm{Var} \left(  \sum_{i=1}^r \frac{ (\bar{d}_i^{\,2} - \E \bar{d}_i^{\,2}) }{ v_{ii} }  \right) = 1.
\end{equation}
%%}}
\iffalse
By Lemma \ref{lemma:var:quadra}, we have
\begin{equation}\label{eq-lemma1-b}
\lim_{r\to\infty} \frac{1}{2r} \mathrm{Var} \left( \sum_{i=1}^r \frac{ (\bar{d}_i^{\,2} - \E \bar{d}_i^{\,2}) }{ v_{ii} } \right)=1.
\end{equation}
\fi
Because $\bar{a}_{ij}$, $i=1, \ldots, r$, $j=i+1,\ldots,n$, are independent, we have
\begin{eqnarray}
\nonumber
\mathrm{Var}\left(\sum_{i=1}^r \sum_{j=1}^n \frac{ \bar{a}_{ij}^2 }{ v_{ii} } \right)
& = & 4\mathrm{Var}\left( \sum_{i=1}^{r-1} \sum_{j=i+1}^{r} \frac{ \bar{a}_{ij}^2 }{ v_{ii} } \right)
+ \sum_{i=1}^r \sum_{j=r+1}^n \mathrm{Var}\left( \frac{ \bar{a}_{ij}^2 }{ v_{ii} } \right) \\
\nonumber
& \le & 4\times \frac{(r-1)r }{2c_n} \times \frac{ b_n^2 }{ (n-1)^2} + r(n-r) \frac{ b_n^2 }{ (n-1)^2c_n} \\
\label{eq-lemma1-c}
& \lesssim & \frac{ r b_n^2 }{ nc_n }.
\end{eqnarray}
Note that $v_{ii} = \sum_{j\neq i} \mathbb{E} ( \bar{a}_{ij}^2 )$. Therefore,
if $b_n^2/c_n = o(n)$, then
\[
\frac{1}{r^{1/2}} \left( \sum_{i=1}^r \sum_{j=1, j\neq i}^n \frac{ \bar{a}_{ij}^2 }{ v_{ii} }- r \right)=o_p(1).
\]
Therefore, by \eqref{eq-lemma1-a} and \eqref{eq-lemma1-b}, it is sufficient to demonstrate
\begin{equation}\label{eq-martingale-clt}
\frac{1}{ (2r)^{1/2} } \sum_{i=1}^r \sum_{j=1,j\neq i}^n \sum_{k=1, k\neq i,j}^n \frac{ \bar{a}_{ij} \bar{a}_{ik} }{ v_{ii} }
~~\stackrel{\mathcal{L}}{\longrightarrow}~~ N(0,1),
\end{equation}
as $r$ goes to infinity.

We shall apply Martingale theory to prove the central limit theorem in \eqref{eq-martingale-clt}.
The martingale sequence is constructed as follows.
Define $\sigma$-fields $\mathcal{F}_i$, $i=1,\ldots, r$, as follows:
\begin{eqnarray*}
\mathcal{F}_1  &  = & \sigma\left\{ \bar{a}_{12}, \bar{a}_{13}, \ldots, \bar{a}_{1n} \right\}, \\
\mathcal{F}_2  &  = & \sigma \big\{ \underbrace{\bar{a}_{12}, \ldots, \bar{a}_{1n}} , \underbrace{\bar{a}_{23}, \ldots, \bar{a}_{2n}} \big\}, \\
  & \vdots & \\
\mathcal{F}_r  & = & \sigma \big\{ \underbrace{\bar{a}_{12}, \ldots, \bar{a}_{1n}} , \underbrace{\bar{a}_{23}, \ldots, \bar{a}_{2n}}, \ldots, \underbrace{\bar{a}_{r,r+1}, \ldots, \bar{a}_{r,n}} \big\},
\end{eqnarray*}
where $\sigma\left\{ X_1, \ldots, X_t\right\}$ denotes the $\sigma$-field generated by random variables $X_1, \ldots, X_t$.
That is, $\mathcal{F}_t$ is the $\sigma$-field generated by elements of the first $t$ rows of the upper triangular matrix of $A$.
For convenience, define $\mathcal{F}_0=\sigma\{ \emptyset, \Omega \}$, where $\E( X|\mathcal{F}_0 )= \E (X) $ for any random variable $X$.
Observe that
\begin{eqnarray*}
 \sum_{i=1}^r \sum_{j=1,j\neq i}^n \sum_{k=1, k\neq i,j}^n \frac{ \bar{a}_{ij} \bar{a}_{ik} }{ v_{ii} }
  =  2 \sum_{i=1}^r \sum_{ \begin{smallmatrix} 1\le j < k \le n\\ j,k \neq i \end{smallmatrix} }
 \frac{\bar{a}_{ij} \bar{a}_{ik}}{v_{ii}}.
\end{eqnarray*}
%%{\color{red}{
For $i=3, \ldots, r$, we divide $\sum_{ \begin{smallmatrix} 1\le j < k \le n\\ j,k \neq i \end{smallmatrix} }
 \bar{a}_{ij} \bar{a}_{ik} $ into two parts:
\begin{eqnarray}
\nonumber
\sum_{ \begin{smallmatrix} 1\le j < k \le n\\ j,k \neq i \end{smallmatrix} }
 \bar{a}_{ij} \bar{a}_{ik} & = &  \sum_{j=1,j\neq i}^{n-1} \bar{a}_{ij} \Big(  \sum_{k=j+1,k\neq i,j}^n \bar{a}_{ik} \Big)
   =   \sum_{j=1}^{i-1} \bar{a}_{ij} \Big(  \sum_{k=j+1}^n \bar{a}_{ik} \Big) + \sum_{j=i+1}^{n-1} \bar{a}_{ij} \Big(  \sum_{k=j+1}^n \bar{a}_{ik} \Big)
  \\
\label{eq-divid-XY}
 & = &
 \underbrace{ \sum_{j=1}^{i-2} \sum_{ k =j+1}^{i-1} \bar{a}_{ij} \bar{a}_{ik} }_{ X_i }
 + \underbrace{ (\sum_{j=1}^{i-1} \bar{a}_{ij}) (\sum_{k=i+1}^n \bar{a}_{ik} ) + \sum_{j=i+1}^{n-1} \sum_{k=j+1}^n \bar{a}_{ij} \bar{a}_{ik} }_{ Y_i }.
\end{eqnarray}
Define
\[
X_1=0,~~ X_2=0, ~~ Y_1 = \sum_{ \begin{smallmatrix} 1\le j < k \le n\\ j,k \neq 1 \end{smallmatrix} }
 \bar{a}_{1j} \bar{a}_{1k},~~
Y_2=\sum_{ \begin{smallmatrix} 1\le j < k \le n\\ j,k \neq 2 \end{smallmatrix} }
 \bar{a}_{2j} \bar{a}_{2k}.
\]
It follows from \eqref{eq-divid-XY} that
\begin{equation}
\label{eq-XY-division}
\sum_{i=1}^r \sum_{ \begin{smallmatrix} 1\le j < k \le n\\ j,k \neq i \end{smallmatrix} }
 \frac{\bar{a}_{ij} \bar{a}_{ik}}{v_{ii}}= \sum_{i=1}^r \frac{X_i}{v_{ii}}  +  \sum_{i=1}^r \frac{Y_i}{v_{ii}}.
\end{equation}
Define $\sigma$-fields $\mathcal{F}_i$, $i=1,\ldots, r$, as follows:
\begin{eqnarray*}
\mathcal{F}_1  &  = & \sigma\left\{ \bar{a}_{12}, \bar{a}_{13}, \ldots, \bar{a}_{1n} \right\}, \\
\mathcal{F}_2  &  = & \sigma \big\{ \underbrace{\bar{a}_{12}, \ldots, \bar{a}_{1n}} , \underbrace{\bar{a}_{23}, \ldots, \bar{a}_{2n}} \big\}, \\
  & \vdots & \\
\mathcal{F}_r  & = & \sigma \big\{ \underbrace{\bar{a}_{12}, \ldots, \bar{a}_{1n}} , \underbrace{\bar{a}_{23}, \ldots, \bar{a}_{2n}}, \ldots, \underbrace{\bar{a}_{r,r+1}, \ldots, \bar{a}_{r,n}} \big\},
\end{eqnarray*}
where $\sigma\left\{ X_1, \ldots, X_t\right\}$ denotes the $\sigma$-field generated by random variables $X_1, \ldots, X_t$.
For convenience, define $\mathcal{F}_0=\sigma\{ \emptyset, \Omega \}$, where $\E( X|\mathcal{F}_0 )= \E (X) $ for any random variable $X$.
Because $\bar{a}_{ij}$, $j=i+1, \ldots, n$, are independent of $\mathcal{F}_{i-1}$ and $\bar{a}_{ij}\in \mathcal{F}_{i-2}, j=1,\ldots, i-1$, according the definition of $Y_i$,
we have
\begin{equation}
\label{ma-definition-Yi}
\E(Y_i | \mathcal{F}_{i-1} ) =0.
\end{equation}
However, $\E( X_i | \mathcal{F}_i ) \neq 0$. Therefore, it is desirable to rearrange the terms in $\sum_i X_i$ such that those terms can be expressed as a martingale difference as well.

In view of that $\bar{a}_{ij}=\bar{a}_{ji}$, $X_i$ defined in \eqref{eq-divid-XY} can be rewritten as
\begin{eqnarray*}
X_i &  = &  \left\{\bar{a}_{i1}(\bar{a}_{i2} + \ldots + \bar{a}_{i,i-1})\right\} + \left\{ \bar{a}_{i2}( \bar{a}_{i3} + \ldots + \bar{a}_{i,i-1}) \right\}
+ \cdots + \bar{a}_{i,i-2} \bar{a}_{i,i-1}
\\
& = &  \left\{ \bar{a}_{1i}(\bar{a}_{2i} + \ldots + \bar{a}_{i-1,i})\right\} + \left\{ \bar{a}_{2i}( \bar{a}_{3i} + \ldots + \bar{a}_{i-1,i}) \right\}
+ \cdots + \bar{a}_{i-2,i} \bar{a}_{i-1,i}
\\
& = & \bar{a}_{2i} \bar{a}_{1i} + \bar{a}_{3i}( \bar{a}_{1i} + \bar{a}_{2i} ) + \bar{a}_{4i}(\bar{a}_{1i} + \bar{a}_{2i} + \bar{a}_{3i} ) + \bar{a}_{i-1,i}( \bar{a}_{1i} + \ldots + \bar{a}_{i-2,i} ),
\end{eqnarray*}
where the terms in the last equation is a rearrangement of those in the second equation.
(In paired comparisons data, we let $\bar{a}_{ji}=-\bar{a}_{ij}$, $j=1,\ldots, i-2$, because $\bar{a}_{ij}+\bar{a}_{ji}=0$.)
It follows that
\begin{eqnarray}
\nonumber
\sum_{i=3}^r \frac{ X_i }{v_{ii} } & = &  \sum_{i=3}^r \frac{ \bar{a}_{2i}\bar{a}_{1i} }{ v_{ii} }  +
\sum_{i=4}^r \frac{ \bar{a}_{3i}(\bar{a}_{1i} + \bar{a}_{2i}) }{ v_{ii} } \\
\label{eq-sum-Xi}
&&
+ \sum_{i=5}^r \frac{ \bar{a}_{4i}( \bar{a}_{1i} + \bar{a}_{2i} + \bar{a}_{3i} ) }{ v_{ii} }
+ \cdots + \frac{ \bar{a}_{r-1,i}( \bar{a}_{1r} + \cdots + \bar{a}_{r-2,r} ) }{ v_{rr} }
\end{eqnarray}
For convenience, define the terms in the above right hand as
\begin{equation}
\label{definition-ZZ}
Z_1 = 0, \qquad  Z_t = \sum_{i=t+1}^r \Big\{ \frac{ \bar{a}_{ti} }{ v_{ii}} \big( \sum_{j=1}^{t-1} \bar{a}_{ji} \big) \Big\},~~ t=2, \ldots, r-1, \qquad  Z_r =0.
\end{equation}
Because $\bar{a}_{tj}$, $j=t+1, \ldots, r$, are independent of $\mathcal{F}_{t-1}$ and $\bar{a}_{ji}\in \mathcal{F}_{t-1}$ for $j=1,\ldots, t-1$ and $i=t+1, \ldots, r$, we have
\[
\E(Z_t | \mathcal{F}_{t-1} ) =0.
\]
This, together with \eqref{ma-definition-Yi}, shows
\[
\E(  Y_t/v_{tt} + Z_t | \mathcal{F}_{t-1} ) =0.
\]
Therefore, $Y_t/v_{tt} + Z_t$, $t=1,\ldots, r$ is a martingale difference.
Combining \eqref{eq-XY-division}, \eqref{eq-sum-Xi} and \eqref{definition-ZZ}, it yields
\begin{equation}
\label{eq-sum-y-z}
\sum_{i=1}^r \sum_{ \begin{smallmatrix} 1\le j < k \le n\\ j,k \neq i \end{smallmatrix} }
 \frac{\bar{a}_{ij} \bar{a}_{ik}}{v_{ii}}= \sum_{i=1}^r \frac{X_i}{v_{ii}}  +  \sum_{i=1}^r \frac{Y_i}{v_{ii}}= \sum_{i=1}^r  Z_i +  \sum_{i=1}^r \frac{Y_i}{v_{ii}} = \sum_{i=1}^r  (Y_i/v_{ii} +  Z_i).
\end{equation}

We shall apply \citeauthor{Brown1971}'s (\citeyear{Brown1971}) martingale limit theorem   to show the asymptotic normality of $\sum_{t=1}^r (Y_t/v_{tt} + Z_t)$ in \eqref{eq-sum-y-z}. %; see Hall and Heyde [9, Corollary 3.1, p.58].
It needs to check two conditions:
\begin{equation}\label{condition-a}
\frac{1}{r} \sum_{t=1}^r \E \left\{ (Y_t/v_{tt} + Z_t )^2 1( | Y_t/v_{tt} + Z_t | > r^{1/2} \epsilon  ) \right\}  \to 0,
\end{equation}
as $n\to\infty$ for each $\epsilon>0$, and
\begin{equation}\label{condition-b}
\frac{1}{2r} \sum_{t=1}^r \E\left\{ ( Y_t/v_{tt} + Z_t )^2 | \mathcal{F}_{t-1} \right\} \to 1 \quad \mbox{in probability},
\end{equation}
as $n\to\infty$. They are shown in two steps below.

{\bf Step 1}. We show \eqref{condition-a}. It is sufficient to demonstrate
\begin{equation}\label{condition-a-check-b}
 \sum_{t=1}^r \P\left(  |  Y_t/v_{tt}  + Z_t  | > r^{1/2} \epsilon \right)   \to 0, \quad r\to\infty,
\end{equation}
and
\begin{equation}\label{condition-a-check}
\frac{1}{r^2} \sum_{t=1}^r \E \{ ( Y_t/v_{tt} + Z_t )^4 \} \to 0, \quad r\to\infty.
\end{equation}
By generalized Chebyshev's inequality, \eqref{condition-a-check} implies \eqref{condition-a-check-b}.
Thus, we only need to show \eqref{condition-a-check}.
By $c_r$-inequality, as in \citeauthor{2010Probability} (\citeyear{2010Probability}, page 97), we have
\[
\E \{ ( Y_t/v_{tt} + Z_t )^4 \le 8 \E( Y_t/v_{tt} )^4 + 8\E  Z_t^4.
\]
This in turn requires us to demonstrate
\begin{equation}\label{condition-a-check2}
\frac{b_n^4}{r^2n^4} \sum_{t=1}^r \E ( Y_t^4 ) \to 0,   \quad  \mbox{and} \quad \frac{1}{r^2} \sum_{t=1}^r \E ( Z_t^4 ) \to 0,
\end{equation}
as $r\to\infty$, by noticing $v_{tt} \ge (n-1)/b_n$.
\iffalse
The Chauchy-Schwarz inequality gives that
\[
2Y_t^2Z_t^2 \le  Y_t^4 +  Z_t^4, ~~ 4 Y_t^3Z_t \le 2 Y_t^4 + 2 Y_t^2Z_t^2 \le 3 Y_t^4 + Z_t^4.
\]
It follows that
\begin{eqnarray}
\nonumber
\E (Y_t + Z_t)^4 & = & \E (Y_t^4 + 4Y_t^2Z_t^2 + Z_t^4 + 4Y_t^3Z_t + 2Y_t^2Z_t^2 + 4Y_tZ_t^3 ) \\
\label{eq-EYZt4-7}
& \le & 7\E (Y_t^4 +Z_t^4).
\end{eqnarray}
\fi
To show \eqref{condition-a-check2}, we shall derive the upper bounds of $\E Y_t^4$ and $\E Z_t^4$. This is done in two sub-steps.\\
{\bf Step 1(a)}. We first derive the upper bound of $\E Y_t^4, t=1, \ldots, r$.
To gain some intuitions, we write detailed expressions of several $Y_t$ below:
\begin{eqnarray*}
Y_1 & = & \bar{a}_{12} ( \bar{a}_{13} + \cdots + \bar{a}_{1n} ) + \bar{a}_{13}( \bar{a}_{14} + \cdots + \bar{a}_{1n}) + \cdots + \bar{a}_{1,n-1} \bar{a}_{1n}, \\
Y_2 & = & \bar{a}_{21} ( \bar{a}_{23}+ \cdots + \bar{a}_{2n} ) + \bar{a}_{23}( \bar{a}_{24} + \cdots + \bar{a}_{2n}) + \cdots + \bar{a}_{2,n-1} \bar{a}_{2n}, \\
Y_3 & = & (\bar{a}_{31}+\bar{a}_{32})(\bar{a}_{34} + \cdots + \bar{a}_{3n})  + \left\{ \bar{a}_{34}( \bar{a}_{35} + \cdots + \bar{a}_{3n} ) \right. \\
    &&\left.+ \bar{a}_{35}( \bar{a}_{36} + \cdots + \bar{a}_{3n} ) + \cdots + \bar{a}_{3,n-1} \bar{a}_{3n} \right\}, \\
Y_4 & = & (\bar{a}_{41}+\bar{a}_{42}+\bar{a}_{43})(\bar{a}_{45} + \cdots + \bar{a}_{4n})  + \left\{ \bar{a}_{45}( \bar{a}_{46} + \cdots + \bar{a}_{4n} ) \right. \\
    &&\left.+ \bar{a}_{46}( \bar{a}_{47} + \cdots + \bar{a}_{4n} ) + \cdots + \bar{a}_{4,n-1} \bar{a}_{4n} \right\}, \\
%Y_5 & = & (\bar{a}_{51} + \cdots + \bar{a}_{54} )(\bar{a}_{56} + \cdots + \bar{a}_{5n})  + \left\{ \bar{a}_{56}( \bar{a}_{57} + \cdots + \bar{a}_{5n} ) \right. \\
%    &&\left.+ \bar{a}_{57}( \bar{a}_{58} + \cdots + \bar{a}_{5n} ) + \cdots + \bar{a}_{5,n-1} \bar{a}_{5n} \right\}, \\
    & \vdots & \\
Y_{r-1} & = &  (\bar{a}_{r-1,1} + \cdots + \bar{a}_{r-1,r-2} )(\bar{a}_{r-1, r} + \cdots + \bar{a}_{r-1,n})  + \left\{ \bar{a}_{r-1,r}( \bar{a}_{r-1,r+1} + \cdots + \bar{a}_{r-1,n} ) \right. \\
    &&\left.+ \bar{a}_{r-1,r+1}( \bar{a}_{r-1, r+2} + \cdots + \bar{a}_{r-1,n} ) + \cdots + \bar{a}_{r-1,n-1} \bar{a}_{r-1,n} \right\}, \\
Y_{r} & = &  (\bar{a}_{r,1} + \cdots + \bar{a}_{r,r-1} )(\bar{a}_{r, r+1} + \cdots + \bar{a}_{r,n})  + \left\{ \bar{a}_{r,r+1}( \bar{a}_{r,r+2} + \cdots + \bar{a}_{r,n} ) \right. \\
    &&\left.+ \bar{a}_{r,r+2}( \bar{a}_{r, r+3} + \cdots + \bar{a}_{r,n} ) + \cdots + \bar{a}_{r,n-1} \bar{a}_{r,n} \right\}.
\end{eqnarray*}
As we can see, $Y_t$ can be divided into two parts:
\begin{equation}
Y_t = \underbrace{ ( \sum_{i_1=1}^{t-1} \bar{a}_{t,i_1} )( \sum_{j_1=t+1}^n \bar{a}_{t,j_1} )}_{ Y_{t1} } +
\underbrace{ \sum_{i_1=t+1}^{n-1} \sum_{j_1=i_1+1}^n \bar{a}_{t,i_1} \bar{a}_{t,j_1} }_{ Y_{t2} }.
\end{equation}
Therefore, we have
\begin{equation}\label{eq-EYt-4-a}
\E Y_t^4 = \E ( Y_{t1}^4 + Y_{t2}^4 + 4Y_{t1}^3 Y_{t2} + 4 Y_{t1}Y_{t2}^3 + 6Y_{t1}^2 Y_{t2}^2 ).
\end{equation}
Because $\bar{a}_{t,i_1}, i_1=1,\ldots, t-1$ are independent of $\bar{a}_{t,j_1}, j_1=t+1, \ldots, n$, we have
\begin{equation}\label{eq-EYt-4-b}
\E Y_{t1}Y_{t2}^3 = \E ( \sum_{i_1=1}^{t-1} \bar{a}_{t,i_1} ) \E \{(\sum_{j_1=t+1}^n \bar{a}_{t,j_1}) Y_{t2}^3\}=0,
\end{equation}
and
\begin{equation}\label{eq-EYt-4-12}
\E Y_{t1}^3 Y_{t2} = \E ( \sum_{i_1=1}^{t-1} \bar{a}_{t,i_1} )^3 \E \{(\sum_{j_1=t+1}^n \bar{a}_{t,j_1})^3 Y_{t2}\}.
\end{equation}
Because $\bar{a}_{t,i_1}, i_1=1, \ldots, t-1$ are independent and $\bar{a}_{t,i_1}=0$, we have
\begin{equation}\label{eq-EYt-4-b2}
\E ( \sum_{i_1=1}^{t-1} \bar{a}_{t,i_1} )^3 = \sum_{i_1=1}^{t-1} \E (  \bar{a}_{t,i_1}^3 )\le \frac{ (t-1)}{ c_n}.
\end{equation}
Note that
\[
\E (Y_{t1})^3 Y_{t2} = \sum_{i_1=t+1}^{n-1} \sum_{j_1=i_1+1}^n \sum_{i_2,i_3,i_4=t+1}^n \bar{a}_{t,i_1}\bar{a}_{t,j_1}
\bar{a}_{t,i_2} \bar{a}_{t,i_3} \bar{a}_{t,i_4}
\]
If the product $\bar{a}_{t,i_1}
\bar{a}_{t,i_2} \bar{a}_{t,i_3} \bar{a}_{t,i_4}\bar{a}_{t,i_5}$ is not equal to $0$, it must be in the forms of
$\bar{a}_{t,i_1}^5$ or $\bar{a}_{t,i_1}^2 \bar{a}_{t,i_2}^3$. Therefore,
\begin{eqnarray*}
\E (\sum_{j_1=t+1}^n \bar{a}_{t,j_1})^3 Y_{t2} & = & \sum_{i_1=t+1}^{n-1} \sum_{j_1=i_1+1}^n \sum_{i_2,i_3,i_4=t+1}^n \bar{a}_{t,i_1}\bar{a}_{t,j_1}
\bar{a}_{t,i_2} \bar{a}_{t,i_3} \bar{a}_{t,i_4} \\
& = & 3\sum_{i_1=t+1}^{n-1} \sum_{j_1=i_1+1}^n \E \bar{a}_{t,i_1}^3 \bar{a}_{t,j_1}^2  \\
& \le & \frac{ 3(n-t)(n-t-1) }{ 2c_n^2 }.
\end{eqnarray*}
In view of \eqref{eq-EYt-4-12} and \eqref{eq-EYt-4-b2}, we have
\begin{equation}\label{ineq-EYt13-Yt2}
\E Y_{t1}^3 Y_{t2} \le \frac{ 3(t-1)(n-t)(n-t-1) }{ c_n^3 }.
\end{equation}
Next, we calculate $\E Y_{t1}^4$. Because $\sum_{i_1=1}^{t-1} \bar{a}_{t,i_1}$ is independent of $\sum_{j_1=t+1}^n \bar{a}_{t,j_1}$, we have
\[
\E Y_{t1}^4 = \E ( \sum_{i_1=1}^{t-1} \bar{a}_{t,i_1} )^4 \E( \sum_{j_1=t+1}^n \bar{a}_{t,j_1} )^4.
\]
Because $\bar{a}_{t,i_1}, i_1=1, \ldots, t-1$ are independent and $\E \bar{a}_{i,i_1}=0$, we have
\begin{eqnarray*}
\E ( \sum_{i_1=1}^{t-1} \bar{a}_{t,i_1} )^4 & = & \sum_{i_1=1}^{t-1} \E \bar{a}_{t,i_1}^4 + 3\sum_{i_1=1}^{t-1} \sum_{i_2=1, i_2\neq i_1}^{t-1}\E \bar{a}_{t,i_1}^2\E \bar{a}_{t,i_2}^2 \\
 & \le & \frac{(t-1)}{c_n} + \frac{3(t-1)(t-2)}{c_n^2},
\end{eqnarray*}
and
\[
 \E( \sum_{j_1=t+1}^n \bar{a}_{t,j_1} )^4 \le \frac{(n-t)}{c_n} + \frac{3(n-t)(n-t-1)}{c_n^2}.
\]
It follows that
\begin{equation}\label{ineq-EYt1-4}
\E Y_{t1}^4 \le \frac{ n-1}{ c_n} + \frac{ 3((t-1)^2 + (n-t)^2 )}{ c_n^2 }.
\end{equation}
Now, we calculate $\E Y_{t2}^4$:
\[
\E Y_{t2}^4 = \E \left( \sum_{i_1=t+1}^{n-1} \sum_{j_1=i_1+1}^n \bar{a}_{t,i_1} \bar{a}_{t,j_1} \right)^4.
\]
It has $8$ summarizations:
\[
 \sum_{i_1=t+1}^{n-1} \sum_{j_1=i_1+1}^n \sum_{i_2=t+1}^{n-1} \sum_{j_2=i_2+1}^n\sum_{i_3=t+1}^{n-1} \sum_{j_3=i_3+1}^n
  \sum_{i_4=t+1}^{n-1} \sum_{j_4=i_4+1}^n \bar{a}_{t,i_1} \bar{a}_{t,j_1}\bar{a}_{t,i_2} \bar{a}_{t,j_2}\bar{a}_{t,i_3} \bar{a}_{t,j_3}\bar{a}_{t,i_4} \bar{a}_{t,j_4}.
\]
Observe that for $i_1<j_1$, $i_2<j_2$, $i_3<j_3$, $i_4<j_4$, if
$\bar{a}_{t,i_1} \bar{a}_{t,j_1} \bar{a}_{t,i_2} \bar{a}_{t,j_2} \bar{a}_{t, i_3} \bar{a}_{t,j_3} \bar{a}_{t,i_4} \bar{a}_{t,j_4}\neq 0$,
it must belongs to one of the four forms:
\begin{align*}
(\bar{a}_{t,k_1})^4 (\bar{a}_{t,k_2})^4, ~~(\bar{a}_{t,k_1})^4 (\bar{a}_{t,k_2})^2(\bar{a}_{t,k_3})^2,
\\
(\bar{a}_{t,k_1})^2(\bar{a}_{t,k_2})^2(\bar{a}_{t,k_3})^2(\bar{a}_{t,k_4})^2, ~~
(\bar{a}_{t,k_1})^3(\bar{a}_{t,k_2})^3(\bar{a}_{t,k_3})^2,
\end{align*}
where $k_1, k_2, k_3, k_4$ are four distinct values. \\
(Case 1) For the type of $(\bar{a}_{t,k_1})^4 (\bar{a}_{t,k_2})^4$, it must have
$i_1=i_2=i_3=i_4$ and $j_1=j_2=j_3=j_4$ and the number of such terms is at most
\[
(n-t-1)+(n-t-2)+\cdots +1=\frac{1}{2}(n-t)(n-t-1).
\]
(Case 2) For the type of $(\bar{a}_{t,k_1})^4 (\bar{a}_{t,k_2})^2(\bar{a}_{t,k_3})^2$, it must have $i_1=i_2=i_3=i_4$ or $j_1=j_2=j_3=j_4$.
If $i_1=i_2=i_3=i_4$, then the number of such terms is at most
\[
6\{(n-t-1)^2+(n-t-2)^2+\cdots +1\} = (n-t-1)(n-t)(2(n-t-1)+1).
\]
If $j_1=j_2=j_3=j_4$, then the number of such terms is at most at most $3(n-1-t)(n-t)(n-t-1)$. \\
(Case 3) For the type of $(\bar{a}_{t,k_1})^2(\bar{a}_{t,k_2})^2(\bar{a}_{t,k_3})^2(\bar{a}_{t,k_4})^2$, it has
at most $c_2 (n-t-1)^4$ such terms, where $c_2$ is an absolute constant. \\
(Case 4) For the type of $(\bar{a}_{t,k_1})^3(\bar{a}_{t,k_2})^3(\bar{a}_{t,k_3})^2$, it has at most
$c_3 (n-t-1)^3$ such terms, where $c_3$ is an absolute constant. \\
As a result, we have
\[
\E Y_{t2}^4 \lesssim \frac{ (n-r)^2}{c_n^2 } + \frac{ (n-r)^3 }{ c_n^3 } +  \frac{ (n-r)^4 }{ c_n^4}.
\]
In view of \eqref{ineq-EYt1-4}, we have
\begin{equation*}\label{ineq-EYt4}
\E Y_t^4 \lesssim \frac{ n-1}{ c_n} + \frac{ 3((t-1)^2 + (n-t)^2 )}{ c_n^2 } + \frac{ t^2 }{ c_n^2 } + \frac{ (n-t)^4 }{ c_n^4 }.
\end{equation*}
It follows that if $b_n^4/c_n^4 = o(r)$ and $b_n=o(n)$, then
\begin{equation}\label{ineq-EYt4}
 \frac{ b_n^4 }{ r^2 n^4 } \sum_{t=1}^r \E Y_t^4 \lesssim \frac{ b_n^4 }{ r^2 n^4 }  \cdot \big( \frac{ nr}{c_n} + \frac{r^3 }{ c_n^2 } + \frac{ n^2r}{c_n^2 } + \frac{n^4 r}{ c_n^4 } \big)=o(1).
\end{equation}

{\bf Step 1 (b)}. We calculate $\E Z_t^4$. Note that for $t=2, \ldots,r-1$, we have
\[
\E Z_t^4 = \E \left\{ \sum_{i_1=t+1}^r \frac{ \bar{a}_{t, i_1} }{ v_{i_1, i_1}} ( \sum_{i_2=1}^{t-1} \bar{a}_{i_2, i_1}) \right\}^4.
\]
Because $\bar{a}_{t, i_1}$ and  $( \sum_{i_2}^{t-1} \bar{a}_{i_2, i_1})$ are independent for  $i_1=t+1, \ldots, r$, we have
\begin{eqnarray}
\nonumber
\E Z_t^4 & = & \underbrace{ \sum_{i_1=t+1}^r \E \left\{  \frac{ \bar{a}_{t, i_1} }{ v_{i_1, i_1}} ( \sum_{i_2=1}^{t-1} \bar{a}_{i_2, i_1}) \right\}^4 }_{ (I)} \\
\label{ineq-EZt-four}
&&+ \underbrace{ \sum_{i_1, j_1=t+1,i_1\neq j_1}^r  \E \left\{  \frac{ \bar{a}_{t, i_1} }{ v_{i_1, i_1}} ( \sum_{i_2=1}^{t-1} \bar{a}_{i_2, i_1}) \right\}^2
\E \left\{  \frac{ \bar{a}_{t, j_1} }{ v_{j_1,j_1} } ( \sum_{i_2=1}^{t-1} \bar{a}_{i_2, j_1}) \right\}^2 }_{ (II)} .
\end{eqnarray}
For the first term in the right-hand side of the above equation, we have
\begin{eqnarray*}
%&&\E \left\{  \frac{ \bar{a}_{t, i_1} }{ v_{i_1,i_1}} ( \sum_{i_2=1}^{t-1} \bar{a}_{i_2, i_1}) \right\}^4  \\
(I) & = & \sum_{i_1=t+1}^r \frac{1}{ v_{i_1,i_1}^4} \sum_{i_2=1}^{t-1} \E \left\{  \bar{a}_{t, i_1} \bar{a}_{i_2, i_1}) \right\}^4
\\
&&+ \sum_{i_1=t+1}^r \frac{1}{ v_{i_1,i_1}^4} \sum_{i_2, i_3=1, i_2\neq i_3}^{(t-1)} \E \left\{  \bar{a}_{t, i_1} \bar{a}_{i_2, i_1}) \right\}^2\E \left\{  \bar{a}_{t, i_1} \bar{a}_{i_3, i_1}) \right\}^2
 \\
& \le & (r-t-1)\big\{ (\frac{ t-1}{ c_n } + \frac{ (t-1)^2 }{ c_n^2 }) \big\} \cdot \frac{ b_n^4 }{ (n-1)^4 } .
\end{eqnarray*}
For the second term in the right-hand side of \eqref{ineq-EZt-four}, we have
\[
(II) \le (r-t)^2 \cdot \big\{  \frac{ b_n^2 (t-1) }{ (n-1)^2 c_n^2 } \big\}^2
\]
Combining these, it yields
\begin{equation}\label{ineq-EZt-4}
\E Z_t^4 \lesssim r\times \left( \frac{ t}{ c_n } + \frac{ t^2 }{ c_n^2 } \right)\cdot \frac{ b_n^4 }{ n^4 }  + r^2 \cdot   \frac{ b_n^4 t^2 }{ n^4 c_n^4 } .
\end{equation}
It follows that if $b_n^4/c_n^4=o(n)$ and $b_n=o(n)$, then
\begin{eqnarray*}
 \frac{1}{r^2} \sum_{i=1}^r \E  ( Z_i^4 )  \lesssim \frac{ b_n^4 }{ r^2 n^4 } \times \big( \frac{ r^3 }{ c_n } + \frac{r^4}{ c_n^2 } + \frac{ r^5 }{ c_n^4 } \big) \to 0,
\end{eqnarray*}
which shows \eqref{condition-a-check2}.

{\bf Step 2}. We show \eqref{condition-b}. We first show
\begin{equation}\label{eq-lemma1-d}
\lim_{r\to\infty} \frac{1}{2r} \sum_{i=1}^r \E\left( Y_i/v_{ii}  +  Z_i \right)^2 \to 1.
\end{equation}
In view of \eqref{eq-lemma1-a},  \eqref{eq-lemma1-b} and \eqref{eq-lemma1-c}, it is sufficient to demonstrate
\begin{equation}\label{eq-lemma1-e}
\lim_{r\to\infty} \frac{1}{r} \mathrm{Cov}\left( \sum_{i=1}^r \sum_{j=1}^n \frac{ (\bar{a}_{ij}^2 - \E \bar{a}_{ij}^2 ) }{ v_{ii} }
, \sum_{i=1}^r \sum_{j=1,j\neq i}^n \sum_{k=1, k\neq i,j}^n \frac{ \bar{a}_{ij} \bar{a}_{ik} }{ v_{ii} } \right) = 0
\end{equation}
Note that $\E\bar{a}_{i_2,j_2}\bar{a}_{i_2,j_3} )= 0$ for $j_2\neq j_3$.
If $\mathrm{Cov}( \bar{a}_{i_1,j_1}^2, \bar{a}_{i_2,j_2}\bar{a}_{i_2,j_3} )\neq 0$ for $j_2\neq j_3$, it must have $i_1=i_2$, $j_2=j_3=0$.
Therefore, we have
\begin{eqnarray*}
&&\sum_{i_1=1}^r \sum_{j_1=1}^n \sum_{i_2=1}^r \sum_{j_2=1,j_2\neq i_2}^n \sum_{j_3=1, j_3\neq i_2,j_2}^n \mathrm{Cov}( \bar{a}_{i_1,j_1}^2, \bar{a}_{i_2,j_2}\bar{a}_{i_2,j_3} ) \\
& = & \sum_{i_1=1}^r \sum_{j_1=1}^n  \sum_{j_2=1,j_2\neq i_2}^n \sum_{j_3=1, j_3\neq i_2,j_2}^n \mathrm{Cov}( \bar{a}_{i_1,j_1}^2, \bar{a}_{i_1,j_2}\bar{a}_{i_1,j_3} ) \\
& = & \sum_{i_1=1}^r \sum_{j_1=1}^n\mathrm{Cov}( \bar{a}_{i_1,j_1}^2, \bar{a}_{i_1,j_1}^2 ) \\
& \lesssim & \frac{rn}{c_n}.
\end{eqnarray*}
If $b_n^2/c_n=o(n)$, then \eqref{eq-lemma1-e} holds.
Therefore,  it is sufficient to demonstrate
\begin{equation}\label{eq-lemma1-f}
\frac{1}{r^2} \mathrm{Var}\left( \sum_{i=1}^r \E\{ ( Y_i/v_{ii} + Z_i )^2 | \mathcal{F}_{i-1} \} \right) \to 0.
\end{equation}
to show \eqref{condition-b}.
It  essentially requires us to  calculate the variance:
\begin{eqnarray*}
 && \mathrm{Var}\left( \frac{1}{2r} \sum_{i=1}^r \E\{ ( Y_i/v_{ii} + Z_i )^2 | \mathcal{F}_{i-1} \}  \} \right) \\
& = & \frac{1}{4r^2} \E \left(  \sum_{i=1}^r \E \left[ \frac{ \{ ( Y_i/v_{ii} + Z_i )^2 - \E ( Y_i/v_{ii} + Z_i )^2 \} }{ v_{ii}^2 } \Big| \mathcal{F}_{i-1} \right] \right)^2 \\
& = & \frac{1}{4r^2} \sum_{i=1}^r \sum_{j=1}^r \E \left\{ \left(   \E \left[ \frac{ \{ ( Y_i/v_{ii} + Z_i )^2 - \E( Y_i/v_{ii} + Z_i )^2 \} }{ v_{ii}^2 } \Big| \mathcal{F}_{i-1} \right] \right) \right. \\
&&~~~~~~~~~~~~~~~~~~~~ \times \left. \left(   \E \left[ \frac{ \{ ( Y_j/v_{jj} + Z_j )^2 - \E( Y_j/v_{jj} + Z_j )^2 \} }{ v_{jj}^2 } \Big| \mathcal{F}_{j-1} \right] \right)
\right\}.
\end{eqnarray*}
Therefore, showing \eqref{eq-lemma1-f} is equivalent to showing
\begin{equation}\label{eq-condition-YZ}
\frac{1}{r^2} \sum_{i=1}^r \E  \left(   \E \left[  \{ ( Y_i/v_{ii} + Z_i )^2 - \E( Y_i/v_{ii} + Z_i )^2 \}  \Big| \mathcal{F}_{i-1} \right] \right)^2\to 0,
\end{equation}
and
\begin{eqnarray}
\nonumber
H&:=&\frac{1}{r^2}  \sum_{i,j=1; i\neq j}^r \left| \E \left\{ \left(   \E \left[  \{ ( Y_i/v_{ii} + Z_i )^2 - \E( Y_i/v_{ii} + Z_i )^2 \}  \Big| \mathcal{F}_{i-1} \right] \right) \right. \right. \\
\label{eq-condition-YZ2}
&&~~~~~~~~~~~~ \times \left.\left. \left(   \E \left[  \{ ( Y_j/v_{jj} + Z_j )^2 - \E( Y_j/v_{jj} + Z_j )^2 \}   \Big| \mathcal{F}_{j-1} \right] \right)
\right\}\right|\to 0.
\end{eqnarray}
This is done in two steps.

{\bf Step 3}. We show \eqref{eq-condition-YZ}.
We derive the explicit expression of the condition expectation:
\begin{eqnarray*}
\E \left[  (Y_t +Z_t)^2 | \mathcal{F}_{t-1} \right]= \E \left(  Y_t^2 | \mathcal{F}_{t-1} \right)
+ \E \left(  Z_t^2 | \mathcal{F}_{t-1} \right) + 2 \E \left(  Y_tZ_t | \mathcal{F}_{t-1} \right).
\end{eqnarray*}
Recall that
\[
Y_t=\underbrace{ ( \sum_{i_1=1}^{t-1} \bar{a}_{t,i_1} )( \sum_{j_1=t+1}^n \bar{a}_{t,j_1} )}_{ Y_{t1} } +
\underbrace{ \sum_{i_1=t+1}^{n-1} \sum_{j_1=i_1+1}^n \bar{a}_{t,i_1} \bar{a}_{t,j_1} }_{ Y_{t2} },
\]
and
\[
Z_1=0, ~~Z_t =  \sum_{i_1=t+1}^r \frac{ \bar{a}_{t, i_1} }{ v_{i_1,i_1} } ( \sum_{i_2=1}^{(t-1)} \bar{a}_{i_2, i_1}), t=2,\ldots, r-1,~~Z_r=0.
\]
It is easy to see that
\[
\E (Y_1Z_1) =0, \E [(Y_2Z_2)| \mathcal{F}_1 ] =0.
\]
The conditional expectation of $Y_tZ_t$ is
\begin{eqnarray}
\nonumber
\E ( Y_tZ_t | \mathcal{F}_{t-1} ) &  =  & \E \left[ ( \sum_{i_1=1}^{t-1} \bar{a}_{t,i_1} )( \sum_{j_1=t+1}^n \bar{a}_{t,j_1} )
\{\sum_{i_3=t+1}^r \frac{ \bar{a}_{t, i_3} }{ v_{i_3,i_3} } ( \sum_{i_4=1}^{(t-1)} \bar{a}_{i_4, i_3})\} \big| \mathcal{F}_{t-1} \right] \\
\nonumber
&& + \E \left[ \sum_{i_1=t+1}^{n-1} \sum_{j_1=i_1+1}^n \bar{a}_{t,i_1} \bar{a}_{t,j_1}\{\sum_{i_3=t+1}^r \frac{ \bar{a}_{t, i_3} }{ v_{i_3,i_3}} ( \sum_{i_4=1}^{(t-1)} \bar{a}_{i_4, i_3})\} \big| \mathcal{F}_{t-1} \right] \\
\nonumber
&=& \sum_{i_1=1}^{t-1}\sum_{j_1=t+1}^n \sum_{i_3=t+1}^r\sum_{i_4=1}^{(t-1)} \frac{ \bar{a}_{t,i_1} \bar{a}_{i_4, i_3} }{ v_{i_3,i_3} } \E \bar{a}_{t,j_1}\bar{a}_{t, i_3} \\
\label{condition-expe-YtZt}
&&+ \sum_{i_1=t+1}^{n-1} \sum_{j_1=i_1+1}^n \sum_{i_3=t+1}^r\sum_{i_4=1}^{(t-1)} \frac{ \bar{a}_{i_4, i_3} }{ v_{i_3,i_3} } \E \bar{a}_{t,i_1}\bar{a}_{t,j_1}\bar{a}_{t, i_3}.
\end{eqnarray}
Therefore, we have
\[
\E \{ \E ( Y_tZ_t | \mathcal{F}_{t-1} ) \} =0.
\]
The conditional expectation of $Y_t^2$ is
\begin{eqnarray}
\nonumber
&&\E \left(  Y_t^2 | \mathcal{F}_{t-1} \right) \\
\nonumber
& = & \E \left(  Y_{t1}^2 | \mathcal{F}_{t-1} \right)
+ 2\E \left(  Y_{t1} | \mathcal{F}_{t-1} \right) \E \left(  Y_{t2} | \mathcal{F}_{t-1} \right)
+ \E \left(  Y_{t2}^2 | \mathcal{F}_{t-1} \right) \\
\label{condition-expe-Yt2}
& = &  ( \sum_{i_1=1}^{t-1} \bar{a}_{t,i_1} )^2\E( \sum_{j_1=t+1}^n \bar{a}_{t,j_1} )^2 +
2( \sum_{i_1=1}^{t-1} \bar{a}_{t,i_1} )\E\left\{( \sum_{j_1=t+1}^n \bar{a}_{t,j_1} )Y_{t2}\right\}
 + \E Y_{t2}^2.
\end{eqnarray}
The conditional expectation of $Z_t^2$ is
\begin{eqnarray}
\nonumber
&& \E \left( Z_t^2 | \mathcal{F}_{t-1} \right) \\
\nonumber
& = & \E \left\{ \sum_{i_1=t+1}^r \frac{ \bar{a}_{t,i_1} }{ v_{i_1,i_1} } ( \sum_{i_2=1}^{t-1} \bar{a}_{i_2,i_1} ) \cdot
\sum_{j_1=t+1}^r \frac{ \bar{a}_{t,j_1} }{ v_{j_1,j_1} } (\sum_{j_2=1}^{t-1} \bar{a}_{j_2,j_1} ) | \mathcal{F}_{t-1} \right\} \\
\nonumber
& = & \sum_{i_1=t+1}^r \sum_{i_2=1}^{t-1}\sum_{j_1=t+1}^r\sum_{j_2=1}^{t-1} \frac{ \bar{a}_{i_2,i_1}\bar{a}_{j_2,j_1} }{ v_{i_1,i_1} v_{j_1,j_1} } \E (\bar{a}_{t,i_1}\bar{a}_{t,j_1} ) \\
\label{condition-expe-Zt2}
& = & \sum_{i_1=t+1}^r \sum_{i_2=1}^{t-1}\sum_{j_2=1}^{t-1} \frac{ \bar{a}_{i_2,i_1}\bar{a}_{j_2,i_1} }{ v_{i_1,i_1}^2  }  \E (\bar{a}_{t,i_1}^2 ).
\end{eqnarray}
By $c_r$-inequality, as in \citeauthor{2010Probability} (\citeyear{2010Probability}, page 97), we have
\begin{eqnarray*}
&&\E \left[ \{ \E(Y_i^2|\mathcal{F}_{i-1}) - \E Y_i^2 \} + \{\E(Z_i^2|\mathcal{F}_{i-1}) - \E Z_i^2 \}
+ 2\E(Y_iZ_i|\mathcal{F}_{i-1}) \right]^2
\\
&\le & 2\E \{ \E(Y_i^2|\mathcal{F}_{i-1}) - \E Y_i^2 \}^2 + 2\E \{\E(Z_i^2|\mathcal{F}_{i-1}) - \E Z_i^2 \}^2
+  4\E \{\E(Y_iZ_i|\mathcal{F}_{i-1}) \}^2.
\end{eqnarray*}
The proof of \eqref{eq-condition-YZ} is divided into three sub-steps.
Step 3(a). We derive the upper bound of $\E \{\E(Y_tZ_t|\mathcal{F}_{t-1}) \}^2$. Note that
\begin{eqnarray*}
\E \{\E(Y_tZ_t|\mathcal{F}_{t-1}) \}^2
&\le& 2\E \left( \sum_{i_1=1}^{t-1}\sum_{j_1=t+1}^n \sum_{i_3=t+1}^r\sum_{i_4=1}^{t-1} \frac{ \bar{a}_{t,i_1} \bar{a}_{i_4, i_3} }{ v_{i_3,i_3} } \E \bar{a}_{t,j_1}\bar{a}_{t, i_3} \right)^2 \\
&&+ 2\E \left( \sum_{i_1=t+1}^{n-1} \sum_{j_1=i_1+1}^n \sum_{i_3=t+1}^r\sum_{i_4=1}^{t-1} \frac{ \bar{a}_{i_4, i_3} }{ v_{i_3,i_3} }  \E \bar{a}_{t,i_1}\bar{a}_{t,j_1}\bar{a}_{t, i_3}\right)^2.
\end{eqnarray*}
Because $j_1>i_1$ and $\bar{a}_{t,s}=0$ for any pair $(t,s)$, we have
\[
\E \bar{a}_{t,i_1}\bar{a}_{t,j_1}\bar{a}_{t, i_3}=0.
\]
It follows that
\begin{eqnarray*}
&&\E \{\E(Y_tZ_t|\mathcal{F}_{t-1}) \}^2 \\
& \le & 2\E \left( \sum_{i_1=1}^{t-1} \sum_{i_3=t+1}^r\sum_{i_4=1}^{t-1} \frac{ \bar{a}_{t,i_1} \bar{a}_{i_4, i_3} }{ v_{i_3,i_3} } \E \bar{a}_{t, i_3}^2 \right)^2\\
& = & 2\sum_{i_1=1}^{t-1} \sum_{j_1=1}^{t-1} \sum_{i_4=1}^{ t-1} \sum_{j_4=1}^{t-1} \sum_{i_3=t+1}^r \sum_{j_3=r+1}^r
\frac{1}{ v_{i_3,i_3}v_{j_3,j_3} }\E\bar{a}_{t,i_1} \bar{a}_{i_4,i_3} \bar{a}_{t,j_1} \bar{a}_{j_4,j_3} \E \bar{a}_{t,i_3}^2 \E \bar{a}_{t,j_3}^2.
\end{eqnarray*}
If $\E\bar{a}_{i_1,i_2} \bar{a}_{i_3,i_4} \bar{a}_{i_5,i_6} \bar{a}_{i_7, i_8}$ is not zero, it must be in the forms of
$\E\bar{a}_{ij}^4$ or $\E\bar{a}_{ij}^2\bar{a}_{kl}^2$. Because
$t$ is fixed in $\bar{a}_{t,i_1} \bar{a}_{i_4,i_3} \bar{a}_{t,j_1} \bar{a}_{j_4,j_3}$, we have
\begin{eqnarray}
\nonumber
&& \E \{\E(Y_tZ_t|\mathcal{F}_{t-1}) \}^2 \\
\nonumber
&\le & 2 \sum_{i_1=1}^{t-1}  \sum_{i_4=1}^{ t-1}  \sum_{i_3=t+1}^r
\E\bar{a}_{t,i_1}^2 \bar{a}_{i_4,i_3}^2  \E \bar{a}_{t,i_3}^2 \E \bar{a}_{t,j_3}^2 \\
\label{ineq-con-YtZt-2}
&\le & \frac{ 2(t-1)^2(r-t) }{ c_4^2}.
\end{eqnarray}

Step 3(b). We derive the upper bound of $\E \{ \E(Y_t^2|\mathcal{F}_{t-1}) - \E Y_t^2 \}^2$.
Note that
\begin{eqnarray}
\nonumber
&&\E \left(  Y_t^2 | \mathcal{F}_{t-1} \right) \\
\nonumber
& = & \E \left(  Y_{t1}^2 | \mathcal{F}_{t-1} \right)
+ 2\E \left(  Y_{t1} | \mathcal{F}_{t-1} \right) \E \left(  Y_{t2} | \mathcal{F}_{t-1} \right)
+ \E \left(  Y_{t2}^2 | \mathcal{F}_{t-1} \right) \\
\label{eq-condition-expe-Yt2}
& = &  ( \sum_{i_1=1}^{t-1} \bar{a}_{t,i_1} )^2\E( \sum_{j_1=t+1}^n \bar{a}_{t,j_1} )^2 +
2( \sum_{i_1=1}^{t-1} \bar{a}_{t,i_1} )\E\left\{( \sum_{j_1=t+1}^n \bar{a}_{t,j_1} )Y_{t2}\right\}
 + \E Y_{t2}^2.
\end{eqnarray}
It follows that
\begin{eqnarray}
\nonumber
& &\E \{ \E(Y_t^2|\mathcal{F}_{t-1}) - \E Y_t^2 \}^2 \\
\nonumber
&= & \E \left[ \left\{( \sum_{i_1=1}^{t-1} \bar{a}_{t,i_1} )^2-\E( \sum_{i_1=1}^{t-1} \bar{a}_{t,i_1} )^2\right\}\E( \sum_{j_1=t+1}^n \bar{a}_{t,j_1} )^2 \right.\\
&& + \left. 2( \sum_{i_1=1}^{t-1} \bar{a}_{t,i_1} )\E\left\{( \sum_{j_1=t+1}^n \bar{a}_{t,j_1} )Y_{t2}\right\} \right]^2 \\
\nonumber
& \le & 2\E \left[ \left\{( \sum_{i_1=1}^{t-1} \bar{a}_{t,i_1} )^2-\E( \sum_{i_1=1}^{t-1} \bar{a}_{t,i_1} )^2\right\}\E( \sum_{j_1=t+1}^n \bar{a}_{t,j_1} )^2\right]^2 \\
\nonumber
&& + 4\E\left[( \sum_{i_1=1}^{t-1} \bar{a}_{t,i_1} )\E\left\{( \sum_{j_1=t+1}^n \bar{a}_{t,j_1} )Y_{t2}\right\}
\right]^2 \\
\nonumber
& \le & \frac{2(n-t)^2}{c_n^2} \E \left\{( \sum_{i_1=1}^{t-1} \bar{a}_{t,i_1} )^2-\E( \sum_{i_1=1}^{t-1} \bar{a}_{t,i_1} )^2\right\}^2 \\
\label{ineq-cond-expe-Yt2}
&& + \frac{4(t-1)}{ c_n} \E\left\{( \sum_{j_1=t+1}^n \bar{a}_{t,j_1} )Y_{t2}\right\}^2.
\end{eqnarray}
The upper bounds of two expectations in the above last inequality are derived as follows.
Note that
\begin{eqnarray*}
 &&\E\left\{( \sum_{j_1=t+1}^n \bar{a}_{t,j_1} )Y_{t2}\right\}^2 \\
 & = & \E \left( \sum_{i_1=t+1}^{n-1} \sum_{j_1=t+1}^n \bar{a}_{t,i_1} \bar{a}_{t,j_1}\right)^2 \left( \sum_{k_1=t+1}^n \bar{a}_{t,k_1}\right)^2 \\
 & = & \sum_{i_1=t+1}^{n-1}\sum_{j_1=i_1+1}^n \sum_{i_2=t+1}^{n-1} \sum_{j_2=i_2+1}^n \sum_{k_1=t+1}^n \sum_{k_2=t+1}^n
 \E \bar{a}_{t,i_1} \bar{a}_{t,j_1} \bar{a}_{t,i_2} \bar{a}_{t,j_2} \bar{a}_{t,k_1}\bar{a}_{t,k_2}
\end{eqnarray*}
If $\E \bar{a}_{t,i_1} \bar{a}_{t,j_1} \bar{a}_{t,i_2} \bar{a}_{t,j_2} \bar{a}_{t,k_1}\bar{a}_{t,k_2}$ is not zero, it must be
one of the forms: $\bar{a}_{i_1,j_1}^6$, $\bar{a}_{i_1,j_1}^3\bar{a}_{i_2,j_2}^3$, $\bar{a}_{i_1,j_1}^4 \bar{a}_{i_1,j_1}^2$
and $\bar{a}_{i_1,j_1}^2 \bar{a}_{i_2,j_2}^2 \bar{a}_{i_3,j_3}^2$ for three distinct random variables $\bar{a}_{i_1,j_1}$, $\bar{a}_{i_2,j_2}$
and $\bar{a}_{i_3,j_3}$. Therefore, we have
\begin{equation}\label{ineq-cond-expe-Yt2-a}
\E\left\{( \sum_{j_1=t+1}^n \bar{a}_{t,j_1} )Y_{t2}\right\}^2 \lesssim \frac{ (n-t)^2}{ c_n^2} + \frac{ (n-t)^3 }{ c_n^3}.
\end{equation}
Note that
\begin{eqnarray*}
\E \left\{( \sum_{i_1=1}^{t-1} \bar{a}_{t,i_1} )^2-\E( \sum_{i_1=1}^{t-1} \bar{a}_{t,i_1} )^2\right\}^2
 =  \E ( \sum_{i_1=1}^{t-1} \bar{a}_{t,i_1} )^4- \left\{\E( \sum_{i_1=1}^{t-1} \bar{a}_{t,i_1} )^2\right\}^2,
\end{eqnarray*}
and
\begin{eqnarray*}
\E(\sum_{i_1=1}^{t-1} \bar{a}_{t,i_1} )^4 & = & \sum_{i_1=1}^{t-1} \sum_{i_2=1}^{t-1} \sum_{i_3=1}^{t-1} \sum_{i_4=1}^{t-1}
\bar{a}_{t,i_1} \bar{a}_{t,i_2} \bar{a}_{t,i_3} \bar{a}_{t,i_4} \\
& = &  \sum_{i_1=1}^{t-1}\E \bar{a}_{t,i_1}^4 +
3\sum_{i_1=1}^{t-1} \sum_{i_2=1,i_2\neq i_1}^{t-1} \E \bar{a}_{t,i_1}^2 \E \bar{a}_{t,i_2}^2  \\
& \lesssim & \frac{ (t-1)^2 }{ c_n^2 }
\end{eqnarray*}
In view of \eqref{ineq-cond-expe-Yt2} and \eqref{ineq-cond-expe-Yt2-a}, it follows that
\begin{equation}\label{ineq-con-exp-Yt2}
\E \{ \E(Y_t^2|\mathcal{F}_{t-1}) - \E Y_t^2 \}^2 \lesssim \frac{ (n-t)^2(t-1)^2 }{ c_n^4 } +
\left\{ \frac{(n-t)^2}{c_n^2} + \frac{ (n-t)^3}{ c_n^3}  \right\} \frac{ (t-1)}{ c_n }.
\end{equation}

Step 3(c). We derive the upper bound of $\E \left( Z_t^2 | \mathcal{F}_{t-1} \right)$.
The conditional expectation of $Z_t^2$ is
\begin{eqnarray}
\nonumber
&& \E \left( Z_t^2 | \mathcal{F}_{t-1} \right) \\
\nonumber
& = & \E \left\{ \sum_{i_1=t+1}^r \frac{ \bar{a}_{t,i_1} }{ v_{i_1,i_1} }  ( \sum_{i_2=1}^{t-1} \bar{a}_{i_2,i_1} ) \cdot
\sum_{j_1=t+1}^r \frac{  \bar{a}_{t,j_1} }{ v_{j_1,j_1} } (\sum_{j_2=1}^{t-1} \bar{a}_{j_2,j_1} ) | \mathcal{F}_{t-1} \right\} \\
\nonumber
& = & \sum_{i_1=t+1}^r \sum_{i_2=1}^{t-1}\sum_{j_1=t+1}^r\sum_{j_2=1}^{t-1} \frac{ \bar{a}_{i_2,i_1}\bar{a}_{j_2,j_1} }{ v_{i_1,i_1} v_{j_1,j_1} }  \E (\bar{a}_{t,i_1}\bar{a}_{t,j_1} ) \\
\label{eq-conditional-Zt2}
& = & \sum_{i_1=t+1}^r \sum_{i_2=1}^{t-1}\sum_{j_2=1}^{t-1} \frac{ \bar{a}_{i_2,i_1}\bar{a}_{j_2,i_1} }{ v_{i_1,i_1}^2 } \E (\bar{a}_{t,i_1}^2 ).
\end{eqnarray}
It follows that
\begin{eqnarray}
\nonumber
&& \E \left\{ \E \left( Z_t^2 | \mathcal{F}_{t-1} \right) \right\}^2 \\
\nonumber
& = &
\E \left\{ \sum_{i_1=t+1}^r \sum_{i_2=1}^{t-1}\sum_{j_2=1}^{t-1} \frac{ \bar{a}_{i_2,i_1}\bar{a}_{j_2,i_1} }{ v_{i_1,i_1} v_{j_1,j_1} } \E (\bar{a}_{t,i_1}^2 ) \right\}^2 \\
\nonumber
& \le & \frac{ b_n^4}{ n^4 c_n^2} \left\{  \sum_{i_1=t+1}^{r} \sum_{i_2=t+1}^{t-1} \E \bar{a}_{i_1,i_2}^4
+  \left(\sum_{i_1=t+1}^r \sum_{i_2=1}^{t-1} \E \bar{a}_{i_1,i_2}^2  \right)^2
+ \sum_{i_1=t+1}^r \sum_{i_2=1}^{t-1} \sum_{s_2=1}^{t-1} \E \bar{a}_{i_1,i_2}^2 \E \bar{a}_{i_1,s_2}^2 \right\} \\
\label{ineq-upper-EZt2}
& \lesssim & \frac{ b_n^4}{ n^4 } \Big( \frac{ (r-t)(t-1) }{ c_n^3 } + \frac{ (r-t)^2(t-1)^2 }{ c_n^4 } + \frac{ (r-t)(t-1)^2 }{ c_n^4 } \Big).
\end{eqnarray}

By combining \eqref{ineq-con-YtZt-2}, \eqref{ineq-con-exp-Yt2} and \eqref{ineq-upper-EZt2}, it yields
\begin{eqnarray*}
&&\frac{1}{r^2} \sum_{i=1}^r \E  \left(   \E \left[  \{ (Y_i/v_{ii} +Z_i)^2 - \E (Y_i/v_{ii} +Z_i)^2 \}  \Big| \mathcal{F}_{i-1} \right] \right)^2 \\
& \lesssim & \frac{b_n^4}{r^2n^4}  \sum_{t=1}^r \left(\frac{ t^2(r-t) }{c_n^2} + \frac{ (n-t)^3t }{c_n^4} + \frac{(r-t)^2t}{c_n^4} \right) \\
& \lesssim & \frac{b_n^4}{r^2n^4} \times \frac{ n^3 r^2 }{ c_n^4 } \to 0.
\end{eqnarray*}
This shows \eqref{eq-condition-YZ}.
%%%%%%%%%%%%%%%

{\bf Step 4}. We show \eqref{eq-condition-YZ2}.
\iffalse
\begin{eqnarray}
\nonumber
&&\frac{b_n^4}{r^2n^4}  \sum_{i,j=1; i\neq j}^r \left| \E \left\{ \left(   \E \left[  \{ (Y_i +Z_i)^2 - \E (Y_i +Z_i)^2 \}  \Big| \mathcal{F}_{i-1} \right] \right) \right. \right. \\
\label{eq-condition-YZ2}
&&~~~~~~~~~~~~~~~~~~~~ \times \left.\left. \left(   \E \left[  \{ (Y_j +Z_j)^2 - \E (Y_j +Z_j)^2 \}   \Big| \mathcal{F}_{j-1} \right] \right)
\right\}\right|.
\end{eqnarray}
\fi
This requires us to calculate
\begin{eqnarray*}
&&\mathrm{Cov}\left\{ \E(Y_t^2|\mathcal{F}_{t-1}) + \E(Z_t^2|\mathcal{F}_{t-1}) + \E(Y_tZ_t|\mathcal{F}_{t-1})\right.,\\
&&~~~~~~~~\left.\E(Y_s^2|\mathcal{F}_{s-1}) + \E( Z_s^2|\mathcal{F}_{s-1}) + \E(Y_sZ_s|\mathcal{F}_{s-1} ) \right\}.
\end{eqnarray*}
This is done in steps 4(a)-4(f). In what follows, we assume $t<s$.\\
Step 4(a). We derive the upper bound of $\mathrm{Cov}\left\{ \E(Y_t^2|\mathcal{F}_{t-1}), \E(Y_s^2|\mathcal{F}_{s-1})\right\}$.
\iffalse
Note that
\[
\E Y_{t2}^2 =\E (\sum_{i_1=t+1}^{n-1} \sum_{j_1=i_1+1}^n \bar{a}_{t,i_1} \bar{a}_{t,j_1} )^2 \le \frac{ (n-t)(n-1-t) }{ 2c_n^2 },
\]
and
\begin{eqnarray*}
\E\left\{( \sum_{j_1=t+1}^n \bar{a}_{t,j_1} )Y_{t2}\right\} & \le &
 \left\{\E( \sum_{j_1=t+1}^n \bar{a}_{t,j_1} )^2 \right\}^{1/2} \left\{\E Y_{t2}^2\right\}^{1/2}
\\
& \lesssim & \frac{ (n-t)^{3/2} }{ c_n^{3/2} }.
\end{eqnarray*}
\fi
Note that $\sum_{i_1=1}^{t-1} \bar{a}_{t,i_1}$ is independent of $\sum_{i_1=1}^{s-1} \bar{a}_{s,i_1}$ for $t\neq s$.  By \eqref{eq-condition-expe-Yt2}, we have
\begin{eqnarray}
\nonumber
&&\mathrm{Cov}\left\{ \E(Y_t^2|\mathcal{F}_{t-1}), \E(Y_s^2|\mathcal{F}_{s-1})\right\} \\
\nonumber
&=&\mathrm{Cov}\left\{( \sum_{i_1=1}^{t-1} \bar{a}_{t,i_1} )^2\E( \sum_{j_1=t+1}^n \bar{a}_{t,j_1} )^2 +
2( \sum_{i_1=1}^{t-1} \bar{a}_{t,i_1} )\E\left\{( \sum_{j_1=t+1}^n \bar{a}_{t,j_1} )Y_{t2}\right\}, \right. \\
\nonumber
&&
~~~~~~\left.( \sum_{i_1=1}^{s-1} \bar{a}_{s,i_1} )^2\E( \sum_{j_1=s+1}^n \bar{a}_{s,j_1} )^2 +
2( \sum_{i_1=1}^{s-1} \bar{a}_{s,i_1} )\E\left\{( \sum_{j_1=s+1}^n \bar{a}_{s,j_1} )Y_{s2}\right\} \right\}\\
\label{le1-4a}
& = & 0.
\end{eqnarray}

Step 4(b). We derive the upper bound of $\mathrm{Cov}( \E(Y_t^2|\mathcal{F}_{t-1}) \E ( Z_s^2 | \mathcal{F}_{s-1} ) )$.
Let
\begin{equation}
\label{definition-eta-t}
\eta_t = \E( \sum_{j_1=t+1}^n \bar{a}_{t,j_1} )^2 \le \frac{ (n-1-t) }{ c_n }.
\end{equation}
Then, for $t<s$,
\begin{eqnarray*}
& &  \mathrm{Cov}\left\{( \sum_{i_1=1}^{t-1} \bar{a}_{t,i_1} )^2\eta_t,
( \sum_{i_1=1}^{s-1} \bar{a}_{s,i_1} )^2\eta_s \right\} \\
& = & \eta_t \eta_s \sum_{i_1=1}^{t-1} \sum_{i_2=1}^{t-1} \sum_{j_1=1}^{s-1} \sum_{j_2=1}^{s-1} \mathrm{Cov}( \bar{a}_{t,i_1} \bar{a}_{t,i_2},
\bar{a}_{s,j_1} \bar{a}_{s,j_2} ) \\
& = & \eta_t \eta_s \sum_{i_1=1}^{t-1} \sum_{i_2=1}^{t-1} \sum_{j_1=1}^{t-1} \sum_{j_2=1}^{t-1} \mathrm{Cov}( \bar{a}_{t,i_1} \bar{a}_{t,i_2},
\bar{a}_{s,j_1} \bar{a}_{s,j_2} )
\end{eqnarray*}
Because $\bar{a}_{t,i_1}, i_1=1,\ldots, t-1$ are independent of $\bar{a}_{s,j_1}, j_1=1,\ldots, s-1$, we have
\begin{equation*}
\mathrm{Cov}\left\{( \sum_{i_1=1}^{t-1} \bar{a}_{t,i_1} )^2\eta_t,
( \sum_{i_1=1}^{s-1} \bar{a}_{s,i_1} )^2\eta_s \right\} =0.
\end{equation*}
Recall that in \eqref{condition-expe-Yt2}, we show
\[
\E(Y_t^2|\mathcal{F}_{t-1})=( \sum_{i_1=1}^{t-1} \bar{a}_{t,i_1} )^2\E( \sum_{j_1=t+1}^n \bar{a}_{t,j_1} )^2 +
2( \sum_{i_1=1}^{t-1} \bar{a}_{t,i_1} )\E\left\{( \sum_{j_1=t+1}^n \bar{a}_{t,j_1} )Y_{t2}\right\} + \E Y_{t2}^2,
\]
and in \eqref{eq-conditional-Zt2}, we show
\begin{eqnarray*}
 \E \left( Z_s^2 | \mathcal{F}_{s-1} \right) =  \sum_{i_1=s+1}^r \sum_{i_2=1}^{s-1}\sum_{j_2=1}^{s-1} \frac{ \bar{a}_{i_2,i_1}\bar{a}_{j_2,i_1} }{ v_{i_1,i_1}^2 } \E (\bar{a}_{s,i_1}^2 ).
\end{eqnarray*}
It follows that
\begin{equation}\label{le1-4b}
\mathrm{Cov}\left\{ \E(Y_t^2|\mathcal{F}_{t-1}), \E( Z_s^2|\mathcal{F}_{s-1}) \right\}=0
\end{equation}

Step 4(c). We derive the upper bound of $\mathrm{Cov}\left\{ \E(Y_t^2|\mathcal{F}_{t-1}), \E ( Y_sZ_s | \mathcal{F}_{s-1} ) \right\}$.
Recall that
\begin{eqnarray*}
\E ( Y_sZ_s | \mathcal{F}_{s-1} ) &=& \sum_{i_1=1}^{s-1}\sum_{j_1=s+1}^n \sum_{i_3=s+1}^r\sum_{i_4=1}^{(s-1)} \frac{ \bar{a}_{s,i_1} \bar{a}_{i_4, i_3} }{ v_{i_3,i_3} } \E \bar{a}_{s,j_1}\bar{a}_{s, i_3} \\
&&+ \sum_{i_1=s+1}^{n-1} \sum_{j_1=i_1+1}^n \sum_{i_3=s+1}^r\sum_{i_4=1}^{(s-1)} \frac{  \bar{a}_{i_4, i_3} }{ v_{i_3,i_3} } \E \bar{a}_{s,i_1}\bar{a}_{s,j_1}\bar{a}_{s, i_3}.
\end{eqnarray*}
It follows that
\begin{equation}\label{le1-4c}
\mathrm{Cov}\left\{ \E(Y_t^2|\mathcal{F}_{t-1}), \E ( Y_sZ_s | \mathcal{F}_{s-1} ) \right\}=0.
\end{equation}

Step 4(d). We derive the upper bound of $\mathrm{Cov}(\E(Z_t^2|\mathcal{F}_{t-1}),
\E(Y_s^2|\mathcal{F}_{s-1}) )$.
Let
\[
\eta_{t2}=\E\left\{( \sum_{j_1=t+1}^n \bar{a}_{t,j_1} )Y_{t2}\right\} = \sum_{j_1=t+1}^n \sum_{i_2=t+1}^{n-1} \sum_{j_2=i_2+1}^n \E \bar{a}_{t,j_1} \bar{a}_{t,i_2} \bar{a}_{t,j_2} .
\]
Because $i_2\neq j_2$, $\E \bar{a}_{t,j_1} \bar{a}_{t,i_2} \bar{a}_{t,j_2}$ must be equal to zero. (Otherwise, $i_2=j_2$ if $\E \bar{a}_{t,j_1} \bar{a}_{t,i_2} \bar{a}_{t,j_2}\neq 0$.)
Therefore,
\begin{equation}\label{def-eta-t2}
\eta_{t2}=0.
\end{equation}
Recalling the definition of $\eta_t$ in \eqref{definition-eta-t},   we have
\begin{eqnarray}
\nonumber
 & & \mathrm{Cov}\left(\E [Z_t^2|\mathcal{F}_{t-1}], \E(Y_s^2|\mathcal{F}_{s-1}) \right) \\
\nonumber
& = &  \mathrm{Cov}\left( \sum_{i_1=t+1}^r \sum_{i_2=1}^{t-1}\sum_{j_2=1}^{t-1} \frac{  \bar{a}_{i_2,i_1}\bar{a}_{j_2,i_1} }{ v_{i_1,i_1}^2 } \E (\bar{a}_{t,i_1}^2), ( \sum_{i_1=1}^{s-1} \bar{a}_{s,i_1} )^2\eta_s +
2( \sum_{i_1=1}^{s-1} \bar{a}_{s,i_1} )\eta_{s2}  \right) \\
\nonumber
& = & \sum_{i_1=t+1}^r \sum_{i_2=1}^{t-1}\sum_{j_2=1}^{t-1} \sum_{i_3=1}^{s-1} \frac{1}{ v_{i_1,i_1}^2 } \left( \sum_{i_4=1}^{s-1} \mathrm{Cov}( v_{t,i_1} \bar{a}_{i_2,i_1}\bar{a}_{j_2,i_1}, \eta_s \bar{a}_{s,i_3}\bar{a}_{s,i_4})
\right) \\
%\nonumber
%&&~~~~+ \left. 2\eta_{s2} \mathrm{Cov}(  v_{t,i_1} \bar{a}_{i_2,i_1}\bar{a}_{j_2,i_1}, \bar{a}_{s,i_3})  \right)\\
\nonumber
& \lesssim & \frac{(n-t)b_n^2}{n^2 c_n^2} \Big( \sum_{i_2,j_2,i_3,i_4=1}^{s-1} | \mathrm{Cov}( \bar{a}_{s,i_2} \bar{a}_{s,j_2}, \bar{a}_{s,i_3}\bar{a}_{s,i_4} ) |  \Big) \\
\label{le1-4d}
& \lesssim &  \frac{ (n-t)(s-1)^2 b_n^2 }{ n^2 c_n^4 }.
\end{eqnarray}

Step 4(e). We derive the upper bound of $\mathrm{Cov}\left(\E(Z_t^2|\mathcal{F}_{t-1}), \E( Z_s^2|\mathcal{F}_{s-1}) \right)$ and \\
 $\mathrm{Cov}\left( \E(Z_t^2|\mathcal{F}_{t-1}), \E(Y_sZ_s|\mathcal{F}_{s-1} ) \right)$.
Note that $t<s$. Then we have
\begin{eqnarray}
\nonumber
&&\mathrm{Cov}\left(\E(Z_t^2|\mathcal{F}_{t-1}), \E( Z_s^2|\mathcal{F}_{s-1}) \right) \\
\nonumber
& = & \mathrm{Cov}\left( \sum_{i_1=t+1}^r \sum_{i_2=1}^{t-1}\sum_{j_2=1}^{t-1} \frac{ v_{t,i_1} }{ v_{i_1,i_1}^2 } \bar{a}_{i_2,i_1}\bar{a}_{j_2,i_1},
\sum_{i_3=s+1}^r \sum_{i_4=1}^{s-1}\sum_{j_3=1}^{s-1} \frac{ v_{s,i_3} }{ v_{i_3,i_3}^2 } \bar{a}_{i_4,i_3}\bar{a}_{j_3,i_3} \right) \\
\nonumber
& = & \sum_{i_1=t+1}^r \sum_{i_2=1}^{t-1}\sum_{j_2=1}^{t-1}\sum_{i_3=s+1}^r \sum_{i_4=1}^{s-1}\sum_{j_3=1}^{s-1} \frac{ v_{s,i_3}v_{t,i_1} }{ v_{i_1,i_1}^2 v_{i_3,i_3}^2 } \mathrm{Cov}
(\bar{a}_{i_1, i_2}\bar{a}_{i_1, j_2},  \bar{a}_{i_3, i_4}\bar{a}_{i_3, j_3}) \\
\nonumber
& = &  \sum_{i_2=1}^{t-1}\sum_{j_2=1}^{t-1}\sum_{i_3=s+1}^r \sum_{i_4=1}^{t-1}\sum_{j_3=1}^{t-1} \frac{ v_{s,i_3}v_{t,i_3} }{ v_{i_3,i_3}^4 } \mathrm{Cov}
(\bar{a}_{i_3, i_2}\bar{a}_{i_3, j_2},  \bar{a}_{i_3, i_4}\bar{a}_{i_3, j_3}) \\
\label{le1-4e}
& \lesssim & \frac{ (t-1)^2(r-s)b_n^4 }{ n^4c_n^4}.
\end{eqnarray}
Recall that
\begin{eqnarray*}
\E ( Y_sZ_s | \mathcal{F}_{s-1} ) &=& \sum_{i_1=1}^{s-1}\sum_{j_1=s+1}^n \sum_{i_3=s+1}^r\sum_{i_4=1}^{s-1} \frac{ \bar{a}_{s,i_1} \bar{a}_{i_4, i_3} }{ v_{i_3,i_3} } \E \bar{a}_{s,j_1}\bar{a}_{s, i_3} \\
&&+ \sum_{i_1=s+1}^{n-1} \sum_{j_1=i_1+1}^n \sum_{i_3=s+1}^r\sum_{i_4=1}^{s-1} \frac{ \bar{a}_{i_4, i_3} }{ v_{i_3,i_3} }  \E \bar{a}_{s,i_1}\bar{a}_{s,j_1}\bar{a}_{s, i_3}.
\end{eqnarray*}
Because $i_1<j_1$, we have $\E \bar{a}_{s,i_1}\bar{a}_{s,j_1}\bar{a}_{s, i_3}=0$. This leads to
\begin{eqnarray}
\nonumber
\E ( Y_sZ_s | \mathcal{F}_{s-1} ) & = & \sum_{i_1=1}^{s-1}\sum_{j_1=s+1}^n \sum_{i_3=s+1}^r\sum_{i_4=1}^{s-1} \frac{ \bar{a}_{s,i_1} \bar{a}_{i_4, i_3} }{ v_{i_3,i_3} } \E \bar{a}_{s,j_1}\bar{a}_{s, i_3},
\\
\label{eq-YsZs-condi}
& = & \sum_{i_1=1}^{s-1} \sum_{i_3=s+1}^r \sum_{i_4=1}^{s-1} \frac{  \bar{a}_{s,i_1} \bar{a}_{i_4, i_3} }{ v_{i_3,i_3} } \E \bar{a}_{s, i_3}^2.
\end{eqnarray}
It follows that
\begin{equation}\label{le1-4ee}
\mathrm{Cov}\left( \E(Z_t^2|\mathcal{F}_{t-1}), \E(Y_sZ_s|\mathcal{F}_{s-1} ) \right) =0.
\end{equation}
\iffalse
\[
\mathrm{Cov}\left( \sum_{i_1=t+1}^r \sum_{i_2=1}^{t-1}\sum_{j_2=1}^{t-1} v_{t,i_1}\bar{a}_{i_2,i_1}\bar{a}_{j_2,i_1},
\sum_{i_3=1}^{s-1}\sum_{j_3=s+1}^n \sum_{i_4=s+1}^r\sum_{i_5=1}^{s-1} \bar{a}_{s,i_3} \bar{a}_{i_5, i_4}\E \bar{a}_{s,j_3}\bar{a}_{s, i_4} \right)
\]
$\E \bar{a}_{i_1,i_2} \bar{a}_{i_1,j_2} \bar{a}_{s,i_3} \bar{a}_{i_4,i_5}$ must be zero because $\bar{a}_{s,i_3}$ is independent of $\bar{a}_{i_4,i_5}$.
%%%%%%%%%%%%  martingale central limit theorem
Note that $t<s$.
\fi

Step 4(f). We derive the upper bound of
\[
\mathrm{Cov}\left\{ \E(Y_tZ_t|\mathcal{F}_{t-1}),
\E(Y_s^2|\mathcal{F}_{s-1}) + \E( Z_s^2|\mathcal{F}_{s-1}) + \E(Y_sZ_s|\mathcal{F}_{s-1} ) \right\}.
\]
Recall that in \eqref{eq-YsZs-condi}, we show
\begin{eqnarray*}
\E ( Y_tZ_t | \mathcal{F}_{t-1} ) &=& \sum_{i_1=1}^{t-1} \sum_{i_3=t+1}^r \sum_{i_4=1}^{t-1} \frac{ \bar{a}_{t,i_1} \bar{a}_{i_4, i_3} }{ v_{i_3,i_3} } \E \bar{a}_{t, i_3}^2,
\end{eqnarray*}
and, by \eqref{condition-expe-Yt2} and \eqref{def-eta-t2}, we have
\begin{eqnarray*}
\E \left(  Y_s^2 | \mathcal{F}_{s-1} \right)
& = &  ( \sum_{i_1=1}^{s-1} \bar{a}_{s,i_1} )^2\E( \sum_{j_1=s+1}^n \bar{a}_{s,j_1} )^2 %+
%2( \sum_{i_1=1}^{s-1} \bar{a}_{s,i_1} )\E\left\{( \sum_{j_1=s+1}^n \bar{a}_{s,j_1} )Y_{s2}\right\}
 + \E Y_{s2}^2.
\end{eqnarray*}
Because $\bar{a}_{t,i_1}$ ($i_1=1,\ldots, t-1$) is independent of $\bar{a}_{i_4, i_3}$ for $i_3=t+1,\ldots,r$ and $i_4=1,\ldots,t-1$, we have
\[
\E\sum_{i_3=t+1}^r \sum_{i_4=1}^{t-1} \bar{a}_{t,i_1} \bar{a}_{i_4, i_3}\E \bar{a}_{t, i_3}^2=0.
\]
It follows that
\begin{eqnarray*}
&&\mathrm{Cov}\left\{ \sum_{i_1=1}^{t-1} \sum_{i_3=t+1}^r \sum_{i_4=1}^{t-1} \bar{a}_{t,i_1} \bar{a}_{i_4, i_3}\E \bar{a}_{t, i_3}^2,
( \sum_{i_1=1}^{s-1} \bar{a}_{s,i_1} )^2\E( \sum_{j_1=s+1}^n \bar{a}_{s,j_1} )^2 \right\} \\
& = & \E \left\{ \sum_{i_1=1}^{t-1} \sum_{i_3=t+1}^r \sum_{i_4=1}^{t-1} \bar{a}_{t,i_1} \bar{a}_{i_3,i_4}\E \bar{a}_{t, i_3}^2
( \sum_{i_1=1}^{s-1} \bar{a}_{s,i_1} )^2\E( \sum_{j_1=s+1}^n \bar{a}_{s,j_1} )^2 \right\} \\
& = & 0,
\end{eqnarray*}
which is due to that $\bar{a}_{t,i_1}$ is independent of $\bar{a}_{i_3,i_4}, \bar{a}_{s,i_1}$. Therefore,
\begin{equation}\label{le1-4fa}
\mathrm{Cov}( \E ( Y_tZ_t | \mathcal{F}_{t-1} ), \E \left(  Y_s^2 | \mathcal{F}_{s-1} \right) ) =0.
\end{equation}
By observing that
\begin{eqnarray}
\nonumber
&&\mathrm{Cov}\left( \sum_{i_1=1}^{t-1} \sum_{i_3=t+1}^r \sum_{i_4=1}^{t-1} \bar{a}_{t,i_1} \bar{a}_{i_4, i_3}\E \bar{a}_{t, i_3}^2, \sum_{i_1=1}^{s-1} \sum_{i_3=s+1}^r \sum_{i_4=1}^{s-1} \bar{a}_{s,i_1} \bar{a}_{i_4, i_3}\E \bar{a}_{s, i_3}^2 \right) \\
\nonumber
& = &\mathrm{Cov}\left( \sum_{i_1=1}^{s-1} \sum_{i_3=s+1}^r \sum_{i_4=1}^{s-1} \bar{a}_{t,i_1} \bar{a}_{i_3, i_4}\E \bar{a}_{t, i_3}^2,
\sum_{j_1=1}^{s-1} \sum_{j_3=s+1}^r \sum_{j_4=1}^{s-1} \bar{a}_{s,j_1} \bar{a}_{j_3, j_4}
\E \bar{a}_{s, j_3}^2 \right) \\
\nonumber
& = & \E \left\{ \sum_{i_1=1}^{s-1}\sum_{j_1=1}^{s-1}\bar{a}_{t,i_1}\bar{a}_{s,j_1} \left( \sum_{i_3=s+1}^r \sum_{i_4=1}^{s-1}  \bar{a}_{i_3, i_4}\E \bar{a}_{t, i_3}^2\right) \cdot
\left(\sum_{j_3=s+1}^r \sum_{j_4=1}^{s-1}  \bar{a}_{j_3, j_4} \E \bar{a}_{s, j_3}^2\right) \right\}\\
\nonumber
&=&0,
\end{eqnarray}
similarly, we have
\begin{eqnarray}
\label{le1-4fb}
\mathrm{Cov}( \E ( Y_tZ_t | \mathcal{F}_{t-1} ), \E ( Y_sZ_s | \mathcal{F}_{s-1} ) ) =0,
\end{eqnarray}
and
\begin{eqnarray}
\nonumber
&&\mathrm{Cov}\left\{ \E(Y_tZ_t|\mathcal{F}_{t-1}), \E( Z_s^2|\mathcal{F}_{s-1}) \right\} \\
\label{le1-4fc}
& = & \E \sum_{i_1=1}^{t-1} \sum_{i_3=t+1}^r \sum_{i_4=1}^{t-1} \frac{ \bar{a}_{t,i_1} \bar{a}_{i_4, i_3} }{ v_{i_3,i_3} } \E \bar{a}_{t, i_3}^2
\sum_{i_5=s+1}^r \sum_{i_6=1}^{s-1}\sum_{j_2=1}^{s-1} \frac{ \bar{a}_{i_6,i_5}\bar{a}_{j_2,i_5} }{ v_{i_5,i_5}^2 } \E (\bar{a}_{s,i_5}^2 )=0.
\end{eqnarray}

By combining \eqref{le1-4a}--\eqref{le1-4fc}, $H$ in \eqref{eq-condition-YZ2} can be bounded above by
\[
H \lesssim \frac{ b_n^4}{r^2n^4} \sum_{t,s=1,t\neq s}^r \left(\frac{ (n-t)(s-1)^2 }{c_n^4 } + \frac{ (t-1)^2(r-s) }{ c_n^4 }\right)
\lesssim \frac{ b_n^4 }{ nc_n^4} \to 0.
\]
This shows \eqref{eq-condition-YZ2}.
\end{proof}

\section{Bernstein's inequality and Martingale central limit theorem}
\label{section-bernstein}

This section collects a user-friendly version of the Bernstein inequality on bounded random variables and a martingale central limit theorem.
The following Bernstein inequality can be easily found in textbooks such as \cite{Boucheron2013book}. The proof is omitted.

\begin{lemma}[Bernstein's inequality]\label{lemma:bernstein}
Suppose $n$ independent random variables $x_{i}$ ($1\le i \le n$)
each satisfying $\left|x_{i}\right|\leq B$. For any $a\geq2$, one
has
\[
\left|\sum_{i=1}^{n}x_{i}-\E\left[\sum_{i=1}^{n}x_{i}\right]\right|\leq\sqrt{2a\log n\sum_{i=1}^{n}\E\left[x_{i}^{2}\right]}+\frac{2a}{3}B\log n
\]
with probability at least $1-2n^{-a}$.
\end{lemma}

Next, we present the martingale central limit theorem by \cite{Brown1971}.

\begin{lemma}[\cite{Brown1971}]\label{lemma-martingale-clt}
Let $\{ S_n, \mathcal{F}_n, n=1, 2, \ldots \}$ be a martingale on the probability space $\{\Omega, \mathcal{F}, \P\}$, with $S_0=0$, and $X_n = S_n - S_{n-1}$, $n=1,2,\ldots$.
Define
\[
\sigma_n^2 = \E( X_n^2|\mathcal{F}_{n-1}), \quad V_n^2 = \sum_{j=1}^n \sigma_j^2, \quad s_n^2 = \E V_n^2 = \E S_n^2.
\]
If the condition
\begin{equation*}
\frac{ V_n^2 }{s_n^2} \stackrel{p.}{\to} 1,\quad, n\to\infty,
\end{equation*}
and the Lindeberg condition
\begin{equation*}
\frac{1}{s_n^2} \sum_{j=1}^n \E [ X_j^2 1(|X_j| \ge \epsilon s_n ) ] \stackrel{p.}{\to} 0, \quad, n\to\infty
\end{equation*}
hold, then $S_n/s_n$ converges in distribution to the standard normal distribution as $n\to\infty$.
\end{lemma}

\section{Figures}
\label{sec-figure}
The last 5 pages show the QQ-plots in the simulation section for other $n$ in the $\beta$-model and those plots in the Bradley--Terry model.

\newpage

\begin{figure}[htbp]
\centering
%\captionstyle{flushleft}
%\onelinecaptionsfalse
\caption{QQ plots for the $\beta$-model ($n=200$). The horizontal and vertical axes in each QQ-plot are the respective theoretical (based on the normal or chi-square distribution) and empirical quantiles.
The straight lines correspond to $y=x$. The first, second, and third columns correspond to $L_n=0.2\log n, 0.4\log n, 0.5\log n$, respectively.
The QQ plots for $L_n=0$ are very similar to those for $L_n=0.2\log n$ and thus are not shown.
Furthermore, the QQ plots under the fixed dimensional null $H_{03}$ and $H_{04}$ are similar and only those under $H_{03}$ are shown here to save space. }
\subfigure[QQ-plot for normalized log-likelihood ratio statistic in \eqref{statistics-beta} under $H_{01}$]{\includegraphics[width=0.9\textwidth]{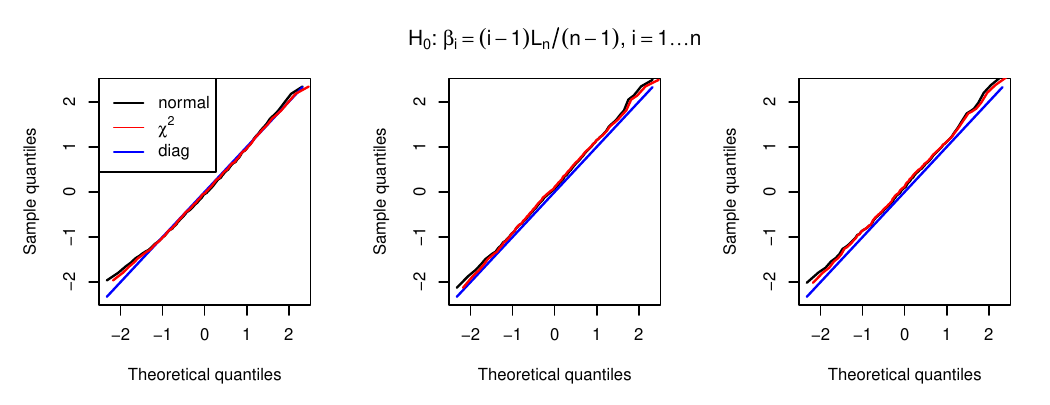}}
\subfigure[QQ-plot for normalized log-likelihood ratio statistic in \eqref{statistics-beta} under $H_{02}$]{\includegraphics[width=0.9\textwidth]{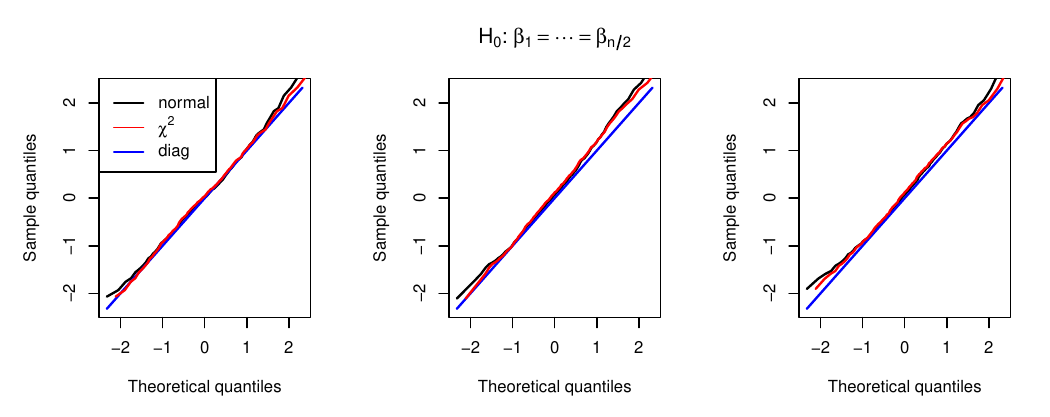}}
\subfigure[QQ-plot for log-likelihood ratio statistic under $H_{03}$]{\includegraphics[width=0.9\textwidth]{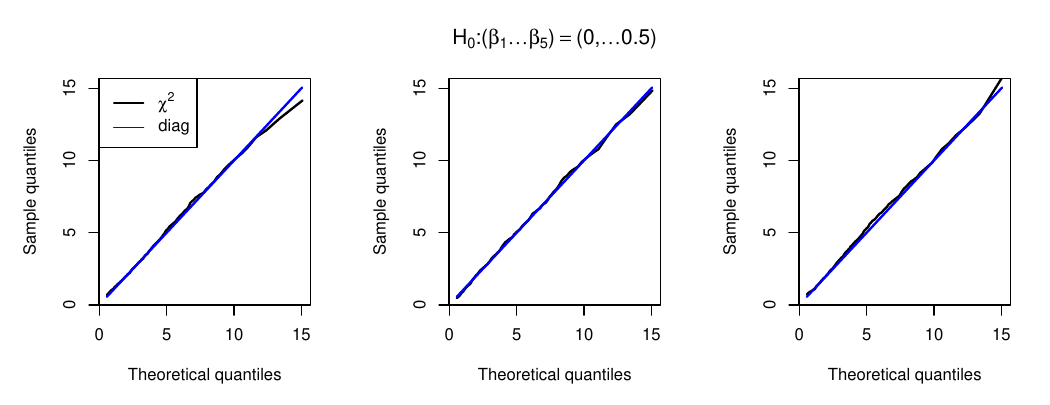}}
%\subfigure[QQ-plot for log-likelihood ratio statistic under $H_{04}$]{\includegraphics[width=0.9\textwidth]{beta-com-200-5.eps}}
\label{fig:beta}     %% label for entire figure
\end{figure}

\begin{figure}[htbp]
\centering
%\captionstyle{flushleft}
%\onelinecaptionsfalse
\caption{QQ plots for the $\beta$-model ($n=500$). The horizontal and vertical axes in each QQ-plot are the respective theoretical (based on the normal or chi-square distribution) and empirical quantiles.
The straight lines correspond to $y=x$. The first, second, and third columns correspond to $L_n=0.2\log n, 0.4\log n, 0.5\log n$, respectively.
The QQ plots for $L_n=0$ are very similar to those for $L_n=0.2\log n$ and thus are not shown.
Furthermore, the QQ plots under the fixed dimensional null $H_{03}$ and $H_{04}$ are similar and only those under $H_{03}$ are shown here to save space. }
\subfigure[QQ-plot for normalized log-likelihood ratio statistic in \eqref{statistics-beta} under $H_{01}$]{\includegraphics[width=0.9\textwidth]{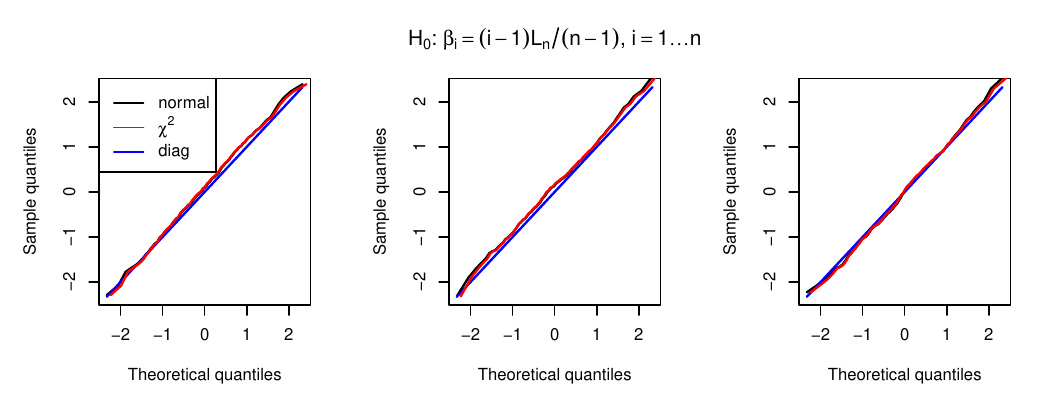}}
\subfigure[QQ-plot for normalized log-likelihood ratio statistic in \eqref{statistics-beta} under $H_{02}$]{\includegraphics[width=0.9\textwidth]{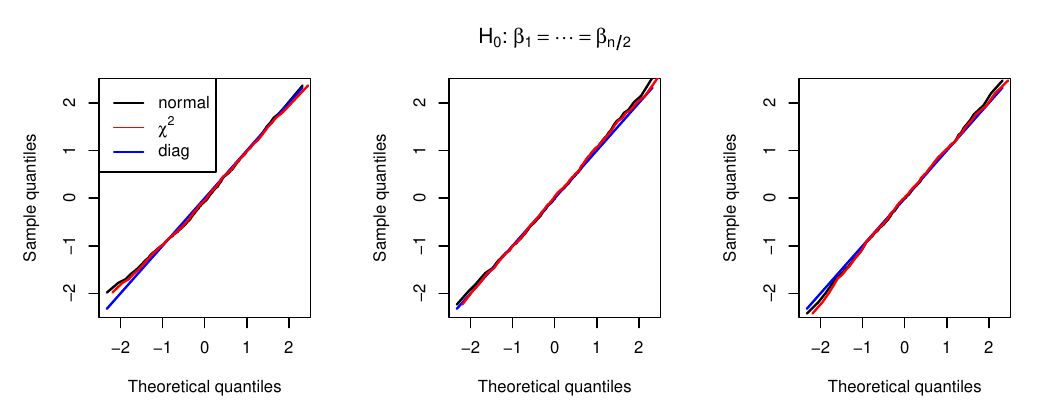}}
\subfigure[QQ-plot for log-likelihood ratio statistic under $H_{03}$]{\includegraphics[width=0.9\textwidth]{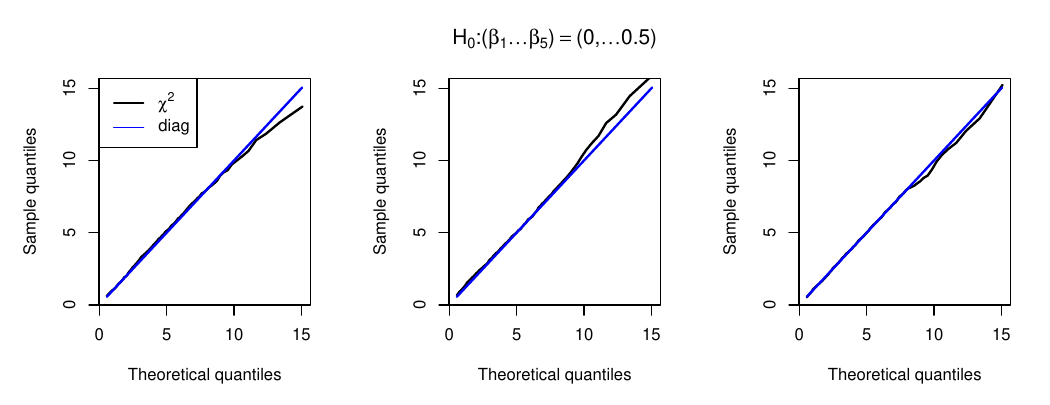}}
\label{fig:beta}     %% label for entire figure
\end{figure}

\begin{figure}[htbp]
\centering
%\captionstyle{flushleft}
%\onelinecaptionsfalse
\caption{QQ plots for the Bradley--Terry model ($n=200$). The horizontal and vertical axes in each QQ-plot are the respective theoretical (based on the normal or chi-square distribution) and empirical quantiles.
The straight lines correspond to $y=x$. The first, second, and third columns correspond to $L_n=0.2\log n, 0.4\log n, 0.5\log n$, respectively.
The QQ plots for $L_n=0$ are very similar to those for $L_n=0.2\log n$ and thus are not shown.
 }
\subfigure[QQ-plot for normalized log-likelihood ratio statistic in \eqref{statistics-beta} under $H_{01}$]{\includegraphics[width=0.9\textwidth]{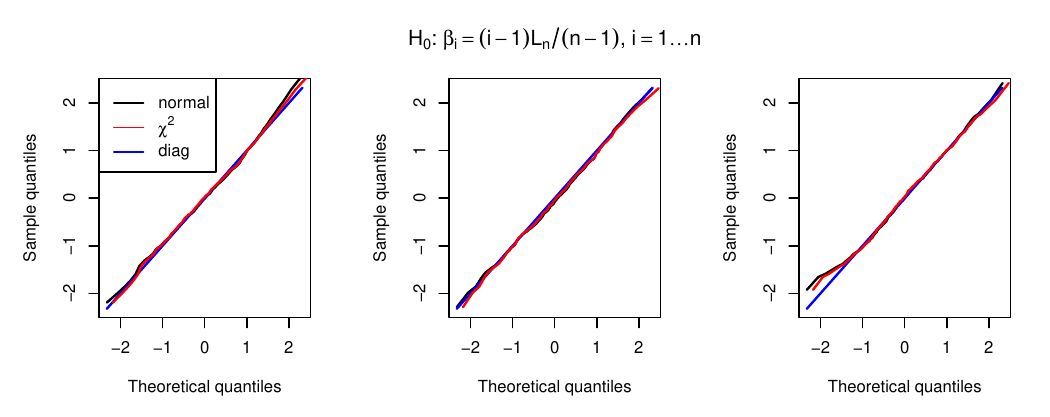}}
\subfigure[QQ-plot for normalized log-likelihood ratio statistic in \eqref{statistics-beta} under $H_{02}$]{\includegraphics[width=0.9\textwidth]{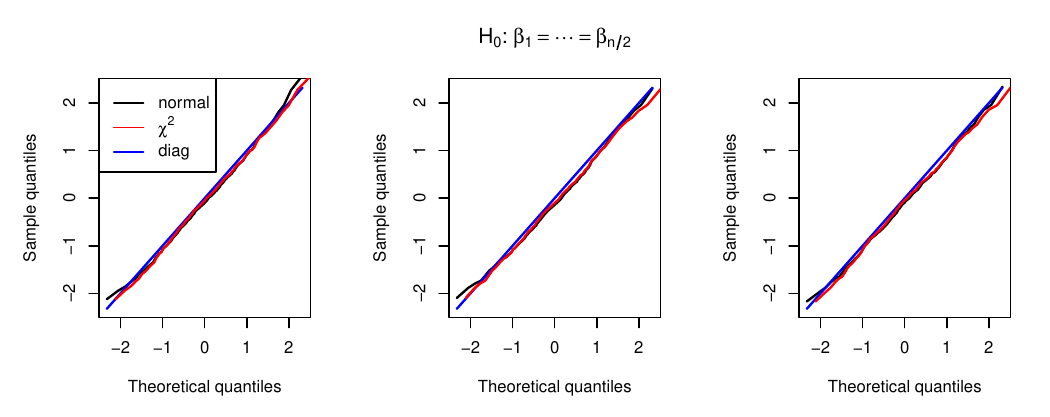}}
\subfigure[QQ-plot for log-likelihood ratio statistic under $H_{04}$]{\includegraphics[width=0.9\textwidth]{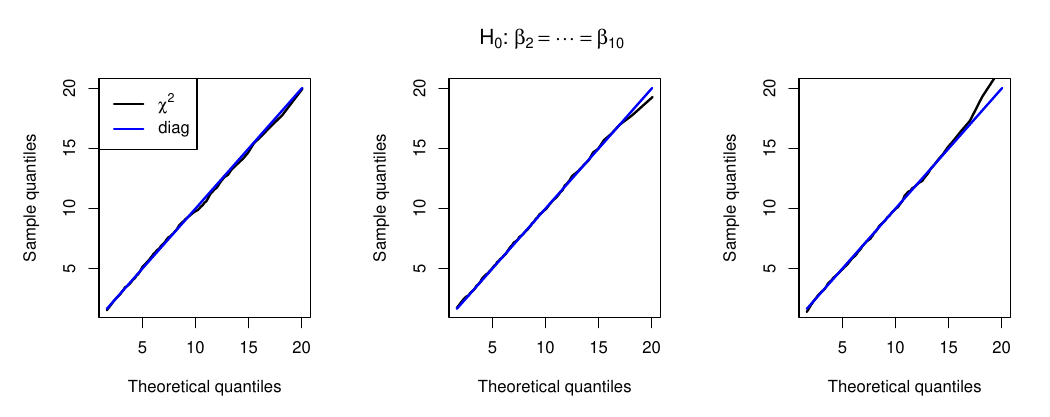}}
\label{fig:beta}     %% label for entire figure
\end{figure}

\begin{figure}[htbp]
\centering
%\captionstyle{flushleft}
%\onelinecaptionsfalse
\caption{QQ plots for the Bradley--Terry model ($n=500$). The horizontal and vertical axes in each QQ-plot are the respective theoretical (based on the normal or chi-square distribution) and empirical quantiles.
The straight lines correspond to $y=x$. The first, second, and third columns correspond to $L_n=0.2\log n, 0.4\log n, 0.5\log n$, respectively.
The QQ plots for $L_n=0$ are very similar to those for $L_n=0.2\log n$ and thus are not shown.
}
\subfigure[QQ-plot for normalized log-likelihood ratio statistic in \eqref{statistics-beta} under $H_{01}$]{\includegraphics[width=0.9\textwidth]{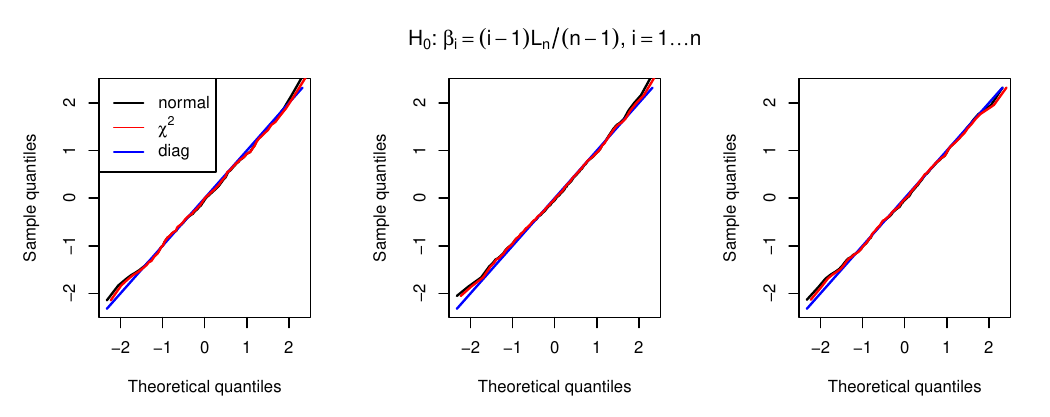}}
\subfigure[QQ-plot for normalized log-likelihood ratio statistic in \eqref{statistics-beta} under $H_{02}$]{\includegraphics[width=0.9\textwidth]{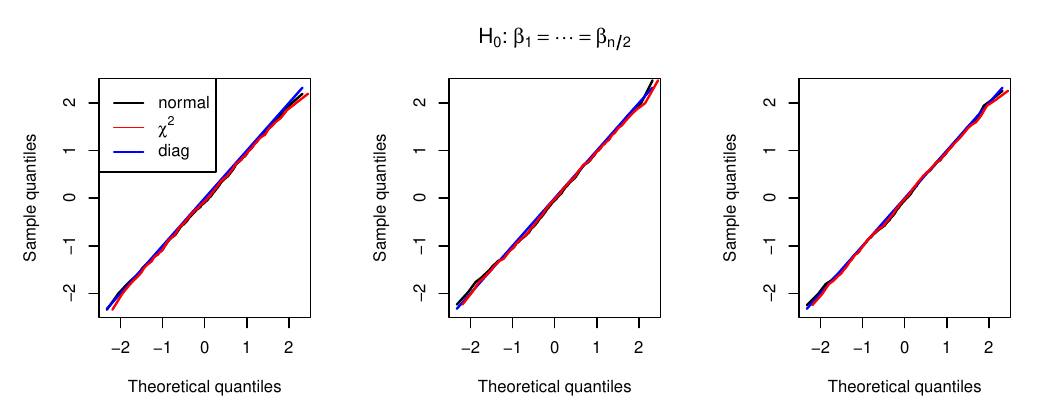}}
\subfigure[QQ-plot for log-likelihood ratio statistic under $H_{04}$]{\includegraphics[width=0.9\textwidth]{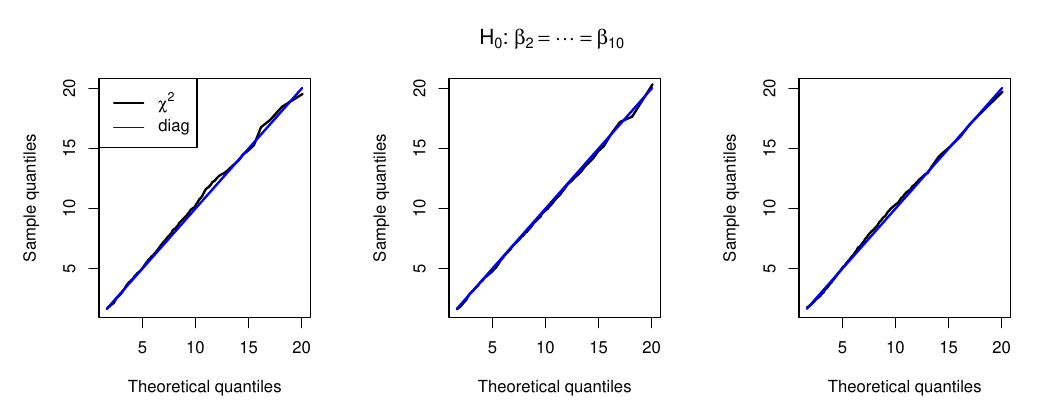}}
\label{fig:beta}     %% label for entire figure
\end{figure}

\begin{figure}[htbp]
\centering
%\captionstyle{flushleft}
%\onelinecaptionsfalse
\caption{QQ plots for the Bradley--Terry model ($n=1000$). The horizontal and vertical axes in each QQ-plot are the respective theoretical (based on the normal or chi-square distribution) and empirical quantiles.
The straight lines correspond to $y=x$. The first, second, and third columns correspond to $L_n=0.2\log n, 0.4\log n, 0.5\log n$, respectively.
The QQ plots for $L_n=0$ are very similar to those for $L_n=0.2\log n$ and thus are not shown.
 }
\subfigure[QQ-plot for normalized log-likelihood ratio statistic in \eqref{statistics-beta} under $H_{01}$]{\includegraphics[width=0.9\textwidth]{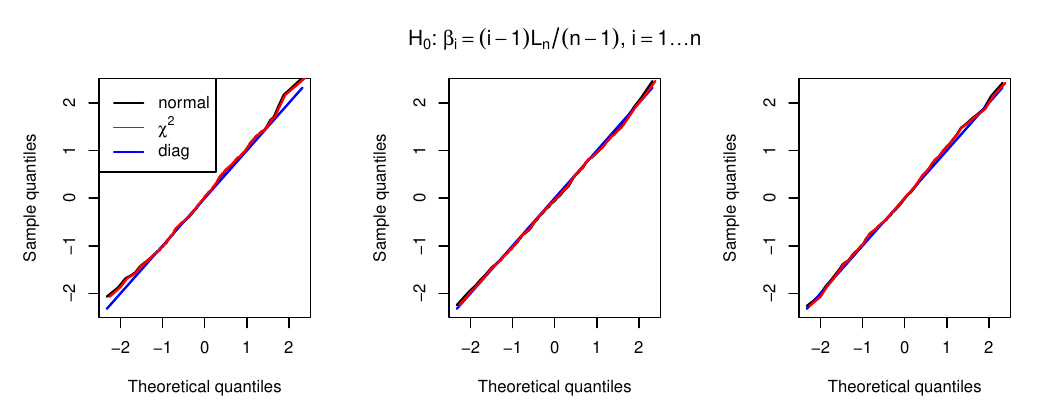}}
\subfigure[QQ-plot for normalized log-likelihood ratio statistic in \eqref{statistics-beta} under $H_{02}$]{\includegraphics[width=0.9\textwidth]{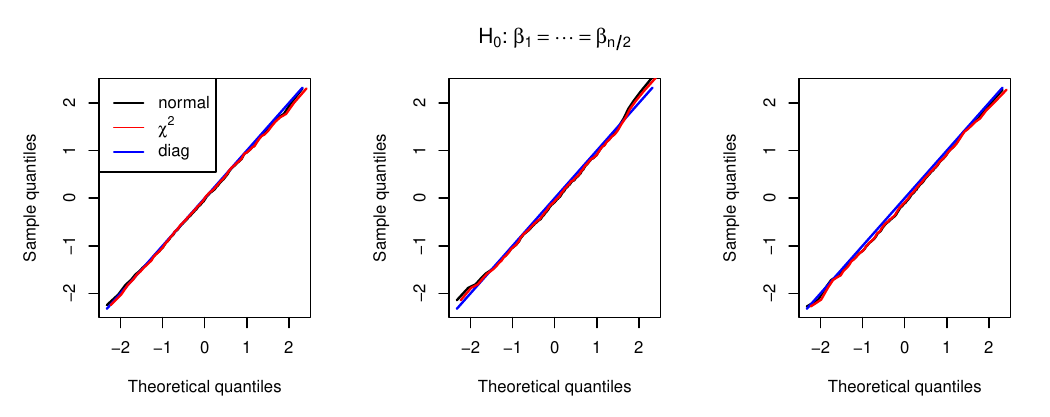}}
\subfigure[QQ-plot for log-likelihood ratio statistic under $H_{04}$]{\includegraphics[width=0.9\textwidth]{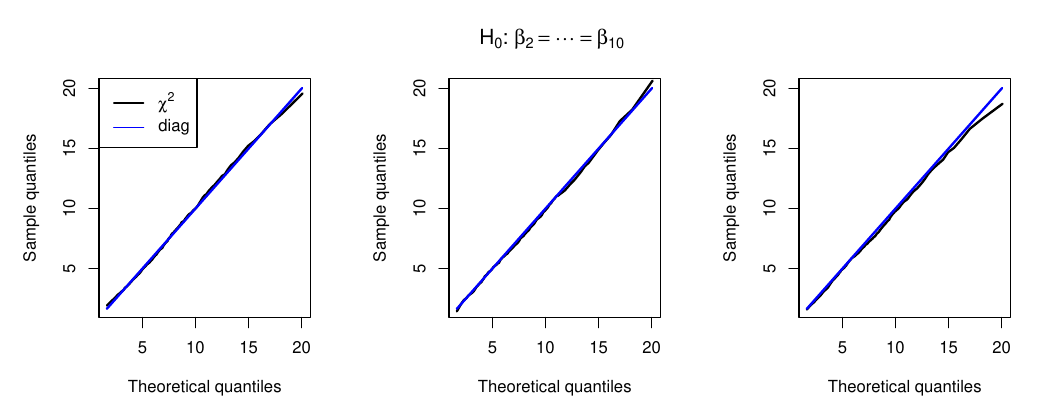}}
\label{fig:beta}     %% label for entire figure
\end{figure}

\end{document}